\documentclass[letterpaper,reqno,10pt,twoside]{amsart}
\usepackage{amsmath,amsthm,amsfonts,amssymb,euscript,mathrsfs,graphics,color,latexsym,marginnote}
\usepackage{stmaryrd}
\usepackage[dvips]{graphicx}
\usepackage[left=1 in, right=1 in,top=1 in, bottom=1 in]{geometry}
\usepackage{hyperref}
\usepackage[toc,page]{appendix}
\usepackage{relsize}
\usepackage[shortlabels]{enumitem}
\usepackage[all]{xy}
\usepackage{float}

\usepackage{xargs}
\usepackage[dvipsnames]{xcolor}

\usepackage{hyperref}
\hypersetup{colorlinks=true, pdfstartview=FitV,linkcolor=blue!70!black,citecolor=red!70!black, urlcolor=green!60!black}
\definecolor{labelkey}{rgb}{0.6,0,0}

\usepackage[colorinlistoftodos,prependcaption,textsize=small]{todonotes}
\newcommandx{\change}[2][1=]{\todo[#1]{#2}}
\newcommandx{\unsure}[2][1=]{\todo[linecolor=red,backgroundcolor=red!25,bordercolor=red,#1]{#2}}
\newcommandx{\rmk}[2][1=]{\todo[linecolor=blue,backgroundcolor=blue!25,bordercolor=blue,#1]{#2}}
\newcommandx{\info}[2][1=]{\todo[linecolor=OliveGreen,backgroundcolor=OliveGreen!25,bordercolor=OliveGreen,#1]{#2}}
\newcommandx{\improvement}[2][1=]{\todo[linecolor=Plum,backgroundcolor=Plum!25,bordercolor=Plum,#1]{#2}}
\newcommandx{\thiswillnotshow}[2][1=]{\todo[disable,#1]{#2}}

\makeatletter
\renewcommand \theequation {%
\ifnum \c@section>\z@ \@arabic\c@section.%
\fi\@arabic\c@equation} \@addtoreset{equation}{section}
\@namedef{subjclassname@2020}{2020 Mathematics Subject Classification}
\makeatother

\newtheorem{theorem}{Theorem}[section]

\newtheorem{lemma}[theorem]{Lemma}

\theoremstyle{definition}

\theoremstyle{remark}

\def\XXint#1#2#3{{\setbox0=\hbox{$#1{#2#3}{\int}$ }
\vcenter{\hbox{$#2#3$ }}\kern-.6\wd0}}

\def\p{\partial}

\def\Om{\Omega}

\def\dive{\mathop{\rm div}\nolimits}

\def\d{\partial}

\def\b{\mathcal{B}}

\def\b{\mathcal{B}}

\title{GLOBAL DYNAMIC STABILITY OF CONTACT LINES IN FLUIDS: 2-D DROPLET PROBLEM}
\author{Xiaoding Yang}

\begin{document}
\maketitle
\begin{abstract}
    In this paper, we investigate the dynamics of an incompressible viscous Navier–Stokes fluid evolving above a one-dimensional flat surface. The fluid is subject to a uniform gravitational field and capillary forces acting along the free boundary. The interface between the fluid and the surrounding air is a free surface whose motion is driven by gravity, surface tension, and the fluid velocity field. The triple-phase intersections where the fluid, the air above the vessel, and the solid vessel wall meet are referred to as contact points, and the angles formed there are called contact angles.
The model under consideration incorporates boundary conditions that allow for full motion of the contact points and dynamic contact angles. Under these conditions, \cite{Yang} established the existence of equilibrium configurations for the model. These equilibria consist of a quiescent fluid occupying a domain whose upper boundary can be represented as the graph of a function in polar coordinates, minimizing a gravity–capillary energy functional subject to a fixed mass constraint. The equilibrium contact angles may take any value in $(0,\pi)$ depending on the choice of capillary parameters.
In the present work, we develop a framework of a priori estimates for this model. We prove that, for initial data sufficiently close to equilibrium, the system admits global solutions that converge exponentially fast to a (horizontally) shifted equilibrium state.
\end{abstract}
\tableofcontents
\setcounter{tocdepth}{1}
\section{Introduction}

\begin{figure}{\label{figure1}}
\centering
\scalebox{1.0}{\begin{tikzpicture}
  \draw  (-4, 0)--(3, 0);
  \filldraw [blue!5, draw=black] (1, 0) arc(330: 570: 2);
  \node at(-0.6,1) {$\Omega$};
\end{tikzpicture}}
\caption{A droplet.}
\end{figure}
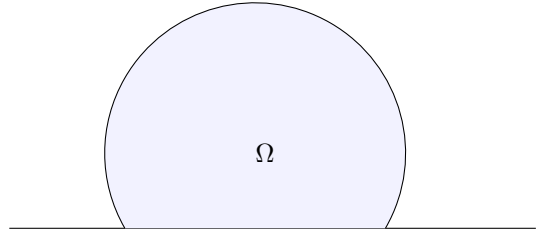
\subsection{Problem Formulation }
Consider a two-dimensional droplet of viscous incompressible fluid evolving above a one-dimensional flat surface. The interface between the droplet and the surrounding vapor cannot, in general, be represented as a graph over the horizontal Cartesian coordinate $x$ (see Figure~\ref{figure1}). To overcome this difficulty, we introduce polar coordinates and denote the spatial variable by $z = (\theta, r) \in \mathbb{R}^2$. We then assume that, for each time $t \geq 0$, the fluid occupies a time-dependent droplet domain.
\begin{align}
\Omega(t) := \{z \in  \mathbb{R}^{2}
: 0 < r < \rho(t, \theta)\},
\end{align}
where the free surface of the droplet is given by the unknown function $\rho(t,\cdot):[0,\pi]\rightarrow \mathbb{R}^{+}$. We write down the free surface at the top of the droplet as:
\begin{align}
    \Sigma(t):=\{(\theta,r):\theta\in [0,\pi], r=\rho(t,\theta)\},
\end{align}
\noindent and at the bottom as
\begin{align}
    \Sigma_{s}:=\{(\theta,r):\theta=0~\operatorname{or}~\theta=\pi\}.
\end{align}
See Figure \ref{figure1} for an example of such a fluid droplet domain. The state of the fluid at each time is determined by its velocity and pressure $(u(t,\cdot), P(t,\cdot)):\Omega(t)\rightarrow \mathbb{R}^{2}\times\mathbb{R}$, for which the associated viscous stress tensor is given by $S(P,u):\Omega(t)\rightarrow \mathbb{R}^{2\times 2}$ via
\begin{align}
    S(P,u):=PI-\mu\mathbb{D}u
\end{align}
\noindent where I is the $2\times 2$ identity, $\mu>0$ is the fluid viscosity, and the symmetrized gradient is $\mathbb{D}u=Du+(Du)^{T}$. We note that a simple computation reveals that if $\dive_{z} u=0$, then $\dive_{z}S(P,u)=-\mu\Delta_{z}u+\nabla_{z}P$.

To formulate the equations of motion, we first introduce the physical parameters governing the dynamics. The fluid is assumed to have unit density and to be subject to a uniform gravitational field acting vertically downward with strength $g>0$. The constant $\sigma>0$ denotes the surface tension coefficient along the fluid–vapor interface. Finally, $\beta>0$ represents the inverse slip length and arises in Navier’s slip boundary condition imposed on the vessel side walls.

The energetic parameters $\gamma_{sv}, \gamma_{sf} \in \mathbb{R}$ denote the free energy per unit length associated with the solid--vapor and solid--fluid interfaces, respectively, and serve as analogs of $\sigma$ for the remaining interfaces. We define
\begin{align}
    [\![\gamma]\!] := \gamma_{sv} - \gamma_{sf},
\end{align}
and assume that $[\![\gamma]\!]$ and $\sigma$ satisfy the classical Young relation
\begin{align}
    \frac{|[\![\gamma]\!]|}{\sigma} < 1.
\end{align}

Finally, we introduce the contact-point velocity response function
$\mathcal{W} : \mathbb{R} \to \mathbb{R}$ as a $C^{2}$, strictly increasing diffeomorphism satisfying $\mathcal{W}(0) = 0$.

We now present the equations of motion for the unknown triple $(u,P,\rho)$, which determine the dynamics for $t>0$:

\begin{equation}{\label{equ:fix_1}}
    \begin{cases}
        \partial_{t}u+u\cdot \nabla u+\dive S(P,u)=0~~~&\operatorname{in}~\Omega(t) \\
        \dive u=0&\operatorname{in}~\Omega(t)\\
        S(P,u)\nu=(g\rho\sin\theta-\sigma\mathcal{H}(\rho)))\nu~~&\operatorname{on}~\Sigma(t)\\
        \partial_{t}\rho=\frac{1}{\rho}u\cdot \mathcal{N}~~&\operatorname{on}~\Sigma(t)\\
        (S(P,u)\cdot \nu-\beta u)\cdot \tau=0~~&\operatorname{on}~\Sigma_{s}\\
        u\cdot \nu=0~&\operatorname{on}~\Sigma_{s}\\
       {\mathcal{W}}(\p_{t}\rho(\pi,t))=\sigma\frac{\rho'}{(\rho^{2}+\rho'^{2})^{\frac{1}{2}}}(\pi,t)+[\![\gamma]\!]\\
        {\mathcal{W}}(\partial_{t}\rho(0,t))=-(\sigma\frac{\rho'}{(\rho^{2}+\rho'^{2})^{\frac{3}{2}}}(0,t))+[\![\gamma]\!]
    \end{cases}
\end{equation}
\noindent Here, $\nu$ denotes the outward-pointing unit normal, $\tau$ the associated unit tangent, and
\begin{align}
\mathcal{H}(\rho)
= -\frac{\rho}{\sqrt{\rho^{2}+\rho'^{2}}}
+ \partial_{\theta}\!\left(\frac{\rho'}{\sqrt{\rho^{2}+\rho'^{2}}}\right)
\end{align}
is the mean curvature operator.

In \cite{Guo}, the system is formulated in Cartesian coordinates. To derive \eqref{equ:fix_1}, it is therefore necessary to reformulate the equations on the free surface $\Sigma(t)$ in polar coordinates and to establish several fundamental properties of the transformed system, including the conservation of total mass in polar coordinates. These derivations are carried out in Section~2.

Finally, we shift the gravitational potential to eliminate the atmospheric pressure by redefining the physical pressure $\bar{P}$ as
\begin{align}
P = \bar{P} + g\rho(\theta)\sin\theta - P_{\mathrm{atm}} .
\end{align}

\begin{figure}
\centering
\label{figure2}
\includegraphics[width=0.9\linewidth]{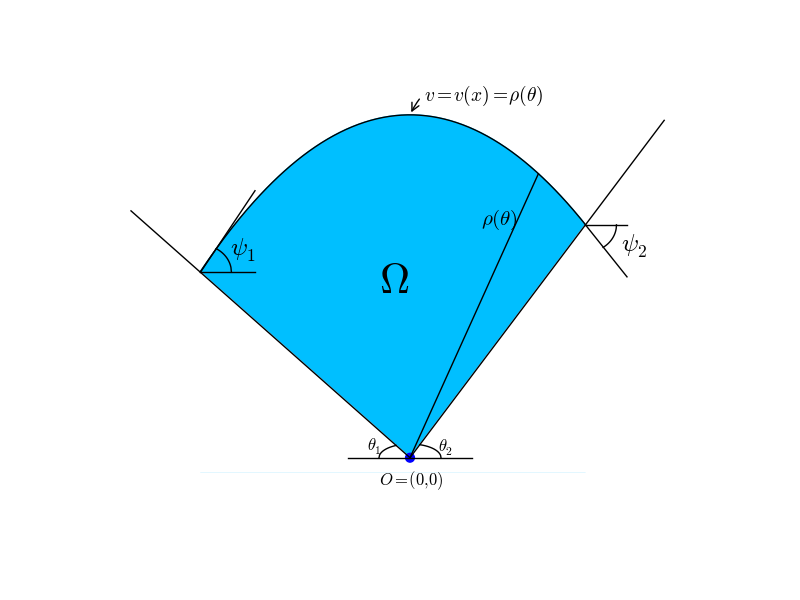}
\caption{}
\end{figure}

\subsection{Equilibrium and Minimizer} To establish a generalized linear stability result, we allow the inclination angles $\theta_{1}$ and $\theta_{2}$ to be arbitrary positive acute angles. The special case $\theta_{1}=\theta_{2}=0$ corresponds to a sessile drop; see Figure~\ref{figure2}. The existence of an equilibrium configuration under these assumptions was proved in \cite{Yang}. Moreover, the equilibrium is uniquely determined once the two contact points are prescribed. In particular, for any given constant $V>0$, there exists a smooth equilibrium profile $\rho_{0}$ and a constant pressure level $P_{0}$ such that $u=0$ and
\begin{equation}\label{equ:equi}
\begin{cases}
g\rho_{0}\sin\theta-\sigma \mathcal{H}(\rho_{0})=P_{0},\\[0.4em]
\dfrac{\rho_{0}'}{\sqrt{\rho_{0}^{2}+\rho_{0}'^{2}}}\!\left(\frac{\pi}{2}\pm\frac{\pi}{2}\right)
=\mp\dfrac{[\![\gamma]\!]}{\sigma},\\[0.8em]
\displaystyle \int_{0}^{\pi}\rho_{0}^{2}\, d\theta=V,
\qquad
P_{0}=C_{1}\bigl(V,g,[\![\gamma]\!],\sigma\bigr).
\end{cases}
\end{equation}

It is worth noting that \eqref{equ:equi} is precisely the Euler--Lagrange system associated with the energy functional $E$ defined by
\begin{align}\label{equ:E}
E(\rho)
&=\frac{1}{3}g\int_{\theta_2}^{\pi-\theta_1}\rho^{3}\sin \theta \,d\theta
+\int_{\theta_2}^{\pi-\theta_1}\sigma \sqrt{\rho^{2}+\rho'^{2}}\,d\theta
-[\![\gamma]\!]\bigl(\rho(\theta_1)+\rho(\theta_2)\bigr),
\end{align}
subject to the volume constraint
\begin{align}\label{equ:C}
\int_{\theta_{2}}^{\pi-\theta_{1}} \rho^{2}\,d\theta=V,
\end{align}
for any prescribed constant $V>0$.

A natural question arises: does the energy functional $E$ admit a minimizer under the constraint \eqref{equ:C}, and is the equilibrium constructed in \cite{Yang} such a minimizer? This issue remains open. To make progress, we state the following theorem, whose proof is presented in Sections~3 and~5.

\begin{theorem}{\label{thm:sta}}

 The following statements hold

    $\bullet$ \textbf{Existence of a $\epsilon$-regularized minimizer}
    
    For any $\epsilon>0$, there exists a minimizer $\rho_{0}^{\epsilon}$ of the following energy functional $E^{\epsilon}$
    \begin{align}
        E^{\epsilon}(\rho)=\frac{1}{3}g\int_{\theta_2}^{\pi-\theta_1}\rho^{3}\sin \theta d\theta+\int_{\theta_2}^{\pi-\theta_1}\sigma \sqrt{\rho^{2}+\rho'^{2}}d\theta-[\![\gamma]\!](\rho(\theta_1)+\rho(\theta_2))+\epsilon\int_{\theta_{2}}^{\pi-\theta_{1}}\rho'^2d\theta
    \end{align}
    \noindent subject to the constraint
    \begin{align}
        \int_{\theta_{2}}^{\pi-\theta_{1}}\rho^{2}d\theta=V
    \end{align}
    \noindent where $V$ is the same constant as in \eqref{equ:C}.

    $\bullet$ \textbf{Uniform boundedness.}
    
    When $M$ is fixed, $\rho_{0}^{\epsilon}(\theta)'$ is uniformly bounded for all $\epsilon>0$ and $\theta\in (\theta_{2},\pi-\theta_{1})$.

    $\bullet$ \textbf{Convergence and existence of the minimizer.}
   As $\epsilon \to 0$, the sequence ${\rho_{0}^{\epsilon}}$ converges to $\rho_{0}$, which is a minimizer of the functional $E$ under the constraint \eqref{equ:C}. This convergence establishes the existence of a minimizer.
\end{theorem}
The proof of the theorem will be given in Section 3.

\subsection{The Positivity Argument}

 We now restrict our attention to the sessile drop case by assuming $\theta_{1}=\theta_{2}=0$, and we maintain this assumption throughout the remainder of this section. Since $\rho_{0}$ minimizes the energy functional $E$, it satisfies
\begin{align}
E'(\rho_{0}) = 0,
\end{align}
and
\begin{align}
E''(\rho_{0}) \ge 0,
\end{align}
where $E'$ and $E''$ denote the first and second variations of $E$, respectively.

It remains unclear, however, whether the second variation is strictly positive, a property that is crucial for the analysis of dynamical stability. When $\theta_{1}=\theta_{2}=0$, the energy functional $E$ is invariant under horizontal translations, which implies that the second variation admits a nontrivial kernel. Consequently, $E''$ is not positive definite with respect to arbitrary perturbations in polar coordinates. It is therefore essential to identify this kernel and develop a mechanism to remove its influence. We summarize this observation in the following theorem, whose proof is deferred to Section~5.

 \begin{theorem}
     The following statements hold:

     $\bullet$ The horizontal shifting function of $\rho_{0}$, denoted by $\xi_{s}$, is given explicitly by
     \begin{align}
         \xi_{s}=\cos\theta+\frac{\rho_{0}'}{\rho_{0}}\sin\theta
     \end{align}

     $\bullet$ Suppose that $\xi$ is a function satisfying the orthogonality condition $\int_{0}^{\pi}\xi\rho_{0}=0$. Then
    \begin{align}
        E^{\prime \prime}(\rho_{0})\big(\xi,\xi\big)=0
    \end{align}
    \noindent if and only if $\rho=C\xi_{s}$ for some constant $C$. In other words
    \begin{align}
        \ker\{E^{\prime\prime}(\rho_{0})\}=\operatorname{span}\{\xi_{s}\}.
    \end{align}
 \end{theorem}
 
 \noindent This theorem shows that the translation mode $\xi_{s}$ is the unique element spanning the kernel of the second-order functional derivative. To remove the influence of this degeneracy, we employ a moving polar coordinate framework. This is an essential and original aspect of this work.

\subsection{Geometric reformulation} To analyze the PDE system \eqref{equ:fix_1}, we reformulate the equations on the equilibrium domain $\Omega_{0}$. The main idea is to introduce a time-dependent diffeomorphism that maps the moving domain $\Omega(t)$ onto the fixed equilibrium domain $\Omega_{0}$. We denote this diffeomorphism by $\Phi(t)$ and define $\mathcal{A}(t)=((\nabla \Phi(t))^{-1})^{T}$. The detailed construction of this diffeomorphism will be presented in Section~4. At this stage, we state only the reformulated system obtained by applying this transformation to \eqref{equ:fix_1}.
\begin{equation}{\label{equ:fix_2}}
    \begin{cases}
        \partial_{t}u-(\cos\theta W\partial_{t}\bar{\rho},\sin\theta W\partial_{t}\bar{\rho})\mathcal{A}(\partial_{x}u,\partial_{y}u)^{T}+u\cdot \nabla_{\mathcal{A}} u+\dive_{\mathcal{A}}S_{\mathcal{A}}(P,u)=0~~~&\operatorname{in}~\Omega \\
        \dive_{\mathcal{A}} u=0&\operatorname{in}~\Omega\\
        S_{\mathcal{A}}(P,u)\mathcal{N}=(g\rho\sin\theta-\sigma\mathcal{H}(\rho)))\mathcal{N}~~&\operatorname{on}~\Sigma\\
        \partial_{t}\rho=\frac{1}{\rho}u\cdot \mathcal{N}~~&\operatorname{on}~\Sigma\\
        (S_{\mathcal{A}}(P,u)\cdot \nu-\beta u)\cdot \tau=0~~&\operatorname{on}~\Sigma_{s}\\
        u\cdot \nu=0~&\operatorname{on}~\Sigma_{s}\\
       {\mathcal{W}}(\p_{t}\rho(\pi,t))=\sigma\frac{\rho'}{(\rho^{2}+\rho'^{2})^{\frac{1}{2}}}(\pi,t)+[\![\gamma]\!]\\
        {\mathcal{W}}(\partial_{t}\rho(0,t))=-(\sigma\frac{\rho'}{(\rho^{2}+\rho'^{2})^{\frac{3}{2}}}(0,t))+[\![\gamma]\!].
    \end{cases}
\end{equation}
\noindent Here, $W=\frac{\bar{\phi}}{\rho_{0}}$ for a smooth cutoff function $\phi$ to be defined in Section 4,  $\mathcal{N}=-\rho'(\theta)\hat{e}_{\phi}+\rho(\theta)\hat{e}_{r}$, and ${\rho_{0}}$ and $\hat{e}_{r}$, $\hat{e}_{\phi}$ denote unit vectors in the angular direction and the radical direction, respectively. Moreover, in equation \eqref{equ:fix_2},  the definition of differential operator $\nabla_{\mathcal{A}}$, $\dive_{\mathcal{A}}$ and $\Delta_{\mathcal{A}}f=\nabla_{\mathcal{A}}\cdot \nabla_{\mathcal{A}}f$ and $\mathbb{D}_{\mathcal{A}}$ are given as follows:
\begin{align}
    (\nabla_{\mathcal{A}}f)_{i}=\mathcal{A}_{ij}\p_{j}f, ~\nabla_{\mathcal{A}}\cdot g=\mathcal{A}_{ij}\p_{j}g_{i},~\Delta_{\mathcal{A}}f=\nabla_{\mathcal{A}}\cdot\nabla_{\mathcal{A}}f,~\operatorname{and} ~(\mathbb{D}u)_{ij}=\mathcal{A}_{ik}\p_{k}u_{j}+\mathcal{A}_{jk}\p_{k}u_{i}
\end{align}

\subsection{Moving polar coordinates} Our major innovation of this paper is introducing the moving polar coordinates. For the sessile drop on the flat plane, it can move horizontally, leading to configurations in which the new free surface can no longer be represented within the original polar coordinate frame. Because of this difficulty, we  introduce a moving polar to solve this problem. Suppose that at time $t=0$, the pole of the coordinate system is located at $(0,0)$. Let $\rho(\theta,t)$ denote the solution function to the equation system \eqref{equ:fix_2}. Then at later time $t$, we construct a new polar coordinates centered at the polar point $(\mathfrak{n}(t),0)$, where $\mathfrak{n}$ satisfies the following \textbf{orthogonality condition with respect to the kernel $\xi_{s}$}
\begin{align}{\label{equ:m}}
    \int_{0}^{\pi}(\rho_{0,\mathfrak{n}(t)}(\theta)-\rho(t,\theta))^{2}d\theta=\min_{c\in(x_{1}(t),x_{2}(t))}\int_{0}^{\pi}(\rho_{0,c}(\theta)-\rho(t,\theta))^{2}d\theta
\end{align}
\noindent In equation \eqref{equ:m}, $x_{1}(t)$, $x_{2}(t)$ are two contact points at time $t$ and $\rho_{0,c}(\theta)$ is the equilibrium centered at the pole $(c,0)$. 

Having identified the new polar center $\mathfrak{n}(t)$, we describe the free surface of the sessile drop in this moving coordinate system. For simplicity, we continue to denote the free surface by $\rho(t,\theta)$ in the moving polar coordinates, slightly abusing notation.

Moreover, the choice of $\mathfrak{n}(t)$ in \eqref{equ:m} is made so that the perturbation from equilibrium, defined by $\xi = \rho - \rho_{0}$, is orthogonal to the kernel of the second variation of the energy functional $E$, namely the translational mode $\xi_s$. This orthogonality ensures the positivity of the quadratic form associated with the second variation; that is,
\begin{align}{\label{equ:pos_0}}
    (\xi,\xi)_{1,\Sigma}\gtrsim \|\xi\|_{H^{1}}^{2}~~~\operatorname{if}~\int_{0}^{\pi}\xi \rho_{0}d\theta=0~~\operatorname{and}~~\int_{0}^{\pi}\xi\xi_{s}d\theta=0
\end{align}
\noindent where $(\xi,\xi)_{1,\Sigma}$ denotes the bilinear form of of the second order functional derivative of $E$ at point $\rho=\rho_{0}$. It can be expressed as follows
\begin{align}
    (\rho_1,\rho_2)_{1,\Sigma}=E^{\prime \prime}(\rho_{0})(\rho_{1},\rho_{2})=&g\int_{0}^{\pi}\rho_{0}\rho_1\rho_2\sin\theta d\theta+\sigma\int_{0}^{\pi}\frac{\rho_{0}\rho_1'\rho_2'}{(\rho_{0}^{2}+\rho_{0}'^{2})^{\frac{3}{2}}}d\theta-\sigma\int_{0}^{\pi}\frac{\rho_{0}'\rho_1'\rho_2}{(\rho_{0}^{2}+\rho_{0}'^{2})^{\frac{3}{2}}}d\theta \notag\\
    &-\sigma \int_{0}^{\pi}\frac{\rho_{0}'\rho_1\rho_2'}{(\rho_{0}^{2}+\rho_{0}'^{2})^{\frac{3}{2}}}d\theta+\sigma\int_{0}^{\pi}\frac{\rho_{0}''\rho_{0}-\rho_{0}'^{2}-\rho_{0}^{2}}{(\rho_{0}^{2}+\rho_{0}'^{2})^{\frac{3}{2}}}\rho_1\rho_2d\theta
\end{align}
\noindent The positivity of this bilinear form is far from evident at first glance, and the evolution law for $\mathfrak{n}(t)$ cannot be derived directly from the definition \eqref{equ:m}. We shall carefully design a dynamic scheme for the evolution of the polar point and rigorously prove the positivity of $(\xi,\xi)_{1,\Sigma}$. The details of these discussion will be given in Section 5.


The positivity argument and the introduction of the moving polar coordinate framework are two of the major technical challenges of this work, and they constitute the main distinction between the present analysis and that of Guo-Tice \cite{Guo}. These ingredients are essential for establishing the a priori estimates. After applying this new moving polar coordinate, we transform the equation system \eqref{equ:fix_2} to the following system

\begin{equation}{\label{equ:mov}}
     \begin{cases}
        \partial_{t}u-(\cos\theta W\partial_{t}\bar{\xi},\sin\theta W\partial_{t}\bar{\xi})\mathcal{A}(\partial_{x}u,\partial_{y}u)^{T}-\mathfrak{n}'(t)\p_{x}u+u\cdot \nabla_{\mathcal{A}}u+\dive_{\mathcal{A}}S_{\mathcal{A}}(p,u)=0~~~&in~\Omega \\
        \dive_{\mathcal{A}}u=0&in~\Omega\\
        S_{\mathcal{A}}(P,u)\mathcal{N}=(g\rho\sin\theta-\sigma \mathcal{H}(\rho))\mathcal{N}~~&on~\Sigma\\
        \partial_{t}\rho+\mathfrak{{n}}^{\prime}(t)\xi_s=\frac{1}{\rho_{0}}u\cdot \mathcal{N}+G~~&on~\Sigma\\
        (S_{\mathcal{A}}(P,u)\cdot \nu-\beta u)\cdot \tau=0~~&on~\Sigma_{s}\\
        u\cdot \nu=0~&on~\Sigma_{s}\\
        {\mathcal{W}}(\p_{t}\rho(\pi,t)+\mathfrak{n}'(t))=\sigma\frac{\rho'}{(\rho^{2}+\rho'^{2})^{\frac{1}{2}}}(\pi,t)+[\![\gamma]\!]\\
        {\mathcal{W}}(\partial_{t}\rho(0,t)-\mathfrak{n}'(t))=-\sigma\frac{\rho'}{(\rho^{2}+\rho'^{2})^{\frac{3}{2}}}(0,t))+[\![\gamma]\!],
    \end{cases}
\end{equation}
\noindent coupled with dynamic for $\mathfrak{n}(t)$
\begin{equation}
    \mathfrak{n}^{\prime}(t)=\lambda \int_{0}^{\pi} \frac{1}{\rho}u\cdot \mathcal{N}\xi_{s}d\theta,
\end{equation}
\noindent where $G$ and $\lambda$ are functions determined by $\rho$ and $u$. Their explicit expressions will be provided in Section~5.

\subsection{Perturbation form} We now study solutions to the full problem as perturbations of the equilibrium state $(0, P_{0}, \rho_{0})$ obtained in the previous section. To this end, we introduce the perturbation variables
\begin{align}
p(t) := P(t) - P_{0}, \qquad \xi(t) := \rho(t) - \rho_{0}, \qquad u := u - 0,
\end{align}
and rewrite the system \eqref{equ:mov} in terms of the variables $(u, p, \xi)$ as follows:
\begin{equation}{\label{equ:mov_1}}
     \begin{cases}
        \partial_{t}u-(\cos\theta W\partial_{t}\bar{\xi},\sin\theta W\partial_{t}\bar{\xi})\mathcal{A}(\partial_{x}u,\partial_{y}u)^{T}-\mathfrak{n}'(t)\p_{x}u+u\cdot \nabla_{\mathcal{A}}u+\dive_{\mathcal{A}}S_{\mathcal{A}}(p,u)=0~~~&in~\Omega \\
        \dive_{\mathcal{A}}u=0&in~\Omega\\
        S_{\mathcal{A}}(p,u)\mathcal{N}=(g\xi\sin\theta+\sigma(\mathcal{P}_1(\rho_0,\rho_0')\xi+\mathcal{P}_{2}(\rho_{0},\rho_{0}')\xi'-\frac{1}{\rho_{0}}\partial_{\theta}(\frac{\rho_{0}^{2}\xi'}{(\rho_{0}^{2}+\rho_{0}'^{2})^{\frac{3}{2}}}-\frac{\rho_{0}'\rho_{0}\xi}{(\rho_{0}^{2}+\rho_{0}'^{2})^{\frac{3}{2}}}+\mathcal{R}_{1})+\mathcal{R}_{2})\mathcal{N}~~&on~\Sigma\\
        \partial_{t}\xi+\mathfrak{n}^{\prime}(t)\xi_s=\frac{1}{\rho_{0}}u\cdot \mathcal{N}+G~~&on~\Sigma\\
        (S_{\mathcal{A}}(p,u)\cdot \nu-\beta u)\cdot \tau=0~~&on~\Sigma_{s}\\
        u\cdot \nu=0~&on~\Sigma_{s}\\
        \kappa\partial_{t}\xi(\pi,t)+\kappa\mathfrak{n}^{\prime}(t)=\sigma\frac{\rho_{0}^{2}\xi'}{(\rho_{0}^{2}+\rho_{0}'^{2})^{\frac{3}{2}}}(\pi,t)-\sigma\frac{\rho_0'\rho_{0}\xi}{(\rho_0^{2}+\rho_{0}^{2})^{\frac{3}{2}}}(\pi,t)\\
        \kappa\partial_{t}\xi(0,t)-\kappa\mathfrak{n}^{\prime}(t)=-(\sigma\frac{\rho_{0}^{2}\xi'}{(\rho_{0}^{2}+\rho_{0}'^{2})^{\frac{3}{2}}}(0,t)-\sigma\frac{\rho_0'\rho_{0}\xi}{(\rho_0^{2}+\rho_{0}^{2})^{\frac{3}{2}}}(0,t)),
    \end{cases}
\end{equation}
\noindent coupled with dynamic for $\mathfrak{n}(t)$:
\begin{equation}
    \mathfrak{n}^{\prime}(t)=\lambda \int_{0}^{\pi} \frac{1}{\rho}u\cdot .\mathcal{N}\xi_{s}d\theta
\end{equation}
The explicit forms of $\mathcal{P}_{1}$, $\mathcal{P}_{2}$, $\mathcal{R}_{1}$ and $\mathcal{R}_{2}$ will be given in Section 4. Moreover, to derive perturbed system \eqref{equ:mov_1}, we expand all nonlinear terms using Taylor series around the equilibrium state. We will give all of the detailed computation in Section 4. 

\subsection{Main results and the structure of this paper}
In order to state our main results, we must first define a number of energy
and dissipation functionals. Following the framework developed in \cite{Guo} and \cite{Guo2024}, the apriori estimate and stability result can be developed based on the second-order energy-dissipation structure. Therefore, we define the corresponding energy and dissipation separately. 

Suppose that $\theta_{eq}\in (0,\pi)$ and $\omega=\pi-\theta_{eq}$. We define the basic or parallel energy by
\begin{align}
        \mathcal{E}_{||}=\sum_{k=0}^{2}\vert \vert \partial_{t}^{k}u\vert \vert^{2}_{L^{2}}+\vert \vert \partial_{t}^{k}\xi\vert \vert_{H^{1}}^{2}.
    \end{align}

    \noindent Then improved energy is defined by

    \begin{align}{\label{equ:energy_0}}
        \mathcal{E}=\mathcal{E}_{||}+\vert \vert u\vert \vert_{W^{2,q_{+}}}^{2}+\vert \vert \partial_{t}u\vert \vert^{2}_{H^{1+\frac{\epsilon_{-}}{2}}}+\vert \vert \partial_{t}^{2}u\vert \vert^{2}_{H^{0}}+\vert \vert p\vert \vert^{2}_{W^{1,q_{+}}}+\vert \vert \partial_{t}p\vert \vert_{L^{2}}^{2}\notag\\
        +|\mathfrak{n}^{\prime}(t)|^{2}+|\mathfrak{n}^{\prime\prime}(t)|^{2}+\vert \vert \xi\vert \vert^{2}_{W^{3-\frac{1}{q_{+}}}}+\vert \vert \partial_{t}\xi\vert \vert_{H^{\frac{3}{2}+\frac{\epsilon_{-}-\alpha}{2}}}^{2}+\vert \vert \partial_{t}^{2}\xi\vert \vert^{2}_{H^{1}}
    \end{align}

    \noindent Next, we define the parallel dissipation functional as follows

    \begin{align}
        \mathcal{D}_{||}=\sum_{k=0}^{2}\vert \vert \partial_{t}^{k}u\vert \vert^{2}_{H^{1}}+\vert \vert \partial_{t}^{k}u\vert\vert^{2}_{L^{2}}+[\p_{t}^{k}u\cdot \mathcal{N}]_{\theta}^{2},
    \end{align}

    \noindent and the improved dissipation by
    
    \begin{align}{\label{equ:dis_0}}
        \mathcal{D}=\mathcal{D}_{||}+\vert \vert u\vert \vert^{2}_{W^{2,q_{+}}}+\vert \vert \partial_{t}u\vert \vert^{2}_{W^{2,q_{-}}}+\sum_{k=0}^{2}[\partial_{t}^{k+1}\xi]^{2}_{\theta}+\sum_{k=0}^{2}[\partial_{t}^{k}\partial_{\theta}\xi]_{\theta}^{2}+\vert \vert p\vert \vert^{2}_{W^{1,q_{+}}}+\vert \vert \partial_{t}p\vert \vert^{2}_{W^{1,q_{-}}}\notag\\
        +\sum_{k=0}^{2}|\mathfrak{n}^{(i+1)}|^{2}+\sum_{k=0}^{2}\vert \vert \partial_{t}^{k}\xi\vert \vert^{2}_{H^{\frac{3}{2}-\alpha}}+\vert \vert \xi\vert \vert_{W^{3-\frac{1}{q_{+}},q_{+}}}^{2}+\vert \vert \partial_{t}\xi\vert \vert^{2}_{W^{3-\frac{1}{q_{-}},q_{-}}}+\vert \vert \partial_{t}^{3}\xi\vert \vert^{2}_{H^{\frac{1}{2}-\alpha}}.
    \end{align}
    \noindent where $[f,g]_{\theta}=f(\pi)g(\pi)+f(0)g(0)$ and $[f]_{\theta}=\sqrt{[f,f]_{\theta}}$.  The exponents $q_{+},q_{-}$ are defined as follows
    \begin{align}
        q_{\pm}=\frac{2}{2-\varepsilon_{\pm}},
    \end{align}
    for two parameters $0<\varepsilon_{-}<\varepsilon_{+}\leq \varepsilon_{\operatorname{max}}$ where
    \begin{align}
        \varepsilon_{\operatorname{max}}=\varepsilon_{\max}(\omega_{eq})=\min\{1,-1+\frac{\pi}{\omega_{eq}}\}\in (0,1].
    \end{align}
    Moreover, 
    we assume that the constant $\alpha$ satisfies the following constraint
    \begin{align}{\label{equ:alpha}}
        0<\alpha<\varepsilon_{-}~~\operatorname{and}~~\alpha<\min\{\frac{1-\varepsilon_{+}}{2},\frac{\varepsilon_{+}-\varepsilon_{-}}{2}\}, ~and~\varepsilon_{+}\leq \frac{\varepsilon_{-}+1}{2}.
    \end{align}
    
    \textbf{Main Theorem} Our main result establishes an apriori estimate for solutions to \eqref{equ:mov}, which shows that if solutions exist in a time horizon $[0,T)$, then the energy decays exponentially. Moreover, we have quantitative estimates in terms of the initial data.

    \begin{theorem}{\label{thm:base_0}}
         Let $\omega_{eq}\in(0,\pi)$, $\epsilon_{-}$ and $\epsilon_{+}$ be defined as above. Suppose that $\alpha$ satisfies the constraint \eqref{equ:alpha}, and that $\mathcal{E}$ and $\mathcal{D}$ are energy and dissipation defined via \eqref{equ:energy_0} and \eqref{equ:dis_0}, respectively. Then there exists a universal constant $0<\delta_{0}<1$ such that if a solution to \eqref{equ:mov_1} exists on the time horizon $[0,T)$ for $0<T\leq \infty$ and obeys the estimate
    \begin{align}
        \sup_{0\leq t<T} \mathcal{E}(t)\leq \delta_{0},
    \end{align}
    \noindent {and the conservation of mass subject to the prescribed constant $V$}, then there exist universal constants $C$ and $\lambda>0$ such that
    \begin{align}
        \sup_{0\leq t<T}e^{\lambda t}\mathcal{E}(t)+\int_{0}^{T}\mathcal{D}dt\leq C\mathcal{E}(0).
    \end{align}
    
    \end{theorem} 
    
    The a priori estimates of this theorem may be coupled with a local existence theory to ensure that the small
energy condition is satisfied, provided the data are sufficiently small and all necessary compatibility conditions are satisfied. Such a local well-posedness theorem will not be given here to keep the present paper of reasonable length.  

    \textbf{Technical overview and layout of paper} Our strategy for proving theorems \ref{thm:base_0} is approximately following that employed for the vessel problem in \cite{Guo}. We establish a nonlinear energy estimate in Section 7.1 and Section 8 based on our energy-dissipation structure defined by equations \eqref{equ:energy_0} and \eqref{equ:dis_0}.  Then we use elliptic estimate and $\frac{3}{2}-\alpha$ estimate developed in Section 7.2 and Section 8 to close the energy-dissipation estimate. The theoretical foundations of this elliptic estimate and $\frac{3}{2}-\alpha$ estimate can be found in \cite{Guo}, and we refer to the introduction of \cite{Guo} for a summary of these underlying methods and techniques.

   There are, however, several new difficulties arising from the droplet geometry and the adoption of polar coordinates. 
   \begin{itemize}
       \item The first concerns the examination of $\ker\{E''(\rho_{0})\}$ and the positivity of the $(1,\Sigma)$ inner product introduced in Section 1.5. {In particular, in order to prove Theorem 1.2, we need to establish an ODE which the function in $\ker\{E''(\rho_{0})\}$ must satisfy. After deriving the math expression for $E''(\rho_{0})$, we derive the ODE \eqref{equ:3.2.7} combined with boundary conditions \eqref{equ:3.2.38},\eqref{equ:3.2.39} and \eqref{equ:3.2.40}. This derivation utilizes integration by parts and the non-negative property of $E''(\rho_{0})$. The resulting equation is a second-order ODE with an unknown parameter $C$. 
       
       Since $\xi_{s}\in \ker\{E''(\rho_{0})\}$, it is a solution to ODE \eqref{equ:3.2.7} and subject to the required boundary condition. Consequently, any solution $\xi$ to \eqref{equ:3.2.7} can be expressed as the following form
       \begin{align}
           \xi(\theta)=\chi(\theta)\xi_{s}(\theta),
       \end{align}
       \noindent where $\chi(\theta)$ is the solution function to the following ODE
       \begin{align}
           \chi''(\theta)+Q(\theta)\chi'(\theta)=C.
       \end{align}
       In the equation above, $Q(\theta)$ is a function depending on $\rho_{0}$ and $\xi_{s}$. This can be interpreted as a first-order ODE for $\chi'(\theta)$. Thus, we obtain the explicit form of $\chi$ and $\xi$ by solving this ODE.
       
       Finally, by requiring the solution function $\xi$ satisfying the boundary condition and the principle of conservation of mass, we demonstrate that $\xi=D\xi_{s}$ for some constant $D$. The detailed computation will be given in Section 5. }
       \item The second difficulty is that all estimates must be carried out in polar coordinates, where expressions for both the $(1,\Sigma)$ inner product and the horizontal shifting function $\xi_{s}$ become considerably more intricate. Moreover, energy estimates in polar coordinates require particular care to ensure consistency with the underlying geometry. We shall use the whole section 2 to establish the dynamics in polar coordinates and derive the following zero-order energy-dissipation relation.
       \begin{equation*}
    \frac{d}{dt}(\int_{\Omega(t)}\frac{1}{2}\vert u\vert^{2}+\mathcal{F}(\rho))+\int_{\Omega(t)}\frac{\mu}{2}\vert \mathbb{D}u(t)\vert^{2}+\int_{\Sigma_{s}(t)}\frac{\beta}{2}\vert u\vert^{2}+\partial_{t}LW(\partial_{t}L)+\partial_{t}RW(\partial_{t}R)=0
     \end{equation*}

\noindent where

\begin{align*}
\mathcal{F}(\rho)=\frac{g}{3}\int_{\theta_2}^{\pi-\theta_1}\rho^{3}+\sigma\sqrt{\rho^{2}+\rho'^{2}}d\theta-[\![\gamma]\!](\rho(\pi-\theta_1)+\rho(\theta_2))
\end{align*}

       \item The final difficulty concerns the new function $\mathfrak{n}(t)$ introduced in Section 1.5, which represents the position of the polar point. Both $\mathfrak{n}(t)$ and its time derivatives are hard to determine. {We will solve this problem by first selecting a new reference frame that moves horizontally at a speed of $\mathfrak{n}'(t)$. Subsequently, we will transform the kinematic boundary condition from the static system to this moving reference frame, which can be expressed as follows
       \begin{align}{\label{equ:difficulty_1}}
           \p_{t}\rho=\frac{1}{\rho}(u-(\mathfrak{n}'(t),0))\cdot \mathcal{N}
       \end{align}
       \noindent This equation establishes a relation among $\mathfrak{n}^{\prime}$, $u$, $\rho$, and $\p_{t}\rho$. However, this relation is not ideal since it contains a first-order term $\p_{t}\rho$. To eliminate this term, we shall test equation \eqref{equ:difficulty_1} by the test function $\xi_{s}$ and integrate the equation from $0$ to $\pi$.
       Using the orthogonality relation, the term
       \begin{align}
           \int_{0}^{\pi}\p_{t}\rho\xi_{s}d\theta
       \end{align} 
       vanishes. Therefore, through more detailed computations presented in Section 5, we derive the ODE \eqref{equ:4.1.18} for $\mathfrak{n}(t)$.}
   \end{itemize}

    Except for the Section 7 and Section 8, which contain the major a priori estimate,  the rest of the paper is organized as follows. In Section 2 we perform some preliminary computations to reformulate the free-boundary Navier-Stokes equation system from Cartesian coordinates into polar coordinates. In Section 3, we show that the steady state $\rho_{0}$ derived in \cite{Yang} is the minimizer of functional energy $E$ defined by \eqref{equ:E}. To our best knowledge, there is no existing systematic result showing the existence of minimizer of this energy functional for arbitrary inclined angles $\theta_{1}$ and $\theta_{2}$, as we developed in Section 3. In Section 4, we introduce the geometric transformation and derive the perturbed form.  In Section 5, we develop a moving polar coordinates and a scheme to determine $\mathfrak{n}(t)$. Moreover, we show that the horizontal shifting function $\xi_{s}$ spans the kernel of $1,\Sigma$ inner product. This is the section that contains the core conceptual innovations and preparatory results essential for the a priori estimate. In Section 6, we briefly introduce the theory about $\frac{3}{2}-\alpha$ for the boundary curve function $\xi$, where we need to borrow what we could from \cite{Guo}.
\subsection{Previous work}
The contact lines (or contact points in two
dimension) that form at triple junctions between three distinct phases (fluid, solid, and vapor phases in
the present paper) have been a subject of intense study since the pioneering work of Young \cite{Young}. For an exhaustive overview, we refer to de Gennes \cite{Gennes} . Here, we restrict ourselves to a brief summary of the classical results.

The initial work of Young, Laplace, and Gauss showed that equilibrium configurations not only solve a
particular equation, commonly known as the gravity-capillary equation, but also satisfy fixed contact angle conditions
determined via:
\begin{align}
    \cos(\theta_{eq})=-\frac{[\![\gamma]\!]}{\sigma}
\end{align}
where $\theta_{\mathrm{eq}}$ denotes the equilibrium contact angle, $\sigma$ is the surface tension, and $[\![\gamma]\!]$ represents the difference in surface energies between the solid–fluid and solid–vapor interfaces.

The dynamics at a contact point are a much more complicated issue. The first problem to deal with in the
context of viscous fluids is that the usual no-slip condition (u = 0 at the fluid-solid interface) and the free boundary kinematic equation (the fourth equation in \eqref{equ:mov}) are incompatible. When both conditions are imposed simultaneously, the motion of the contact point is completely inhibited, implying that the fluid would remain fixed at its initial position on the solid boundary. This is an unphysical outcome. Then the no-slip condition must be replaced by the more general Navier-slip
conditions, the sixth and fifth equations in \eqref{equ:mov}. 

The surveys of the dynamic of contacts points come from Dussan \cite{dussan} and Blake \cite{blake}. The general picture is that the dynamic contact
angle, $\theta_{dyn}$, and the equilibrium angle, $\theta_{eq}$, are related via
\begin{align}{\label{equ:theta}}
    V_{cl}=\mathscr{V}(\sigma \cos(\theta_{dyn})-\sigma \cos(\theta_{eq}))
\end{align}
where $V_{cl}$ is the contact point normal velocity and $\mathscr{V}$ is the increasing diffeomorphism such that $\mathscr{V}(0)=0$. This equation with different form of function $\mathscr{V}$ can be derived in a number of ways. Blake-Haynes \cite{blake1969} obtained \eqref{equ:theta} through thermodynamic and
molecular kinetics arguments. Cox \cite{cox} used matched asymptotic analysis and hydrodynamic arguments.
Ren-E \cite{ren2011} derived this equation from thermodynamic principles applied to constitutive equations. Ren-E \cite{ren2007} also
used molecular dynamics simulations and found an equation of the form \eqref{equ:theta}. These simulations
also indicated that the slip of the fluid along the solid obeys the well-known Navier-slip condition. The last two equations in \ref{equ:mov} implement \eqref{equ:theta} in the drop let problem.

 When the Navier-Stokes equations in \eqref{equ:mov} were replaced by the
Stokes equations, which yields a sort of quasi-static evolution, there are a lot of works done recently. In Guo-Tice \cite{guo2018}, they developed the stability result for stokes flow within the context of the fluid evolving inside a vessel. Then Tice-Wu \cite{tice2021} proved
corresponding results for the Stokes droplet problem in which the vessel configuration is replaced with a
droplet sitting atop a flat plane. Moreover,  Tice-Zheng \cite{zheng2017} established local existence results. When it comes to the Navier-Stokes flow, we refer to Guo-Tice \cite{Guo} which established the stability result for fluid evolving in the vessel. Up to our best knowledge, there is no such a theorem for Navier-Stokes flow on a flat plane.

There has also been much prior work devoted to studying contact lines and points in simplified thin-film
models; we will not attempt to list these results here and instead refer to the survey by Bertozzi
\cite{bertozzi1998}. However, there are relatively few results in the literature concerning models that take into account the full fluid mechanics, and, to the best of our knowledge, none that accommodate both a dynamic contact point and a dynamic contact angle. Schweizer \cite{schweizer2001} studied a 2D Navier-Stokes problem with a fixed contact angle
of $\frac{\pi}{2}$. Bodea \cite{bodea2006} studied a similar problem with fixed $\pi/2$ contact angle in 3D channels with periodicity
in one direction. Knupfer-Masmoudi \cite{knupfer2015} studied the dynamics of a 2D drop with fixed contact angle when
the dynamic of the fluid is assumed to be governed by Darcy’s law. Related analysis of the fully stationary Navier-Stokes
system with free, but unmoving boundary, was shown in 2D by Solonnikov \cite{solonnikov1995} with contact angle
fixed at $\pi$, by Jin \cite{ja2005} in 3D with angle $\pi/2$, and by Socolowsky \cite{socolowsky1993} for 2D coating problems with fixed
contact angles.

\section{Basic computation}

In this section, we aim to reform the contact line Navier-Stokes equation system into  polar coordinates and derive the form of the equation system in polar coordinates. To be specific, we aim to derive the kinematic boundary condition in polar coordinates. Recall the equation system in Cartesian coordinates
\begin{equation}\label{eq:navier-stokes}
  \left\{
  \begin{aligned}
&\d_t u+u\cdot\nabla u+\nabla P-\mu\Delta u=0 \quad & \text{in} &\ \Om(t),\\
&\dive u=0 \quad & \text{in} &\ \Om(t),\\
&S(P,u)\nu=g\xi\nu-\sigma \mathcal{H}(\xi)\nu \quad & \text{on} &\ \Sigma(t),\\
&\big(S(P,u)\nu-\beta u\big)\cdot\tau=0 \quad & \text{on} &\ \Sigma_s(t),\\
&u\cdot\nu=0 \quad & \text{on} &\ \Sigma_s(t),\\
&\d_t\xi=u\cdot\nu=u_2-u_1\partial_1\xi \quad & \text{on} &\ \Sigma(t),\\\rule{0ex}{-1.5em}
&\d_t\xi(\pm\ell,t)=\mathscr{V}\left([\![\gamma]\!]\mp\sigma\frac{\partial_1\xi}{(1+|\partial_1\xi|^2)^{1/2}}(\pm\ell,t)\right).
  \end{aligned}
  \right.
\end{equation}

To prove the stability result in the next Section, we assume $\theta_{1}$ and $\theta_{2}$ are arbitrary acute angles as shown in Figure \ref{figure2}. To begin with, let the boundary curve in polar coordinates be denoted by $\rho(\theta,t)=\rho_{0}(\theta)+\eta(\theta,t)$, where $\rho_{0}$ denotes the equilibrium profile and $\eta$ represents the perturbation. The expression of the boundary curve in Cartesian coordinates is denoted by $\xi(x,t)$. Also, we denote $u(x,y)=u(r,\theta)$ as the velocity field in the polar coordinates. 

Before proceeding with the reformulation, we first compute the basic derivative relations in polar coordinates. Using the chain rule, it holds that

\begin{equation*}
    \partial_x u=-\rho \sin\theta\partial_{\theta}u+\cos\theta \partial_{r} u.
\end{equation*}

\noindent Similarly, we have the following result

\begin{equation*}
    \partial_{y}u=\rho\cos\theta \partial_{\theta} u+\cos\theta\partial_{r}u.
\end{equation*}

\noindent Then for the second derivatives especially the Laplacian, we can show the following result

\begin{equation*}
    \Delta u=\frac{1}{r}\partial_{r}(r\frac{\partial u}{\partial r})+\frac{1}{r^{2}}\frac{\partial^{2}u}{\partial^{2}\theta}
\end{equation*}

 We next consider the derivative for the boundary curve in the polar coordinates. By using variation method to the functional $E(\rho)=\int_{\theta_1}^{\theta_2}\sqrt{\rho^{2}(\theta)+\rho'^{2}(\theta)}$, we obtain that

\begin{equation*}
    \mathcal{H}(\theta)=\frac{2(\rho')^{2}-\rho\rho''+\rho^{2}}{(\rho'+\rho^{2})^{\frac{3}{2}}},
\end{equation*}

\noindent where $\mathcal{H}$ denotes the mean curvature of the free surface in polar coordinates.

We also derive the relation between the first derivative in polar coordinates and the first derivative in Cartesian coordinates. We have the following computation: 

\begin{equation*}
    \frac{dy}{dx}=\frac{d(\rho\sin\theta)}{d(\rho\cos\theta)}=\frac{\sin\theta d\rho+\rho\cos\theta d\theta}{\cos\theta d\rho-\rho\sin\theta d\theta}
\end{equation*}

\noindent which leads to the result:

\begin{equation}{\label{equ:1.1.3}}
    \frac{dy}{dx}=\frac{\sin\theta \rho'+\rho\cos\theta}{\cos\theta\rho'-\rho\sin\theta}
\end{equation}

 Having obtained the representation of the first derivative of $u$ and $\rho$, we now  derive an important physical derivative-the conservation of total mass. Before this main theorem, we establish the Navier-Stokes equation system in polar coordinates via the following lemma.

 \begin{lemma}
     Suppose that $(u, p, \rho)$ denote the velocity field, pressure, and surface function in polar coordinates. Locally, when expressed in Cartesian coordinates, these functions satisfy the Navier–Stokes equations \eqref{eq:navier-stokes}. Then their evolution is globally governed by the following system of equations:
     \begin{equation}{\label{equ:navier_stokes_0}}
    \begin{cases}
        \partial_{t}u+u\cdot \nabla u+\dive S(P,u)=0~~~&\operatorname{in}~\Omega(t) \\
        \dive u=0&\operatorname{in}~\Omega(t)\\
        S(P,u)\nu=(g\rho\sin\theta-\sigma\mathcal{H}(\rho)))\nu~~&\operatorname{on}~\Sigma(t)\\
        \partial_{t}\rho=\frac{1}{\rho}u\cdot \mathcal{N}~~&\operatorname{on}~\Sigma(t)\\
        (S(P,u)\cdot \nu-\beta u)\cdot \tau=0~~&\operatorname{on}~\Sigma_{s}\\
        u\cdot \nu=0~&\operatorname{on}~\Sigma_{s}\\
       {\mathcal{W}}(\p_{t}\rho(\pi,t))=\sigma\frac{\rho'}{(\rho^{2}+\rho'^{2})^{\frac{1}{2}}}(\pi,t)+[\![\gamma]\!]\\
        {\mathcal{W}}(\partial_{t}\rho(0,t))=-(\sigma\frac{\rho'}{(\rho^{2}+\rho'^{2})^{\frac{3}{2}}}(0,t))+[\![\gamma]\!]
    \end{cases}
\end{equation}

 \end{lemma}
\begin{proof}
    We prove this result by transforming equation \eqref{eq:navier-stokes} from Cartesian coordinates to polar coordinates. By comparing equation \eqref{equ:navier_stokes_0} with \eqref{eq:navier-stokes}, it suffices to transform the kinematic boundary condition—namely, the sixth equation in the system—into polar coordinates.

    To apply the transformation, the key point is to construct the relation between $\p_{t}\rho$ and $\p_{t}\xi$. By definition, we have

\begin{equation}{\label{equ:1.1.1}}
    y(\theta,t)=\rho(\theta,t)\sin\theta,
\end{equation}
and
\begin{equation}{\label{equ:1.1.2}}
    x(\theta,t)=\rho(\theta,t)\cos\theta.
\end{equation}

\noindent From equations \eqref{equ:1.1.1} and \eqref{equ:1.1.2}, we obtain that:

\begin{equation}{\label{equ:1.1.4}}
    \frac{\partial y}{\partial t}=\frac{\partial \rho}{\partial t} \sin\theta,
\end{equation}
and
\begin{equation}{\label{equ:1.1.5}}
    \frac{\partial x}{\partial t}=\frac{\partial \rho}{\partial t}\cos\theta.
\end{equation}

\noindent Let $\xi$ be the representation of the boundary curve in Cartesian coordinates.(Even though this representation is not global, we only need it for the local computation). Then we have the following identities from its definition.

\begin{equation*}
    \xi(x,t)=\xi(x(\theta,t),t)=y(\theta,t)
\end{equation*}

\noindent Taking time derivative on the equation above, we have

\begin{equation*}
    \partial_{t}y(\theta,t)=\partial_{x}\xi(x,t)\frac{\partial x}{\partial t}+\partial_{t}\xi,
\end{equation*}

\noindent which is equivalent to the following fact

\begin{equation}{\label{equ:1.1.6}}
    \partial_t\xi=\frac{\partial y}{\partial t}-\partial_x \xi(x,t)\frac{\partial x}{\partial t}.
\end{equation}

\noindent Then plugging \eqref{equ:1.1.4} and \eqref{equ:1.1.5} into equation \eqref{equ:1.1.6}, we have

\begin{equation*}
    \partial_t \xi=\frac{\partial \rho}{\partial t}\sin\theta-\partial_x\xi\frac{\partial \rho}{\partial t}\cos\theta,
\end{equation*}

\noindent which implies that:

\begin{equation}{\label{equ:1.1.7}}
    \frac{\partial \rho}{\partial t}=\frac{\partial_t \xi}{\sin\theta-\partial_x\xi \cos\theta}.
\end{equation}

\noindent Substituting equation \eqref{equ:1.1.3} into equation \eqref{equ:1.1.7},  equation \eqref{equ:1.1.7} can be transformed to the following relation.

\begin{equation*}
    \partial_t \rho=\frac{\partial_t \xi}{\sin\theta-\frac{\sin\theta \rho'+\rho\cos\theta}{\cos\theta\rho'-\rho\sin\theta}\cos\theta} =-\frac{\partial_{t}\xi}{\frac{\rho}{\cos\theta\rho'-\rho\sin\theta}}=-\frac{(\cos\theta\rho'-\rho\sin\theta)\partial_{t}\xi}{\rho}.
\end{equation*}

\noindent Hence,

\begin{equation}{\label{equ:1.1.8}}
    \partial_{t}\rho=-\frac{(\cos\theta\rho'-\rho\sin\theta)\partial_{t}\xi}{\rho}.
\end{equation}

     Using the kinematic boundary condition of the contact line equation system(The sixth equation in \eqref{eq:navier-stokes}), we have

 \begin{equation}{\label{equ:1.1.9}}
     \partial_{t}\xi=u_2-u_1\partial_{x}\xi
 \end{equation}

\noindent Applying equation \eqref{equ:1.1.3} to equation \eqref{equ:1.1.9}, we obtain

\begin{equation}{\label{equ:1.1.10}}
    \partial_{t}\xi=u_2-u_1\frac{\sin\theta \rho'+\rho\cos\theta}{\cos\theta \rho'-\rho\sin\theta}
\end{equation}

\noindent We then use \eqref{equ:1.1.10} in \eqref{equ:1.1.8} to substitute the term $\partial_{t}\xi$, which shows the following relation:

\begin{equation}{\label{equ:1.1.11}}
    \partial_{t}\rho=-\frac{(\cos\theta\rho'-\rho\sin\theta)u_2-u_1(\sin\theta\rho'+\rho\cos\theta)}{\rho}
\end{equation}

 Note that the velocity field in polar coordinates is expressed as

\begin{equation}{\label{equ:1.1.12}}
    u=u_{r}\hat{e}_{r}+u_{\theta}\hat{e}_{\phi}=(u_{r}\cos\theta-u_{\theta}\sin\theta,u_{r}\sin\theta+u_{\theta}\cos\theta)=(u_1,u_2)
\end{equation}

\noindent Substituting \eqref{equ:1.1.12} into equation \eqref{equ:1.1.11}, it holds that

\begin{equation}{\label{equ:1.1.13}}
    \partial_{t}\rho=u_{r}-\frac{\rho'}{\rho}u_{\theta}.
\end{equation}

 Moreover, by geometry relation,

\begin{equation}{\label{equ:1.1.14}}
    \mathcal{N}=-\rho'(\theta)\hat{e}_{\phi}+\rho(\theta)\hat{e_{r}},
\end{equation}

\noindent where $\mathcal{N}$ is the normal vector with $\vert \mathcal{N}\vert =\sqrt{\rho^{2}+\rho'^{2}}$. Combining \eqref{equ:1.1.13} and \eqref{equ:1.1.14}, we arrive at the new kinematic boundary condition:

\begin{equation}{\label{equ:1.1.15}}
    \partial_{t}\rho=\frac{1}{\rho}u\cdot \mathcal{N}
\end{equation}

This finishes the proof.
\end{proof}

\begin{theorem}
    Suppose that ($\rho$,u,P) is the solution to equation system \eqref{equ:fix_1}. The total mass enclosed by the free-boundaries and two inclined walls is conserved. This property can be expressed by the following equation

    \begin{align}
        \partial_{t}\int_{\theta_{2}}^{\pi-\theta_{1}}\rho^{2}d\theta=0.
    \end{align}
    
\end{theorem}

\begin{proof}
    
First, the representation of the total mass expressed is given by

\begin{equation*}
    M=\frac{1}{2}\int_{\theta_2}^{\pi-\theta_1} \rho^{2}(\theta)d\theta
\end{equation*}

\noindent Taking time derivative on both sides of the equation, we obtain

\begin{equation}{\label{equ:mass_d}}
    \partial_{t}M=\int_{\theta_2}^{\pi-\theta_1} \rho\partial_{t} \rho d\theta.
\end{equation}

 Applying \eqref{equ:1.1.15} to equation \eqref{equ:mass_d}, we derive the following expression for time derivative of the total mass $M$

\begin{align}
     \partial_{t}M=&\int_{\theta_2}^{\pi-\theta_1} \rho\partial_t\rho d\theta=\int_{\pi-\theta_2}^{\theta_1} \rho \frac{u\cdot \mathcal{N}}{\rho}d\theta\notag\\
     =&\int_{\theta_2}^{\pi-\theta_1} u\cdot \nu \sqrt{\rho^{2}+\rho'^{2}}d\theta \notag\\
      =&\int_{\Sigma(t)} u\cdot \nu=\int_{\partial \Omega(t)} u\cdot \nu=\int_{\Omega(t)} \operatorname{div}~ u=0\label{equ:1.1.16},
\end{align}

\noindent where we used the fact that $u\cdot \nu=0$ on $\Sigma_{s}$. Hence, we conclude that $\partial_t M=0$.

\end{proof}

Now we have the conservation law of the total mass or the total volume. We then derive another important physical law — the energy–dissipation relation.

\begin{theorem}
    Suppose that ($\rho$,u,P) is the solution to equation system \eqref{equ:fix_1}. The energy-dissipation relation holds for this solution:

    \begin{equation*}
    \frac{d}{dt}(\int_{\Omega(t)}\frac{1}{2}\vert u\vert^{2}+\mathcal{F}(\rho))+\int_{\Omega(t)}\frac{\mu}{2}\vert \mathbb{D}u(t)\vert^{2}+\int_{\Sigma_{s}(t)}\frac{\beta}{2}\vert u\vert^{2}+\partial_{t}LW(\partial_{t}L)+\partial_{t}RW(\partial_{t}R)=0
\end{equation*}

\noindent where: 

\begin{align}
\mathcal{F}(\rho)=\frac{g}{3}\int_{\theta_2}^{\pi-\theta_1}\rho^{3}+\sigma\sqrt{\rho^{2}+\rho'^{2}}d\theta-[\![\gamma]\!](\rho(\pi-\theta_1)+\rho(\theta_2)),
\end{align}
and $L$ and $R$ denote the x-coordinate of the left and right contact points, respectively.
\end{theorem}

\begin{proof}
    
   Multiplying $u$ on both sides of the Navier-Stokes equation (the first equation in system \eqref{equ:fix_1}), and integrating it in the domain $\Omega$, we obtain

\begin{equation}{\label{equ:1.1.17}}
    \int_{\Omega(t)} \partial_{t}u\cdot u+\int_{\Omega(t)} (u\cdot \nabla u)\cdot u+\int_{\Omega(t)} \nabla P\cdot u-\mu\int_{\Omega(t)} \Delta u\cdot u=0.
\end{equation}

\noindent For the left hand side of the equation \eqref{equ:1.1.17}, we divide it into several parts:

\begin{equation}{\label{equ:1.1.18}}
    \int_{\Omega(t)} \partial_{t}u\cdot u+\int_{\Omega(t)} (u\cdot \nabla u)\cdot u+\int_{\Omega(t)} \nabla P\cdot u-\mu\int_{\Omega(t)} \Delta u\cdot u=I_1+I_2+I_3+I_4
\end{equation}

\noindent Then we estimate each term separately.

\textbf{Term $I_{1}$}

Taking time derivative and applying the product rule, we have
\begin{equation*}
    I_1=\frac{1}{2}\frac{d}{dt}\int_{\Omega(t)} \vert u\vert^{2}-\frac{1}{2}\int_{\partial \Omega(t)}\vert u\vert^{2} u\cdot \nu
\end{equation*}

 \textbf{Term $I_{2}$} 
 
 Integrating by parts yields

\begin{equation*}
    I_2=\frac{1}{2}\int_{\Omega(t)} u\cdot \nabla \vert u\vert^{2}=\frac{1}{2}\int_{\partial \Omega(t) }\vert u\vert^{2} u\cdot \nu-\frac{1}{2}\int_{\Omega(t)} \nabla \cdot u \vert u\vert^{2}=\frac{1}{2}\int_{\partial \Omega(t) }\vert u\vert^{2} u\cdot \nu
\end{equation*}

Hence,
\begin{equation*}
  I_1+I_{2}=\frac{1}{2}\frac{d}{dt}\int_{\Omega(t)} \vert u\vert^{2}
\end{equation*}

\textbf{Term $I_{3}+I_{4}$} 

Using integration by part, we have

\begin{align}
    I_3+I_4&=\int_{\Omega(t)} \operatorname{div}(S(P,u))\cdot u=\int_{\Omega(t)} \frac{\mu}{2}\vert \mathbb{D}u\vert^{2}-\int_{\Omega(t)} P(\nabla \cdot u)+\int_{\Sigma_{s}}(S(P,u)\nu)\cdot u+\int_{\Sigma(t)}(S(P,u)\nu)\cdot u \notag\\
     &=\int_{\Omega(t)} \frac{\mu}{2}\vert \mathbb{D}u\vert^{2}+\int_{\Sigma_{s}}(S(P,u)\nu)\cdot u+\int_{\Sigma(t)}(S(P,u)\nu)\cdot u\notag\\
     &=I_5+I_6+I_7 \label{equ:1.1.19}
\end{align}

 \textbf{Term $I_{6}$}: 
 
 The term $I_{5}$ has the form we want. Therefore, it remains to estimate $I_6$ and $I_7$. For $I_6$, we have the following computation for it

\begin{align}
     I_6=&\int_{\Sigma_{s}}(S(P,u)\nu)\cdot u \notag \\
     =&\int_{\Sigma_{s}}((S(P,u)\nu)\cdot \nu)(u\cdot \nu)+\int_{\Sigma_{s}}((S(P,u)\nu)\cdot \tau)(u\cdot \tau) \label{equ:1.1.20}.
\end{align}

\noindent Since $u\cdot \nu=0$ on the boundary $\Sigma_s$, we have

\begin{equation}{\label{equ:1.1.21}}
    I_6=\int_{\Sigma_{s}}((S(P,u)\nu)\cdot \tau) (u\cdot \tau)=\int_{\Sigma_{s}}\beta (u\cdot \tau)^{2}.
\end{equation}

\textbf{Term $I_{7}$}

Using kinematic boundary condition on $\Sigma$ (the third equation in system \eqref{equ:fix_1}), we have

\begin{align}
     I_{7}=&\int_{\Sigma}(S(P,u)\nu)\cdot u=\int_{\Sigma} (g\xi+\sigma H(\theta))(u\cdot \nu) \notag\\
      =&\int_{\Sigma} (g\xi)(u\cdot \nu)+\sigma\int_{\Sigma}H(\theta)(u\cdot \nu) \notag\\
      =&I_8+I_9 \label{equ:1.1.22}
\end{align}

\noindent  Using \eqref{equ:1.1.15}, we have

\begin{equation}{\label{equ:1.1.23}}
    \frac{1}{\rho}u\cdot \mathcal{N}=\partial_{t}\rho
\end{equation}

\noindent Moreover, since $\nu$ is the unit normal vector, by the definition of $\mathcal{N}$(the equation \eqref{equ:1.1.14}), we have

\begin{equation*}
    \nu=\frac{\mathcal{N}}{\sqrt{\rho^{2}+\rho'^{2}}}
\end{equation*}

\noindent For $I_8$, using equation \eqref{equ:1.1.23} and converting the integral to an integral in polar coordinates, we obtain

\begin{equation}{\label{equ:1.1.24}}
    I_8=\int_{\theta_2}^{\pi-\theta_2} \rho(g\rho\sin\theta)\partial_{t}\rho d\theta=\frac{1}{3}g\partial_{t}\int_{\theta_2}^{\pi-\theta_1}\rho^{3}\sin\theta d\theta
\end{equation}

\noindent We notice that $\partial_{t}\frac{1}{3}g\int_{\theta_1}^{\theta_2} \rho^{3}\sin\theta d\theta$ is exactly the temporal derivative of the total gravitational energy, which is what we want. Hence, it remains to compute $I_9$. Using a similar argument as $I_{8}$ and transforming the integral into polar coordinates, we have

\begin{equation}{\label{equ:1.1.25}}
    I_9=\sigma \int_{\theta_2}^{\pi-\theta_1} H(\theta) \rho \partial_{t}\rho d\theta
\end{equation}

\noindent Recall that

\begin{align}
     H(\theta)=&\frac{2(\rho')^{2}-\rho\rho''+\rho^{2}}{(\rho'+\rho^{2})^{\frac{3}{2}}}\notag\\
     =&\frac{1}{\rho}(\frac{\rho}{\sqrt{\rho^{2}+\rho'^{2}}}-\partial_{\theta}\frac{\rho'}{\sqrt{\rho^{2}+\rho'^{2}}}) \label{equ:1.1.26}.
\end{align}

\noindent Substituting \eqref{equ:1.1.26} into \eqref{equ:1.1.25}, and then integrating it by part, we have

\begin{equation*}
    I_9=\sigma\int_{\theta_2}^{\pi-\theta_1}(\frac{\rho\partial_t\rho+\rho'\partial_t\rho'}{\sqrt{\rho^{2}+\rho'^{2}}})d\theta-\sigma\partial_{t}\rho(\theta_2)\frac{\rho'}{\sqrt{\rho^{2}+\rho'^{2}}}(\theta_2)+\sigma\partial_{t}\rho(\theta_1)\frac{\rho'}{\sqrt{\rho^{2}+\rho'^{2}}}(\theta_1)=I_{9,1}+I_{9,2}+I_{9,3}
\end{equation*}

\noindent For $I_{9,1}$, we have:

\begin{equation}{\label{equ:1.1.27}}
    I_{9,1}=\sigma\int_{\theta_2}^{\pi-\theta_1}(\frac{\rho\partial_t\rho+\rho'\partial_t\rho'}{\sqrt{\rho^{2}+\rho'^{2}}})d\theta=\sigma \partial_{t}\int_{\theta_2}^{\pi-\theta_1}\sqrt{\rho^{2}+\rho'^{2}}d\theta
\end{equation}

\noindent Then for the boundary terms $I_{9,2}$ and $I_{9,3}$,  note that

\begin{equation*}
    \frac{\rho'}{\sqrt{\rho^{2}+\rho'^{2}}}(\theta_1)=\sin\gamma_1,
\end{equation*}

\noindent and

\begin{equation*}
    \frac{\rho'}{\sqrt{\rho^{2}+\rho'^{2}}}(\theta_2)=\sin\gamma_2,
\end{equation*}

\noindent where $\gamma_1$ and $\gamma_2$ are two angle between the tangent direction of the curve and the normal direction of the wall at $\theta_1$ and $\theta_2$. By the definitions of $\gamma_1$, $\gamma_2$ and the contact line condition of \eqref{equ:navier_stokes_0}, we have

\begin{equation*}
    W(\partial_tL)=\sigma\sin\gamma_1+[\![\gamma]\!],
\end{equation*}

\noindent and

\begin{equation*}
    W(\partial_tR)=-\sigma\sin\gamma_2+[\![\gamma]\!],
\end{equation*}

\noindent which implies that

\begin{equation}{\label{equ:1.1.28}}
   \sigma\sin\gamma_1= W(\partial_tL)-[\![\gamma]\!],
\end{equation}

\noindent and

\begin{equation}{\label{equ:1.1.29}}
    \sigma \sin\gamma_{2}=[\![\gamma]\!]-W(\partial_{t}R).
\end{equation}

\noindent Therefore,  combining \eqref{equ:1.1.27},\eqref{equ:1.1.28} and \eqref{equ:1.1.29}, we finally obtain

\begin{equation}{\label{equ:1.1.30}}
    I_9=\sigma \partial_{t}\int_{\theta_2}^{\pi-\theta_1}\sqrt{\rho^{2}+\rho'^{2}}d\theta-\partial_{t}([\![\gamma]\!](\rho(\theta_1)+\rho(\theta_2)))+\partial_{t}LW(\partial_{t}L)+\partial_{t}RW(\partial_{t}R)
\end{equation}

Finally, combining the computations for $I_1$ through $I_9$, we derive the energy–dissipation relation.

\end{proof}

\section{The stability of the steady state}

In our previous paper (\cite{Yang}), we have already proved the existence of the steady state, which is the solution function of the Euler-Lagrange equation. Alternatively, the existence of the steady state can also be demonstrated by studying the minimization problem associated with the energy functional, since any minimizer necessarily satisfies the Euler–Lagrange equation. From the previous paper \cite{Yang}, we know that the steady state might not be represented as a graph of function in Cartesian coordinates, it is more convenient to first use polar coordinates to study the problem.  The energy functional as be written in polar coordinates as follows:

    \begin{equation}{\label{equ:2.1.1}}
        E(\rho(\theta))=\frac{1}{3}g\int_{\theta_2}^{\pi-\theta_1}\rho^{3}\sin \theta d\theta+\int_{\theta_2}^{\pi-\theta_1}\sigma \sqrt{\rho^{2}+\rho'^{2}}d\theta-[\![\gamma]\!](\rho(\pi-\theta_1)+\rho(\theta_2))
    \end{equation}

    \noindent With the constraint

    \begin{equation}{\label{equ:2.1.2}}
    \int_{\theta_2}^{\pi-\theta_1} \rho^{2}(\theta)d\theta=\mathcal{V}(\rho(\theta))=V
    \end{equation}
    
    \noindent where $V$ is a constant.

    From a straightforward observation of equation \eqref{equ:2.1.1}, we notice that $E$ is positive and, therefore, bounded from below when $[\![\gamma]\!]<0$. However, when $[\![\gamma]\!]>0$, it is less straightforward to determine whether the energy functional is bounded from below or not.  To address this, we divide this chapter into two cases. In both cases, we assume that $\vert [\![\gamma]\!]\vert<\sigma $.

    \subsection{The case when $[\![\gamma]\!]<0$}
    
    In this subsection, we discuss the case where $[\![\gamma]\!]<0$.
    
    Since the functional energy is bounded from below, we can use Helly's selection principle to show that $E$ has a minimizer $\rho_{0}\in BV$. However, the regularity of this minimizer $\rho_{0}$ is not enough to be improved and to derive the Euler-Lagrange equation. To obtain a minimizer with higher regularity,  we need to introduce a new $\epsilon$ regularized energy functional:

    \begin{align}{\label{equ:2.1.3}}
        E^{\epsilon}=&\frac{1}{3}g\int_{\theta_2}^{\pi-\theta_1} \rho^{3}\sin\theta d\theta +\sigma\int_{\theta_2}^{\pi-\theta_1} \sqrt{\rho^{2}+\rho'^{2}}d\theta
        \notag\\
        &-[\![\gamma]\!](\rho(\theta_2)+\rho(\pi-\theta_1))+\frac{\epsilon}{2} \int_{\theta_2}^{\pi-\theta_1}\rho'^{2} d\theta
    \end{align}
    
    \noindent with the restriction:

    \begin{equation}{\label{equ:2.1.4}}
       V=\mathcal{V}(\rho(\theta))=\int_{\theta_2}^{\pi-\theta_1} \rho^{2}d\theta
    \end{equation}
    
    Using the Banach-Alaoglu theorem and lower semi-continuity, we know  that there exists a $H^{1}$ minimizer $\rho_{\epsilon}$ for this problem( By the elliptic theory, it is actually a smooth function), solving the following Euler-Lagrange equation system

    \begin{equation}{\label{equ:2.1.5}}
    \begin{cases}
        g\rho^{2}\sin\theta+\sigma \frac{\rho}{\sqrt{\rho^{2}+\rho'^{2}}}-\sigma \partial_{\theta}\frac{\rho'}{\sqrt{\rho^{2}+\rho'^{2}}}-\epsilon \rho''=P_{0}^{\epsilon}\rho\\
        (\sigma \frac{\rho'}{\sqrt{\rho^{2}+\rho'^{2}}}+\epsilon \rho')(\pi-\theta_1)=[\![\gamma]\!]\\
        (\sigma \frac{\rho'}{\sqrt{\rho^{2}+\rho'^{2}}}+\epsilon \rho')(\theta_2)=-[\![\gamma]\!],
    \end{cases}
    \end{equation}

    \noindent where $P_{0}^{\epsilon}$ is the Lagrangian multiplier. To move on to the next step, we first show that $P_{0}^{\epsilon}$ is positive when $[\![\gamma]\!]<0$. This is stated in the following theorem.

    \begin{theorem}

        For any $[\![\gamma]\!]<0$ and $V>0$, $P_{0}^{\epsilon}$ is positive for any $\epsilon$
    \end{theorem}
    
    \begin{proof}

         Using definition of $\rho_{\epsilon}$, $\rho_{\epsilon}$ is the minimizer of the following energy functional

         \begin{align}
             {E}^{\epsilon}_{p}(\rho)=&\frac{1}{3}g\int_{\theta_2}^{\pi-\theta_1} \rho^{3}\sin\theta d\theta +\sigma\int_{\theta_2}^{\pi-\theta_1} \sqrt{\rho^{2}+\rho'^{2}}d\theta-[\![\gamma]\!](\rho(\theta_2)+\rho(\pi-\theta_1))\\
              &+\frac{\epsilon}{2} \int_{\theta_2}^{\pi-\theta_1}\rho'^{2} d\theta-P_{0}^{\epsilon}\int_{\theta_2}^{\pi-\theta_1} \rho^{2}d\theta \label{equ:2.1.6}.
         \end{align}

    Suppose $P^{\epsilon}_{0}$ is negative when $[\![\gamma]\!]$ is negative.  It is straightforward to see that $\rho_{\epsilon}=0$ is the minimizer for this energy functional ${E}_{p}^{\epsilon}$ defined by \eqref{equ:2.1.6}.  Plugging $\rho_{\epsilon}=0$ into  equation \eqref{equ:2.1.4}, we obtain:

    \begin{equation*}
        \mathcal{V}(\rho_{\epsilon})=\int_{\theta_2}^{\pi-\theta_1}\rho_{\epsilon}^{2}d\theta=0\neq V
    \end{equation*}

    \noindent However, this contradicts our assumption that $V>0$. Hence, $P_{0}^{\epsilon}$ should be positive.
    
    \end{proof}    

    Besides the uniform lower estimate for $P_{0}^{\epsilon}$, we also want to derive a uniform upper bound for $P_{0}^{\epsilon}$. Before showing this upper bound, we need to first establish the following lemma

    \begin{lemma}

    $\rho_{\epsilon}$ is defined to be the minimizer of the energy functional $E^{\epsilon}$ defined by equation \eqref{equ:2.1.3}. Then we have the following uniform bound property for $\rho_{\epsilon}$.
    
    \begin{equation*}
        \vert \vert \rho_{\epsilon}\vert \vert_{BV}\leq C~\operatorname{for}~\operatorname{any}~\epsilon\leq \epsilon_{0},
     \end{equation*}

     \noindent and
     
    \begin{equation*}
        \epsilon\int_{\theta_2}^{\pi-\theta_1}\rho_{\epsilon}'^{2}d\theta \leq C~\operatorname{for}~\operatorname{any}~\epsilon\leq \epsilon_{0}
    \end{equation*}
    
     \noindent for some constant $C$ independent of $\epsilon$. 
     
    \end{lemma}

    \begin{proof}
        Noticing that if $\epsilon_1<\epsilon_2$, we have:

        \begin{equation}{\label{equ:2.1.7}}
         E^{\epsilon_1}(\rho_{\epsilon_2})<E^{\epsilon_2}(\rho_{\epsilon_2})
        \end{equation}

        \noindent Then using the fact that $\rho_{\epsilon_1}$ is the minimizer of the energy functional \eqref{equ:2.1.3} with $\epsilon=\epsilon_{1}$, we obtain the following inequality

        \begin{equation}{\label{equ:2.1.8}}
            E^{\epsilon_1}(\rho_{\epsilon_1})\leq E^{\epsilon_1}(\rho_{\epsilon_2}).
        \end{equation}

        \noindent Combining equations \eqref{equ:2.1.7} and \eqref{equ:2.1.8} together, we obtain

        \begin{equation}{\label{equ:2.1.9}}
            0<E^{\epsilon_1}(\rho_{\epsilon_1})<E^{\epsilon_2}(\rho_{\epsilon_2})<E^{\epsilon_0}(\rho_{\epsilon_0}).
        \end{equation}
        
        \noindent Equation \eqref{equ:2.1.9} yields

        \begin{equation*}
            \int_{\theta_2}^{\pi-\theta_1} \sqrt{\rho_{\epsilon}^{2}+\rho_{\epsilon}'^{2}}d\theta+\frac{1}{2}\epsilon\int_{\theta_2}^{\pi-\theta_1}\rho^{\prime 2}_{\epsilon}d\theta\leq E^{\epsilon}(\rho_{\epsilon}(\theta))\leq E^{\epsilon_{0}}(\rho_{\epsilon_0}),
        \end{equation*}

        \noindent which implies that

        \begin{align}
            \vert \vert\rho_{\epsilon}\vert \vert_{BV} ~~\operatorname{and} ~~~\epsilon\vert \vert \rho_{\epsilon}\vert \vert_{L^{2}}
        \end{align}
         are both uniformly bounded.
    \end{proof}
    
    Having established the uniform lower bound for $P_{0}^{\epsilon}$, we now prove the following main theorem:

     \begin{theorem}{\label{thm:uni_pre}}
        Suppose that $P_{0}^{\epsilon}$ is the Euler-Lagrange equation corresponding to energy functional $E^{\epsilon}$. Then $P_{0}^{\epsilon}$ is uniformly bounded for any $0<\epsilon\leq\epsilon_{0}$ where $\epsilon_{0}$ is any given positive constant.
    \end{theorem}
    
    \begin{proof}
    
   Multiplying both sides of the Euler-Lagrange equation (the first equation in the system \eqref{equ:2.1.5}) by $\rho_{\epsilon}$ and integrating it from $\theta_2$ to $\pi-\theta_1$, we have

    \begin{align}
         0=&\int_{\theta_2}^{\pi-\theta_1} \rho_{\epsilon}^{3}\sin\theta d\theta+\sigma\int_{\theta_2}^{\pi-\theta_1} \frac{\rho_{\epsilon}^{2}}{\sqrt{\rho_{\epsilon}^{2}+\rho_{\epsilon}'^{2}}}-\int_{\theta_2}^{\pi-\theta_1}\rho_{\epsilon}\partial_{\theta}\frac{\rho_{\epsilon}'}{\sqrt{\rho_{\epsilon}^{2}+\rho_{\epsilon}'^{2}}}\notag\\
         &-\epsilon\int_{\theta_2}^{\pi-\theta_1}\rho_{\epsilon}''\rho_{\epsilon}d\theta- P_{0}^{\epsilon}\int_{\theta_2}^{\pi-\theta_1} \rho_{\epsilon}^{2} d\theta\label{equ:2.1.10}
    \end{align}

    \noindent Then using integration by part to deal with the third and fourth terms on the right hand side of equation\eqref{equ:2.1.10} and applying the boundary condition(the second and third equations in the system \eqref{equ:2.1.5}), we obtain

    \begin{align*}
        &\int_{\theta_2}^{\pi-\theta_1} \rho_{\epsilon}^{3}\sin\theta d\theta+\sigma\int_{\theta_2}^{\pi-\theta_1} \frac{\rho_{\epsilon}^{2}}{\sqrt{\rho_{\epsilon}^{2}+\rho_{\epsilon}'^{2}}}+\sigma\int_{\theta_2}^{\pi-\theta_1} \frac{\rho_{\epsilon}'^{2}}{\sqrt{\rho_{\epsilon}^{2}+\rho_{\epsilon}'^{2}}}\\
        &-{[\![\gamma]\!]}(\rho(\theta_2)+\rho(\pi-\theta_1))-P_{0}^{\epsilon}\int_{\theta_2}^{\pi-\theta_1} \rho_{\epsilon}^{2} d\theta=0
    \end{align*}

    \noindent Plugging equation \eqref{equ:2.1.2} into the equation above to substitute the integral $\int_{\theta_2}^{\pi-\theta_1}\rho_{\epsilon}^{2} d\theta$, we derive the following inequality

    \begin{align}
         P_{0}^{\epsilon} V\leq & \int_{\theta_2}^{\pi-\theta_1} \rho_{\epsilon}^{3}\sin\theta d\theta+\sigma\int_{\theta_2}^{\pi-\theta_1} \frac{\rho_{\epsilon}^{2}}{\sqrt{\rho_{\epsilon}^{2}+\rho_{\epsilon}'^{2}}}d\theta \notag\\
         &+\epsilon\int_{\theta_2}^{\pi-\theta_1}\rho_{\epsilon}'^{2}d\theta+\int_{\theta_2}^{\pi-\theta_1} \frac{\rho_{\epsilon}'^{2}}{\sqrt{\rho_{\epsilon}^{2}+\rho_{\epsilon}'^{2}}}d\theta -{[\![\gamma]\!]}(\rho(\theta_2)+\rho(\pi-\theta_1)) \label{equ:2.1.11}
    \end{align}

    \noindent  We now compute each term on the right hand side of the equation \eqref{equ:2.1.11} individually. For the first term, we have

    \begin{equation}{\label{equ:102}}
        \int_{\theta_2}^{\pi-\theta_1} \rho_{\epsilon}^{3}\sin\theta d\theta\leq (\pi-\theta_1-\theta_2) (\operatorname{BV}(\rho))^{3}
    \end{equation}

    \noindent which is uniformly bounded by using Lemma 2.2.

    For the second term, we have

    \begin{equation}{\label{equ:103}}
        \sigma \int_{\theta_2}^{\pi-\theta_1}\frac{\rho_{\epsilon}^{2}}{\sqrt{\rho_{\epsilon}^{2}+\rho_{\epsilon}'^{2}}} d\theta\leq \sigma \int_{\theta_2}^{\pi-\theta_1} \vert \rho_{\epsilon}\vert d\theta \leq \sigma (\pi-\theta_1-\theta_2) \operatorname{BV}(\rho)
    \end{equation}

    \noindent which is also bounded by using Lemma 2.2.

    For the third term, it is uniformly bounded using Lemma 2.2.

    For the fourth term, we have

     \begin{equation}{\label{equ:104}}
        \sigma \int_{\theta_2}^{\pi-\theta_1} \frac{\rho_{\epsilon}'^{2}}{\sqrt{\rho_{\epsilon}^{2}+\rho_{\epsilon}'^{2}}}d \theta \leq \sigma \int_{\theta_2}^{\pi-\theta_1} \vert \rho_{\epsilon}'\vert d\theta \leq \sigma (\pi-\theta_1-\theta_2) \operatorname{BV}(\rho_{\epsilon})
    \end{equation}

    \noindent which is again uniformly bounded by using Lemma 2.2.

    For the last term, we have:

    \begin{equation}{\label{equ:105}}
        -[\![\gamma]\!](\rho(\theta_2))+\rho(\pi-\theta_1))\leq 2|[\![\gamma]\!]|\operatorname{BV}(\rho)
    \end{equation}

    \noindent which is uniformly bounded by using Lemma 2.2.

     Combining the results in equations \eqref{equ:102}-\eqref{equ:105} and plugging them into the equation \eqref{equ:2.1.11},  we then derive the upper bound for $P_{0}^{\epsilon}$. Moreover, $P_{0}^{\epsilon}$ is positive from Theorem 3.1, which means that it is bounded from below. Hence, $P_{0}^{\epsilon}$ is uniformly bounded.

    \end{proof}
    
    Lemma 3.2 provides us a uniform BV bound for the minimizer $\rho_{\epsilon}$. This regularity is still not good enough for us to let $\epsilon\rightarrow 0$. To overcome this difficulty,  we then derive Theorem to show the uniform bound for  $\rho_{\epsilon}'$. Before stating the main theorem, we should first establish the following lemma:

    \begin{lemma}

        The following relation holds for all $\theta\in (\theta_{2},\pi-\theta_{1})$ when $\epsilon$ is small enough.

        \begin{equation*}
        g\rho_{\epsilon}(\theta)\sin\theta-P_{0}^{\epsilon}\leq0
        \end{equation*}
    \end{lemma}

    \begin{proof}

        Using the boundary condition

        \begin{equation}{\label{equ:1.1.100}}
            (\sigma\frac{1}{\sqrt{\rho_{\epsilon}^{2}+\rho_{\epsilon}'^{2}}}+\epsilon)\rho_{\epsilon}'(\pi-\theta_1)=[\![\gamma]\!],
        \end{equation}

        \noindent it holds that at the angle $\rho_{\epsilon}$ is decreasing at $\theta=\pi-\theta_1$. Moreover, 

        \begin{equation}{\label{equ:1.1.200}}
            (\sigma\frac{1}{\sqrt{\rho_{\epsilon}^{2}+\rho_{\epsilon}'^{2}}}+\epsilon)\rho
            _{\epsilon}'(\theta_2)=-[\![\gamma]\!]
        \end{equation}

        \noindent implies that $\rho_{\epsilon}'(\theta_2)>0$ and $\rho_{\epsilon}$ is increasing at the point $\theta_2$. Combining \eqref{equ:1.1.100} and \eqref{equ:1.1.200} together, we show that the maximum point of $\rho_{\epsilon}(\theta)\sin\theta$ can only be achieved in the interior. Suppose that $\theta_{\epsilon}$ is a point such that
        \begin{equation*}
            g\rho_{\epsilon}(\theta_{\epsilon})\sin\theta_{\epsilon}-P_{0}^{\epsilon}=\max_{\theta\in (\theta_{2},\pi-\theta_{1})} (g\rho_{\epsilon}(\theta)\sin\theta-P_{0}^{\epsilon}),
        \end{equation*}
        which is equivalent to that $\theta_{\epsilon}$ is the maximum point for $\rho_{\epsilon}\sin\theta$. This implies the following relation:
        \begin{align}{\label{equ:max}}
            \rho_{\epsilon}(\theta_{\epsilon})\cos\theta_{\epsilon}=-\rho'_{\epsilon}(\theta_{\epsilon})\sin\theta_{\epsilon}
        \end{align}
        Then it remains to show that:
        \begin{equation*}
            g\rho_{\epsilon}(\theta_{\epsilon})\sin\theta_{\epsilon}-P_{0}^{\epsilon}\leq 0 
        \end{equation*}
        
        \noindent We use a contradiction argument to prove this.
        
        Suppose that
        
        \begin{equation*}
            g\rho_{\epsilon}(\theta_{\epsilon})\sin\theta_{\epsilon}-P_{0}^{\epsilon}>C.
        \end{equation*}

        \noindent for some constant $C$. For simplification, in the following computation, we use $\rho$ instead of $\rho_{\epsilon}$ to denote the minimizer of $E^{\epsilon}$. From the Euler-Lagrange equation(Equation \eqref{equ:2.1.5}), we have

        \begin{equation*}
        \begin{aligned}
        (\sigma\frac{\rho}{(\rho^{2}+\rho'^{2})^{\frac{3}{2}}}+\epsilon)(\rho\sin\theta)''=&(\sigma\frac{2\rho'^{2}+\rho^{2}}{\sqrt{\rho_{\epsilon}^{2}+\rho_{\epsilon}'^{2}}^{\frac{3}{2}}}+g\rho\sin\theta-P_{0}^{\epsilon})\sin\theta+2(\sigma\frac{\rho}{(\rho^{2}+\rho'^{2})^{\frac{3}{2}}}+\epsilon)(\rho'\cos\theta)\\
        &-(\sigma\frac{\rho}{(\rho^{2}+\rho'^{2})^{\frac{3}{2}}}+\epsilon)(\rho\sin\theta).
        \end{aligned}
        \end{equation*}
        
        \noindent When $\theta=\theta_{\epsilon}$, applying equation \eqref{equ:max} to the equation above we have:

        \begin{equation*}
           \begin{aligned}
        (\sigma\frac{\rho}{(\rho^{2}+\rho'^{2})^{\frac{3}{2}}}+\epsilon)(\rho\sin\theta)''|_{\theta=\theta_{\epsilon}}=&\big(\sigma\frac{2\rho'^{2}+\rho^{2}}{\sqrt{\rho_{\epsilon}^{2}+\rho_{\epsilon}'^{2}}^{\frac{3}{2}}}+g\rho\sin\theta-P_{0}^{\epsilon})\sin\theta-2(\sigma\frac{\rho'^{2}\sin\theta}{(\rho^{2}+\rho'^{2})^{\frac{3}{2}}})-\epsilon(\sigma{\rho\cot\theta\cos\theta})\big)|_{\theta=\theta_{\epsilon}}\\
        &-\big((\sigma\frac{\rho}{(\rho^{2}+\rho'^{2})^{\frac{3}{2}}})(\rho\sin\theta)+\epsilon \rho\sin\theta\big)|_{\theta=\theta_{\epsilon}},
        \end{aligned}
        \end{equation*}
        \noindent which implies that
        \begin{equation*}
            \begin{aligned}
        (\sigma\frac{\rho}{(\rho^{2}+\rho'^{2})^{\frac{3}{2}}}+\epsilon)(\rho\sin\theta)''|_{\theta=\theta_{\epsilon}}=&\big(g\rho\sin\theta-P_{0}^{\epsilon})\sin\theta-\epsilon(\sigma{\rho\cot\theta\cos\theta})\big)|_{\theta=\theta_{\epsilon}}-\big(\epsilon \rho\sin\theta\big)|_{\theta=\theta_{\epsilon}},
        \end{aligned}
        \end{equation*}
        \noindent Since $\rho_{\epsilon}(\theta)$ is uniformly bounded, and suppose $\theta_{1},\theta_{2}\in(0,\pi)$, we can choose $\epsilon$ to be small enough such that
        \begin{align}
            |\epsilon(\sigma{\rho\cot\theta\cos\theta})|_{\theta=\theta_{\epsilon}}-\big(\epsilon \rho\sin\theta\big)|_{\theta=\theta_{\epsilon}}|<C\min(\sin\theta_{1},\sin\theta_{2})
        \end{align}
        \noindent Therefore, we have

        \begin{equation*}
            (\rho_{\epsilon}\sin\theta)''(\theta_{\epsilon})>0.
        \end{equation*}

        \noindent This contradicts to the assumption that $\theta_{\epsilon}$ is the maximum point for $\rho_{\epsilon}\sin\theta$.

        When $\theta_{1}=0$ or $\theta_{2}=0$, using the uniform BV bound and the positivity of $P_{0}$, we obtain
        \begin{align}
            g\rho_{\epsilon}(\theta)\sin\theta-P_{0}\leq g\rho_{\epsilon}(0)\sin0-P_{0}  +\frac{P_{0}}{2}\leq -\frac{P_{0}}{2}<0
        \end{align}
        \noindent for any $\theta\in (\pi-\theta_{1}-\delta,\pi-\theta_{1})\cup (\theta_{2},\theta_{2}+\delta)$ for some positive $\delta$ independent of $\epsilon$. Then we use the discussion for the case when $\theta_{1},\theta_{2}\neq 0$ to prove the statement in this special case. 
    \end{proof}

    Then we state the key lemma establishing the uniform boundedness of $\rho'$
    
    \begin{theorem}
    When $\theta_1\neq 0$ and $\theta_2\neq 0$, $\rho_{\epsilon}$ is the minimizer for $E^{\epsilon}$. Then $\rho^{\prime}(\theta)$ is uniformly bounded in $L^{\infty}$.
    \end{theorem}
    
    \begin{proof}

    We prove it via a contradiction argument.
    
     Suppose that there is a sequence $\epsilon_{n}$ and $\theta_n$ such that $\rho_{\epsilon_{n}}^{\prime}(\theta_n)\rightarrow +\infty$.  Taking derivative with respect to $\theta$ on both sides of the Euler-Lagrange equation (equation \eqref{equ:2.1.5}), we obtain the following equation
     
     \begin{align}
      &(\frac{\rho_{\epsilon}}{(\rho_{\epsilon}^{2}+\rho_{\epsilon}'^{2})^{\frac{3}{2}}}+\epsilon)\rho_{\epsilon}'''+(\frac{\rho_{\epsilon}'(\rho_{\epsilon}^{2}+\rho_{\epsilon}'^{2})-3\rho_{\epsilon}(\rho_{\epsilon}+\rho_{\epsilon}'')\rho_{\epsilon}'}{(\rho_{\epsilon}^{2}+\rho_{\epsilon}'^{2})^{\frac{5}{2}}})\rho_{\epsilon}''\notag\\=&\frac{(4\rho_{\epsilon}''+2\rho_{\epsilon})\rho_{\epsilon}'(\rho_{\epsilon}^{2}+\rho_{\epsilon}'^{2})-3(2\rho_{\epsilon}'^{2}+\rho_{\epsilon}^{2})\rho_{\epsilon}'(\rho_{\epsilon}+\rho_{\epsilon}'')}{{(\rho_{\epsilon}^{2}+\rho_{\epsilon}'^{2})}^{\frac{5}{2}}}+g\rho_{\epsilon}'\sin\theta+g\rho_{\epsilon}\cos\theta \label{equ:2.1.13}
     \end{align}

     We suppose that ${\theta}^{m}_{n}$ is the point such that:

     \begin{equation*}
         \rho_{\epsilon_{n}}'({\theta}^{m}_{n})=\max_{\theta\in(\theta_2,\pi-\theta_1)} \rho_{\epsilon_{n}}'(\theta)
     \end{equation*}
     
     \noindent Since $\lim_{n\rightarrow +\infty}\rho'_{\epsilon_{n}}(\theta_n)=\infty$, we have $\lim_{n\rightarrow +\infty}\rho'_{\epsilon_{n}}({\theta}_{n}^{m})= +\infty$.  Let $\epsilon=\epsilon_{n}$ and $\theta={\theta}^{m}_{n}$ in equation \eqref{equ:2.1.13}. Since $\theta={\theta}_{n}^{m}$ is the maximum point of $\rho_{\epsilon_{n}}'$, we have $\rho_{\epsilon_{n}}''(\theta_{\epsilon_{n}}^{m})=0$. Using this fact in equation \eqref{equ:2.1.13}, we have:

     \begin{align}
         (\frac{\rho_{\epsilon_{n}}}{(\rho_{\epsilon_{n}}^{2}+\rho_{\epsilon_{n}}'^{2})}+\epsilon_{n})\rho_{\epsilon_{n}}'''({\theta}^{m}_{n})=&\frac{2\rho_{\epsilon_{n}}\rho'_{\epsilon_{n}}(\rho_{\epsilon_{n}}^{2}+\rho_{\epsilon_{n}}'^{2})-3\rho_{\epsilon_{n}}\rho_{\epsilon_n}'(2\rho_{\epsilon_{n}}'^{2}+\rho_{\epsilon_n}^{2})}{(\rho_{\epsilon_{n}}^{2}+\rho_{\epsilon_{n}}'^{2})^{\frac{5}{2}}}({\theta}_n^{m})\notag\\ &+g\rho_{\epsilon_{n}}'({\theta}^{m}_n)\sin{\theta}^{m}_{n}+g\rho_{\epsilon_{n}}({\theta}^{m}_{n})\cos{\theta}^{m}_{n} \label{equ:2.1.14}
     \end{align}

     \noindent For the right-hand side of the equation \eqref{equ:2.1.14}, when $\rho_{\epsilon_{n}}'({\theta}^{m}_{n})\rightarrow +\infty$, the second term dominates the asymptotic behavior and diverges to $+\infty$ given that $\sin{\theta}^{m}_n$ is uniformly bounded from below. Using conditions $\theta_1>0$ and $\theta_2>0$, we have

     \begin{equation*}
         \sin{\theta}^{m}_{n}\geq \min(\sin\theta_1,\sin\theta_2)>0,
     \end{equation*}

     \noindent which shows the uniform lower bound for $\sin\theta_{n}^{m}$. Therefore, applying this fact in \eqref{equ:2.1.14}, we derive the following inequality

     \begin{equation*}
         (\frac{\rho_{\epsilon_{n}}}{(\rho_{\epsilon_{n}}^{2}+\rho_{\epsilon_{n}}'^{2})}+\epsilon_{n})\rho_{\epsilon_{n}}'''({\theta}^{m}_{n})> 0,
     \end{equation*}

     \noindent which implies that:

     \begin{equation*}
         \rho_{\epsilon_{n}}({\theta}^{m}_n)'''>0.
     \end{equation*}

     \noindent This contradicts to the fact that ${\theta}^{m}_{n}$ is the maximum point for $\rho_{\epsilon_{n}}'$ if ${\theta}^{m}_{n}$ is not on the boundary. However, on the boundary, the contact angle condition shows that

     \begin{align}
        \vert (\sigma \frac{1}{\sqrt{\rho_{\epsilon_{n}}^{2}+\rho_{\epsilon_{n}}'^{2}}}+\epsilon_{n})(\rho_{\epsilon_{n}}')\vert=-[\![\gamma]\!]\in(0,\sigma),
     \end{align}

     \noindent which implies that

     \begin{align}
          \vert ( \frac{1}{\sqrt{\rho_{\epsilon_{n}}^{2}+\rho_{\epsilon_{n}}'^{2}}})(\rho_{\epsilon_{n}}')\vert<1.
     \end{align}

     \noindent This shows that $\rho_{\epsilon_{n}}'$ is bounded on the boundary. In conclusion, $\rho_{\epsilon_{n}}'$ is bounded from above.

     We can use the same method to show that $\rho_{\epsilon}'$ is also bounded from below and then we have the theorem proved.
    \end{proof}

    Now we have a uniform bound result for $\rho_{\epsilon}'$ when $\theta_1$ and $\theta_2$ are both positive. We then want to examine the case where $\theta_1=0$ or $\theta_2=0$.
    
    \begin{theorem}
    
    When $\theta_1=0$ or $\theta_2=0$, $\rho_{\epsilon}'$ is uniformly bounded for any $\epsilon>0$.

    \end{theorem}
        
    \begin{proof}

    From the discussion in the last theorem, it suffices to show the boundedness of $\rho_{\epsilon}'$ near the boundary.
    
        \textbf{step 1} Uniformly boundedness at the boundary. 
        
        Using the boundary condition of system \eqref{equ:2.1.5}), at the point $\theta=\theta_2$, we have
        
        \begin{equation}{\label{equ:2.1.15}}
            (\epsilon+\sigma \frac{1}{\sqrt{\rho_{\epsilon}^{2}(\theta_2)+\rho_{\epsilon}'^{2}(\theta_2)}})\rho_{\epsilon}'(\theta_2)=- [\![\gamma]\!]>0,
        \end{equation}

        \noindent which implies that

        \begin{equation}{\label{equ:2.1.16}}
        \rho_{\epsilon}'(\theta_2)> 0.
        \end{equation}

        \noindent Using \eqref{equ:2.1.16} in equation \eqref{equ:2.1.15}, we obtain that

        \begin{equation}{\label{equ:2.1.17}}
            0<\frac{\rho_{\epsilon}'}{\sqrt{\rho_{\epsilon}^{2}+\rho_{\epsilon}'^{2}}}(\theta_2)<-\frac{[\![\gamma]\!]}{\sigma}<C<1
        \end{equation}
        
        \noindent The last inequality holds because of our assumption that $0<[\![\gamma]\!]<\sigma$. Inequality \eqref{equ:2.1.17} yields that:

        \begin{equation*}
            0<\frac{\rho_{\epsilon}'}{\rho_{\epsilon}}(\theta_2)\leq \sqrt{\frac{C^{2}}{1-C^{2}}}.
        \end{equation*}

        \noindent This implies that $\rho_{\epsilon}'(\theta_2)$ is uniformly bounded, since $\rho(\theta_2)$ is uniformly bounded by Lemma 3.2.

        Using the similar discussion, we can show that $\rho_{\epsilon}'(\pi-\theta_1)$ is uniformly bounded.

        \textbf{step 2} Uniform boundedness in a neighborhood of two boundary points.
        
        Now we have a uniform bound at two boundary points. We then compute the derivative of $\frac{\rho_{\epsilon}'}{\rho_{\epsilon}}$ as follows:

        \begin{equation}{\label{equ:2.1.18}}
            (\frac{\rho_{\epsilon}'}{\rho_{\epsilon}})'=\frac{\rho_{\epsilon}''\rho_{\epsilon}-\rho_{\epsilon}'^{2}}{\rho_{\epsilon}^{2}}
        \end{equation}

         In order to derive an inequality for $\rho_{\epsilon}^{\prime}$ and $\rho_{\epsilon}$, we need to eliminate the term with second order spatial derivative in equation \eqref{equ:2.1.18}.  We use the first equation in the system \eqref{equ:2.1.5} to show that

        \begin{equation}{\label{equ:2.1.19}}
            (\frac{\rho}{(\rho^{2}+\rho'^{2})^{\frac{3}{2}}}+\epsilon)\rho''=\frac{2\rho'^{2}+\rho^{2}}{\sqrt{\rho_{\epsilon}^{2}+\rho_{\epsilon}'^{2}}^{3}}+g\rho\sin\theta-P_{0}^{\epsilon}
        \end{equation}
        
        \noindent  Using Lemma 3.4, we have

        \begin{equation*}
        g\rho_{\epsilon}\sin\theta-P_{0}^{\epsilon}\leq 0
        \end{equation*}
        Therefore, substituting this relation in equation \eqref{equ:2.1.19}, we obtain an inequality for the second order derivative of $\rho_{\epsilon}$

        \begin{align}
            \rho_{\epsilon}'' <\frac{\frac{2\rho_{\epsilon}'^{2}+\rho_\epsilon{}^{2}}{\sqrt{\rho_{\epsilon}^{2}+\rho_{\epsilon}'^{2}}^{3}}}{\frac{\rho_{\epsilon}}{\sqrt{\rho_{\epsilon}^{2}+\rho_{\epsilon}'^{2}}^{3}}+\epsilon}<\frac{\frac{2\rho_{\epsilon}'^{2}+\rho_{\epsilon}^{2}}{\sqrt{\rho_{\epsilon}^{2}+\rho_{\epsilon}'^{2}}^{3}}}{\frac{\rho_{\epsilon}}{\sqrt{\rho_{\epsilon}^{2}+\rho_{\epsilon}'^{2}}^{3}}}=\frac{2\rho_{\epsilon}'^{2}+\rho_{\epsilon}^{2}}{\rho_{\epsilon}},
        \end{align}

        \noindent which is equivalent to the following relation

        \begin{equation}{\label{equ:2.1.20}}
            \rho_{\epsilon}\rho_{\epsilon}''\leq 2\rho_{\epsilon}'^{2}+\rho_{\epsilon}^{2}.
        \end{equation}

         Plugging the inequality \eqref{equ:2.1.20} back into the equation \eqref{equ:2.1.18} to eliminate the second derivative, we obtain that

        \begin{equation}{\label{equ:2.1.21}}
            (\frac{\rho_{\epsilon}'}{\rho_{\epsilon}})'\leq 1+\frac{\rho_{\epsilon}'^{2}}{\rho_{\epsilon}^{2}}
        \end{equation}

        \noindent Combining inequality \eqref{equ:2.1.21} and the boundedness of $\rho_{\epsilon}'$ on the boundary, we derive the following result by Gronwall type argument:

        \begin{equation}{\label{equ:2.1.22}}
            \rho_{\epsilon}'(\theta)\leq C\rho_{\epsilon}(\theta)
        \end{equation}

        \noindent for some constant $C$ and all $\theta\in (\theta_2,\theta_2+\delta)\cup (\pi-\theta_1-\delta,\pi-\theta_1)$. The two constants $C$ and $\delta$ are both independent of $\epsilon$.

        Moreover, it is also important to show the following inequality

        \begin{equation*}
            \rho_{\epsilon}'\geq -C\rho_{\epsilon}
        \end{equation*}

        \noindent To prove this we construct a new function $\tilde{\rho}$ such that $\tilde{\rho}(\theta)=\rho(\pi+\theta_2-\theta_1-\theta)$. It is easy to show that this function satisfies the same equation \eqref{equ:2.1.19}. Hence, we obtain the Gronwall type estimate for this new constructed function

        \begin{equation*}
            \tilde{\rho}_{\epsilon}'(\theta)\leq C\tilde{\rho}_{\epsilon}(\theta)
        \end{equation*}

        \noindent for all $\theta\in(\theta_2,\theta_2+\delta)\cup (\pi-\theta_1-\delta,\pi-\theta_1)$. This means that

        \begin{equation}{\label{equ:2.1.23}}
            \rho_{\epsilon}'(\theta)\geq -C\rho_{\epsilon}(\theta)
        \end{equation}

        \noindent for all  $\theta\in(\theta_2,\theta_2+\delta)\cup (\pi-\theta_1-\delta,\pi-\theta_1)$. Combining \eqref{equ:2.1.22} and \eqref{equ:2.1.23} together and using Lemma 2.2 again, we derive that $\rho'(\theta)$ is bounded for any  $\theta\in(\theta_2,\theta_2+\delta)\cup (\pi-\theta_1-\delta,\pi-\theta_1)$. 

        \textbf{Step 3} The uniform boundedness for $\rho_{\epsilon}'(\theta)$. 

        We use the same discussion as in Theorem 3.5. $\epsilon_n$ and ${\theta}_{n}^{m}$ are as defined in Theorem 3.5. From the proof in Theorem 3.5, it suffices to derive a uniform positive lower bound for $\sin\tilde\theta_n$. By the definition of $\tilde{\theta}_n$, we know that $\lim_{n\rightarrow}\rho_{\epsilon_{n}}'(\tilde\theta_{n})=+\infty$.  From Step 1 and Step 2,  $\rho_{\epsilon_{n}}'(\theta)$ is uniformly bounded for all  $\theta\in(\theta_2,\theta_2+\delta)\cup (\pi-\theta_1-\delta,\pi-\theta_1)$. Therefore, we have

        \begin{equation*}
            {\theta}_{n}^{m}\in(\theta_2+\delta,\pi-\theta_1-\delta)
        \end{equation*}

        \noindent where $\delta>0$ is independent of $\epsilon$. Hence, we have

        \begin{equation*}
            \sin({\theta}_{n}^{m})>\min(\sin(\theta_2+\delta),\sin(\pi-\theta_1-\delta))>0.
        \end{equation*}

        \noindent Therefore,  $\rho_{\epsilon}'(\theta)$ is uniformly bounded for all $\theta\in(\theta_2,\pi-\theta_1)$ and $\epsilon>0$.
    \end{proof}
    
    From Theorem 3.5, $P_{0}^{\epsilon}$ is bounded from above by some constant independent of $\epsilon$. Also, $\rho_{\epsilon}$ is uniformly bounded in BV. If we want to pass the limit, we need to show at least the uniform $H^{1}$ bound. This is established in the following Lemma.

    \begin{lemma}

    $\rho_{\epsilon}$ is the minimizer of $E^{\epsilon}$ subject to the conservation of total mass $V$. Then it is uniformly bounded in $H^{1}$.
    
    \end{lemma}

    \begin{proof}

    From Theorem 2.6 and Lemma 2.2 both $\vert \rho_{\epsilon}'(\theta)\vert$ and $\rho_{\epsilon}(\theta)$ are bounded by a constant $C$ independent of $\epsilon$. Hence, we have:

    \begin{equation*}
        \int_{\theta_2}^{\pi-\theta_1}\rho_{\epsilon}(\theta)^{2}d\theta \leq C^{2}(\pi-\theta_1-\theta_2),
    \end{equation*}

    \noindent and

    \begin{equation*}
        \int_{\theta_2}^{\pi-\theta_1}\rho_{\epsilon}(\theta)'^{2}d\theta \leq C^{2}(\pi-\theta_1-\theta_2).
    \end{equation*}

    \noindent Therefore, $\rho_{\epsilon}$ is uniformly bounded in $H^{1}$.
    
    \end{proof}

     From Theorem 3.7, $\rho_{\epsilon}$ is uniformly bounded in $H^{1}$. Let $\epsilon\rightarrow 0$. Then $\rho_{\epsilon}\rightarrow \rho_{0}$ weakly in $H^{1}$ and strongly in $L^{p}$ for any $1<p<\infty$ by Banach-Alaoglu theorem. Then we can easily show from the weakly lower semicontinuity of the functional $E$ that

    \begin{equation}{\label{equ:1.1.300}}
        E(\rho_{0})\leq \liminf_{\epsilon\rightarrow 0} E^{\epsilon}(\rho_{\epsilon})
    \end{equation}

    \noindent Finally, it remains to show that $\rho_{0}$ is exactly the minimizer of energy $E$. 

    \begin{theorem}{\label{thm:mini}}
    $\rho_{0}$ is the function defined above. It is the minimizer in the function set $H^{1}$ of energy functional $E$
    \end{theorem}

    \begin{proof}

        Suppose that $\rho_{0}$ is not the minimizer of the functional energy $E$. Then there exits another function $\rho_{1}\in H^{1}$ such that

        \begin{equation*}
            E(\rho_{1})<E(\rho_{0})-2\delta
        \end{equation*}

        \noindent for some small $\delta>0$. Then using \eqref{equ:1.1.300}, we choose a small $\epsilon>0$ such that

        \begin{equation*}
             E^{\epsilon}(\rho_{1})<E(\rho_{0})-\delta\leq E^{\epsilon}(\rho_{\epsilon})
        \end{equation*}

        \noindent which contradicts to the fact that $\rho_{\epsilon}$ is the minimizer of functional energy $E^{\epsilon}$. Hence $\rho_{0}$ is the minimizer of functional $E$ in the set of functions $H^{1}$.
        
    \end{proof}

    Now we already have a $H^{1}$ minimizer for the functional $E$. We want to improve the regularity so that it is the strong solution to the Euler-Lagrange equation. We enhance the regularity for $\rho_{0}$ in the following lemma

    \begin{theorem}

        The minimizer  $\rho_{0}$ derived in Theorem \ref{thm:mini} is smooth.
    \end{theorem}

    \begin{proof}

    Since $\rho_{0}$ is a $H^{1}$ function, we know that it solves the equation in the weak sense. Let $D_{\theta}^{h}$ be an operator such that $D_{\theta}^{h}(\rho)(\theta)=\frac{\rho(\theta+h)-\rho(\theta)}{h}$. Then we choose the test function $\phi D_{h}^{\theta}D_{h}^{\theta}(\phi \rho_{0})$ where $\phi$ is a cut off function such that it equals zero on the boundary. Applying this cut off function, we obtain:

        \begin{align}
        \begin{aligned}
            &\int_{\theta_2}^{\pi-\theta_1}\rho_{0}^{2}\sin\theta \phi D_{h}^{\theta}D_{h}^{\theta} (\phi\rho_{0}) d\theta+\int_{\theta_2}^{\pi-\theta_1} \frac{\rho_{0}}{\sqrt{\rho_{0}^{2}+\rho_{0}'^{2}}} \phi D_{h}^{\theta}D_{h}^{\theta} (\phi\rho_{0})d\theta\\
            &+\int_{\theta_2}^{\pi-\theta_1} \frac{\rho_{0}'}{\sqrt{\rho_{0}^{2}+\rho_{0}'^{2}}}\phi D_{h}^{\theta}D_{h}^{\theta}D_{\theta} (\phi\rho_{0})d\theta -\int_{\theta_2}^{\pi-\theta_1} P_{0} \rho_{0}\phi D_{h}^{\theta}D_{h}^{\theta}(\phi\rho_{0})=0
            \end{aligned}
        \end{align}

        \noindent Then by integration by part, we obtain:

        \begin{align}{\label{equ:324}}
        \begin{aligned}
            &\int_{\theta_2}^{\pi-\theta_1} D_{h}^{\theta}(\phi\rho_{0}^{2}\sin\theta)D_{h}^{\theta}(\phi\rho_{0})d\theta+\int_{\theta_2}^{\pi-\theta_1} D_{h}^{\theta}\frac{\phi\rho_{0}}{\sqrt{\rho_{0}^{2}+\rho_{0}'^{2}}} D_{h}^{\theta}(\phi\rho_{0})d\theta\\
            &+\int_{\theta_2}^{\pi-\theta_1}D_{\theta}^{h}(\frac{\phi\rho_{0}'}{\sqrt{\rho_{0}^{2}+\rho_{0}'^{2}}})D_{h}^{\theta}D_{\theta}(\phi\rho_{0})d\theta -\int_{\theta_2}^{\pi-\theta_1}P_{0}D_{h}^{\theta}(\rho_{0}\phi) D_{h}^{\theta}(\phi\rho_{0})=0
            \end{aligned}
        \end{align}

       \noindent In \eqref{equ:324} above, we need to compute the term $D_{\theta}^{h}(\frac{\phi \rho_{0}'}{\sqrt{\rho_{0}^{2}+\rho_{0}'^{2}}})$ as follows:

        \begin{equation}\label{equ:dis_dif}
        \begin{aligned}
         D_{\theta}^{h}(\frac{\phi \rho_{0}'}{\sqrt{\rho_{0}^{2}+\rho_{0}'^{2}}})&=\frac{1}{h}(\frac{\phi \rho_{0}'}{\sqrt{\rho_{0}^{2}+\rho_{0}'^{2}}}(\theta+h)-\frac{\phi \rho_{0}'}{\sqrt{\rho_{0}^{2}+\rho_{0}'^{2}}}(\theta))\\
            &=\frac{1}{h}\frac{\phi\rho'(\theta+h)\sqrt{\rho_{0}^{2}+\rho_{0}'^{2}}(\theta)-\phi\rho'(\theta)\sqrt{\rho_{0}'^{2}+\rho_{0}^{2}}(\theta+h)}{(\sqrt{\rho_{0}^{2}+\rho_{0}'^{2}}(\theta+h)\sqrt{\rho_{0}^{2}+\rho_{0}'^{2}}(\theta))}\\
            &=\frac{1}{h}\frac{\phi\rho'(\theta+h)\sqrt{\rho_{0}^{2}+\rho_{0}'^{2}}(\theta)-\phi\rho'(\theta)\sqrt{\rho_{0}^{2}+\rho_{0}'^{2}}(\theta)}{\sqrt{\rho_{0}^{2}+\rho_{0}'^{2}}(\theta+h)\sqrt{\rho_{0}^{2}+\rho_{0}'^{2}}(\theta)} \\
            &\quad+\frac{1}{h}\frac{\phi\rho'(\theta)\sqrt{\rho_{0}^{2}+\rho_{0}'^{2}}(\theta)-\phi\rho'(\theta)\sqrt{\rho_{0}^{2}+\rho_{0}'^{2}}(\theta+h)}{\sqrt{\rho_{0}^{2}+\rho_{0}'^{2}}(\theta+h)\sqrt{\rho_{0}^{2}+\rho_{0}'^{2}}(\theta)}\\
            &=\frac{\sqrt{\rho_{0}^{2}+\rho_{0}'^{2}}(\theta)D_{h}^{\theta}(\phi\rho_{0}')}{\sqrt{\rho_{0}^{2}+\rho_{0}'^{2}}(\theta+h)\sqrt{\rho_{0}^{2}+\rho_{0}'^{2}}(\theta)} \\
             &\quad+\frac{1}{h}\frac{(\rho_{0}'^{2}(\theta)-\rho_{0}'^{2}(\theta+h)+\rho_{0}^{2}(\theta)-\rho_{0}^{2}(\theta+h))(\phi\rho_{0}')(\theta)}{\sqrt{\rho_{0}^{2}+\rho_{0}'^{2}}(\theta+h)\sqrt{\rho_{0}^{2}+\rho_{0}'^{2}}(\theta)(\sqrt{\rho_{0}^{2}+\rho_{0}'^{2}}(\theta)+\sqrt{\rho_{0}^{2}+\rho_{0}'^{2}}(\theta+h))}\\
             &=\frac{\phi D_{h}^{\theta}\rho_{0}'}{\sqrt{\rho_{0}^{2}+\rho_{0}'^{2}}(\theta+h)}+\frac{\rho_{0}' D_{h}^{\theta}\phi}{\sqrt{\rho_{0}^{2}+\rho_{0}'^{2}}(\theta+h)}\\
             &\quad - \frac{(\rho_{0}'(\theta)+\rho_{0}'(\theta+h))(D_{h}^{\theta}\rho_{0}')\phi\rho_{0}'(\theta)+(\rho_{0}(\theta)+\rho_{0}(\theta+h))(D_{h}^{\theta}\rho_{0})\phi\rho_{0}'(\theta)}{\sqrt{\rho_{0}^{2}+\rho_{0}'^{2}}(\theta+h)\sqrt{\rho_{0}^{2}+\rho_{0}'^{2}}(\theta)(\sqrt{\rho_{0}^{2}+\rho_{0}'^{2}}(\theta)+\sqrt{\rho_{0}^{2}+\rho_{0}'^{2}}(\theta+h))}
            \end{aligned}
        \end{equation}

         Using H\"older inequality, we have:
        \begin{align}
            \int_{\theta_{2}}^{\pi-\theta_{1}}|D_{h}^{\theta}(\phi\rho_{0})'|^{\alpha}=&\int_{\theta_{2}}^{\pi-\theta_{1}}|D_{h}^{\theta}(\phi\rho_{0})'|^{\alpha}\frac{1}{(\rho_{0}^{2}+\rho_{0}'^{2})^{\frac{\alpha}{4}}}(\rho_{0}^{2}+\rho_{0}'^{2})^{\frac{\alpha}{4}}\notag\\
            &\leq (\int_{\theta_{2}}^{\pi-\theta_{1}}|D_{h}^{\theta}(\phi\rho_{0})'|^{2}\frac{1}{(\rho_{0}^{2}+\rho_{0}'^{2})^{\frac{1}{2}}})^{\frac{\alpha}{2}}(\int_{\theta_{2}}^{\pi-\theta_{1}}({\rho_{0}^{2}+\rho_{0}'^{2}})^{\frac{2}{2-\alpha}})^{\frac{2-\alpha}{2}}
        \end{align}

        \noindent for any $\alpha<2$. Combining this equation and equation \eqref{equ:dis_dif}, we have
        \begin{equation}{\label{equ:prelim}}
        \begin{aligned}
            &(\int_{\theta_{2}}^{\pi-\theta_{1}}({\rho_{0}^{2}+\rho_{0}'^{2}})^{\frac{2}{2-\alpha}})^{-\frac{2-\alpha}{2}}\int_{\theta_2}^{\pi-\theta_1}(D_{\theta}^{h}D_{\theta}(\rho_{0}\phi))^{\alpha}\notag\\
            &\leq \frac{1}{2}\delta\int_{\theta_2}^{\pi-\theta_1}(D_{\theta}^{h}D_{\theta}(\rho_{0}\phi))^{\alpha}+\frac{C''}{\epsilon}(\vert \vert \rho_{0}\vert \vert_{W^{1,\frac{2}{2-\alpha}}}+\vert \vert \phi\vert \vert_{H^{2}})
            \end{aligned}
        \end{equation}
        
        \noindent where $C''$ is some number depending on $\|\rho_{0}\|_{W^{1,k}}$ for some large number $k$. Choosing 
        \begin{align}
            \delta<\frac{1}{2}(\int_{\theta_{2}}^{\pi-\theta_{1}}({\rho_{0}^{2}+\rho_{0}'^{2}})^{\frac{2}{2-\alpha}})^{-\frac{2-\alpha}{2}},
        \end{align}
         we derive from \eqref{equ:prelim} that $\rho_{0}''$ is uniformly bounded in $W^{2,\alpha}$ for all $0<\alpha<2$. Hence, from Sobolev embedding, $\rho_{0}\in W^{1,+\infty}$.
        
        We now use the fact that $\rho_{0}\in W^{1,+\infty}$ to discuss the main term in equation \eqref{equ:324}.  From computation \eqref{equ:dis_dif}, we have
        
        \begin{equation}{\label{equ:325}}
            \int_{\theta_2}^{\pi-\theta_1}D_{\theta}^{h}(\frac{\phi\rho_{0}'}{\sqrt{\rho_{0}^{2}+\rho_{0}'^{2}}})D_{\theta}^{h}D_{\theta}(\phi\rho_{0})\geq C'\int_{\theta_2}^{\pi-\theta_1} (D_{h}^{\theta}(\phi\rho_{0})')^{2}+\mathcal{R}(\rho_{0},\rho_{0}',\phi)
        \end{equation}

        \noindent where $C'$ is some constant only depends on $\vert \vert \rho_{0}\vert \vert_{L^{\infty}}$ and $\vert \vert \rho_{0}'\vert \vert_{L^{\infty}}$, and $\mathcal{R}$ is the combination of all of the lower order terms. Plugging \eqref{equ:325} back into the \eqref{equ:324}, and using H\"older's inequality and Cauchy inequality, we obtain

        \begin{equation}
            C'\int_{\theta_2}^{\pi-\theta_1}(D_{\theta}^{h}D_{\theta}(\rho_{0}\phi))^{2}
            \leq \frac{1}{2}\epsilon\int_{\theta_2}^{\pi-\theta_1}(D_{\theta}^{h}D_{\theta}(\rho_{0}\phi))^{2}+\frac{C''}{\epsilon}(\vert \vert \rho_{0}\vert \vert_{H^{1}}+\vert \vert \phi\vert \vert_{H^{1}})
        \end{equation}

        \noindent where $C''$ depends only on $\vert \vert \rho_{0}\vert \vert_{W^{1,\infty}}$. Choosing $\epsilon\leq C'$, we obtain
        
        \begin{equation*}
            \frac{1}{2}C'\int_{\theta_2}^{\pi-\theta_1}(D_{\theta}^{h}D_{\theta}(\rho_{0}\phi))^{2}\leq \frac{C''}{\epsilon}(\vert \vert \rho_{0}\vert \vert_{H^{1}}+\vert \vert \phi\vert \vert_{H^{1}})
        \end{equation*}

        \noindent Since $h$ can be arbitrarily chosen, we have $\rho_{0}\in H^{2}$.

        After showing the $H^{2}$ regularity for $\rho_{0}$, we apply a standard elliptic estimate to show that $\rho_{0}$ is smooth.
        
    \end{proof}

    \subsection{ The case when$[\![\gamma]\!]>0$} In this subsection, we discuss the case when $[\![\gamma]\!]>0$
    
   In this case, we have $[![\gamma]!] > 0$. Unlike the previous situation, it is not immediately evident that the energy functional $E$, defined in \eqref{equ:1.1.1}, is bounded from below. Instead, we establish this important property through the following theorem.

    \begin{theorem}
        When $[\![\gamma]\!]>0$, the energy functional \eqref{equ:2.1.1} is bounded from below.
   \end{theorem}

    \begin{proof}

     For any $H^{1}$ function $\rho$, we have:
        \begin{align}{\label{equ:1.2.1}}
        \begin{aligned}
             E(\rho)=&\frac{1}{3}g\int_{\theta_2}^{\pi-\theta_1}\rho^{3}\sin\theta d\theta +\sigma\int_{\theta_2}^{\pi-\theta_1}\sqrt{\rho^{2}+\rho'^{2}}-[\![\gamma]\!](\rho(\theta_2)+\rho(\pi-\theta_1))\\
             \geq& \sigma \int_{\theta_2 }^{\pi-\theta_1}\vert \rho'\vert d\theta-[\![\gamma]\!](\rho(\theta_2)+\rho(\pi-\theta_1))\\
              \geq& \sigma (\rho(\theta_2)-\rho(\theta_{3}))+\sigma(\rho(\pi-\theta_1)-\rho(\theta_{3}))-[\![\gamma]\!](\rho(\theta_2)+\rho(\pi-\theta_1))\\
              =&-2\sigma \rho(\theta_3)+\sigma (\rho(\pi-\theta_1)+\rho(\theta_2))-[\![\gamma]\!](\rho(\theta_2)+\rho(\pi-\theta_1))
            \end{aligned}
        \end{align}

     \begin{equation}
         =-2\sigma \rho(\theta_3)+\sigma (\rho(\pi-\theta_1)+\rho(\theta_2))-[\![\gamma]\!](\rho(\theta_2)+\rho(\pi-\theta_1))
     \end{equation}

    \noindent  where $\theta_{3}$ is an angle between $\theta_2$ and $\pi-\theta_1$ to be determined .
     
     Applying the condition $[\![\gamma]\!]<\sigma$ to the inequality \eqref{equ:1.2.1}, we show

     \begin{equation}{\label{equ:1.2.2}}
         E(\rho)\geq -2\sigma\rho(\theta_3)\geq -2\sigma \vert\rho(\theta_3)\vert
     \end{equation}

     \noindent We then choose $\theta_3$ such that

     \begin{equation*}
         \vert \rho(\theta_3)\vert=\min_{\theta\in(\theta_2,\pi-\theta_1)} \vert \rho(\theta)\vert
     \end{equation*}

     \noindent By using the equation \eqref{equ:2.1.2} and the definition of $\theta_{3}$ above, we have

     \begin{equation*}
         V=\int_{\theta_2}^{\pi-\theta_1} \rho^{2}d\theta\geq (\pi-\theta_1-\theta_2) \rho(\theta_3)^{2}
     \end{equation*}

     \noindent which implies that

     \begin{equation}{\label{equ:1.2.3}}
         \vert \rho(\theta_{3})\vert\leq \sqrt{\frac{V}{\pi-\theta_1-\theta_2}}.
     \end{equation}

     \noindent Plugging \eqref{equ:1.2.3} into the inequality \eqref{equ:1.2.2}, we obtain 

     \begin{equation}{\label{equ:1.2.4}}
         E(\rho)\geq -2\sigma\sqrt{\frac{V}{\pi-\theta_1-\theta_2}}
     \end{equation}

     \noindent Hence, the energy functional is bounded from below.

     \end{proof}

     As a consequence of this theorem, we immediately deduce that $E^{\epsilon}$ is uniformly bounded from below.

As in Theorem~\ref{thm:uni_pre}, we next derive certain boundedness properties of the pressure $P_{0}^{\epsilon}$. We proceed with the following discussion. From a physical perspective, $P_{0}$ represents the pressure along the free surface and should therefore be positive. Since $P_{0}^{\epsilon}$ is a small perturbation of $P_{0}$, it is natural to expect that it remains bounded from below. Guided by this intuition, we establish the following theorem.

     \begin{theorem}{\label{thm:uni_pre1}}
         $P_{0}^{\epsilon}$ is uniformly bounded from below. 
     \end{theorem}

     \begin{proof}

         Suppose, for contradiction that there is a sequence $\epsilon_n$ such that: $P_{0}^{\epsilon_n}\leq -n$. Then from the definition of $\rho_{\epsilon}$ we know that it is the minimizer of the following energy functional

         \begin{equation*}
         \begin{aligned}
             {E}_{p}^{\epsilon}(\rho_{\epsilon})=&\frac{1}{3}g\int_{\theta_2}^{\pi-\theta_1}\rho_{\epsilon}^{3}\sin\theta d\theta +\sigma\int_{\theta_2}^{\pi-\theta_1}\sqrt{\rho_{\epsilon}^{2}+\rho_{\epsilon}'^{2}} d\theta\\
              &\quad-[\![\gamma]\!](\rho_{\epsilon}(\theta_2)+\rho_{\epsilon}(\pi-\theta_1))-P^{\epsilon}_{0}\int_{\theta_2}^{\pi-\theta_1}\rho_{\epsilon}^{2}d\theta
         \end{aligned}
     \end{equation*}
    Using the equation \eqref{equ:1.2.4} in Theorem 2.9, we obtain the following boundedness estimate:

    \begin{equation*}
    \begin{aligned}
         \tilde{E}_{p}^{\epsilon_{n}}(\rho_{\epsilon_{n}})\geq& -2\sigma \sqrt{\frac{V}{\pi-\theta_1-\theta_2}}-P_{0}^{\epsilon_n}\int_{\theta_2}^{\pi-\theta_1}\rho_{\epsilon_{n}}^{2}d\theta\\
          \geq& -2\sigma \sqrt{\frac{V}{\pi-\theta_1-\theta_2}}+n\int_{\theta_2}^{\pi-\theta_1}\rho_{\epsilon_{n}}^{2}d\theta 
    \end{aligned}
    \end{equation*}

     \noindent Choosing $n$ such that

     \begin{equation*}
         n\int_{\theta_2}^{\pi-\theta_1}\rho_{\epsilon_{n}}^{2}d\theta=nV> 2\sigma \sqrt{\frac{V}{\pi-\theta_1-\theta_2}},
     \end{equation*}

     \noindent we then have the following relation

     \begin{equation*}
         \tilde{E}^{\epsilon_{n}}(\rho_{\epsilon_{n}})>0=\tilde{E}^{\epsilon_{n}}(0).
     \end{equation*}

     \noindent which contradicts to the fact that $\rho_{\epsilon_{n}}$ is the non-trivial minimizer of the energy functional $\tilde{E}^{\epsilon_{n}}$.  This finishes the proof.
     
     \end{proof}

     Furthermore, we want to prove the uniform BV bound for $\rho_{\epsilon}$. This is stated in the following theorem.

     \begin{theorem}
         Suppose that $\rho_{\epsilon}$ is the minimizer of energy functional $E^{\epsilon}$ subject to the conservation of total mass. Then it is uniformly bounded in BV norm. And the following norm with respect to $\rho_{\epsilon}$ is also uniformly bounded:

         \begin{equation}{\label{equ:1.2.100}}
              \int_{\theta_2}^{\pi-\theta_1}\sqrt{\rho_{\epsilon}^{2}+\rho_{\epsilon}'^{2}}d\theta+\epsilon\int_{\theta_2}^{\pi-\theta_1}\rho_{\epsilon}'^{2}d\theta
        \end{equation}
     \end{theorem}

     \begin{proof}

     From the definition of energy functional $E^{\epsilon}$, we can transform it to:

     \begin{align*}
         E^{\epsilon}(\rho_{\epsilon})=&\frac{1}{3}g\int_{\theta_2}^{\pi-\theta_1}\rho_{\epsilon}^{3}\sin\theta d\theta+(\sigma-[\![\gamma]\!])\int_{\theta_2}^{\pi-\theta_1} \sqrt{\rho_{\epsilon}^{2}+\rho_{\epsilon}'^{2}}d\theta\\
         &+[\![\gamma]\!]\int_{\theta_2}^{\pi-\theta_1}\sqrt{\rho_{\epsilon}^{2}+\rho_{\epsilon}'^{2}}d\theta-[\![\gamma]\!](\rho_{\epsilon}(\theta_2)+\rho_{\epsilon}(\pi-\theta_1))+\epsilon\int_{\theta_2}^{\pi-\theta_1}\rho_{\epsilon}'^{2}d\theta\label{equ:1.2.11}
     \end{align*}
    
     \noindent Using the computation as in Theorem 3.10, we have

     \begin{equation*}
         E^{\epsilon}(\rho_{\epsilon})\geq \frac{1}{3}g\int_{\theta_2}^{\pi-\theta_1}\rho_{\epsilon}^{3}\sin\theta d\theta+(\sigma-[\![\gamma]\!])\int_{\theta_2}^{\pi-\theta_1}\sqrt{\rho_{\epsilon}^{2}+\rho_{\epsilon}'^{2}}d\theta+\epsilon\int_{\theta_2}^{\pi-\theta_1}\rho_{\epsilon}'^{2}-C
     \end{equation*}

    \noindent  for some constant $C$ depending on the total volume $V$ and $[\![\gamma]\!]$. Combining this with the fact that $E^{\epsilon_1}(\rho_{\epsilon_1})\leq E^{\epsilon_2}(\rho_{\epsilon_2})$ when $\epsilon_1\leq \epsilon_2$, we have the uniform bound for the integral:

    \begin{equation*}
        \int_{\theta_2}^{\pi-\theta_1}\sqrt{\rho_{\epsilon}^{2}+\rho_{\epsilon}'^{2}}d\theta+\epsilon\int_{\theta_2}^{\pi-\theta_1}\rho_{\epsilon}'^{2}d\theta
    \end{equation*}

    \noindent  which implies that $\operatorname{BV}(\rho_{\epsilon})$ and \eqref{equ:1.2.100} are uniformly bounded.
     \end{proof}

     Now we have the uniform lower bound for $P_{0}^{\epsilon}$. We then derive the uniform upper bound for $P_{0}^{\epsilon}$. This is stated in the following Theorem

     \begin{theorem}
         $P_{0}^{\epsilon}$ is uniformly bounded from above for any $\epsilon>0$.
     \end{theorem}

     \begin{proof}

        The proof follows the same argument as that of Theorem 3.3.
     \end{proof}

    Just as the previous case where $[\![\gamma]\!]<0$, we establish the following theorem proving that $\rho_{\epsilon}^{\prime}(\theta)$ is uniformly bounded for any $\theta\in(\pi-\theta_{1},\theta_{2})$ and $\epsilon>0$. 
    
    \begin{theorem}

         When $[\![\gamma]\!]>0$, $\rho_{\epsilon}'(\theta)$ is uniformly bounded for any $\theta\in (\theta_2,\pi-\theta_1)$ , and any $\epsilon>0$. 
         
     \end{theorem}

     \begin{proof}

     We follow the same idea as in the proof of Theorem 3.6.

     \textbf{step 1} The uniform bound at the boundary.

         Using the boundary condition (the second and the third equation of system \eqref{equ:1.1.5}), we have

         \begin{equation}{\label{equ:1.2.5}}
             \sigma (\epsilon+\frac{1}{\sqrt{\rho_{\epsilon}^{2}+\rho_{\epsilon}'^{2}}})(\rho_{\epsilon}')(\theta_2)=-[\![\gamma]\!],
         \end{equation}

         \begin{equation}{\label{equ:1.2.6}}
            \sigma (\epsilon+\frac{1}{\sqrt{\rho_{\epsilon}^{2}+\rho_{\epsilon}'^{2}}})(\rho_{\epsilon}')(\pi-\theta_1)=[\![\gamma]\!],
         \end{equation}

         \noindent from which we derive that $\rho_{\epsilon}'(\theta_2)<0$ and $\rho_{\epsilon}'(\pi-\theta_1)>0$. Therefore, at the point $\theta_2$, we have

         \begin{equation*}
             0>\frac{\rho_{\epsilon}'}{\sqrt{\rho_{\epsilon}'^{2}+\rho_{\epsilon}^{2}}}(\theta_2)>-\frac{[\![\gamma]\!]}{\sigma}>C>-1,
         \end{equation*}

         \noindent which implies that:
         
         \begin{equation*}
             0>\rho_{\epsilon}'(\theta_2)>-\sqrt{\frac{C^{2}}{1-C^{2}}}.
         \end{equation*}

         \noindent Hence, $\rho'(\theta_2)$ is uniformly bounded. Moreover, we use the same discussion to show that $\rho'(\pi-\theta_1)$ is also uniformly bounded.

         \textbf{step 2} The uniform local bound for $\rho_{\epsilon}'$ near the boundary.
         
         We use the same idea as shown in the step 2 of Theorem 3.6. We take the derivative of function $\frac{\rho'}{\rho}$ with respect to $\theta$. And then we use the Euler-Lagrange equation to eliminate the second derivative in the derivative of $\frac{\rho'}{\rho}$. Following this idea, we have:

         \begin{equation}{\label{equ:1.2.7}}
            (\frac{\rho_{\epsilon}'}{\rho_{\epsilon}})'=\frac{\rho_{\epsilon}''\rho_{\epsilon}-\rho_{\epsilon}'^{2}}{\rho_{\epsilon}^{2}}
        \end{equation}

        \noindent  Applying the first equation to the system \eqref{equ:2.1.5} yields

        \begin{equation}{\label{equ:1.2.8}}
            (\frac{\rho_{\epsilon}}{(\rho_{\epsilon}^{2}+\rho_{\epsilon}'^{2})^{\frac{3}{2}}}+\epsilon)\rho_{\epsilon}''=\frac{2\rho_{\epsilon}'^{2}+\rho_{\epsilon}^{2}}{\sqrt{\rho_{\epsilon}^{2}+\rho_{\epsilon}'^{2}}^{3}}+g\rho_{\epsilon}\sin\theta-P_{0}^{\epsilon}
        \end{equation}

         \noindent Using Theorem \ref{thm:uni_pre1} in the inequality \eqref{equ:1.2.8}, we show the following inequality
         \begin{equation}{\label{equ:1.2.9}}
         (\frac{\rho_{\epsilon}}{(\rho_{\epsilon}^{2}+\rho_{\epsilon}'^{2})^{\frac{3}{2}}}+\epsilon)\rho_{\epsilon}''\leq \frac{2\rho_{\epsilon}'^{2}+\rho_{\epsilon}^{2}}{\sqrt{\rho_{\epsilon}^{2}+\rho_{\epsilon}'^{2}}^{3}}+g\rho_{\epsilon}\sin\theta+C,
         \end{equation}

         \noindent from which we obtain

         \begin{equation}{\label{equ:1.2.10}}
             \rho_{\epsilon}''\leq g(\rho_{\epsilon}^{2}+\rho_{\epsilon}'^{2})^{\frac{3}{2}}\sin\theta +\frac{\rho_{\epsilon}^{2}+\rho_{\epsilon}'^{2}}{\rho_{\epsilon}}+\frac{\rho_{\epsilon}'^{2}}{\rho_{\epsilon}}+C\frac{(\sqrt{\rho_{\epsilon}^{2}+\rho_{\epsilon}'^{2}})^{3}}{\rho_{\epsilon}}
         \end{equation}
         \noindent where $C$ is a positive constant independent of $\epsilon$

         Plugging \eqref{equ:1.2.10} back into the equation \eqref{equ:1.2.7}, we obtain:

         \begin{align}
              (\frac{\rho_{\epsilon}'}{\rho_{\epsilon}})'\leq &C+\frac{\rho_{\epsilon}'^{2}}{\rho_{\epsilon}^{2}}+\frac{g(\rho_{\epsilon}^{2}+\rho_{\epsilon}'^{2})^{\frac{3}{2}}}{\rho_{\epsilon}} +C\frac{(\rho_{\epsilon}^{2}+\rho_{\epsilon}'^{2})^{\frac{3}{2}}}{\rho_{\epsilon}^{2}}\notag\\
              =&1+C+(\frac{\rho_{\epsilon}'}{\rho_{\epsilon}^{2}})^{2}+g\rho_{\epsilon}^{2}(1+\frac{\rho_{\epsilon}'^{2}}{\rho_{\epsilon}^{2}})^{\frac{3}{2}} +C\rho_{\epsilon}(1+\frac{\rho_{\epsilon}'^{2}}{\rho_{\epsilon}^{2}})^{\frac{3}{2}} \label{equ:1.2.11}.
         \end{align}

         \noindent  Using Theorem 3.12, $\rho_{\epsilon}$ is uniformly BV bounded.  Hence, applying the Gronwall inequality to the \eqref{equ:1.2.10}, we obtain the uniform local upper bound for $\rho_{\epsilon}'$ near two boundary points. We express it in math as follows:

         \begin{equation*}
             \rho_{\epsilon}'(\theta)\leq C\rho_{\epsilon}(\theta)
         \end{equation*}

         \noindent for some constant $C$ independent of $\epsilon$, and $\theta\in (\theta_2,\theta_2+\delta)\cup(\pi-\theta_1-\delta,\pi-\theta_1)$. Again, we construct a new function $\tilde{\rho}_{\epsilon}(\theta)=\rho_{\epsilon}(\pi-\theta_1+\theta_2-\theta)$ and use the same idea to obtain that:

         \begin{equation*}
             \tilde{\rho}_{\epsilon}'(\theta)\leq C\tilde{\rho_{\epsilon}}(\theta),
         \end{equation*}

         \noindent Therefore,

         \begin{equation*}
             -C\rho_{\epsilon}(\theta) \leq \rho_{\epsilon}'(\theta)\leq C\rho_{\epsilon}(\theta)
         \end{equation*}

         \noindent for all $\theta\in(\theta_2,\theta_2+\delta)\cup(\pi-\theta_1-\delta,\pi-\theta_1)$ with $\delta$ independent of $\epsilon$.
         
         \textbf{step 3} Uniform global bound.

         The proof is the same as the proof of Theorem 3.6.
     \end{proof}

     Now we have the uniform boundedness for the first derivative. The rest of the proof is the same as in case one. Let $\epsilon\rightarrow 0$,  $\rho_{\epsilon}\rightarrow \rho_{0}$ weakly in $H^{1}$. This function $\rho_{0}$ is exactly the smooth minimizer we want.
     
\section{The transformation of the domain(Without moving polar)}

\subsection{The construction of the transformation}
In this section, we discuss a transformation map for any angle $\theta_1$, $\theta_2$. This map maps the changing domain $\Omega(t)$ to the static domain $\Omega(0)$. We suppose that $\xi$ is a function from the interval $(\theta_2,\pi-\theta_2)$ to $\mathbb{R}$. Then we can choose $\mathscr{E}$ to be a bounded linear extension operator that maps $C^{m}(\theta_2,\pi-\theta_1)$ to $C^{m}(\mathbb{R})$ for all $0\leq m\leq 5$ and $W^{s,p}(\theta_2,\pi-\theta_1)$ to
$W^{s,p}(\mathbb{R})$ for all $0\leq s\leq 5$ and $1<p<+\infty$ (such a map is readily constructed with the help of higher-order reflections, Vandermonde matrices, and a cutoff function.

Then we define the extension of $\xi$ as a function $\bar{\xi}:\{(r,\theta)\in \mathbb{R}^{+}\times \mathbb{R}| r\leq E\rho_{0}(\theta)\}\times \mathbb{R}^{+}\rightarrow \mathbb{R}$:

\begin{equation*}
    \bar{\xi}(\theta,r,t)=\mathcal{P}\mathscr{E}\xi(\theta,r-E\rho_{0}(\theta),t)
\end{equation*}

\noindent where $\mathcal{P}$ is the lower Poisson extension of which the definition and properties are given in Appendix B in \cite{Guo}. We then choose $\phi\in C^{\infty}$ such that $\phi(r)=0$ for $r\leq \frac{1}{4}\min\rho_0$ and $\phi(r)=r$ for $r\geq \frac{1}{2}\min\xi_{0}$. Suppose that $\xi$ is the surface perturbation function, we then combine $\phi$ and the extension map $\bar{\xi}$ to define a map from the equilibrium domain $\Omega_{0}$ to the moving domain $\Omega(t)$ by the following definition:

\begin{equation*}
    \Omega \ni (\theta,r)\rightarrow (\theta,r+\frac{\phi(r)}{\rho_{0}(\theta)}\bar{\xi}(\theta,r,t))
\end{equation*}

\noindent We call this map $\Phi$. It will map two boundary walls to boundary walls and also maps the boundary curve of the domain $\Omega_{0}$ to the boundary curve of the domain $\Omega(t)$. Also when the value of $\xi$ is small enough, $\Phi$ is a diffeomorphism and maps $\Omega_0$ to $\Omega(t)$.

Now we have a transformation map. We then want to compute $\nabla \Phi$. We compute it in Cartesian coordinates

\begin{equation*}
    \nabla \Phi=
    \begin{pmatrix}
        \frac{\partial x'}{\partial x} & \frac{\partial x'}{\partial y}\\
        \frac{\partial y'}{\partial x} & \frac{\partial y'}{\partial y}
    \end{pmatrix}
\end{equation*}

We first compute $\frac{\partial x'}{\partial x}$ (notice that $\theta'=\theta$)

\begin{align}
\frac{\partial x'}{\partial x}=&\frac{\partial \rho'\cos\theta'}{\partial x }=\frac{\partial \rho'\cos\theta'}{\partial \rho}\frac{\partial \rho}{\partial x}+\frac{\partial \rho'\cos\theta'}{\partial \theta}\frac{\partial \theta}{\partial x}    \notag\\
=&\cos\theta \frac{\partial \rho'}{\partial \rho} \frac{\partial \rho}{\partial x}+\cos\theta\frac{\partial \rho'}{\partial \theta}\frac{\partial \theta}{\partial x}-\rho'\sin\theta\frac{\partial \theta}{\partial x}\notag\\
=&\frac{\partial \rho'}{\partial \rho}-\cos\theta\frac{y}{x^{2}+y^{2}}\frac{\partial \rho'}{\partial \theta}+\rho'\sin\theta\frac{y}{x^{2}+y^{2}} \label{equ:2.0.1}
\end{align}

\noindent From the definition of the map $\Phi$, we have the following relation

\begin{equation*}
    \frac{\partial \rho'}{\partial \rho}=1+W\partial_{\rho}\bar{\xi}+\frac{\phi'\bar{\xi} }{\rho_{0}},
\end{equation*}

\noindent where $W=\frac{\phi}{\rho_{0}}$. Moreover, we have the following computation

\begin{equation}{\label{equ:2.0.2}}
    \frac{\partial \rho'}{\partial \theta}=W\partial_{t} \bar{\xi}-\frac{W}{\rho_{0}}\partial_{\theta}\rho_{0}\bar{\xi}
\end{equation}

\noindent Then plugging \eqref{equ:2.0.2} back into the representation of $\frac{\partial x'}{\partial x}$ (Equation \eqref{equ:2.0.1}), we have

\begin{equation}{\label{equ:2.0.3}}
    \frac{\partial x'}{\partial x}=(1+W\partial_{\rho}\bar{\xi}+\frac{\phi'\bar{\xi}}{\rho_{0}})-\cos\theta\frac{y}{x^{2}+y^{2}}(W\partial_{\theta}\bar{\xi}-\frac{W}{\rho_{0}}\partial_{\theta}\rho_{0}\bar{\xi})+\rho'\sin\theta\frac{y}{x^{2}+y^{2}}
\end{equation}

\noindent By similar computation to the derivation of \eqref{equ:2.0.3}, we obtain the following computation for $\frac{\p x'}{\p y}$

\begin{equation}{\label{equ:2.0.4}}
    \frac{\partial x'}{\partial y}=\cos\theta\sin\theta(1+W\partial_{\rho}\bar{\xi}+\frac{\phi'\bar{\xi}}{\rho_{0}})+\cos\theta\frac{x}{x^{2}+y^{2}}(W\partial_{\theta} \bar{\xi}-\frac{W}{\rho_{0}}\partial_{\theta}\rho_{0}\bar{\xi})-\rho'\sin\theta\frac{x}{x^{2}+y^{2}},
\end{equation}

\noindent and

\begin{equation}{\label{equ:2.0.5}}
    \frac{\partial y'}{\partial x}=\sin\theta\cos\theta(1+W\partial_{\rho}\bar{\xi}+\frac{\phi'\bar{\xi}}{\rho_{0}})-\sin\theta \frac{y}{x^{2}+y^{2}}(W\partial_{\theta} \bar{\xi}-\frac{W}{\rho_{0}}\partial_{\theta}\rho_{0}\bar{\xi})-\rho'\cos\theta\frac{y}{x^{2}+y^{2}},
\end{equation}

\noindent and
\begin{equation}{\label{equ:2.0.6}}
    \frac{\partial y'}{\partial y}=\sin^{2}\theta(1+W\partial_{\rho}\bar{\xi}+\frac{\phi'\bar{\xi}}{\rho_{0}})+\sin\theta \frac{x}{x^{2}+y^{2}}(W\partial_{\theta} \bar{\xi}-\frac{W}{\rho_{0}}\partial_{\theta}\rho_{0}\bar{\xi})+\rho'\sin\theta\frac{x}{x^{2}+y^{2}}
\end{equation}

Although the representation formulas seem to be complicated, they have the similar form. Or in other words, the highest derivative involved is the first derivative of $\bar{\xi}$, which means that $\nabla \Phi$ can be controlled by $1+\vert \vert \nabla \bar{\eta}\vert \vert_{H^{1}}$. The same type of bound also applies to $(\nabla \Phi)^{-1}$. 

For the inverse matrix, we have

\begin{equation*}
    (\nabla \Phi)^{-1}=
    \begin{pmatrix}
    \frac{\partial x}{\partial x'} & \frac{\partial x}{\partial y'}\\
    \frac{\partial y}{\partial x'} & \frac{\partial y}{\partial y'}
    \end{pmatrix},
\end{equation*}

\noindent We then replace $x$ and $y$ in above equations (\eqref{equ:2.0.3} with \eqref{equ:2.0.6}) to $x'$ and $y'$ to obtain the following four relations

\begin{equation*}
    \frac{\partial x}{\partial x'}=\cos^{2}\theta\frac{\partial \rho}{\partial \rho'}-\cos\theta \frac{y'}{x'^{2}+y'^{2}}\frac{\partial \rho}{\partial \theta'}+\rho\sin\theta \frac{y'}{x'^{2}+y'^{2}}
\end{equation*}

\begin{equation*}
     \frac{\partial x}{\partial y'}=\cos\theta\sin\theta\frac{\partial \rho}{\partial \rho'}+\cos\theta\frac{x'}{x'^{2}+y'^{2}}\frac{\partial \rho}{\partial \theta'}-\rho\sin\theta\frac{x'}{x'^{2}+y'^{2}}
\end{equation*}

\begin{equation*}
    \frac{\partial y}{\partial x'}=\sin\theta\cos\theta\frac{\partial \rho}{\partial \rho'}-\sin\theta \frac{y'}{x'^{2}+y'^{2}}\frac{\partial \rho}{\partial \theta'}-\rho\cos\theta\frac{y'}{x'^{2}+y'^{2}}
\end{equation*}

\begin{equation*}
    \frac{\partial y}{\partial y'}=\sin^{2}\theta\frac{\partial \rho}{\partial \rho'}+\sin\theta \frac{x'}{x'^{2}+y'^{2}}\frac{\partial \rho}{\partial \theta'}+\rho\sin\theta\frac{x'}{x'^{2}+y'^{2}}
\end{equation*}

\noindent Then we make the following decomposition for $(\nabla \Phi)^{-1}$

\begin{equation*}
    (\nabla \Phi)^{-1}=DM+B,
\end{equation*}

\noindent where

\begin{equation*}
    B=
    \begin{pmatrix}
    \rho\sin\theta\frac{y'}{x'^{2}+y'^{2}} & -\rho\sin\theta\frac{x'}{x'^{2}+y'^{2}}\\
    -\rho \cos\theta \frac{y'}{x'^{2}+y'^{2}} & \rho \sin\theta \frac{x'}{x'^{2}+y'^{2}}
    \end{pmatrix},
\end{equation*}
and
\begin{equation*}
    M=
    \begin{pmatrix}
        \frac{\partial \rho'}{\partial x'} & \frac{\partial\rho'}{\partial y'}\\
         \frac{\partial \theta}{\partial x'} & \frac{\partial \theta}{\partial y'}
    \end{pmatrix}
    =
    \begin{pmatrix}
        \cos\theta  & \sin\theta\\
        -\frac{y'}{x'^{2}+y'^{2}} & \frac{x'}{x'^{2}+y'^{2}}
    \end{pmatrix}
    ,
\end{equation*}
and
\begin{equation*}
    D=
    \begin{pmatrix}
        \cos\theta\frac{\partial \rho}{\partial \rho'} & \cos\theta\frac{\partial \rho}{\partial \theta'}\\
        \cos\theta\frac{\partial \rho}{\partial \rho'} & \cos\theta\frac{\partial \rho}{\partial \theta'}.
    \end{pmatrix}
\end{equation*}

\noindent Then it remains to compute $\frac{\partial \rho}{\partial \rho'}$ and $\frac{\partial \rho}{\partial \theta'}$. In fact,  from the definition of the transformation $\Phi$, we have

\begin{equation*}
    \begin{pmatrix}
        \frac{\partial \theta'}{\partial\theta} & \frac{\partial \theta'}{\partial \rho}\\
        \frac{\partial \rho'}{\partial\theta} & \frac{\partial \rho'}{\partial \rho}
    \end{pmatrix}
    =
    \begin{pmatrix}
        1 & 0\\
         W\partial_{\theta} \bar{\xi}-\frac{W}{\rho_{0}}\partial_{\theta}\rho_{0}\bar{\xi} & 1+W\partial_{\rho}\bar{\xi}+\frac{\phi'\bar{\xi} }{\rho_{0}}
    \end{pmatrix}
    .
\end{equation*}

   \noindent Hence,

    \begin{equation*}
         \begin{pmatrix}
        \frac{\partial \theta}{\partial\theta'} & \frac{\partial \theta}{\partial \rho'}\\
        \frac{\partial \rho}{\partial\theta'} & \frac{\partial \rho}{\partial \rho'}
    \end{pmatrix}
    =\frac{1}{1+W\partial_{\rho}\bar{\xi}+\frac{\phi'\bar{\xi} }{\rho_{0}}}
    \begin{pmatrix}
        1+W\partial_{\rho}\bar{\xi}+\frac{\phi'\bar{\xi} }{\rho_{0}} & 0\\
         -W\partial_{\theta} \bar{\xi}-\frac{W}{\rho_{0}}\partial_{\theta}\rho_{0}\bar{\xi} & 1
    \end{pmatrix}
    ,
    \end{equation*}

    \noindent which implies that

    \begin{equation*}
        \frac{\partial \rho}{\partial \rho'}=\frac{1}{1+W\partial_{\rho}\bar{\xi}+\frac{\phi'\bar{\xi} }{\rho_{0}}},
    \end{equation*}

    \noindent and 

    \begin{equation*}
        \frac{\partial \rho}{\partial \theta'}=\frac{-W\partial_{\theta} \bar{\xi}-\frac{W}{\rho_{0}}\partial_{\theta}\rho_{0}\bar{\xi}}{1+W\partial_{\rho}\bar{\xi}+\frac{\phi'\bar{\xi} }{\rho_{0}}}.
    \end{equation*}

    \noindent Combining all the computations above, we find out that the norm of $(\nabla \Phi)^{-1}$ can also be bounded by $1+\vert \vert \bar{\xi}\vert \vert_{H^{1}}$
    
Having derived all of the transformation matrices,  we then show  the transformed equation system(mostly the same as the special case ). For the first equation in system \eqref{equ:fix_1}, we first deal with the term $\partial_{t}u$. We have the following computation by setting $v(t,x,y)=u(t,x'(x,y,t),y'(x,y,t))$

\begin{equation*}
   \partial_{t}v(t,x,y)=\frac{d}{dt}u(t,x',y')=\frac{d}{dt}u(t,x'(t,x,y),y'(t,x,y))=\partial_{t}u+\partial_{x'}u\frac{\partial x'}{\partial t}+\partial_{y'}u\frac{\partial y'}{\partial t}
\end{equation*}

\noindent  For $\frac{\partial x_1'}{\partial t}$, from the definition of $\rho^{\prime}$, we have the following computation 

\begin{equation}{\label{equ:2.0.7}}
    \frac{\partial x'}{\partial t}=\frac{\partial \rho'\cos\theta}{\partial t}=\cos\theta \frac{\partial \rho'}{\partial t}=\cos\theta W\partial_{t}\bar{\xi}.
\end{equation}

\noindent Similarly, we  have

\begin{equation}{\label{equ:2.0.8}}
    \frac{\partial y'}{\partial t}=\frac{\partial \rho'\sin\theta}{\partial t}=\sin\theta W\partial_{t}\bar{\xi}.
\end{equation}

\noindent Besides the transformation formula for temporal derivative, we have the following formulas for spatial derivatives of velocity field after the transformation

\begin{equation}{\label{equ:2.0.9}}
    \partial_{x'}u=\partial_{x}u\frac{\partial x}{\partial x'}+\partial_{y}u\frac{\partial y}{\partial x'},
\end{equation}

\noindent and

\begin{equation}{\label{equ:2.0.10}}
    \partial_{y'}u=\partial_{x}u \frac{\partial {x}}{\partial y'}+\partial_{y}u\frac{\partial y}{\partial y'}.
\end{equation}

\noindent To make those representations simpler, we set

\begin{equation*}
    \mathcal{A}=((\nabla \Phi)^{-1})^{T}.
\end{equation*}

\noindent Using the notation, we have the simplified version of transformation formulas as follows:

\begin{equation*}
    (\partial_{x'}u,\partial_{y'}u)^{T}=\mathcal{A}(\partial_{x}u,\partial_{y}u)^{T}.
\end{equation*}

\noindent  Combining \eqref{equ:2.0.7} to \eqref{equ:2.0.10} together, we obtain

\begin{equation}
\begin{aligned}
    \partial_{t}v(t,x,y)&=\partial_{t}u+(\frac{\partial x'}{\partial t},\frac{\partial y'}{\partial t})(\partial_{x'}u,\partial_{y'}u)^{T}\\
     &=\partial_{t}u+(\cos\theta W\partial_{t}\bar{\xi},\sin\theta W\partial_{t}\bar{\xi})\mathcal{A}(\partial_{x}u,\partial_{y}u)^{T}\\
     &=\partial_{t}u+(\cos\theta W\partial_{t}\bar{\xi},\sin\theta W\partial_{t}\bar{\xi})\mathcal{A}(\partial_{x}v,\partial_{y}v)^{T}.
\end{aligned}
\end{equation}

Define $\gamma_{1}\in(-\frac{\pi}{2},\frac{\pi}{2})$ and $\gamma_{2}\in(-\frac{\pi}{2},\frac{\pi}{2})$ as follows
\begin{align}
    \sin\gamma_{1}=-\frac{\rho'}{(\rho^{2}+\rho'^{2})^{\frac{1}{2}}}(\pi-\theta_{1}),
\end{align}
\noindent and
\begin{align}
    \sin\gamma_{2}=-\frac{\rho'}{(\rho^{2}+\rho'^{2})^{\frac{1}{2}}}(\theta_{2}).
\end{align}
Using all of the computation for the transformation above, and the definition for $\gamma_{1}$ and $\gamma_{2}$, we then have the following equation system in fixed domain by applying this transformation to system \eqref{equ:navier_stokes_0}. (For simplification, we still use $u$ to represent the velocity function in the fixed domain):

\begin{equation}{\label{equ:2.0.11}}
    \begin{cases}
        \partial_{t}u-(\cos\theta W\partial_{t}\bar{\xi},\sin\theta W\partial_{t}\bar{\xi})\mathcal{A}(\partial_{x}u,\partial_{y}u)^{T}+u\cdot \nabla_{\mathcal{A}} u +\operatorname{div}_{\mathcal{A}}(S_{\mathcal{A}}(P,u))=0~~~&in~\Omega \\
        \dive_{\mathcal{A}}u=0&in~\Omega\\
        S_{\mathcal{A}}(P,u)\mathcal{N}=(g\rho+\sigma\mathcal{H}(\rho))\mathcal{N}~~&on~\Sigma\\
        \partial_{t}\rho=\frac{1}{\rho}u\cdot \mathcal{N}~~&on~\Sigma\\
        (S_{\mathcal{A}}(P,u)\cdot \nu-\beta u)\cdot \tau=0~~&on~\Sigma_{s}\\
        u\cdot \nu=0~&on~\Sigma_{s}\\
        \mathcal{W}(\partial_{t}\rho(\pi-\theta_1,t))=-\sigma\sin\gamma_1+[\![\gamma]\!]\\
        \mathcal{W}(\partial_{t}\rho(\theta_2,t))=\sigma\sin\gamma_2+[\![\gamma]\!]
    \end{cases}
\end{equation}

\noindent where $\mathcal{N}$ is the outer normal vector of the new surface curve expressed by 
\begin{align}
    \mathcal{N}=-\rho'(\theta)\hat{e}_{\phi}+\rho(\theta)\hat{e}_{r}.
\end{align}

\section{The Navier Stokes Equation for Sessile Drop case $\theta_1=\theta_2=0$}

\subsection{New polar coordinate (A moving polar)}. 

From the previous subsection, after the transformation, $(\rho,u,P)$ solves the following equation system

\begin{equation}{\label{equ:fix_3}}
    \begin{cases}
        \partial_{t}u-(\cos\theta W\partial_{t}\bar{\xi},\sin\theta W\partial_{t}\bar{\xi})\mathcal{A}(\partial_{x}u,\partial_{y}u)^{T}+u\cdot \nabla_{\mathcal{A}} u +\dive_{\mathcal{A}}(S_{\mathcal{A}}(P,u))=0~~~&in~\Omega \\
        \dive_{\mathcal{A}}u=0&in~\Omega\\
        S_{\mathcal{A}}(P,u)\mathcal{N}=(g\rho+\sigma\mathcal{H}(\rho))\mathcal{N}~~&on~\Sigma\\
        \partial_{t}\xi=\frac{1}{\rho}u\cdot \mathcal{N}~~&on~\Sigma\\
        (S_{\mathcal{A}}(P,u)\cdot \nu-\beta u)\cdot \tau=0~~&on~\Sigma_{s}\\
        u\cdot \nu=0~&on~\Sigma_{s}\\
        \mathcal{W}(\partial_{t}\rho(\pi,t))=-\sigma\sin\gamma_1+[\![\gamma]\!]\\
        \mathcal{W}(\partial_{t}\rho(0,t))=\sigma\sin\gamma_2+[\![\gamma]\!]
    \end{cases}
\end{equation}

\noindent where $\mathcal{A}$ is defined to be $((\nabla \Phi)^{-1})^{T}$ and $\mathcal{N}$ is defined by equation \eqref{equ:1.1.14}. However, it is not a good idea to use the original polar coordinate to discuss this problem since the droplet can move on the horizontal plane and may move to a place where the new boundary function of the droplet can not be represented in the original polar coordinates. Based on this observation, we introduce $\mathfrak{n}(t)$. This is defined as the $x$ component of the coordinate of the new polar point. We construct a new polar coordinate based on this polar point. Suppose that at time $t$, the $x$ components of two contact points are $x_1(t)$ and $x_2(t)$. $\mathfrak{n}(t)$ is chosen to satisfy the following two properties:

\begin{itemize}
    \item $\mathfrak{n}(t)\in(x_1(t),x_2(t))$.
    \item For any fixed time t, we have\begin{equation}{\label{equ:4.1.1}}
        \int_{0}^{\pi} (\rho_{0,\mathfrak{n}(t)}(\theta)-\rho(t,\theta))^{2}d\theta=\min_{c\in(x_1,x_2)}\int_{0}^{\pi}(\rho_{0,c}(\theta)-\rho(t))^{2}d\theta
    \end{equation}
\end{itemize}

\noindent where $\rho_{0,c}(\theta)$ is the equilibrium function centering at the point $(c,0)$.

\textbf{Remark} To compute the integral
\begin{align}
     \int_{0}^{\pi} (\rho_{0,c}(\theta)-\rho(t,\theta))^{2}d\theta
\end{align}
for any $c\in (x_{1}(t),x_{2}(t))$, we need to first fix a polar point such that both $\rho(t,\theta)$ and $\rho_{0,\mathfrak{n}(t)}(t,\theta)$ can be represented in polar coordinates centering at this pole. Since the perturbation is small corresponding to the energy norms, there exists such a point. After we have already derived $\mathfrak{n}(t)$, we choose this $(\mathfrak{n}(t),0)$ as the new polar point.

Now, suppose that the value of $\mathfrak{n}(t)$ is given. We construct the polar point at this $\mathfrak{n}(t)$. For simplification, {we still use $\rho(t,\theta)$ to denote the free surface of the droplet in the new polar coordinates centering at $(\mathfrak{n}(t),0)$ and use $\rho_{0,\mathfrak{n}(t)}(\theta)$ to denote the steady state of the droplet centering at $(\mathfrak{n}(t),0)$. Since in polar coordinates centered at $(\mathfrak{n}(t),0)$, the representation for $\rho_{0,\mathfrak{n}(t)}(\theta)$ stays the same, we use $\rho_{0}(\theta)$ to denote its representation}. Furthermore, we use $\xi$ to denote the perturbation of boundary curve given as follows

\begin{align}
\xi(t,\theta)=\rho(t,\theta)-\rho_{0,\mathfrak{n}(t)}(\theta).
\end{align}

{In the polar coordinates centering at $\mathfrak{n}(t)$, we define a multi-variable function $\rho_{0}^{f}(c,t)$ as follows
\begin{align}
    \rho_{0}^{f}(c,t):=\rho_{0,c}(\theta).
\end{align}
 From the right-hand side of the \eqref{equ:4.1.1}, we have

\begin{equation}{\label{equ:4.1.2}}
   \frac{d}{dc}|_{c=\mathfrak{n}(t)} \int_{0}^{\pi}(\rho_{0,c}(\theta)-\rho(t))^{2}d\theta=2\int_{0}^{\pi}(\rho_{0}(c,t)-\rho(t))\frac{d}{dc}|_{c=\mathfrak{n}(t)}\rho_{0}^{f}(c,\theta)d\theta
\end{equation}

\noindent We then introduce the following notation for the shifting function as follows

\begin{equation}{\label{equ:4.1.3}}
    \xi_s:=\frac{d}{dc}\rho_{0}^{f}(c,\theta)|_{c=\mathfrak{n}(t)}
\end{equation}

\noindent First, we note that in Cartesian coordinates, any point $(\rho_{0}(\theta),\theta)$ on the boundary function $\rho_0(\theta)$ after this horizontal shift can be easily represented as:

\begin{equation}{\label{equ:4.1.4}}
    (\rho_{0}(\theta)\cos\theta+c,\rho_{0}(\theta)\sin\theta)
\end{equation}

\noindent From equation \eqref{equ:4.1.4}, we show that for any fixed $\theta$, the new angle of the point after shifting defined by \eqref{equ:4.1.4} can be represented as follows

\begin{equation}{\label{equ:4.1.5}}
    \tan\theta'=\frac{\rho_{0}\sin\theta}{\rho_{0}\cos\theta+c}
\end{equation}

\noindent Moreover, the radius of the point can be computed as follows:

\begin{equation}{\label{equ:4.1.6}}
    r(\theta')=\sqrt{\rho_{0}^{2}(\theta)\sin^{2}\theta+(\rho_{0}\cos\theta+c)^{2}}=\sqrt{\rho_{0}^{2}(\theta)+2\rho_{0}c\cos\theta+c^{2}}
\end{equation}

\noindent Using this, we then express $\rho_{0}^{f}(c,\theta)$ by $c$ and $\rho_{0}(\theta)$. Then we derive the representation for $\frac{d}{dc}\rho_{0}(c,\theta)$. This representation is complicated and will be derived by a theorem in the subsection 4.3. We first express it as follows

\begin{equation}{\label{equ:4.1.7}}
    \xi_s=\cos\theta+\frac{\rho_{0}'}{\rho_{0}}\sin\theta
\end{equation}

\noindent Using \eqref{equ:4.1.7} and \eqref{equ:4.1.2}, we obtain

\begin{equation}{\label{equ:4.1.8}}
    \int_{0}^{\pi} \xi_s\xi d\theta=0
\end{equation}

\noindent Moreover, we have the following relation by the symmetry of $\rho_{0}$ and $\xi_{s}$:

\begin{equation}{\label{equ:4.2.8}}
     \int_{0}^{\pi} \xi_s\rho_{0} d\theta=0
\end{equation}

{\textbf{Remark} At any time $t$, $\xi_{s}$ is the representation of the shift function in polar coordinates centering at $\mathfrak{n}(t)$}

Now we have the definition of $\mathfrak{n}(t)$. However, it is difficult to show the representation for $\mathfrak{n}(t)$ directly from its definition. We then establish the following theorem to derive the representation of $\mathfrak{n}(t)$ and the new kinematic boundary condition.

\begin{theorem}
    $\mathfrak{n}(t)$ is defined as above. We have the new kinematic boundary condition as follows:

    \begin{align}
         \partial_{t}\xi=\partial_{t}(\rho-\rho_{0})=\frac{1}{\rho_{0}} u\cdot \mathcal{N}-\mathfrak{n}'(t)\xi_s+G,
    \end{align}

    \noindent where

    \begin{align}
       G=(\frac{1}{\rho}-\frac{1}{\rho_{0}})u\cdot \mathcal{N}+\mathfrak{n}'(t)\sin\theta (\frac{\rho'}{\rho}-\frac{\rho_{0}'}{\rho_{0}}).
    \end{align}

    \noindent Moreover, $\mathfrak{n}'(t)$ can be expressed as follows

    \begin{equation}{\label{equ:4.1.18}}
    \mathfrak{n}'(t)= \lambda\int_{0}^{\pi}\frac{1}{\rho} u\cdot \mathcal{N} \xi_sd\theta,
   \end{equation}

\noindent where

\begin{equation}{\label{equ:4.1.19}}
\begin{aligned}
    \lambda(t) =&\frac{1}{\int_{0}^{\pi}\xi_s\xi_{3}(t) d\theta}\\
    \xi_{3}(\theta,t)=&(\cos\theta+\frac{\rho'(\theta,t)}{\rho(\theta,t)}\sin\theta)
    \end{aligned}
\end{equation}
\end{theorem}

\begin{proof}

 In the following calculation and the definition in the statement of the theorem, $\xi_s$, $\rho_{0}$, and $\rho$ are their representations in polar coordinate centering at $\mathfrak{n}(t)$. Hence, $\xi_s$ and $\rho_{0}$ are functions independent of time $t$.  Using the kinematic boundary condition of the equation system, we have the following formula

\begin{equation}
    \partial_{t}\rho=\frac{1}{\rho}u\cdot \mathcal{N}
\end{equation}

\noindent This is only true in fixed polar coordinates.  Then we compute the time derivative of $\rho$ in the moving polar coordinates. We choose a new reference system which moves at a speed of $\mathfrak{n}'(t)$. In the new system, the polar point is static, and the velocity field is $u-\mathfrak{n}'(t)$. Using this fact, we have:

\begin{equation}{\label{equ:kin_1}}
   \partial_{t}\rho=\frac{1}{\rho}(u-(\mathfrak{n}'(t),0))\cdot \mathcal{N}
\end{equation}

\noindent Recall the definition of $\mathcal{N}$ from equation \eqref{equ:1.1.14}
\begin{align}
    \mathcal{N}=-\rho'(\theta)\hat{e_{\phi}}+\rho(\theta)\hat{e}_{r},
\end{align}
 equation \eqref{equ:kin_1} above is equivalent to:

\begin{equation}
    \partial_{t}(\rho)=\frac{1}{\rho}u\cdot \mathcal{N}-(\mathfrak{n}'(t)\cos\theta-\mathfrak{n}'(t)\frac{\rho'}{\rho}\sin\theta)
\end{equation}

\noindent Substituting the definition of $\xi_s$ (equation \eqref{equ:4.1.7}) in the equation above, we have

\begin{equation}{\label{equ:4.1.9}}
   \partial_{t}\rho=\frac{1}{\rho_{0}}u\cdot \mathcal{N}-\mathfrak{n}'(t)\xi_2+G
\end{equation}

\noindent where $G$ contains nonlinear terms defined as follows:

\begin{equation}{\label{equ:4.1.10}}
    G=(\frac{1}{\rho}-\frac{1}{\rho_{0}})u\cdot \mathcal{N}+\mathfrak{n}'(t)\sin\theta (\frac{\rho'}{\rho}-\frac{\rho_{0}'}{\rho_{0}})
\end{equation}

\noindent  We then define $\xi_{3}(t,\theta)$ as follows

\begin{align}
    \xi_{3}(t,\theta)=(\cos \theta-\frac{\rho'}{\rho}\sin\theta).
\end{align}

\noindent Using this definition, the equation \eqref{equ:4.1.9} can be transformed into:

\begin{align}
    \partial_{t}\rho=\frac{1}{\rho}u\cdot \mathcal{N}-\mathfrak{n}'(t)\xi_{3}(t) \label{equ:4.0.9}
\end{align}

\noindent  Using the same idea as in the proof of equation \eqref{equ:4.2.8}, we show that:

\begin{align}
    \int_{0}^{\pi}\rho(t,\theta)\xi_{3}(t,\theta)d\theta=0\label{equ:4.0.8}
\end{align}

\noindent Furthermore, from the definition of $\xi_{s}$, we have $\xi_{3}=\xi_{s}-G$. Since in the moving polar coordinates the representation of $\rho_{0}$ is independent of $t$, we have:

\begin{equation*}
    \partial_{t}\xi=\partial_{t}(\rho-\rho_{0})=\frac{1}{\rho_{0}} u\cdot \mathcal{N}-\mathfrak{n}'(t)\xi_3=\frac{1}{\rho_{0}} u\cdot \mathcal{N}-\mathfrak{n}'(t)\xi_s+G,
\end{equation*}
\noindent which is the new kinematic boundary condition we want.

Now we have the new kinematic boundary condition, it remains to find the representation for $\mathfrak{n}'(t)$. It is hard to derive the representation for the function $\mathfrak{n}(t)$ from its definition \eqref{equ:4.1.1}. Instead, we use the new kinematic boundary condition to derive an ODE for $\mathfrak{n}(t)$.  Multiplying both sides of the equation \eqref{equ:4.0.9} by the function $\xi_s$ and then integrating it from $0$ to $\pi$, we obtain

\begin{equation}{\label{equ:4.1.12}}
   \int_{0}^{\pi}\xi_s\partial_{t}\xi=\int_{0}^{\pi}\frac{1}{\rho}u\cdot \mathcal{N}\xi_s-\int_{0}^{\pi}\mathfrak{n}'(t)\xi_s\xi_{3}d\theta.
\end{equation}

\noindent After adjusting the order of terms in equation above, we obtain

\begin{equation}{\label{equ:4.1.13}}
    \mathfrak{n}'(t)\int_{0}^{\pi}\xi_s\xi_{3}(t) d\theta=\int_{0}^{\pi}\frac{1}{\rho}u\cdot \mathcal{N}\xi_sd\theta-\int_{0}^{\pi}\partial_{t}\xi \xi_sd\theta
\end{equation}

\noindent Observing that there is a temporal derivative term on the right-hand side of the equation \eqref{equ:4.1.13} which cannot be determine at this point, we use the following computation to eliminate it. Using integration by part and chain rule, we have

\begin{equation}{\label{equ:4.1.14}}
    \int_{0}^{\pi}\partial_{t}\xi \xi_s d\theta=\partial_{t}\int_{0}^{\pi} \xi \xi_s d\theta-\int_{0}^{\pi}\xi\partial_{t}\xi_s.
\end{equation}

\noindent Since $\xi_s$ is a function independent of $t$, equation \eqref{equ:4.1.14} indicates that

\begin{equation}{\label{equ:4.1.15}}
    \int_{0}^{\pi}\partial_{t}\xi \xi_s d\theta=\partial_{t}\int_{0}^{\pi} \xi \xi_s d\theta=\partial_{t}(\int_{0}^{\pi}\rho\xi_s-\int_{0}^{\pi}\rho_{0}\xi_s)=\partial_{t}(\int_{0}^{\pi}\rho\xi_s),
\end{equation}

\noindent where we have used the vertical condition \eqref{equ:4.2.8}. Moreover, we have the following property by \eqref{equ:4.1.8} and \eqref{equ:4.2.8}

\begin{equation}{\label{equ:4.1.16}}
    \int_{0}^{\pi}\rho\xi_s d\theta=\int_{0}^{\pi} (\rho_{0}+\xi)\xi_{s}d\theta=0
\end{equation}

\noindent Combining equation \eqref{equ:4.1.15} and equation \eqref{equ:4.1.16}, we obtain

\begin{equation}{\label{equ:4.1.17}}
    \int_{0}^{\pi} \partial_{t}\xi \xi_s d\theta=0
\end{equation}

\noindent Therefore,  applying equation \eqref{equ:4.1.17} to equation \eqref{equ:4.1.14} yields that

\begin{equation}{\label{equ:ODE}}
    \mathfrak{n}'(t)= \lambda\int_{0}^{\pi}\frac{1}{\rho} u\cdot \mathcal{N} \xi_sd\theta,
\end{equation}

\noindent where

\begin{equation}{\label{equ:lambda}}
    \lambda(t) =\frac{1}{\int_{0}^{\pi}\xi_s\xi_{3}(t) d\theta}
\end{equation}

\noindent is a function of $t$. This completes the proof of the theorem.

\end{proof}

Having derived the kinematic boundary condition in the moving polar coordinates, we now establish the remaining equations of \eqref{equ:fix_1} in this coordinates. Since the moving polar coordinates only affect the equations on the free boundary, it suffices to examine the capillary equation and the contact-point condition. 

\textbf{Capillary equation} In Cartesian coordinates, after applying the horizontal shift $(x,y)\rightarrow (x',y')=(x+\mathfrak{n}(t),y)$ to Cartesian coordinates, we have the following relation for any function $f(x,y)=g(x-\mathfrak{n}(t),y)$
\begin{align}
    \p_{x}(f(x,y))|_{x'=a}=\p_{x}g(x',y)|_{x'=a-\mathfrak{n}(t)},
\end{align}
which shows that the capillary equation in Cartesian coordinates
\begin{align}
    S(u,p)\mathcal{N}=gh(x)-\sigma\mathcal{H}(h)
\end{align}
preserves its form under this horizontal translation. For the capillary equation in polar coordinates, we recall that the formula obtained in Section~2 was derived by transforming the Cartesian capillary equation into polar coordinates. Therefore, under the same horizontal transformation, the polar-coordinate formulation of capillary equation retains the same structure either:
\begin{align}
    S(P,u)\mathcal{N}=(g\rho \sin\theta+\sigma\mathcal{H}(\rho))\mathcal{N}
\end{align}

\textbf{Contact points condition}: In the static polar coordinates, the contact points condition are given as follows
\begin{align}{\label{trans_1}}
    \mathcal{W}(\p_{t}\rho(\pi,t))=\sigma \frac{\rho'}{(\rho^{2}+\rho'^{2})^{\frac{1}{2}}}(\pi,t)+[\![\gamma]\!],
\end{align}
and
\begin{align}{\label{trans_2}}
    \mathcal{W}(\p_{t}\rho(0,t))=-\sigma \frac{\rho'}{(\rho^{2}+\rho'^{2})^{\frac{1}{2}}}(\pi,t)+[\![\gamma]\!].
\end{align}
Transforming these into Cartesian coordinates yields
\begin{align}{\label{trans_3}}
     \mathcal{W}(-\p_{t}x_{1}(t))=\sigma \frac{\rho'}{(\rho^{2}+\rho'^{2})^{\frac{1}{2}}}(\pi,t)-[\![\gamma]\!],\notag\\
     \mathcal{W}(\p_{t}x_{2}(t))=\sigma \frac{\rho'}{(\rho^{2}+\rho'^{2})^{\frac{1}{2}}}(\pi,t)+[\![\gamma]\!].
\end{align}
\noindent where $x_{1}(t)$ and $x_{2}(t)$ denote the left and right contact points, respectively. The right-hand sides remain unchanged because they are invariant under the horizontal translation from the discussion for Capillary equation. At time $t$, after applying the horizontal shift $(x,y)\rightarrow (x',y')=(x+\mathfrak{n(t)},y)$, we have
\begin{align}
    x_{1}'(t)=x_{1}(t)-\mathfrak{n}(t)\notag\\
    x_{2}'(t)=x_{2}(t)-\mathfrak{n}(t)
\end{align}
\noindent Therefore, substituting these two relations to \eqref{trans_3}, we obtain
\begin{align}
    \mathcal{W}(-\p_{t}x_{1}'(t)+\mathfrak{n}'(t))=\sigma \frac{\rho'}{(\rho^{2}+\rho'^{2})^{\frac{1}{2}}}(\pi,t)-[\![\gamma]\!]\notag\\
     \mathcal{W}(\p_{t}x_{2}'(t)-\mathfrak{n}'(t))=\sigma \frac{\rho'}{(\rho^{2}+\rho'^{2})^{\frac{1}{2}}}(\pi,t)+[\![\gamma]\!]
\end{align}
\noindent Transforming it back to the polar coordinates, we have
\begin{align}
    \mathcal{W}(\p_{t}\rho(t)+\mathfrak{n}'(t))(\pi)=\sigma \frac{\rho'}{(\rho^{2}+\rho'^{2})^{\frac{1}{2}}}(\pi,t)-[\![\gamma]\!]\notag\\
     \mathcal{W}(\p_{t}\rho(t)-\mathfrak{n}'(t))(0)=\sigma \frac{\rho'}{(\rho^{2}+\rho'^{2})^{\frac{1}{2}}}(\pi,t)+[\![\gamma]\!]
\end{align}

In conclusion, we have the following system in moving polar coordinates centering at $\mathfrak{n}(t)$

\begin{equation}{\label{equ:origin}}
    \begin{cases}
        \partial_{t}u+u\cdot \nabla u +\operatorname{div}(S(P,u))=0~~~&in~\Omega(t) \\
        \operatorname{div} u=0&in~\Omega(t)\\
        S(P,u)\mathcal{N}=(g\rho \sin\theta+\sigma\mathcal{H}(\rho))\mathcal{N}~~&on~\Sigma(t)\\
        \partial_{t}\xi=\frac{1}{\rho_{0}}u\cdot \mathcal{N}-\mathfrak{n}'(t)\xi_s+G~~&on~\Sigma(t)\\
        (S(P,u)\cdot \nu-\beta u)\cdot \tau=0~~&on~\Sigma_{s}\\
        u\cdot \nu=0~&on~\Sigma_{s}\\
        \mathcal{W}(\partial_{t}\xi(\pi,t)+\mathfrak{n}'(t))=-\sigma\sin\gamma_1+[\![\gamma]\!]\\
        \mathcal{W}(\partial_{t}\xi(0,t)-\mathfrak{n}'(t))=\sigma\sin\gamma_2+[\![\gamma]\!]
    \end{cases}
\end{equation}

\textbf{Remark}: All functions in the two equations on the boundary $\Sigma(t)$ are their representation in polar coordinates centering at $\mathfrak{n}(t)$, which means that the function $\rho$,$\rho_{0}$, $\xi$ and $\xi_s$ are all functions depending only on $\theta$. They do not depend on $\mathfrak{n}(t)$.

 Then applying the transformation $\Phi$ established in Section 4 to system \eqref{equ:origin}, we obtain the following equation system

\begin{equation*}
    \begin{cases}
        \partial_{t}u-(\cos\theta W\partial_{t}\bar{\xi},\sin\theta W\partial_{t}\bar{\xi})\mathcal{A}(\partial_{x}u,\partial_{y}u)^{T}+u\cdot \nabla_{\mathcal{A}} u +\dive_{\mathcal{A}}(S_{\mathcal{A}}(P,u))=0~~~&in~\tilde{\Omega}(t) \\
        \dive_{\mathcal{A}}u=0&in~\tilde{\Omega}(t)\\
        S_{\mathcal{A}}(P,u)\mathcal{N}=(g\rho+\sigma\mathcal{H}(\rho))\mathcal{N}~~&on~\tilde{\Sigma}(t)\\
       \partial_{t}\xi=\frac{1}{\rho_{0}}u\cdot \mathcal{N}-\mathfrak{n}'(t)\xi_s+G~~&on~\tilde{\Sigma}(t)\\
        (S_{\mathcal{A}}(P,u)\cdot \nu-\beta u)\cdot \tau=0~~&on~\Sigma_{s}\\
        u\cdot \nu=0~&on~\Sigma_{s}\\
        \mathcal{W}(\partial_{t}\xi(\pi,t)+\mathfrak{n}'(t))=-\sigma\sin\gamma_1+[\![\gamma]\!]\\
        \mathcal{W}(\partial_{t}\xi(0,t)-\mathfrak{n}'(t))=\sigma\sin\gamma_2+[\![\gamma]\!].
    \end{cases}
\end{equation*}
\noindent {where $\tilde{\Omega}(t)$ denotes the domain congruent to $\Omega_{0}$ but centering at $(\mathfrak{n}(t),0)$}. 

 Finally, by applying the shift transformation $(x,y)\mapsto (x-\mathfrak{n}(t),y)$, we map the time-dependent domain $\tilde{\Omega}(t)$ onto the fixed reference domain $\Omega_{0}$. Under this change of variables, the system of equations above can be rewritten in the fixed domain as follows:

\begin{equation}{\label{equ:4.1.11}}
    \begin{cases}
         \partial_{t}u-\mathfrak{n}'(t)\partial_{x}u-(\cos\theta W\partial_{t}\bar{\xi},\sin\theta W\partial_{t}\bar{\xi})\mathcal{A}(\partial_{x}u,\partial_{y}u)^{T}+u\cdot \nabla_{\mathcal{A}} u +\dive_{\mathcal{A}}(S_{\mathcal{A}}(P,u))=0~~~&in~{\Omega}_{0} \\
        \dive_{\mathcal{A}}u=0&in~{\Omega}_{0}\\
        S_{\mathcal{A}}(P,u)\mathcal{N}=(g\rho+\sigma\mathcal{H}(\rho))\mathcal{N}~~&on~{\Sigma}\\
        \partial_{t}\xi=\frac{1}{\rho_{0}}u\cdot \mathcal{N}-\mathfrak{n}'(t)\xi_s+G&on~\Sigma\\
        (S_{\mathcal{A}}(P,u)\cdot \nu-\beta u)\cdot \tau=0~~&on~\Sigma_{s}\\
        u\cdot \nu=0~&on~\Sigma_{s}\\
        \mathcal{W}((\mathfrak{n}'(t)+\partial_{t}\xi)(\pi,t))=-\sigma\sin\gamma_1+[\![\gamma]\!]\\
        \mathcal{W}((\mathfrak{n}'(t)-\partial_{t}\xi)(0,t))=\sigma\sin\gamma_2+[\![\gamma]\!]
    \end{cases}
\end{equation}

\noindent where $\mathcal{A}=(\nabla \Phi^{-1})^{T}$. We focus on discussing this equation system in the next subsection.

\subsection{Equation for the perturbation function(Taylor Expansion)}
In this subsection, we write $u=0+u$, $P=P_0+p$, $\rho=\rho_0+\xi$ to study the equation for the perturbed functions $(u,p,\xi)$ from equation \eqref{equ:4.1.11}. To this end, we perform a Taylor expansion of all the nonlinear terms in system \eqref{equ:4.1.11}.

\textbf{Contact angle Terms}

For $\sin\gamma_1$,by definition of the contact angles, we have the following computation 

\begin{equation*}
    \sin\gamma_1=-\frac{\rho'}{\sqrt{\rho^{2}+\rho'^{2}}}(\pi-\theta_1)
\end{equation*}

\noindent Using Taylor expansion, we have

\begin{equation*}
    \frac{\rho'}{\sqrt{\rho^{2}+\rho'^{2}}}=\frac{\rho_0'+\xi'}{\sqrt{(\rho_0+\xi)^{2}+(\rho_0'+\xi ')^{2}}}=\frac{\rho_0'}{\sqrt{\rho_0^{2}+\rho_0'^{2}}}+\frac{\rho_{0}^{2}\xi'}{(\rho_{0}^{2}+\rho_{0}'^{2})^{\frac{3}{2}}}-\frac{\rho_0'\rho_{0}\xi}{(\rho_0^{2}+\rho_{0}'^{2})^{\frac{3}{2}}}+\mathcal{R}_1(\partial_{\theta}\xi,\xi,\partial_{\theta}\rho_{0},\rho_{0})
\end{equation*}

For the residue term $\mathcal{R}$, there are three smooth functions $\mathcal{R}_i,i=1,2,3$ depending on $\partial_{t}\rho_0$ and $\rho_{0}$ such that:

\begin{equation}
\begin{aligned}
     \frac{\rho'}{\sqrt{\rho^{2}+\rho'^{2}}}=\frac{\rho_0'+\xi'}{\sqrt{(\rho_0+\xi)^{2}+(\rho_0'+\xi ')^{2}}}=&\frac{\rho_0'}{\sqrt{\rho_0^{2}+\rho_0'^{2}}}+\frac{\rho_{0}^{2}\xi'}{(\rho_{0}^{2}+\rho_{0}'^{2})^{\frac{3}{2}}}-\frac{\rho_0'\rho_{0}\xi}{(\rho_0^{2}+\rho_{0}^{2})^{\frac{3}{2}}}\\
&+\mathcal{R}_{1,1}\xi^{2}+\mathcal{R}_{1,2}\xi\xi'+\mathcal{R}_{1,3}(\xi')^{2},
\end{aligned}
\end{equation}

\noindent where the three functions have the following forms:

\begin{equation*}
    \mathcal{R}_{1,1}=\int_{0}^{1}(1-t)\frac{((\rho_0+t\xi)^{2}(\rho_0'+t\xi')-(\rho_{0}'+t\xi')^{3}}{((\rho_{0}+t\xi)^{2}+(\rho_0'+t\xi')^{2})^{\frac{5}{2}}}dt,
\end{equation*}

\noindent and

\begin{equation*}
    \mathcal{R}_{1,2}=2\int_{0}^{1}(1-t)\frac{(\rho_0+t\xi)(\rho_{0}'+t\xi')^{2}-(\rho_{0}+t\xi)^{3}}{((\rho_{0}+t\xi)^{2}+(\rho_{0}'+t\xi')^{2})^{\frac{5}{2}}}dt,
\end{equation*}
and
\begin{equation*}
    \mathcal{R}_{1,3}=-3\int_{0}^{1}(1-t)\frac{(\rho_0+t\xi)^{2}(\rho_0'+t\xi')}{((\rho_0+t\xi)^{2}+
    (\rho_{0}'+t\xi')^{2})^{\frac{5}{2}}}dt.
\end{equation*}

\textbf{The normal vector}

We now decompose another highly non-linear term is the normal vector $\mathcal{N}$. We have the following decomposition in Cartesian coordinates. 

\begin{equation}
\begin{aligned}
    {\mathcal{N}}=&-\rho'(\theta)(\cos(\theta+\frac{\pi}{2}),\sin(\theta+\frac{\pi}{2}))+\rho(\theta)(\cos\theta,\sin\theta)\\
    =&-\rho'(\theta)(-\sin\theta,\cos\theta)+\rho(\theta)(\cos\theta,\sin\theta)\\
    =&(\rho'(\theta)+\rho(\theta)\cos\theta,-\rho'(\theta)\cos\theta+\rho(\theta)\sin\theta)
    \end{aligned}
\end{equation}

\textbf{Curvature term}

Finally, we decompose the curvature term. We have the math expression of curvature $\mathcal{H}(\rho)$ as follows

\begin{equation*}
    H(\rho)=H(\rho_{0}+\xi)=\frac{1}{\rho}\frac{\rho_{0}+\xi}{\sqrt{(\rho_{0}+\xi)^{2}+(\rho_{0}'+\xi')^{2}}}-\frac{1}{\rho}\partial_{\theta}\frac{\rho_{0}'+\xi'}{\sqrt{(\rho_{0}+\xi)^{2}+(\rho_{0}'+\xi')^{2}}}
\end{equation*}

\noindent  Applying the previous computation for contact angles to the second term, we obtain

\begin{equation}
\begin{aligned}
    \partial_{\theta}\frac{\rho_{0}'+\xi'}{\sqrt{(\rho_{0}+\xi)^{2}+(\rho_{0}'+\xi')^{2}}}
    =\partial_{\theta}(\frac{\rho_0'}{\sqrt{\rho_0^{2}+\rho_0'^{2}}}+\frac{\rho_{0}^{2}\xi'}{(\rho_{0}^{2}+\rho_{0}'^{2})^{\frac{3}{2}}}-\frac{\rho_0'\rho_{0}\xi}{(\rho_0^{2}+\rho_{0}^{2})^{\frac{3}{2}}}+\mathcal{R}_1)
\end{aligned}
\end{equation}

\noindent Then we have:

\begin{equation}
\begin{aligned}
    \frac{1}{\rho_{0}+\xi}\partial_{\theta}\frac{\rho_{0}'+\xi'}{\sqrt{(\rho)^{2}+(\rho')^{2}}}
    =&\frac{1}{\rho_{0}}\partial_{\theta}\frac{\rho_{0}'}{\sqrt{\rho_{0}^{2}+\rho_{0}'^{2}}}-\frac{\xi}{\rho_{0}^{2}}\partial_{\theta}\frac{\rho_{0}'}{\sqrt{\rho_{0}^{2}+\rho_{0}'^{2}}}\\&+\frac{1}{\rho_{0}}\partial_{\theta}(\frac{\rho_{0}^{2}\xi'}{(\rho_{0}^{2}+\rho_{0}'^{2})^{\frac{3}{2}}}-\frac{\rho_0'\rho_{0}\xi}{(\rho_0^{2}+\rho_{0}^{2})^{\frac{3}{2}}}+\mathcal{R}_{1})+{\mathcal{R}_{4}}
    \end{aligned}
\end{equation}

\noindent Moreover, for the first term of the curvature, we decompose it to

\begin{equation}
\begin{aligned}
    \frac{1}{\rho}\frac{\rho_{0}+\xi}{\sqrt{(\rho_{0}+\xi)^{2}+(\rho_{0}'+\xi')^{2}}}=&\frac{1}{\sqrt{(\rho_{0}+\xi)^{2}+(\rho_{0}'+\xi')^{2}}}\\
    &=\frac{1}{\sqrt{\rho^{2}+\rho_{0}'^{2}}}-\frac{\rho_{0}\xi}{(\rho_{0}^{2}+\rho_0'^{2})^{\frac{3}{2}}}-\frac{\rho_0'\xi'}{(\rho_{0}^{2}+\rho'^{2})^{\frac{3}{2}}}+\mathcal{R}_{5}
    \end{aligned}
\end{equation}

\noindent where $\mathcal{R}_{5}$ is the second order term which has the similar form compared to $\mathcal{R}_1$ and $\mathcal{R}_{4}$. Let $\mathcal{R}_{2}=\mathcal{R}_{4}+\mathcal{R}_{5}$, we have
\begin{align}
    \mathcal{H}(\rho)=&\frac{1}{\sqrt{\rho^{2}+\rho_{0}'^{2}}}-\frac{\rho_{0}\xi}{(\rho_{0}^{2}+\rho_0'^{2})^{\frac{3}{2}}}-\frac{\rho_0'\xi'}{(\rho_{0}^{2}+\rho'^{2})^{\frac{3}{2}}}+\frac{1}{\rho_{0}}\partial_{\theta}\frac{\rho_{0}'}{\sqrt{\rho_{0}^{2}+\rho_{0}'^{2}}}-\frac{\xi}{\rho_{0}^{2}}\partial_{\theta}\frac{\rho_{0}'}{\sqrt{\rho_{0}^{2}+\rho_{0}'^{2}}}\notag\\&+\frac{1}{\rho_{0}}\partial_{\theta}(\frac{\rho_{0}^{2}\xi'}{(\rho_{0}^{2}+\rho_{0}'^{2})^{\frac{3}{2}}}-\frac{\rho_0'\rho_{0}\xi}{(\rho_0^{2}+\rho_{0}^{2})^{\frac{3}{2}}}+\mathcal{R}_{1})+\mathcal{R}_{2}
\end{align}

Now we have the Taylor expansion for all of nonlinear terms. Then we establish the equation system for these perturbation functions. First, we have the equation for the equilibrium

\begin{equation*}
    P_{0}=g\rho_{0}\sin\theta-\sigma \mathcal{H}(\rho_{0}),
\end{equation*}

We then analyze the contact line condition. Since $\mathcal{W}$ s monotone increasing, we set:

\begin{equation*}
    \kappa= \mathcal{W}'(0)>0
\end{equation*}

\noindent Then we define $\hat{\mathcal{W}}(z)=\frac{1}{\kappa}\mathcal{W}(z)-z$.

To keep the mass conserved, we have:

\begin{equation*}
    \int_{\theta_1}^{\theta_2} (\rho_{0}+\xi)^{2}=\int_{\theta_1}^{\theta_2}\rho_{0}^{2}
\end{equation*}

\noindent which is equivalent to the fact that:

\begin{equation*}
     \int_{\theta_1}^{\theta_2} 2\rho_0\xi d\theta=-\int_{\theta_1}^{\theta_2}\xi^{2} d\theta. 
\end{equation*}

\noindent Therefore, we have:

\begin{equation*}
   \int_{\theta_1}^{\theta_{2}} \rho_{0}\xi=-\frac{1}{2}\int_{\theta_{1}}^{\theta_{2}} \xi^{2}d\theta
\end{equation*}

Finally, combining all of the computation and decomposition above, we derive the following equation system from equation \eqref{equ:4.1.11} by using the equation for the steady surface curve to cancel the zero order term out).

\begin{equation}\label{equ:4.1.20}
    \begin{cases}
        \partial_{t}u-(\cos\theta W\partial_{t}\bar{\xi},\sin\theta W\partial_{t}\bar{\xi})\mathcal{A}(\partial_{x}u,\partial_{y}u)^{T}+u\cdot \nabla_{\mathcal{A}}u-\mathfrak{n}'(t)\p_{x}u+\dive_{\mathcal{A}}S_{\mathcal{A}}(p,u)=0~~~&in~\Omega \\
        \dive_{\mathcal{A}}u=0&in~\Omega\\
        S_{\mathcal{A}}(p,u)\mathcal{N}=(g\xi\sin\theta+\sigma(\mathcal{P}_1(\rho_0,\rho_0')\xi+\mathcal{P}_{2}(\rho_{0},\rho_{0}')\xi'-\frac{1}{\rho_{0}}\partial_{\theta}(\frac{\rho_{0}^{2}\xi'}{(\rho_{0}^{2}+\rho_{0}'^{2})^{\frac{3}{2}}}-\frac{\rho_{0}\rho_{0}'\xi}{(\rho_{0}^{2}+\rho_{0}'^{2})^{\frac{3}{2}}}+\mathcal{R}_{1}))\mathcal{N}+\mathcal{R}_{2}\mathcal{N}~~&on~\Sigma\\
        \partial_{t}\xi+\mathfrak{n}^{\prime}(t)\xi_{s}=\frac{1}{\rho_{0}}u\cdot \mathcal{N}+G~~&on~\Sigma\\
        (S_{\mathcal{A}}(p,u)\cdot \nu-\beta u)\cdot \tau=0~~&on~\Sigma_{s}\\
        u\cdot \nu=0~&on~\Sigma_{s}\\
        \kappa\partial_{t}\xi(\theta_1,t)+\kappa\mathfrak{n}^{\prime}(t)+\kappa\hat{\mathcal{W}}(\partial_{t}\xi(\theta_1,t))=\sigma\frac{\rho_{0}^{2}\xi'}{(\rho_{0}^{2}+\rho_{0}'^{2})^{\frac{3}{2}}}(\theta_1,t)-\sigma\frac{\rho_0'\rho_{0}\xi}{(\rho_0^{2}+\rho_{0}^{2})^{\frac{3}{2}}}(\theta_1,t)+\sigma\mathcal{R}_{1}(\theta_1,t)\\
        \kappa\partial_{t}\xi(\theta_1,t)-\kappa\mathfrak{n}^{\prime}(t)+\kappa\hat{\mathcal{W}}(\partial_{t}\xi(\theta_1,t))=-(\sigma\frac{\rho_{0}^{2}\xi'}{(\rho_{0}^{2}+\rho_{0}'^{2})^{\frac{3}{2}}}(\theta_2,t)-\sigma\frac{\rho_0'\rho_{0}\xi}{(\rho_0^{2}+\rho_{0}^{2})^{\frac{3}{2}}}(\theta_2,t)+\sigma\mathcal{R}_{1}(\theta_2,t))
    \end{cases}
\end{equation}

\noindent where $(\xi,u,p)$ are perturbation functions defined above and $\mathcal{P}_{1}$ and $\mathcal{P}_{2}$ are defined as:

\begin{equation}{\label{equ:4.1.21}}
    \mathcal{P}_{1}(\rho_0.\rho_0')=\frac{\rho_{0}''-\rho_{0}-\frac{\rho_{0}'^{2}}{\rho_{0}}}{(\rho_{0}^{2}+\rho_{0}'^{2})^{\frac{3}{2}}}
\end{equation}

\begin{equation}{\label{equ:4.1.22}}
    \mathcal{P}_{2}=-\frac{\rho_{0}'}{(\rho_{0}^{2}+\rho_{0}'^{2})}.
\end{equation}

\noindent They are both smooth functions in the polar coordinates derived by Taylor expansion results.

Furthermore,  taking temporal derivative in system \eqref{equ:4.1.20}, we show the following system

\begin{equation*}
     \begin{cases}
        \partial_{t}v-(\cos\theta W\partial_{t}\bar{\zeta},\sin\theta W\partial_{t}\bar{\zeta})\mathcal{A}(\partial_{x}v,\partial_{y}v)^{T}-\mathfrak{n}'(t)\p_{x}v+u\cdot \nabla_{\mathcal{A}}v+\dive_{\mathcal{A}}S_{\mathcal{A}}(q,v)=F^{1}~~~&in~\Omega \\
        \dive_{\mathcal{A}}v=F^{2}&in~\Omega\\
        S_{\mathcal{A}}(q,v)\mathcal{N}=(g\zeta\sin\theta+\sigma(\mathcal{P}_1(\rho_0,\rho_0')\zeta+\mathcal{P}_{2}(\rho_{0},\rho_{0}')\zeta'-\frac{1}{\rho_{0}}\partial_{\theta}(\frac{\rho_{0}^{2}\zeta'}{(\rho_{0}^{2}+\rho_{0}'^{2})^{\frac{3}{2}}}-\frac{\rho_{0}'\rho_{0}\zeta}{(\rho_{0}^{2}+\rho_{0}'^{2})^{\frac{3}{2}}}+F^{3}))\mathcal{N}+F^{4}~~&on~\Sigma\\
        \partial_{t}\zeta+\mathfrak{n}^{(i)}(t)\xi_s=\frac{1}{\rho_{0}}v\cdot \mathcal{N}+F^6~~&on~\Sigma\\
        (S_{\mathcal{A}}(q,v)\cdot \nu-\beta v)\cdot \tau=F^{5}~~&on~\Sigma_{s}\\
        v\cdot \nu=0~&on~\Sigma_{s}\\
        \kappa\partial_{t}\zeta(\pi,t)+\kappa\mathfrak{n}^{(i)}(t)+\kappa F^{7}=\sigma\frac{\rho_{0}^{2}\zeta'}{(\rho_{0}^{2}+\rho_{0}'^{2})^{\frac{3}{2}}}(\pi,t)-\sigma\frac{\rho_0'\rho_{0}\zeta}{(\rho_0^{2}+\rho_{0}^{2})^{\frac{3}{2}}}(\pi,t)+F^{3}\\
        \kappa\partial_{t}\zeta(\pi,t)-\kappa\mathfrak{n}^{(i)}(t)+\kappa F^{7}=-(\sigma\frac{\rho_{0}^{2}\zeta'}{(\rho_{0}^{2}+\rho_{0}'^{2})^{\frac{3}{2}}}(0,t)-\sigma\frac{\rho_0'\rho_{0}\zeta}{(\rho_0^{2}+\rho_{0}^{2})^{\frac{3}{2}}}(0,t)+F^{3})
    \end{cases}
\end{equation*}

\noindent where $v=\p_{t}^{i}u$, $\zeta=\partial_{t}^{i}\xi$ and $q=\p_{t}^{i}p$ with the form of $\mathcal{P}_1$ and $\mathcal{P}_{2}$ given by equations \eqref{equ:4.1.21} and \eqref{equ:4.1.22}, respectively (i=1 or 2):
 
We now consider the following simplified system:

\begin{equation}{\label{equ:4.1.23}}
    \begin{cases}
        \partial_{t}v+\operatorname{div}_{\mathcal{A}}S_{\mathcal{A}}(q,v)=b^{1}~~&\operatorname{in}~~\Omega\\
        \operatorname{div}_{\mathcal{A}}v=b^{2}~~&\operatorname{in}~~\Omega\\
        S_{\mathcal{A}}(q,v)\mathcal{N}=(\mathcal{K}(\zeta)-\sigma\partial_{\theta}{F}^{3})\mathcal{N}+{b}^{4}~~&\operatorname{on}~~\Sigma\\
        \partial_{t}\zeta=\frac{1}{\rho_{0}}v\cdot \mathcal{N}-\mathfrak{n}^{i}(t)\xi_s+b^{6}~~&\operatorname{on}~~\Sigma\\
        (S_{\mathcal{A}}(q,v)\nu-\beta v)\cdot \tau=b^{5}~~&\operatorname{on}~~\Sigma_{s}\\
        v\cdot \nu=0~~&\operatorname{on}~~\Sigma_{s}\\
        (\mp\sigma \frac{\rho_{0}^{2}\partial_{\theta}(\zeta)}{(\rho_{0}^{2}+\rho_{0}'^{2})^{\frac{3}{2}}}\pm\sigma \frac{\rho_{0}'\rho_{0}(\zeta)}{(\rho_{0}^{2}+\rho_{0}'^2)^{\frac{3}{2}}})(\frac{\pi}{2}\pm \frac{\pi}{2})=\kappa ((\frac{1}{\rho_{0}}v\cdot \mathcal{N})\pm b^{3}+b^{7}-b^{6})(\frac{\pi}{2}\pm\frac{\pi}{2})
    \end{cases}
\end{equation}

\noindent where the function $\mathcal{K}$ is defined to be:

\begin{align*}
     \mathcal{K}(\xi)=&g(\xi)\sin\theta\mathcal+\sigma {P}_1(\rho_{0},\rho_{0}')(\xi)\\
     &+\sigma(\mathcal{P}_{2}(\rho_{0},\rho_{0}')(\xi)'-\frac{1}{\rho_{0}}\partial_{\theta}(\frac{\rho_{0}^{2}(\xi)'}{(\rho_{0}^{2}+\rho_{0}'^{2})^{\frac{3}{2}}}-\frac{\rho_{0}\rho_{0}'(\xi)}{(\rho_{0}^{2}+\rho_{0}'^{2})^{\frac{3}{2}}})),
\end{align*}
    
\noindent and $F^{6}=G$ defined by equation \eqref{equ:4.1.10}.

To establish the estimate for the solution function, we write down the weak form of the equation system:

\begin{align}
    &(\partial_{t}v,\omega)_{\mathcal{H}^{0}}+((v,\omega))+(\zeta,\frac{1}{\rho_{0}}\omega \cdot \mathcal{N})_{1,\Sigma}+[\frac{1}{\rho_{0}}v\cdot \mathcal{N},\frac{1}{\rho_{0}}\omega \cdot \mathcal{N}]_{\theta}\notag\\
    &=\int_{\Omega} b^{1}\cdot \omega J-\int_{0}^{\pi}b^{3}\partial_{\theta}(\omega\cdot \mathcal{N})+b^{4}\cdot \omega-\int_{\Sigma_{s}}b^{5}(\omega \cdot \tau)J-[b^{7}+b^{8},\frac{1}{\rho_{0}}\omega\cdot \mathcal{N}]_{\theta} \label{equ:4.1.24},
\end{align}

\noindent combined with the kinematic boundary condition

\begin{equation}
    \partial_{t}\zeta=\frac{1}{\rho_{0}}v\cdot \mathcal{N}-\gamma^{i}(t)\xi_s+F^{6},
\end{equation}

\noindent and the velocity condition for the pole

\begin{equation}{\label{equ:4.1.100}}
    \mathfrak{n}'(t)= \lambda\int_{0}^{\pi}\frac{1}{\rho} u\cdot \mathcal{N} \xi_sd\theta,
\end{equation}

\noindent where $\lambda$ is defined by equation \eqref{equ:4.1.18}. The $(1,\Sigma)$ inner product is defined by

\begin{align}
    (\rho_1,\rho_2)_{1,\Sigma}=&g\int_{0}^{\pi}\rho_{0}\rho_1\rho_2\sin\theta d\theta+\sigma\int_{0}^{\pi}\frac{\rho_{0}\rho_1'\rho_2'}{(\rho_{0}^{2}+\rho_{0}'^{2})^{\frac{3}{2}}}d\theta-\sigma\int_{0}^{\pi}\frac{\rho_{0}'\rho_1'\rho_2}{(\rho_{0}^{2}+\rho_{0}'^{2})^{\frac{3}{2}}}d\theta \notag\\
    &-\sigma \int_{0}^{\pi}\frac{\rho_{0}'\rho_1\rho_2'}{(\rho_{0}^{2}+\rho_{0}'^{2})^{\frac{3}{2}}}d\theta+\sigma\int_{0}^{\pi}\frac{\rho_{0}''\rho_{0}-\rho_{0}'^{2}-\rho_{0}^{2}}{(\rho_{0}^{2}+\rho_{0}'^{2})^{\frac{3}{2}}}\rho_1\rho_2d\theta ,\label{equ:4.1.25}
\end{align}

\noindent which is a symmetric scalar product. To show the apriori estimate, it is important to show the positivity of this inner product $1,\Sigma$. We will use the next subsection to show this crucial property. 

\subsection{The property of $\xi_{s}$ and the positivity of the inner product}

In this subsection, we study the property of $\xi_{s}$ by showing that it is the unique kernel of the $1,\Sigma$ inner product, within a certain functional space, and subsequently establish the positivity of this inner product. We begin by computing the second variation of the following functional energy.

     \begin{equation*}
         E(\rho(\theta))=\frac{1}{3}g\int_{\theta_1}^{\theta_2}\rho^{3}\sin \theta d\theta+\int_{\theta_1}^{\theta_2}\sigma \sqrt{\rho^{2}+\rho'^{2}}d\theta-[\![\gamma]\!](\rho(\theta_1)+\rho(\theta_2))
     \end{equation*}

       We consider any perturbation $\xi(\theta)$ near the minimizer we derived in Section 2. From the conservation of total mass, the perturbation satisfies

      \begin{equation}{\label{equ:4.3.1}}
          \int_{\theta_2}^{\pi-\theta_1} \rho_{0}(\theta)\xi(\theta) d\theta=0
      \end{equation}

      \noindent To make the setting precise, we define two function spaces:

      \begin{equation*}
          A=\{\xi:\xi\in C^{\infty}(\theta_2,\pi-\theta_1)\}
      \end{equation*}

      \begin{equation*}
          B=\{\xi:\xi\in A~and~\int_{\theta_2}^{\pi-\theta_1}\rho_{0}(\theta)\xi(\theta)d\theta=0\}
      \end{equation*}

     \noindent  From the definition, it follows that 

     \begin{equation*}
         \operatorname{codim}(B) = 1
     \end{equation*}

     \noindent that is, $B$ is a codimension-one subspace of $A$.
     
      Having identified the minimizer $\rho_{0}$,  we then want to compute the second-order derivative of the energy.

     \begin{theorem}
         Suppose that $\rho_{0}$ is the minimizer of energy functional $E$ derived in Section 2. For any perturbation function $\xi$,  the representation of the second derivative of the functional energy $E$ is expressed as follows:

           \begin{align}
          F_{0}(\xi)=&g\int_{\theta_2}^{\pi-\theta_1} \rho_{0}\xi^{2}\sin\theta d\theta \notag\\
            &+\sigma \int_{\theta_2}^{\pi-\theta_1}(-\frac{\rho_0^{2}+2\rho_0'^{2}-\rho_0''\rho_0}{(\rho_0^{2}+\rho_0'^{2})^{\frac{3}{2}}}\xi^{2}+\frac{\xi^{2}+\xi'^{2}}{(\rho_0^{2}+\rho_0'^{2})^{\frac{1}{2}}}-\frac{(\xi\rho_0+\rho_0'\xi')^{2}}{(\rho_0^{2}+\rho_0'^{2})^{\frac{3}{2}}})d\theta \label{equ:4.3.2}
     \end{align}

     \noindent Especially, when $\theta_1=\theta_2=0$, we have:

     \begin{align}
         F_{0}(\xi)=&g\int_{0}^{\pi} \rho_0\xi^{2}\sin\theta d\theta \notag\\
            &+\sigma \int_{0}^{\pi}(-\frac{\rho_0^{2}+2\rho_0'^{2}-\rho_0''\rho_0}{(\rho_0^{2}+\rho_0'^{2})^{\frac{3}{2}}}\xi^{2}+\frac{\xi^{2}+\xi'^{2}}{(\rho_0^{2}+\rho_0'^{2})^{\frac{1}{2}}}-\frac{(\xi\rho_0+\rho_0'\xi')^{2}}{(\rho_0^{2}+\rho_0'^{2})^{\frac{3}{2}}})d\theta \label{equ:4.3.3}
     \end{align}
     \end{theorem}

     \begin{proof}

         We first give the definition of the nonlinear perturbation function $f(t,\theta)$ as follows:

         \begin{equation*}
       f(t,\theta)=t\xi(\theta)+t^{2}\eta(\theta)+o(t^{3})
   \end{equation*}
   
   \noindent We then compute the energy corresponding to this perturbation and subtract the original energy associated with the minimizer $\rho_{0}$.

   \begin{align}
       &E(\rho_0(\theta)+f(t,\theta))-E(\rho_{0}(\theta))\notag\\
        =&tg\int_{\theta_2}^{\pi-\theta_1}\rho_0^{2}(\theta)\xi(\theta)\sin\theta d\theta+t\sigma\int_{\theta_2}^{\pi-\theta_1} \frac{\rho_0\xi}{\sqrt{\rho_0^{2}+\rho_0'^{2}}}d\theta+t\sigma\int_{\theta_2}^{\pi-\theta_1}\frac{\rho_0'\xi'}{\sqrt{\rho_0^{2}+\rho_0'^{2}}}\notag\\
        &-t[\![\gamma]\!](\xi(\theta_1)+\xi(\theta_2))\notag\\
        & +t^{2}g\int_{\theta_2}^{\pi-\theta_1}\rho_0(\theta)\xi^{2}(\theta)\sin\theta d\theta+\frac{1}{2}t^{2}\sigma\int_{\theta_2}^{\pi-\theta_1}\frac{\xi^{2}+\xi'^{2}}{(\rho_0^{2}+\rho_0'^{2})^{\frac{1}{2}}}d\theta-\frac{1}{8}\sigma\int_{\theta_2}^{\pi-\theta_1}\frac{(2t\xi\rho_0+2t\rho_0'\xi')^{2}}{(\rho_0^{2}+\rho_0'^{2})^{\frac{3}{2}}}\notag\\
        & +t^{2}g\int_{\theta_2}^{\pi-\theta_1}\rho_0^{2}(\theta)\eta(\theta)\sin\theta d\theta+t^{2}\sigma\int_{\theta_2}^{\pi-\theta_1} \frac{\rho_0\eta}{\sqrt{\rho_0^{2}+\rho_0'^{2}}}d\theta+t^{2}\sigma\int_{\theta_2}^{\pi-\theta_1}\frac{\rho_0'\eta'}{\sqrt{\rho_0^{2}+\rho_0'^{2}}} \notag\\
        &-t^{2}[\![\gamma]\!](\eta(\theta_1)+\eta(\theta_2)) \notag\\
         &+O(t^{3}) \label{equ:3.4.2}
   \end{align}

    Using the conservation law of the total mass, we have

   \begin{equation*}
       \int_{\theta_{2}}^{\pi-\theta_1} \rho_{0}^{2}d\theta=\int_{\theta_{2}}^{\pi-\theta_1} (\rho_{0}+t\xi+t^{2}\eta+O(t^{3}))^{2}d\theta
   \end{equation*}

   \noindent By comparing the first and second-order terms with respect to $t$ on both sides of the equation, we obtain the following two relations for $\xi$ and $\eta$

   \begin{equation}{\label{equ:3.4.3}}
       \int_{\theta_2}^{\pi-\theta_1} 2\rho_0\xi=0,
   \end{equation}

   \noindent and

   \begin{equation}{\label{equ:3.4.4}}
        -\int_{\theta_2}^{\pi-\theta_1} 2\rho_0\eta= \int_{\theta_2}^{\pi-\theta_1} \xi^{2}.
   \end{equation}

  \noindent Using Euler-Lagrange equation satisfied by the steady state $\rho_{0}$ , we obtain the following relations

  \begin{align}
      P_{0}t\int_{\theta_2}^{\pi-\theta_1}\rho_0\xi d\theta= &tg\int_{\theta_2}^{\pi-\theta_1}\rho_0^{2}(\theta)\xi(\theta)\sin\theta d\theta+t\sigma\int_{\theta_2}^{\pi-\theta_1}\frac{\rho_0\xi}{\sqrt{\rho_0^{2}+\rho_0'^{2}}}d\theta\notag\\
     & +t\sigma\int_{\theta_2}^{\pi-\theta_1}\frac{\rho_0'\xi'}{\sqrt{\rho_0^{2}+\rho_0'^{2}}}d\theta-t[\![\gamma]\!](\xi(\theta_1)+\xi(\theta_2))\label{equ:3.4.5},
  \end{align}

  \noindent and

  \begin{align}
      P_{0}t^{2}\int_{\theta_2}^{\pi-\theta_1}\rho_{0}\eta d\theta= &t^{2}g\int_{\theta_2}^{\pi-\theta_1}\rho_{0}^{2}(\theta)\eta(\theta)\sin\theta d\theta+t^{2}\sigma\int_{\theta_2}^{\pi-\theta_1}\frac{\rho_0{}\eta}{\sqrt{\rho_{0}^{2}+\rho_{0}'^{2}}}\notag\\
     & +t^{2}\sigma\int_{\theta_2}^{\pi-\theta_1}\frac{\rho_{0}'\eta'}{\sqrt{\rho_{0}^{2}+\rho_{0}'^{2}}}-t^{2}[\![\gamma]\!](\eta(\theta_1)+\eta(\theta_2))\label{equ:3.4.6}.
  \end{align}
  
  \noindent Applying equation \eqref{equ:3.4.3} to \eqref{equ:3.4.5}, we have

  \begin{align}
       0=&tg\int_{\theta_2}^{\pi-\theta_1}\rho_{0}^{2}(\theta)\xi(\theta)d\theta+t\sigma\int_{\theta_2}^{\pi-\theta_1}\frac{\rho_{0}\xi}{\sqrt{\rho_{0}^{2}+\rho_{0}'^{2}}}\notag\\
     & +t\sigma\int_{\theta_2}^{\pi-\theta_1}\frac{\rho_{0}'\xi'}{\sqrt{\rho_{0}^{2}+\rho_{0}'^{2}}}-t[\![\gamma]\!](\xi(\theta_1)+\xi(\theta_2))\label{equ:3.4.7}
  \end{align}

  \noindent Then plugging \eqref{equ:3.4.6} and \eqref{equ:3.4.7} into the equation \eqref{equ:3.4.2}, we obtain 
  
   \begin{align*}
       &E(\rho_{0}(\theta)+f(t,\theta))-E(\rho_{0}(\theta))\\
       =&t^{2}g\int_{\theta_2}^{\pi-\theta_1}\rho_{0}(\theta)\xi^{2}(\theta)\sin\theta d\theta+\frac{1}{2}t^{2}\sigma\int_{\theta_2}^{\pi-\theta_1}\frac{\xi^{2}+\xi'^{2}}{(\rho_{0}^{2}+\rho_{0}'^{2})^{\frac{1}{2}}}d\theta\\
       &-\frac{1}{8}\sigma\int_{\theta_2}^{\pi-\theta_1}\frac{(2t\xi\rho_{0}+2t\rho_{0}'\xi')^{2}}{(\rho_{0}^{2}+\rho_{0}'^{2})^{\frac{3}{2}}}+P_{0}\int_{\theta_{2}}^{\pi-\theta}\rho_{0}\eta+O(t^{3}).
   \end{align*}

   \noindent Applying \eqref{equ:3.4.4} to the equation above, it follows that

   \begin{align}
       &E(\rho_{0}(\theta)+f(t,\theta))-E(\rho_{0}(\theta))\notag\\
       =&t^{2}g\int_{\theta_2}^{\pi-\theta_1}\rho_{0}(\theta)\xi^{2}(\theta)\sin\theta d\theta+\frac{1}{2}t^{2}\sigma\int_{\theta_2}^{\pi-\theta_1}\frac{\xi^{2}+\xi'^{2}}{(\rho_{0}^{2}+\rho_{0}'^{2})^{\frac{1}{2}}}d\theta \notag\\
       &-\frac{1}{8}\sigma\int_{\theta_2}^{\pi-\theta_1}\frac{(2t\xi\rho_{0}+2t\rho_{0}'\xi')^{2}}{(\rho_{0}^{2}+\rho_{0}'^{2})^{\frac{3}{2}}}-\frac{1}{2}t^{2}P_{0}\int_{\theta_{2}}^{\pi-\theta}\xi^{2}+O(t^{3}) . \label{equ:3.4.8}
   \end{align}

   \noindent In equation \eqref{equ:3.4.8}, the relation between the pressure $P_{0}$ above and the minimizer $\rho_{0}$ is implicit. Therefore,  we use Euler-Lagrange equation 
   \begin{equation*}
       g\rho_{0}^{2}\sin\theta+\sigma\frac{\rho_{0}}{\sqrt{\rho_{0}^{2}+\rho_{0}'^{2}}}-\sigma\frac{\rho_{0}''\rho_{0}^{2}+\rho_{0}''\rho_{0}'^{2}-\rho_{0}\rho_{0}'^{2}-\rho_{0}''\rho_{0}'^{2}}{(\rho_{0}^{2}+\rho_{0}'^{2})^{\frac{3}{2}}}-P_{0}\rho_{0}=0
   \end{equation*}
   for the minimizer $\rho_{0}$ again to eliminate this term. We rewrite the equation for $P_{0}$ as follows

   \begin{equation}{\label{equ:3.4.9}}
       g\rho_{0}^{2}\sin\theta-P_{0}\rho_{0}=-\sigma\frac{\rho_{0}^{3}+2\rho_{0}\rho_{0}'^{2}-\rho_{0}''\rho_{0}^{2}}{(\rho_{0}^{2}+\rho_{0}'^{2})^{\frac{3}{2}}}
   \end{equation}

   \noindent  Plugging this equation \eqref{equ:3.4.9} back into equation \eqref{equ:3.4.8},  we have (where we only consider the second order term):

   \begin{align}
       &E(\rho_{0}(\theta)+f(t,\theta))-E(\rho_{0}(\theta))\notag\\
        =&\frac{1}{2}t^{2}g\int_{\theta_2}^{\pi-\theta_1} \rho_{0}\xi^{2}\sin\theta d\theta\notag\\
         &+\frac{1}{2}t^{2}\sigma\int_{\theta_2}^{\pi-\theta_1}(-\frac{\rho_{0}^{2}+2\rho_{0}'^{2}-\rho_{0}''\rho_{0}}{(\rho_{0}^{2}+\rho_{0}'^{2})^{\frac{3}{2}}}\xi^{2}+\frac{\xi^{2}+\xi'^{2}}{(\rho_{0}^{2}+\rho_{0}'^{2})^{\frac{1}{2}}}-\frac{(\xi\rho_{0}+\rho_{0}'\xi')^{2}}{(\rho_{0}^{2}+\rho_{0}'^{2})^{\frac{3}{2}}})d\theta\notag\\
         =&\frac{1}{2}t^{2}g\int_{\theta_2}^{\pi-\theta_1}\rho_{0}\xi^{2}\sin\theta d\theta \notag\\
         &+\frac{1}{2}t^{2}\sigma\int_{\theta_2}^{\pi-\theta_1} (\frac{-\rho_{0}'^{2}+\rho_{0}''\rho_{0}}{(\rho_{0}^{2}+\rho_{0}'^{2})^{\frac{3}{2}}}\xi^{2}+\frac{\rho_{0}'^{2}\xi'^{2}+\rho_{0}^{2}\xi'^{2}}{(\rho_{0}^{2}+\rho_{0}'^{2})^{\frac{3}{2}}}-\frac{\xi^{2}\rho_{0}^{2}+2\xi\rho_{0}\rho_{0}'\xi'+\rho_{0}'^{2}\xi'^{2}}{(\rho_{0}^{2}+\rho_{0}'^{2})^{\frac{3}{2}}})d\theta \notag\\
          =&\frac{1}{2}t^{2}g\int_{\theta_2}^{\pi-\theta_1}\rho_{0}\xi^{2}\sin\theta d\theta \notag\\
          &+\frac{1}{2}t^{2}\sigma \int_{\theta_2}^{\pi-\theta_1} (\frac{-\rho_{0}'^{2}-\rho_{0}^{2}+\rho_{0}''\rho_{0}}{(\rho_{0}^{2}+\rho_{0}'^{2})^{\frac{3}{2}}}\xi^{2}+\frac{\rho_{0}^{2}\xi'^{2}}{(\rho_{0}^{2}+\rho_{0}'^{2})^{\frac{3}{2}}}-\frac{2\xi\xi'\rho_{0}\rho_{0}'}{(\rho_{0}^{2}+\rho_{0}'^{2})^{\frac{3}{2}}})d\theta\notag\\
           =&\frac{1}{2}t^{2}g\int_{\theta_2}^{\pi-\theta_1}\rho_{0}\xi^{2}\sin\theta d\theta \notag\\
           &+\frac{1}{2}t^{2}\sigma \int_{\theta_2}^{\pi-\theta_1} (\frac{-\rho_{0}'^{2}-\rho_{0}^{2}+\rho_{0}''\rho_{0}}{(\rho_{0}^{2}+\rho_{0}'^{2})^{\frac{3}{2}}}\xi^{2}+\frac{\rho_{0}^{2}\xi'^{2}}{(\rho_{0}^{2}+\rho_{0}'^{2})^{\frac{3}{2}}}-\frac{\frac{d}{d\theta}(\xi^{2})\rho_{0}\rho_{0}'}{(\rho_{0}^{2}+\rho_{0}'^{2})^{\frac{3}{2}}})d\theta \label{equ:3.4.10}
   \end{align}
   This formula is exactly what we want.
     \end{proof}

       The stability result in Section 3 implies that $F_{0}(\xi)$ is nonnegative for any $\xi\in B$. For the remaining part of this section, our goal is to determine the kernel of this functional. Throughout this part, we also restrict our attention to the case where the two inclined angles vanish, i.e $\theta_{1}=\theta_{2}=0$.
     
     It is known that the functional energy $E$ is invariant under the horizontal transformation of the form $(x,y)\rightarrow (x-a,y)$ for any real constant $a$. We denote the disturbing function $\xi$ induced by this shift the function $\xi_s$. Then we have the following result

     \begin{equation}{\label{equ:3.3.3}}
         F_{0}(\xi_s)=0
     \end{equation}

     For any $\xi,\tilde{\xi}\in B$, the following identity holds

     \begin{equation}{\label{equ:3.3.4}}
         F_{0}(t\xi+\tilde\xi)=t^{2}F_{0}(\xi)+F_{0}(\tilde\xi)+2tH(\xi,\tilde\xi),
     \end{equation}

     \noindent where:

     \begin{equation*}
     \begin{aligned}
         H(\xi,\tilde\xi)=&g\int_{0}^{\pi} \rho \xi\tilde\xi\sin\theta d\theta\\
         &+\sigma \int_{0}^{\pi} (-\frac{\rho_{0}^{2}+2\rho_{0}'^{2}-\rho_{0}''\rho_{0}}{(\rho_{0}^{2}+\rho_{0}'^{2})^{\frac{3}{2}}}\xi\tilde\xi+\frac{\xi\tilde\xi+\xi'\tilde\xi'}{(\rho_{0}^{2}+\rho_{0}'^{2})^{\frac{1}{2}}}-\frac{(\xi\rho_{0}+\rho_{0}'\xi')(\tilde\xi\rho_{0}+\rho_{0}'\tilde\xi)}{(\rho_{0}^{2}+\rho_{0}'^{2})^{\frac{3}{2}}})d\theta
        \end{aligned}
     \end{equation*}

     \noindent We notice that $H(\xi,\xi)=F_{0}(\xi)$ and $H$ can be viewed as a scalar product.  Using the fact that $F_{0}(\xi_{s})=0$ in equation \eqref{equ:3.3.4}, we can show the following result by setting $\xi=\xi_{s}$:

     \begin{equation}{\label{equ:3.2.5}}
         F_{0}(t\xi_s+\tilde\xi)=F_{0}(\tilde\xi)+2tH(\xi_s,\tilde \xi)
     \end{equation}

     \noindent We are now ready to establish the following theorem.

     \begin{theorem}{\label{thm:kernel}}

     $\xi_1$ is a function in the set B. Suppose $F_{0}(\xi_1)=0$. Then $H(\xi_1,\xi)=0$ for any $\xi\in B$.
     
     \end{theorem}

     \begin{proof}

         We prove this theorem by a contradiction argument. Suppose result is not true. There exists a function $\xi_{2}\in B$ such that:

         \begin{equation*}
             H(\xi_1,\xi_2)\neq 0
         \end{equation*}

         \noindent Without loss of generality, we suppose that:

         \begin{equation*}
             H(\xi_1,\xi_2)<-c<0,
         \end{equation*}

         \noindent where c is a positive constant. Otherwise, we can choose $-\xi_2$ to be test function. Then applying this relation to equation \eqref{equ:3.2.5}, we have

         \begin{equation}{\label{equ:3.2.6}}
             F_{0}(t\xi_1+\xi_2)<F_{0}(\xi_2)-2tc
         \end{equation}

         \noindent In equation \eqref{equ:3.2.6}, choosing $t$ to be large enough such that:

         \begin{equation*}
             F_{0}(\xi_2)-2tc<0,
         \end{equation*}

         \noindent we have

         \begin{equation*}
             F_{0}(t\xi+\xi_2)<0.
         \end{equation*}

         \noindent This contradicts to the fact that $F_{0}$ is nonnegative. Hence, we have the theorem proved.
     \end{proof}

     In Theorem 5.3, we have shown a weak form of PDE that any kernel function $\xi_1$ should satisfy. We want to proceed to find the strong form of the equation that $\xi_1$ satisfy.  Now we have equation $H(\xi_1,\xi)=0$ for any $\xi\in B$. Using the equation \eqref{equ:4.3.1}, there exists a function $v\in C_{0}^{+\infty}(0,\pi)$ such that:

     \begin{equation*}
         \nabla \cdot v=\rho_{0}(\theta)\xi(\theta)
     \end{equation*}

     \noindent In one-dimension case, $\nabla \cdot v=\partial_{\theta}v$. This intrigues us to establish following theorem

     \begin{theorem}{\label{thm:ode}}

        $\xi_{1}$ is a function in function set $B$. If $H(\xi_1,\xi)=0$ for any $\xi\in B$, $\xi_1$ satisfies the following equation:

         \begin{equation}{\label{equ:3.2.7}}
             g\xi_1\sin\theta-\frac{\sigma}{\rho_0}\frac{\rho_0^{2}+2\rho_0'^{2}-\rho_0''\rho_0}{(\rho_0^{2}+\rho_0'^{2})^{\frac{3}{2}}}\xi_1+\sigma\frac{2\frac{\rho_0'^{2}}{\rho_0}\xi_1-2\rho_0'\xi_1'}{(\rho_0^{2}+\rho_0'^{2})^{\frac{3}{2}}}-\sigma\partial_{\theta}(\frac{\rho_0\xi_1'-\rho_0'\xi_1}{(\rho_0^{2}+\rho_0'^{2})^{\frac{3}{2}}})=C
         \end{equation}

         \noindent for some constant $C$ to be determined.
     \end{theorem}

     \begin{proof}
     
      We write down the weak form equation $H(\xi_1,\xi)=0$ as follows

      \begin{align}
          H(\xi_1,\xi_2)=&g\int_{0}^{\pi} \rho_0 \xi_1\xi\sin\theta d\theta \notag\\
          &+\sigma \int_{0}^{\pi} (-\frac{\rho_0^{2}+2\rho_0'^{2}-\rho_0''\rho_0}{(\rho_0^{2}+\rho_0'^{2})^{\frac{3}{2}}}\xi_1\xi+\frac{\xi_1\xi+\xi_1'\xi'}{(\rho_0^{2}+\rho_0'^{2})^{\frac{1}{2}}}-\frac{(\xi_1\rho_0+\rho_0'\xi_1')(\xi\rho_0+\rho_0'\xi')}{(\rho_0^{2}+\rho_0'^{2})^{\frac{3}{2}}})d\theta \notag\\
          =&g\int_{0}^{\pi} \xi_1\nabla \cdot v\sin\theta d\theta \notag\\
          &+\sigma \int_{0}^{\pi} (-\frac{1}{\rho_0}\frac{\rho_0^{2}+2\rho_0'^{2}-\rho_0''\rho_0}{(\rho_0^{2}+\rho_0'^{2})^{\frac{3}{2}}}\xi_1\nabla \cdot v+\frac{\frac{1}{\rho_0}\xi_1(\nabla \cdot v)+\frac{1}{\rho_0}\xi_1'(\nabla \cdot v)'-\frac{\rho_0'}{\rho_0^{2}}\xi_1'(\nabla \cdot v)}{(\rho_0^{2}+\rho_0'^{2})^{\frac{1}{2}}}d\theta\notag\\
          &-\frac{(\xi_1\rho_0+\rho_0'\xi_1')((\nabla \cdot v)+\frac{\rho_0'}{\rho_0}(\nabla \cdot v)'-\frac{\rho_0'{2}}{\rho_0^{2}}(\nabla \cdot v))}{(\rho_0^{2}+\rho_0'^{2})^{\frac{3}{2}}})d\theta\label{equ:3.5.4},
      \end{align}
      
     \noindent where we use $\nabla\cdot v=\rho_{0}(\theta)\xi(\theta)$ to substitute $\xi$. Since $v\in C_{0}^{\infty}$, using integration by parts in equation \eqref{equ:3.5.4}, we obtain

     \begin{align}
         \int_{0}^{\pi}\p_{\theta} (g\xi_1\sin\theta-\frac{1}{\rho}\frac{\rho_0^{2}+2\rho_0'^{2}-\rho_0''\rho_0}{(\rho_0^{2}+\rho_0'^{2})^{\frac{3}{2}}}\xi_1+\frac{2\frac{\rho_0'^{2}}{\rho_0}\xi_1-2\rho_0'\xi_1'}{(\rho_0^{2}+\rho_0'^{2})^{\frac{3}{2}}}-\partial_{\theta}(\frac{\rho_0\xi_1'-\rho_0'\xi_1}{(\rho_0^{2}+\rho_0'^{2})^{\frac{3}{2}}})) vd\theta=0
     \end{align}
     \noindent for any $v\in C_{0}^{+\infty}(0,\pi)$. This implies
     \begin{equation*}
         \p_{\theta} (g\xi_1\sin\theta-\frac{1}{\rho}\frac{\rho_0^{2}+2\rho_0'^{2}-\rho_0''\rho_0}{(\rho_0^{2}+\rho_0'^{2})^{\frac{3}{2}}}\xi_1+\frac{2\frac{\rho_0'^{2}}{\rho_0}\xi_1-2\rho_0'\xi_1'}{(\rho_0^{2}+\rho_0'^{2})^{\frac{3}{2}}}-\partial_{\theta}(\frac{\rho_0\xi_1'-\rho_0'\xi_1}{(\rho_0^{2}+\rho_0'^{2})^{\frac{3}{2}}}))=0,
     \end{equation*}

     \noindent which means that the term inside the gradient is a constant. Therefore,  $\xi_1$ satisfies the equation \eqref{equ:3.2.7}. 

     \end{proof}
     
     Let $\xi_s$ be the shift function(linearized). We have the following theorem for $\xi_s$:

     \begin{theorem}
         $\xi_s$ is a solution function to \eqref{equ:3.2.7} with $C=0$
     \end{theorem}

     \begin{proof}

          Suppose that $\rho_{0}$ is the solution to the Euler-Lagrange equation. After the shift, this equilibrium becomes

         \begin{equation*}
             \tilde{\rho_{0}}(\theta,\epsilon)=\rho_{0}+\epsilon\xi_s+O(\epsilon)^{2}.
         \end{equation*}
         
         \noindent This new function $\tilde{\rho}_{0}$ is also a solution to equation \eqref{equ:2.1.1} since it is also a minimizer of energy functional $E$. Then  taking the derivative with respect to $\epsilon$ on both sides of equation \eqref{equ:2.1.1}, we obtain the following equation

         \begin{equation}{\label{equ:3.2.8}}
             \frac{d}{d\epsilon}(g\tilde{\rho}_0\sin\theta+\sigma\frac{2\tilde{\rho}_{0}'^{2}+\tilde{\rho}_{0}^{2}-\tilde{\rho}_{0}\tilde{\rho}_{0}''}{(\tilde{\rho}_{0}^{2}+\tilde{\rho}_{0}'^{2})^{\frac{3}{2}}}-P_{0})=0
         \end{equation}

         \noindent Based on a direct computation, we have

         \begin{align}
               \frac{d}{d\epsilon}\tilde{\rho}_{0}|_{\epsilon=0}=&\xi_s \label{equ:3.2.9}\\
               \sigma \frac{d}{d\epsilon}\frac{1}{(\tilde{\rho}_{0}^{2}+\tilde{\rho}_{0}'^{2})^{\frac{3}{2}}}|_{\epsilon=0}=&-3\sigma \frac{\rho_{0}\xi_s+\rho_{0}'\xi_s'}{\sqrt{\rho_{0}^{2}+\rho_{0}'^{2}}^{5}} \label{equ:3.2.10}\\
               \frac{d}{d\epsilon}(2\tilde{\rho}_{0}'^{2}+\tilde{\rho}_{0}^{2}-\tilde{\rho}_{0}\tilde{\rho}_{0}'')|_{\epsilon=0}=&4{\rho}_{0}'\xi_s'+2{\rho}_{0}\xi_s-\rho_{0}''\xi_s-\rho_{0}\xi_s''\label{equ:3.2.11}.
         \end{align}

         \noindent Plugging \eqref{equ:3.2.9}-\eqref{equ:3.2.11} into \eqref{equ:3.2.8}, we obtain that

         \begin{equation}{\label{equ:3.2.12}}
             g\xi_s\sin\theta+\sigma\frac{4{\rho}_{0}'\xi_s'+2{\rho}_{0}\xi_s-\rho_{0}''\xi_s-\rho_{0}\xi_s''}{(\rho_{0}^{2}+\rho_{0}'^{2})^{\frac{3}{2}}}-3\sigma \frac{\rho_{0}\xi_s+\rho_{0}'\xi_s'}{\sqrt{\rho_{0}^{2}+\rho_{0}'^{2}}^{5}}(2{\rho}_{0}'^{2}+{\rho}_{0}^{2}-{\rho}_{0}{\rho}_{0}'')=0
         \end{equation}

         \noindent which is equivalent to:

         \begin{equation}{\label{equ:3.2.13}}
             g\xi_s\sin\theta-\sigma \frac{\rho_{0}\xi_s''}{(\rho_{0}^{2}+\rho_{0}'^{2})^{\frac{3}{2}}}+\sigma\frac{\rho'\rho^{2}-2\rho'^{3}+3\rho\rho'\rho''}{(\rho_{0}^{2}+\rho_{0}'^{2})^{\frac{5}{2}}}\xi_s'+\sigma\frac{2\rho_{0}^{2}\rho_{0}''-\rho_{0}^{3}-4\rho_{0}\rho_{0}'^{2}-\rho_{0}''\rho_{0}'^{2}}{(\rho_{0}^{2}+\rho_{0}'^{2})^{\frac{5}{2}}}\xi_s=0
         \end{equation}

        \noindent By a simple computation, we verify that \eqref{equ:3.2.13} is equivalent to equation \eqref{equ:3.2.7} with $C=0$.
     \end{proof}
     
     From Theorem 4.4, we have obtained a special solution function $\xi_s$  which serves as a kernel of the functional $F_{0}$. In what follows, we analyze the properties of function $\xi_s$ in detail. These properties are summarized in the following theorem.

     \begin{theorem}
         $\xi_s$ is the shift function defined above. It has the following properties:

         \begin{align}
             \xi_s(0)&=-\xi_s(\pi) \label{equ:3.2.14}\\
             \xi_s'(0)&=\xi_s'(\pi)\label{equ:3.2.15}\\
             \xi_s(\frac{\pi}{2})&=0\label{equ:3.2.16}\\
             \xi_s(\frac{\pi}{2})'&\neq 0 \label{equ:3.2.17}\\
         \end{align}

     \end{theorem}

     \begin{proof}
         To prove these properties, it is better to derive the representation for $\xi_s$. After performing the horizontal shift $(x,y)\rightarrow (x-\epsilon,y)$, we have the following relation which gives the coordinates of point $(x-\epsilon,y)$ in polar coordinates:

         \begin{equation}{\label{equ:3.2.18}}
             \tilde{\rho}_{0}(\theta')=\sqrt{(x-\epsilon)^{2}+y^{2}}=\sqrt{\rho_{0}^{2}(\theta)-2\epsilon\rho_{0}(\theta)\cos\theta+\epsilon^{2}},
         \end{equation}

         \noindent where

         \begin{align}
             \sin\theta'=\frac{\rho_{0}(\theta 
             )\sin\theta}{\sqrt{\rho_{0}^{2}(\theta)-2\epsilon\rho_{0}(\theta)\cos\theta+\epsilon^{2}}} \label{equ:3.2.19}\\
             \cos\theta'=\frac{\rho_{0}\cos\theta-\epsilon}{\sqrt{\rho_{0}^{2}(\theta)-2\epsilon\rho_{0}(\theta)\cos\theta+\epsilon^{2}}} \label{equ:3.2.20}.
         \end{align}

         \noindent Therefore, using the definition of shifting function $\xi_{s}$ and equation \eqref{equ:3.2.18}, we have the following computation:

         \begin{equation}{\label{equ:3.2.21}}
             \xi_s(\theta')=-\frac{d}{d\epsilon}(\tilde{\rho}_{0}(\theta')-\rho(\theta'))_{\epsilon=0}=-\frac{d}{d\epsilon}(\rho_{0}(\theta)-\epsilon\cos\theta-\rho(\theta'))|_{\epsilon}=-\rho_{0}'(\theta')\frac{d}{d\epsilon}(\theta-\theta')|_{\epsilon=0}+\cos\theta
         \end{equation}

         \noindent  Using equations \eqref{equ:3.2.19} and \eqref{equ:3.2.20}, we have the following relations if we view $\theta'$ as a function of $\theta$:

         \begin{align}
             \theta'(\pi)=\pi \label{equ:3.2.22}\\
             \theta'(0)=0 \label{equ:3.2.23}
         \end{align}

         \noindent Therefore, $\theta'=\theta$ when $\theta={\pi}$ or $\theta=0$, which means that $\frac{d}{d\epsilon}(\theta-\theta')|_{\epsilon=0}=0$ when $\theta'=0$ or $\theta'={\pi}$. Use this result in equation \eqref{equ:3.2.21}, we obtain

         \begin{equation*}
             \xi_s(0)=1
         \end{equation*}

         \noindent and:

         \begin{equation*}
             \xi_{s}(\pi)=-1=-\xi_s(0)
         \end{equation*}

         \noindent which is \eqref{equ:3.2.14}.

         At the point $\theta'=\frac{\pi}{2}$, we have $\cos\theta=\frac{\epsilon}{\rho_{0}(\theta)}$ from equation \eqref{equ:3.2.21}, which implies that:
         
         \begin{align}
             \frac{d}{d\epsilon}(\theta-\theta')|_{\epsilon=0}=-\frac{1}{\rho_{0}}~\operatorname{when}~\theta=\frac{\pi}{2}
         \end{align}
         
        \noindent Applying this relation to \eqref{equ:3.2.21}, we obtain
        
         \begin{equation}{\label{equ:3.2.24}}
             \xi_s(\frac{\pi}{2})=\rho_{0}'(\frac{\pi}{2})\frac{1}{\rho_{0}}+\cos(\frac{\pi}{2})
         \end{equation}

         \noindent Since $\theta=\frac{\pi}{2}$ is the maximum point for $\rho_{0}$, we have $\rho'(\frac{\pi}{2})=0$. Using this fact in equation \eqref{equ:3.2.24}, we obtain:

         \begin{equation*}
             \xi_s(\frac{\pi}{2})=0
         \end{equation*}
         
          Now taking derivative with respect to $\theta'$ on both sides of equation \eqref{equ:3.2.21}, we obtain

         \begin{equation}{\label{equ:3.2.25}}
             \xi_s'(\theta)=-\rho_{0}'(\theta)\frac{d}{d\epsilon}(\frac{d\theta}{d\theta'}-1)|_{\epsilon=0}-\rho_{0}''(\theta)\frac{d}{d\epsilon}(\theta-\theta')|_{\epsilon=0}-\sin\theta
         \end{equation}

         \noindent  Let $\theta=\frac{\pi}{2}$ in equation \eqref{equ:3.2.25}, and note that $\rho_{0}'(\frac{\pi}{2})=0$. Equation \eqref{equ:3.2.25} can be rewritten as follows:
         
         \begin{equation}{\label{equ:3.2.26}}
             \xi_s'(\frac{\pi}{2})=-\rho''(\frac{\pi}{2})\frac{d}{d\epsilon}(\theta-\theta')-1=\frac{\rho_{0}''(\frac{\pi}{2})}{\rho_{0}(\frac{\pi}{2})}-1.
         \end{equation}

         \noindent Using the Euler-Lagrange equation\eqref{equ:2.1.1} for equilibrium $\rho_{0}$, we have(This is the maximum point for $\rho_{0}$):

         \begin{equation*}
             \sigma \rho_{0}''(\frac{\pi}{2})\frac{1}{\rho_{0}^{2}(\frac{\pi}{2})}=g\rho_{0}(\frac{\pi}{2})-P_{0}+\sigma\frac{1}{\rho_{0}(\frac{\pi}{2})}
         \end{equation*}

         \noindent which implies that:

         \begin{equation*}
             \frac{\rho_{0}''(\frac{\pi}{2})}{\rho_{0}(\frac{\pi}{2})}=\frac{(g\rho_{0}-P_{0})\rho_{0}}{\sigma }+1
         \end{equation*}

         \noindent Using Theorem 2.1 in \cite{Yang} , $(g\rho_{0}-P_{0})(\frac{\pi}{2})<0$, which implies that:

         \begin{equation}{\label{equ:3.2.27}}
             \frac{\rho_{0}''(\frac{\pi}{2})}{\rho_{0}(\frac{\pi}{2})}<C<1
         \end{equation}

         \noindent  Then using equation \eqref{equ:3.2.27} in equation \eqref{equ:3.2.26}, we obtain 

         \begin{equation*}
             \xi_s'(\frac{\pi}{2})=-(1-\frac{\rho_{0}(\frac{\pi}{2})''}{\rho_{0}(\frac{\pi}{2})})<0
         \end{equation*}

         \noindent which finishes the proof of the theorem.
     \end{proof}

     Furthermore, We establish a theorem concerning the quantity $\frac{\xi_s'}{\xi_s}$ on the boundary and the explicit representation for shift function $\xi_{s}$.

     \begin{theorem}{\label{thm:bdd_s}}
        The following relation holds for $\xi_{s}$

         \begin{equation*}
             \frac{\xi_s'(\pi)}{\xi_{s}(\pi)}=-\frac{\xi_{s}'(0)}{\xi_s(0)}=\frac{\rho_{0}'(\pi)}{\rho_{0}(\pi)}=-\frac{\rho_{0}'(0)}{\rho_{0}(0)}.
         \end{equation*}

         \noindent Moreover, the function $\xi_{s}$ is expressed explicitly as follows:

         \begin{align}{\label{equ:rep_s}}
             \xi_{s}=\cos\theta+\frac{\rho_{0}'}{\rho_{0}}\sin\theta
         \end{align}
         
     \end{theorem}

     \begin{proof}

         Using the symmetry of the steady state $\rho_{0}$, it holds that

         \begin{equation*}
             \frac{\rho_{0}'(\pi)}{\rho_{0}(\pi)}=-\frac{\rho_{0}'(0)}{\rho_{0}(0)}
         \end{equation*}

         \noindent Then we compute the derivative of $\xi_s$ at point 0. Using equation \eqref{equ:3.2.22} and \eqref{equ:3.2.23}, we have $\frac{d}{d\epsilon}(\theta-\theta')|_{\epsilon=0}=0$ when $0$. Substituting these into equation \eqref{equ:3.2.25}, we obtain the following result

         \begin{equation}{\label{equ:3.2.37}}
             \xi_s'(0)=-\rho_{0}'(0)\frac{d}{d\epsilon}(\frac{d\theta}{d\theta'}-1)
         \end{equation}

         \noindent Combining equation \eqref{equ:3.2.19} and equation \eqref{equ:3.2.20}, we obtain the following result for $\frac{d\theta^{\prime}}{d\theta}$

         \begin{equation*}
             \frac{d\theta'}{d\theta}=\frac{1}{1+(\frac{(\rho_{0}\sin\theta)^{2}}{(\rho_{0}\cos\theta-\epsilon)^{2}})}\frac{(\rho_{0}'\sin\theta+\rho_{0}\cos\theta)(\rho_{0}\cos\theta-\epsilon)-(\rho_{0}'\cos\theta-\rho_{0}\sin\theta)\rho_{0}\sin\theta}{(\rho_{0}\cos\theta-\epsilon)^{2}},
         \end{equation*}

         \noindent which implies that

         \begin{equation}{\label{equ:3.2.36}}
             \frac{d\theta}{d\theta'}=\frac{\rho_{0}^{2}-2\epsilon\rho_{0}\cos\theta}{\rho_{0}^{2}-\epsilon\rho_{0}\cos\theta-\epsilon\rho_{0}'\sin\theta}+O(\epsilon^{2}).
         \end{equation}

         \noindent Applying equation\eqref{equ:3.2.36} to equation \eqref{equ:3.2.37} and computing the derivative with respect to $\epsilon$, we obtain the following result

         \begin{equation*}
             \xi_s'(0)=\frac{\rho_{0}'(0)}{\rho_{0}(0)},
         \end{equation*}

         \noindent which implies that

         \begin{equation*}
             \frac{\xi_s'(0)}{\xi_s(0)}=\frac{\rho_{0}'(0)}{\rho_{0}(0)}.
         \end{equation*}

         \noindent Here, we used the fact that $\xi_s(0)=1$. Similarly, we use the same computation as for $\theta=\pi$ to derive
         \begin{align}
              \frac{\xi_s'(\pi)}{\xi_s(\pi)}=\frac{\rho_{0}'(\pi)}{\rho_{0}(\pi)}.
         \end{align}

         Then it remains to derive the representation fo $\xi_{s}$. From equation \eqref{equ:3.2.21}, it suffices to compute $\frac{d}{d\epsilon}(\theta-\theta')|_{\epsilon=0}$. By equation \eqref{equ:3.2.19} and equation \eqref{equ:3.2.20}, we have

         \begin{align}
             \frac{d}{d\epsilon}(\sin\theta')=\cos\theta'\frac{d\theta'}{d\epsilon}\label{equ:3.4.37},
         \end{align}

         \noindent and,

         \begin{align}
             \frac{d}{d\epsilon}(\sin\theta')|_{\epsilon=0}=\frac{\rho_{0}^{2}\sin\theta\cos\theta}{\rho_{0}^{3}}=\frac{\sin\theta\cos\theta}{\rho_{0}}\label{equ:3.4.38}.
         \end{align}

         \noindent  Then applying equation \eqref{equ:3.4.38} to the equation \eqref{equ:3.4.37}, we obtain

         \begin{align}
             \frac{d\theta'}{d\epsilon}|_{\epsilon=0}=\frac{\sin\theta}{\rho_{0}}\label{equ:3.4.39}
         \end{align}

         \noindent Finally, applying equation \eqref{equ:3.4.39} to equation \eqref{equ:3.2.21}, we obtain

         \begin{align}
             \xi_{s}=\cos\theta+\frac{\rho_{0}'}{\rho_{0}}\sin\theta
         \end{align}
     \end{proof}

     In addition to the properties of $\xi_s$ at some special points, it is also important to derive a general symmetry property for $\xi_s$. We establish this property in the following theorem

     \begin{lemma} \label{thm:xi_2sym}
         $\xi_s$ is the linearized shift function defined above. It satisfies the following properties:

         \begin{align}
             \xi_{s}(\frac{\pi}{2}-\gamma)=&-\xi_s(\frac{\pi}{2}+\gamma)\label{equ:3.4.40}\\
             \xi_{s}'(\frac{\pi}{2}-\gamma)=&\xi_{s}'(\frac{\pi}{2}+\gamma)\label{equ:3.4.41}
         \end{align}
         
         \noindent For any $\gamma\in (0,\frac{\pi}{2})$
    \end{lemma}
     
     \begin{proof}

     We first prove equation \eqref{equ:3.4.40}. Using equations \eqref{equ:rep_s}, we have the following equation

     \begin{equation}{\label{equ:3.4.42}}
        \xi_{s}=\cos\theta+\frac{\rho_{0}'}{\rho_{0}}\sin\theta
     \end{equation}

     \noindent Using the symmetry for $\rho$, we have $\rho(\frac{\pi}{2}+\theta)=\rho(\frac{\pi}{2}-\theta)$. Therefore, we obtain the following property:

     \begin{equation}{\label{equ:3.4.43}}
          \xi_{s}(\frac{\pi}{2}-\gamma)=\cos(\frac{\pi}{2}-\gamma)+\frac{\rho_{0}'(\frac{\pi}{2}-\gamma)}{\rho_{0}(\frac{\pi}{2}-\gamma)}\sin(\frac{\pi}{2}-\gamma)=-\cos(\frac{\pi}{2}+\gamma)-\frac{\rho_{0}'(\frac{\pi}{2}+\gamma)}{\rho_{0}(\frac{\pi}{2}+\gamma)}\sin(\frac{\pi}{2}+\gamma)=-\xi_{s}(\frac{\pi}{2}+\gamma)
     \end{equation}

     \noindent which is exactly the equation \eqref{equ:3.4.40}. 

    We then prove the equation \eqref{equ:3.4.41}. Taking derivative with respect to $\theta$ on both sides of equation \eqref{equ:rep_s}, we obtain

   \begin{equation*}
       \xi_{s}'(\theta)=-\sin\theta+\frac{\rho_{0}''\rho_{0}-\rho_{0}'^{2}}{\rho_{0}^{2}}\sin\theta+\frac{\rho_{0}'}{\rho_{0}}\cos\theta
   \end{equation*}

   \noindent Using the symmetry property for $\rho_{0}$: $\rho_{0}(\frac{\pi}{2}-\gamma)=\rho_{0}(\frac{\pi}{2}+\gamma)$ and the properties for trigonometric functions we have

   \begin{align*}
       \xi_{s}'(\frac{\pi}{2}-\gamma)=\xi_{s}(\frac{\pi}{2}+\gamma),
   \end{align*}
    \noindent which finishes the proof.
     \end{proof}
     
     Then we want to derive the boundary condition for the solution function of \eqref{equ:3.2.7} in function set $B$.

     \begin{lemma}{\label{lem:bdd}}
        For any function $\xi_k$ satisfying $F_{0}(\xi_k)=0$, it must satisfy one of the following boundary conditions:

         \begin{align}
             \xi_{k}(\pi)=&-\xi_k(0)\label{equ:3.2.38}\\
             -(\rho_{0}(0)\xi_{k}'(0)-\rho_0'(0)\xi_{k}(0))=&\rho_{0}(\pi)\xi_{k}'(\pi)-\rho_0'({\pi})\xi_{k}(\pi)\label{equ:3.2.39},
         \end{align}

         \noindent or

         \begin{equation}{\label{equ:3.2.40}}
            \frac{\xi_k'(\pi)}{\xi_{k}(\pi)}=-\frac{\xi_{k}'(0)}{\xi_k(0)}=\frac{\rho_{0}'(\pi)}{\rho_{0}(\pi)}=-\frac{\rho_{0}'(0)}{\rho_{0}(0)}
         \end{equation}
         
     \end{lemma}

     \begin{proof}

         From theorem \ref{thm:ode}, $F_{0}(\xi_k)=0$ implies that $H(\xi_k, \xi)=0$ for any $\xi\in B$. Especially, we have $H(\xi_k,\xi_s)=H(\xi_s,\xi_k)=H(\xi_{s},\xi_{s})=0$.   Using integration by part, the condition $H(\xi_{k},\xi_{s})=0$ can be rewritten as the following equation

         \begin{align}
             &\int_{0}^{\pi}( g\xi_k\sin\theta-\frac{1}{\rho_0}\frac{\rho_0^{2}+2\rho_0'^{2}-\rho_0''\rho_0}{(\rho_0^{2}+\rho_0'^{2})^{\frac{3}{2}}}\xi_k+\frac{2\frac{\rho_0'^{2}}{\rho_0}\xi_k-2\rho_0'\xi_k'}{(\rho_0^{2}+\rho_0'^{2})^{\frac{3}{2}}}-\partial_{\theta}(\frac{\rho_0\xi_k'-\rho_0'\xi_k}{(\rho_0^{2}+\rho_0'^{2})^{\frac{3}{2}}}))\rho_{0}\xi_{s}\\
             &\quad+  ((\frac{\rho_{0}'\xi_k-\rho_{0}\xi_k'}{(\rho_{0}^{2}+\rho_{0}'^{2})^{\frac{3}{2}}})\rho_{0}\xi_s)(\pi)-((\frac{\rho_{0}'\xi_k-\rho_{0}\xi_k'}{(\rho_{0}^{2}+\rho_{0}'^{2})^{\frac{3}{2}}})\rho_{0}\xi_s)(0)=0
         \end{align}
        
         \noindent Using the fact that $\xi_k$ is the solution to \eqref{equ:3.2.7}, we have:

         \begin{equation}{\label{equ:3.2.41}}
             ((\frac{\rho_{0}'\xi_k-\rho_{0}\xi_k'}{(\rho_{0}^{2}+\rho_{0}'^{2})^{\frac{3}{2}}})\rho_{0}\xi_s)(\pi)=((\frac{\rho_{0}'\xi_k-\rho_{0}\xi_k'}{(\rho_{0}^{2}+\rho_{0}'^{2})^{\frac{3}{2}}})\rho_{0}\xi_s)(0)
         \end{equation}

         \noindent Moreover, using the same computation for $H(\xi_{s},\xi_{k})$, $H(\xi_{k},\xi_{k})$ and $\xi_{s},\xi_{s}$, we obtain the following result

         \begin{equation}{\label{est:bdd}}
             ((\frac{\rho_{0}'\xi_i-\rho_{0}\xi_i'}{(\rho_{0}^{2}+\rho_{0}'^{2})^{\frac{3}{2}}})\xi_j)(\pi)=((\frac{\rho_{0}'\xi_i-\rho_{0}\xi_i'}{(\rho_{0}^{2}+\rho_{0}'^{2})^{\frac{3}{2}}})\xi_j)(0)
         \end{equation}
         
         \noindent for $i,j\in\{3,s\}$, where we used the fact that $\rho_{0}(0)=\rho_{0}(\pi)$ to cancel the $\rho_{0}$ on each side of the equation. 
         
        Let $j=s$ and $i=k$ in \eqref{est:bdd}, we obtain

         \begin{equation*}
             ((\frac{\rho_{0}'\xi_k-\rho_{0}\xi_k'}{(\rho_{0}^{2}+\rho_{0}'^{2})^{\frac{3}{2}}})\xi_s)(\pi)=((\frac{\rho_{0}'\xi_k-\rho_{0}\xi_k'}{(\rho_{0}^{2}+\rho_{0}'^{2})^{\frac{3}{2}}})\xi_s)(0)
         \end{equation*}

        \noindent Using \eqref{equ:3.2.14} in the equation above, we have the following relation:

        \begin{equation}{\label{equ:3.2.42}}
            (\frac{\rho_{0}'\xi_k-\rho_{0}\xi_k'}{(\rho_{0}^{2}+\rho_{0}'^{2})^{\frac{3}{2}}})(\pi)=-(\frac{\rho_{0}'\xi_k-\rho_{0}\xi_k'}{(\rho_{0}^{2}+\rho_{0}'^{2})^{\frac{3}{2}}})(0)
        \end{equation}

        \noindent Using the symmetry of the function $\rho_{0}$, \eqref{equ:3.2.42} is equivalent to the relation \eqref{equ:3.2.39}. We then set $i=j=k$ in equation \eqref{equ:3.2.41} to obtain that:

        \begin{equation}{\label{equ:3.2.43}}
            ((\frac{\rho_{0}'\xi_k-\rho_{0}\xi_k'}{(\rho_{0}^{2}+\rho_{0}'^{2})^{\frac{3}{2}}})\xi_k)(\pi)=((\frac{\rho_{0}'\xi_k-\rho_{0}\xi_k'}{(\rho_{0}^{2}+\rho_{0}'^{2})^{\frac{3}{2}}})\xi_k)(0)
        \end{equation}

        \noindent If the relation \eqref{equ:3.2.40} is true, then the theorem is proved. If it is not true, the term:

        \begin{align}
        (\frac{\rho_{0}'\xi_k-\rho_{0}\xi_k'}{(\rho_{0}^{2}+\rho_{0}'^{2})^{\frac{3}{2}}})
        \end{align}
        
        \noindent is non-zero at points $0$ and $\pi$. Then we can apply equation \eqref{equ:3.2.42} in equation \eqref{equ:3.2.43} to obtain the following equation 

        \begin{equation*}
            \xi_k(0)=-\xi_k(\pi),
        \end{equation*}

        \noindent which is the relation \eqref{equ:3.2.38}. 
     \end{proof}

     Now we have a function $\xi_{s}$ which is a solution function to the second variation equation \eqref{equ:3.2.7}. We then use $\xi_s$ to construct $\xi_k$ which is the other solution explicitly. Using Theorem 5.5, we know that $\xi_s$ is a solution to equation \eqref{equ:3.2.13}, which is equivalent to \eqref{equ:3.2.7} with $C=0$. Since equation \eqref{equ:3.2.13} is a second-order equation, we construct another basis of the solution space.  Let $\xi_5$ be the function defined as follows:
     
          \begin{equation}
              \xi_5(\theta)=\chi(\theta)\xi_{2}(\theta)
          \end{equation}

          \noindent Suppose that $\xi_{5}$ is a solution function to \eqref{equ:3.2.7} with $C=0$. Then we use this assumption to show that the multiplier $\chi$ satisfies the following equation by plugging $\xi_{5}$ into equation \eqref{equ:3.2.7}

           \begin{align}
          0=&\chi(g\xi_s\sin\theta-\sigma \frac{\rho_{0}\xi_s''}{(\rho_{0}^{2}+\rho_{0}'^{2})^{\frac{3}{2}}}+\sigma\frac{\rho_{0}'\rho_{0}^{2}-2\rho_{0}'^{3}+3\rho_{0}\rho_{0}'\rho_{0}''}{(\rho_{0}^{2}+\rho_{0}'^{2})^{\frac{5}{2}}}\xi_s'+\sigma\frac{2\rho_{0}^{2}\rho_{0}''-\rho_{0}^{3}-4\rho_{0}\rho_{0}'^{2}-\rho_{0}''\rho_{0}'^{2}}{(\rho_{0}^{2}+\rho_{0}'^{2})^{\frac{5}{2}}}\xi_s)\notag\\
          &\quad+(-\sigma\frac{\rho_{0}\xi_{s}\chi''}{(\rho_{0}^{2}+\rho_{0}'^{2})^{\frac{3}{2}}}-2\sigma\frac{\rho_{0}\xi_{s}'\chi'}{(\rho_{0}^{2}+\rho_{0}'^{2})^{\frac{3}{2}}}+\sigma\frac{\rho_{0}'\rho_{0}^{2}-2\rho_{0}'^{3}+3\rho_{0}\rho_{0}'\rho_{0}''}{(\rho_{0}^{2}+\rho_{0}'^{2})^{\frac{5}{2}}}\xi_s\chi')
     \end{align}
        
     \noindent Using the fact that $\xi_{s}$ is a solution to \eqref{equ:3.2.7} with $C=0$. The terms with $\chi$ can be canceled out. Then we have 

          \begin{equation}{\label{equ:3.2.28}}
              \chi''(\theta)+Q(\theta)\chi'(\theta)=0,
          \end{equation}

         \noindent where

         \begin{equation}{\label{equ:3.2.29}}
             Q(\theta)=\frac{2\rho_{0}\xi_s'-\frac{\rho_{0}^{2}\rho_{0}'+3\rho_{0}\rho_{0}'\rho_{0}''-2\rho_{0}
             '^{2}}{\rho_{0}^{2}+\rho_{0}'^{3}}\xi_s}{\rho_{0}\xi_s}
         \end{equation}

         \noindent  By solving equation \eqref{equ:3.2.28}, we obtain the following expression for $\chi^{\prime}(\theta)$

         \begin{equation}{\label{equ:3.2.30}}
             \chi'(\theta)=Ce^{-\int_{0}^{\theta}Q(s)ds},
         \end{equation}

         \noindent which implies that

         \begin{equation}{\label{equ:3.2.31}}
             \chi(\theta)=\int_{0}^{\theta}Ce^{-\int_{0}^{s}Q(t)dt}ds+D 
         \end{equation}

         \noindent Note that $\xi_{s}$ vanishes at the point $\frac{\pi}{2}$. We divide the function $\xi_{5}$ into two parts as follows

         \begin{equation}{\label{equ:3.2.32}}
             \xi_5(\theta)=\left\{
             \begin{aligned}
             (\int_{0}^{\theta}C_{1}e^{-\int_{0}^{s}Q(t)dt}ds+D_{1} )\xi_s~~\operatorname{when}~~\theta\in (0,\frac{\pi}{2})\\
             (\int_{\pi}^{\theta}C_{2}e^{-\int_{\pi}^{s}Q(t)dt}ds+D_{2} )\xi_s~~\operatorname{when}~~\theta\in (\frac{\pi}{2},\pi)
             \end{aligned}
             \right.
         \end{equation}
         
          Suppose that $\xi_4$ is the solution function of equation \eqref{equ:3.2.7} with $C=1$. We use the same computation as for $\xi_{5}$ to show that

         \begin{equation}{\label{equ:3.2.33}}
             \xi_4(\theta)=\left\{
             \begin{aligned}
             (\int_{0}^{\theta}(C_{1}+\int_{0}^{\gamma}e^{\int_{0}^{s}Q(t)dt}ds)e^{-\int_{0}^{\gamma}Q(s)ds}d\gamma)\xi_s+D_{1}\xi_s~~\operatorname{when}~~\theta\in (0,\frac{\pi}{2})\\
             (\int_{\pi}^{\theta}(C_{2}+\int_{\pi}^{\gamma}e^{\int_{\pi}^{s}Q(t)dt}ds)e^{-\int_{\pi}^{\gamma}Q(s)ds}d\gamma)\xi_s+D_{2}\xi_s~~\operatorname{when}~~\theta\in (\frac{\pi}{2},\pi)
             \end{aligned}
             \right.
         \end{equation}
         \noindent  Moreover, we introduce the following notation(formal):
         \begin{align}
             \xi_{4}=\xi_{5}+\xi_{6}
         \end{align}
         \noindent where
         \begin{equation}{\label{equ:xi_6}}
             \xi_6(\theta)=\left\{
             \begin{aligned}
             (\int_{0}^{\theta}(\int_{0}^{\gamma}e^{\int_{0}^{s}Q(t)dt}ds)e^{-\int_{0}^{\gamma}Q(s)ds}d\gamma)\xi_s~~\operatorname{when}~~\theta\in (0,\frac{\pi}{2})\\
             (\int_{\pi}^{\theta}(\int_{\pi}^{\gamma}e^{\int_{\pi}^{s}Q(t)dt}ds)e^{-\int_{\pi}^{\gamma}Q(s)ds}d\gamma)\xi_s~~\operatorname{when}~~\theta\in (\frac{\pi}{2},\pi)
             \end{aligned}
             \right..
         \end{equation}
         
         From \eqref{equ:3.2.32} and \eqref{equ:3.2.33}, the following integral takes an important role in the expression of solution functions

         \begin{equation}{\label{equ:3.2.34}}
            e^{-\int_{0}^{\theta}Q(s)ds}
         \end{equation}

         \noindent Especially, when $\theta=\frac{\pi}{2}$, It is crucial to analyze the function

         \begin{equation}{\label{equ:3.2.35}}
             e^{-\int_{0}^{\frac{\pi}{2}}Q(s)ds}.
         \end{equation}

         \noindent  We establish the following lemma to show the properties of the function $Q(\theta)$.

         \begin{lemma}{\label{thm:Q_sym}}

             the function $Q(\theta)$ defined by equation \eqref{equ:3.2.29} has the following two properties

             \begin{align}
                Q(\frac{\pi}{2}-\gamma)=&-Q(\frac{\pi}{2}+\gamma)\label{equ:3.1.36} \\
                \int_{0}^{\frac{\pi}{2}-\theta}Q(s)ds=&\int_{\pi}^{\frac{\pi}{2}+\theta }Q(s)ds \label{equ:3.1.37}
             \end{align}

             \noindent for all $\theta\in(0,\frac{\pi}{2})$
        \end{lemma}

         \begin{proof}

             Using \eqref{equ:3.2.29}, we have:

             \begin{align}
                  Q(\theta)=&\frac{2\rho_{0}\xi_s'-\frac{\rho_{0}^{2}\rho_{0}'+3\rho_{0}\rho_{0}'\rho_{0}''-2\rho_{0}
             '^{3}}{\rho_{0}^{2}+\rho_{0}'^{2}}\xi_s}{\rho_{0}\xi_s}\notag\\
             =& 2\frac{\xi_s'}{\xi_s}-\frac{\rho_{0}^{2}\rho_{0}'+3\rho_{0}\rho_{0}'\rho_{0}''-2\rho_{0}'^{3}}{\rho_{0}(\rho_{0}^{2}+\rho_{0}'^{2})}\label{equ:3.1.38}
             \end{align}
             
            \noindent Using Lemma 5.6, we have $\frac{\xi_2'}{\xi_2}(\frac{\pi}{2}-\theta)=-\frac{\xi_2'}{\xi_2}(\frac{\pi}{2}+\theta)$. For the second term in \eqref{equ:3.1.38}, from the symmetry of $\rho_{0}$, we have the following properties

            \begin{align}
                \rho_{0}(\frac{\pi}{2}-\theta)=&\rho_{0}(\frac{\pi}{2}+\theta)\label{equ:3.1.60}\\
                \rho_{0}'(\frac{\pi}{2}-\theta)=&-\rho_{0}^{\prime}(\frac{\pi}{2}+\theta)\label{equ:3.1.61}\\
                \rho_{0}''(\frac{\pi}{2}-\theta)=&\rho_{0}(\frac{\pi}{2}+\theta)\label{equ:3.1.62}.
            \end{align}

            \noindent Applying \eqref{equ:3.1.60}-\eqref{equ:3.1.62} to the second term of \eqref{equ:3.1.38}, we obtain

            \begin{equation*}
        \frac{\rho_{0}^{2}\rho_{0}'+3\rho_{0}\rho_{0}'\rho_{0}''-2\rho_{0}'^{3}}{\rho_{0}(\rho_{0}^{2}+\rho_{0}'^{2})}(\frac{\pi}{2}+\theta)=-\frac{\rho_{0}^{2}\rho_{0}'+3\rho_{0}\rho_{0}'\rho_{0}''-2\rho_{0}'^{3}}{\rho_{0}(\rho_{0}^{2}+\rho_{0}'^{2})}(\frac{\pi}{2}-\theta),
            \end{equation*}

            \noindent which yields the following symmetry property
            \begin{equation*}
                Q(\frac{\pi}{2}-\gamma)=-Q(\frac{\pi}{2}+\gamma).
            \end{equation*}

            \noindent Finally, equation \eqref{equ:3.1.37} follows directly from equation \eqref{equ:3.1.36}.
         \end{proof}
         
         Examining the definition of function $Q(\theta)$ (equation \eqref{equ:3.2.29}), we observe that it has a critical point at $\theta=\frac{\pi}{2}$ since $\xi_s(\frac{\pi}{2})=0$. From equation \eqref{equ:3.1.37}, it suffices to figure out whether $\theta=\frac{\pi}{2}$ is a critical point for $\xi_4$ and $\xi_5$. We have the following theorem.

         \begin{theorem}{\label{H^1}}
             $\xi_6$ defined by equations  \eqref{equ:xi_6} is not a $H^{1}$ function in the interval $(0,\pi)$.
         \end{theorem}

         \begin{proof}

          From Theorem \ref{thm:Q_sym}, it suffices to discuss the value of function $\xi_{6}$ when $\theta \in (0,\frac{\pi}{2})$. 

            Using the definition of $\xi_s$, we have

         \begin{align}{\label{est:xi_2}}
         \begin{aligned}
             \xi_{s}(\theta)>&0~\operatorname{when}~\theta\in(0,\frac{\pi}{2})\\
             \xi_{s}'(\theta)<&0~\operatorname{when}~\theta\in(\frac{\pi}{2}-\delta,\frac{\pi}{2})
             \end{aligned}
         \end{align}

         \noindent for some $\delta$ arbitrarily small. By Theorem 4.5, we have the following property

         \begin{align}{\label{est:xi_2pi}}
         \begin{aligned}
             \xi_{s}(\frac{\pi}{2})=0\\
             \xi_{s}'(\frac{\pi}{2})\neq 0
             \end{aligned}
         \end{align}
         
         \noindent Therefore, combining \eqref{est:xi_2} and \eqref{est:xi_2pi}, we obtain the following estimate near the point $\frac{\pi}{2}$:

         \begin{equation}{\label{equ:Q_asy}}
             Q(\theta)= \frac{2}{(\theta-\frac{\pi}{2})}+K(\theta)~\operatorname{when}~\theta\rightarrow \frac{\pi}{2}-
         \end{equation}

         \noindent where $K(\theta)$ is a smooth function of $\theta$. This equation implies that:

         \begin{equation}
              e^{-\int_{0}^{\frac{\pi}{2}}Q(\theta)d\theta}\sim \frac{1}{(\frac{\pi}{2}-\theta)^{2}}~\operatorname{when}~\theta \rightarrow \frac{\pi}{2}-\label{equ:asym1}
         \end{equation}
         \noindent Using the similar estimate, we have:
         \begin{align}
              e^{-\int_{\pi}^{{\theta}}Q(\theta)d\theta}\sim \frac{1}{(\frac{\pi}{2}-\theta)^{2}}~\operatorname{when}~\theta \rightarrow \frac{\pi}{2}+\label{equ:asym2},
         \end{align}
         \noindent which implies that:
         \begin{align}
              e^{-\int_{\pi}^{{\theta}}Q(\theta)d\theta}\sim \frac{1}{(\frac{\pi}{2}-\theta)^{2}}~\operatorname{when}~\theta \rightarrow \frac{\pi}{2}\label{equ:asym}.
         \end{align}

         \noindent Then using the previous discussion, we only need to focus on the value of $\xi_6$ when $\theta=\frac{\pi}{2}$.

         We write down the definition of $\xi_6$
             \begin{equation}{\label{equ:3.1.67}}
             \xi_6(\theta)=\left\{
             \begin{aligned}
             (\int_{0}^{\theta}(\int_{0}^{\gamma}e^{\int_{0}^{s}Q(t)dt}ds)e^{-\int_{0}^{\gamma}Q(s)ds}d\gamma)\xi_s~~\operatorname{when}~~\theta\in (0,\frac{\pi}{2})\\
             (\int_{\pi}^{\theta}(\int_{\pi}^{\gamma}e^{\int_{\pi}^{s}Q(t)dt}ds)e^{-\int_{\pi}^{\gamma}Q(s)ds}d\gamma)\xi_s~~\operatorname{when}~~\theta\in (\frac{\pi}{2},\pi)
             \end{aligned}
             \right.
         \end{equation}

          \noindent Using Theorem 5.4, we have the following asymptotic behavior of $\xi_{s}$ near point $\frac{\pi}{2}$:

         \begin{equation*}
             \xi_s(\theta)\sim \theta-\frac{\pi}{2} ~\operatorname{when}~\theta \rightarrow \frac{\pi}{2}-,
         \end{equation*}

         \noindent which implies that

         \begin{equation}{\label{equ:3.1.68}}
            L'|\theta-\frac{\pi}{2}|\leq \xi_s(\theta)\leq L''\vert (\theta-\frac{\pi}{2})\vert.
         \end{equation}
         for some constant $L''>0$. Therefore, using \eqref{equ:asym} and \eqref{equ:3.1.68} in \eqref{equ:3.1.67}, we obtain the following inequality

         \begin{equation}{\label{equ:3.1.69}}
             \vert\xi_6(\frac{\pi}{2}-\epsilon)\vert \geq |(C_1L'''\int_{0}^{\frac{\pi}{2}-\epsilon}\frac{1}{(\frac{\pi}{2}-t)^{2}}dt(\frac{\pi}{2}-\epsilon))|\geq M
         \end{equation}

         \noindent for some positive constant $M$ independent of $\epsilon$.  

         Also, form the definition \eqref{equ:3.1.67}, Theorem \ref{thm:Q_sym} and Theorem \ref{thm:xi_2sym}, we derive the following symmetry relation for $\xi_{6}$

         \begin{align}
             \xi_{6}(\frac{\pi}{2}-\theta)=-\xi_{6}(\frac{\pi}{2}+\theta)
         \end{align}
         \noindent for any $\theta\in (0,\frac{\pi}{2})$. Combining this with \eqref{equ:3.1.69}, we have

         \begin{align}
              \lim_{\theta\rightarrow \frac{\pi}{2}-}\xi_{6}(\theta)=-\lim_{\theta \rightarrow\frac{\pi}{2}+}\xi_{6}(\theta)\neq 0,
         \end{align}

         \noindent which implies that $\xi_{6}$ is not a continuous function. Therefore, $\xi_{6}$ is not a $H^{1}$ function either.
         \end{proof}

         Furthermore, we establish the following theorem for $\xi_{5}$:

         \begin{theorem}{\label{thm:H^1_1}}
             When $C_{1}=C_{2}$, $\xi_{5}$ defined by \eqref{equ:3.2.32} is a $H^{1}$ function in the interval $(0,\pi)$.
         \end{theorem}

         \begin{proof}
             We write down the definition of $\xi_{5}$ from equation \eqref{equ:3.2.32}:

              \begin{equation}
             \xi_5(\theta)=\left\{
             \begin{aligned}
             (\int_{0}^{\theta}C_{1}e^{-\int_{0}^{s}Q(t)dt}ds+D_{1} )\xi_s~~\operatorname{when}~~\theta\in (0,\frac{\pi}{2})\\
             (\int_{\pi}^{\theta}C_{2}e^{-\int_{\pi}^{s}Q(t)dt}ds+D_{2} )\xi_s~~\operatorname{when}~~\theta\in (\frac{\pi}{2},\pi)
             \end{aligned}
             \right.
         \end{equation}
         From the proof in \ref{H^1}, we have:

         \begin{align}{\label{equ:xi_5}}
              (\int_{0}^{\theta}C_{1}e^{-\int_{0}^{s}Q(t)dt}ds)=C_{3}\frac{1}{\frac{\pi}{2}-\theta}+K_{1}(\theta)
         \end{align}
         \noindent for some smooth function $K_{1}(\theta)$. Moreover, we have the symmetry property of $\xi_{5}$
         \begin{align}{\label{equ:xi_5_sym}}
             \xi_{5}(\frac{\pi}{2}-\gamma)=\xi_{5}(\frac{\pi}{2}+\gamma)
         \end{align}
         \noindent for any $\gamma\in (0,\frac{\pi}{2})$. Combining equation \eqref{equ:xi_5} and \eqref{equ:3.1.68}, we have:
         \begin{align}
             |\xi_{5}|\lesssim C_{3}L''
         \end{align}
         \noindent which implies that $\xi_5$ is a bounded function. Then combining this bounded result with \eqref{equ:xi_5_sym}, it follows that $\xi_{5}$ is a continuous function. 

         Taking derivative on function $\xi_5$ with respect to $\theta$, we have
         \begin{align}
             \xi_5'=C_{1}e^{-\int_{0}^{\theta}Q(t)dt}\xi_{s}+(\int_{0}^{\theta}C_{1}e^{-\int_{0}^{s}Q(t)dt}+D_{1})\xi_{s}'
         \end{align}
         \noindent Using equation \eqref{equ:Q_asy}, and $\xi_{s}=\xi_{s}'(\frac{\pi}{2})(\theta-\frac{\pi}{2})+O((\theta-\frac{\pi}{2})^{2})$, we have
         \begin{align}
         \begin{aligned}
             \xi_5'=&\tilde{C_{1}}\frac{1}{(\theta-\frac{\pi}{2})^{2}}\xi_{s}'(\frac{\pi}{2})(\theta-\frac{\pi}{2})+(\int_{0}^{\theta}\tilde{C}_{1}\frac{ds}{(s-\frac{\pi}{2})^{2}})\xi_{s}'+K_{1}(\theta)\\
             =& \tilde{C_{1}}\frac{1}{(\theta-\frac{\pi}{2})}\xi_{s}'(\frac{\pi}{2})-\tilde{C}_{1}\frac{1}{\theta-\frac{\pi}{2}}\xi_{s}'(\frac{\pi}{2})-\tilde{C}_{1}\frac{2}{\pi}\xi_{s}'(\frac{\pi}{2})+K_{1}(\theta)\\
             =&-\tilde{C}_{1}\frac{2}{\pi}\xi_{s}'(\frac{\pi}{2})+K_{1}(\theta)
             \end{aligned}
         \end{align}
        \noindent for some bounded function $K_{1}(\theta)$. This implies that $\xi_{5}'\in L^{\infty}$. Combine this with the fact that $\xi_{5}$ is continuous and smooth for any $\theta\neq \frac{\pi}{2}$, we obtain the fact that $\xi_{5}\in H^{1}$ 
         \end{proof}
         
         Now, we have well-defined functions $\xi_4$ and $\xi_{5}$. We proceed to examine whether they satisfy the prescribed boundary conditions. We first examine the function $\xi_5$, for which we have the following theorem.

        \begin{theorem}{\label{thm:xi_5}}
            If $\xi_{5}$ is {a function in} the kernel of $F_{0}$, then it does not satisfy the conservation law:
            \begin{align}
                \int_{0}^{\pi}\rho_{0}\xi_{5} d\theta=0
            \end{align}
        \end{theorem}
        \begin{proof}
           From Lemma \ref{lem:bdd}, it follows that $\xi_{5}$ satisfies given boundary conditions if $\xi_{5}$ is the kernel of $F_{0}$. Suppose that $\xi_{5}$ satisfy \eqref{equ:3.2.38} and \eqref{equ:3.2.39}. Then $D_{1}=D_{2}$ from \eqref{equ:3.2.38} and $\xi_{s}(0)=-\xi_{s}(\pi)$. Moreover, equation \eqref{equ:3.2.39} implies that
           \begin{align}
               C_{1}\xi_{s}(0)\rho_{0}(0)=-C_{2}\rho_{0}(\pi)\xi_{s}(\pi)
           \end{align}
           \noindent This shows that
           \begin{align}
               C_{1}=C_{2},
           \end{align}
           since $\rho_{0}(0)=\rho_{0}(\pi)$ and $\xi_{s}(0)=-\xi_{s}(\pi)$. Using Theorem \ref{thm:H^1_1} , $\xi_{5}$ is a $H^{1}$ function. Then we have the following computation using the symmetry properties

           \begin{align}
               \int_{0}^{\pi}\xi_{5}\rho_{0}=\int_{0}^{\pi}D_{1}\xi_{s}\rho_{0} d\gamma+2\int_{0}^{\frac{\pi}{2}}C_{1}e^{-\int_{0}^{\gamma}Q(s)ds}\rho_{0}d\gamma 
           \end{align}
           \noindent The first term on the right hand side vanishes since $\xi_{s}$ is perpendicular to $\rho_{0}$. However, the second term vanishes if and only if $C_{1}=0$. Therefore, $\xi_{5}=D_{1}\xi_{s}$ which contradicts to our assumption.

           Suppose $\xi_{5}$ satisfies \eqref{equ:3.2.40}. We have the following computation
           \begin{align}
               \frac{\xi_{5}'(0)}{\xi_{5}(0)}=\frac{C_{1}\xi_{s}(0)+D_{1}\xi_{s}'(0)}{D_{1}\xi_{s}(0)}=\frac{\rho_{0}'(0)}{\rho_{0}(0)}
           \end{align}
           Using the fact that $\frac{\xi_{s}'(0)}{\xi_{s}(0)}=\frac{\rho_{0}'(0)}{\rho_{0}(0)}$ from Theorem \ref{thm:bdd_s}, we have:
           \begin{align}
               \frac{C_{1}\xi_{s}(0)}{D_{1}\xi_{s}(0)}=\frac{\rho_{0}'(0)}{\rho_{0}(0)}-\frac{\xi_{s}'(0)}{\xi_{s}(0)}=0
           \end{align}
           \noindent which implies that $C_{1}=0$. Hence, we have $\xi_{5}(0)=D_{1}\xi_{s}$, which contradicts to our assumption that $\xi_{5}$ is linearly independent of $\xi_{s}$.

           In conclusion, $\xi_{5}$ does not satisfy the conservation law if it serves as another kernel of $F_{0}$.
        \end{proof}
        
        \begin{theorem}{\label{thm:xi_4}}
            If $\xi_{4}$ {is a function} in the kernel of $F_{0}$, and it  satisfies the conservation law
            \begin{align}
                \int_{0}^{\pi}\rho_{0}\xi_{4} d\theta=0,
            \end{align}
            \noindent then it is not a $H^{1}$ function.
        \end{theorem}

        \begin{proof}
        Noticing that:
        \begin{align}
            \xi_{4}=\xi_{5}+\xi_{6},
        \end{align}
        \noindent where
        \begin{equation}
             \xi_6(\theta)=\left\{
             \begin{aligned}
             (\int_{0}^{\theta}(\int_{0}^{\gamma}e^{\int_{0}^{s}Q(t)dt}ds)e^{-\int_{0}^{\gamma}Q(s)ds}d\gamma)\xi_s~~\operatorname{when}~~\theta\in (0,\frac{\pi}{2})\\
             (\int_{\pi}^{\theta}(\int_{\pi}^{\gamma}e^{\int_{\pi}^{s}Q(t)dt}ds)e^{-\int_{\pi}^{\gamma}Q(s)ds}d\gamma)\xi_s~~\operatorname{when}~~\theta\in (\frac{\pi}{2},\pi)
             \end{aligned}
             \right.
         \end{equation}
         From a simple observation, we know that $\xi_{6}(0)=\xi_{6}'(0)=\xi_{6}(\pi)=\xi_{6}'(\pi)=0$ and $\xi_{6}(\frac{\pi}{2}+\theta)=-\xi_{6}(\frac{\pi}{2}-\theta)$ for any $\theta\in (0,
         \frac{\pi}{2})$. Therefore,

         \begin{align}
             \int_{0}^{\pi}\xi_{6}\rho_{0}d\theta=0
         \end{align}
         
            If $\xi_{4}$ satisfies the boundary conditions \eqref{equ:3.2.38} and \eqref{equ:3.2.39} or satisfies \eqref{equ:3.2.40}, we use the same computation as in Theorem \ref{thm:xi_5} to show that $C_{1}=C_{2}$ and $D_{1}=D_{2}$. Then using the same discussion as in Theorem \ref{thm:xi_5}, we have:
            \begin{align}
                \int_{0}^{\pi}\xi_{4}\rho_{0}d\theta=\int_{0}^{\pi}\xi_{5}\rho_{0}d\theta+\int_{0}^{\pi}\xi_{6}\rho_{0}d\theta=\int_{0}^{\pi}\xi_{5}\rho_{0}d\theta
            \end{align}
            \noindent which vanishes if and only if $C_1=0$. Then $\xi_{4}=\xi_{6}+D_{1}\xi_{s}$. Using Theorem \ref{H^1}, $\xi_{4}$ is not a $H^{1}$ function. This finishes the proof.
        \end{proof}

     \begin{theorem}{\label{thm:uni_ker}}
         If $\xi\in D:=\{f\in H^{1}|\int_{0}^{\pi}f\rho_{0}=0\}$ such that $F_{0}(\xi)=0$ then $\xi=C\xi_s$ for some constant C,
     \end{theorem}

     \begin{proof}
         Using Theorem \ref{thm:kernel}, we know that $\xi\in H^{2}$ if $F_{0}(\xi)=0$. Therefore, it satisfies equation \eqref{equ:3.2.7}. Using Theorem \ref{thm:xi_5} and \ref{thm:xi_4}, it holds that $\xi=C\xi_{s}$ if $\xi$ is subject to the conservation law and boundary condition at the same time. Therefore, we have the theorem proved.
     \end{proof}
     
Now, we have already established the fact that $\xi_{s}$ serves as the unique basis of the kernel of the second variation of the energy functional. We then want to show the positivity of the inner product (1,$\Sigma$) in some functional space here. We begin with defining three functions.  Assume that $\xi$ is the perturbation function we defined in Section 5.1. We have the following relation by using conservation of total mass.

\begin{equation*}
    \int_{0}^{\pi} \xi\rho=0
\end{equation*}
\noindent which implies that
\begin{equation*}
    \int_{0}^{\pi} \xi\rho_0=-\frac{1}{2}\int_{0}^{\pi} \xi^{2}d\theta
\end{equation*}
\noindent Taking time derivative on both sides of the equation, we obtain that

\begin{equation}
   \int_{0}^{\pi}\rho\partial_{t}\xi=0.
\end{equation}

\noindent After taking another temporal derivative, we have

\begin{equation}
    \int_{0}^{\pi}\partial_{tt}\xi\rho_{0}=-\int_{0}^{\pi}(\partial_{t}\xi)^{2}-\int_{0}^{\pi}\partial_{t}^{2}\xi \xi.
\end{equation}

\noindent From the equations above, we define three functions with respect to time $t$ which are expressed as follows

\begin{align}{\label{equ:a}}
\begin{aligned}
    a_{0}(t)=&\frac{-\frac{1}{2}\int_{0}^{\pi}\xi^{2}d\theta}{\int_{0}^{t}\rho_{0}^{2}d\theta}\\
     a_{1}(t)=&\frac{-\int_{0}^{\pi}\xi \partial_{t}\xi d\theta}{\int_{0}^{t}\rho_{0}^{2}d\theta}\\
       a_{2}(t)=&\frac{-\int_{0}^{\pi} (\partial_{t}\xi)^{2} d\theta-\int_{0}^{\pi}\partial_{t}^{2}\xi \xi}{\int_{0}^{t}\rho_{0}^{2}d\theta}.
\end{aligned}
\end{align}

\noindent These functions $a_{i}(t)$ are nicely chosen such that

\begin{align}
    \int_{0}^{\pi}\rho_{0}(\partial_{t}^{i}\xi-a_{i}(t)\rho_{0})d\theta=0{\label{equ:vanish}}
\end{align}

In addition, it holds that $\partial_{t}^{i}\xi$ is perpendicular to $\xi_s$ which is the shift function by equation \eqref{equ:4.1.8}. This is expressed as follows

\begin{equation}{\label{equ:3.1.26}}
    \int_{0}^{\pi}\partial_{t}^{i}\xi\xi_sd\theta=0
\end{equation}

\noindent Using the definition of $a_{i}$, equation\eqref{equ:3.1.26}, and equation \eqref{equ:3.2.8}. We have

\begin{align}
    \int_{0}^{\pi}(\partial_{t}^{i}\xi-a_{i}(t)\rho_{0})\rho_{0}d\theta=0\label{equ:3.1.27}
\end{align}
\noindent Therefore, equations \eqref{equ:vanish} and \eqref{equ:3.1.27} imply that $\p_{t}^{i}\xi-a_{i}(t)\rho_{0}\in \tilde{B}$ defined as follows
     \begin{align}
        \tilde{B}:=\{\rho\in H^{1}|\int_{0}^{\pi}\rho\rho_{0}d\theta=0,\int_{0}^{\pi}\rho\xi_2d\theta=0\}
    \end{align}

Having proven the property that $\xi_{s}$ is the unique kernel of the $1,\Sigma$ inner product from Theorem \ref{thm:uni_ker}, we then show the following lemma to establish the positivity of the inner product. 

\begin{theorem}{\label{thm:pos}}
    Suppose that $\xi_{s}$ is the horizontal shift function and $\xi$ is any perturbation function defined above. Moreover, suppose that $\mathcal{E}(u,p,\xi)(t)\leq \delta\ll1$ when $t\in(0,T)$ for some $T>0$. $a_{i}(t)$ are functions defined above. We have the following vanishing and positive estimate for $a_{i}(t)$ and $\xi_{s}$

    \begin{equation}
        (\partial_{t}^{i}\xi-a_{i}(t)\rho_{0},\xi_s)_{1,\Sigma}=0,
    \end{equation}

    \noindent and 

    \begin{equation}{\label{equ:p_0}}
        (\partial_{t}^{i}\xi-a_{i}(t)\rho_{0},\partial_{t}^{i}\xi-a_{i}(t)\rho_{0})_{1,\Sigma}\geq \delta \vert \vert \partial_{t}^{i}\xi-a_{i}(t)\rho_{0}\vert \vert^{2}_{H^{1}}~~\operatorname{for~i=\{0,1,2\}}.
    \end{equation}

    \noindent where $\delta$ is a constant independent of time $t$ and the function $\xi$. Moreover, we have the following estimates

    \begin{align}{\label{equ:p_1}}
        (\xi-a_{0}(t)\rho_{0},\xi-a_{0}(t)\rho_{0})_{1,\Sigma}\geq \delta \vert \vert \xi\vert \vert^{2}_{H^{1}}-\|\xi\|^{4}_{L^{2}},
    \end{align}

    \begin{equation}{\label{equ:p_2}}
          (\partial_{t}\xi-a_{1}(t)\rho_{0},\partial_{t}\xi-a_{1}(t)\rho_{0})_{1,\Sigma}\geq \delta_1 \vert \vert \partial_{t}\xi\vert \vert^{2}_{H^{1}}-\delta_{2}\|\xi\|^{4}_{L^{2}},
    \end{equation}

    \noindent and

    \begin{equation}{\label{equ:p_3}}
          (\partial_{t}^{2}\xi-a_{2}(t)\rho_{0},\partial_{t}^{2}\xi-a_{2}(t)\rho_{0})_{1,\Sigma}\geq \delta_1\vert \vert \partial_{t}^{2}\xi\vert \vert^{2}_{H^{1}}-\delta_2\vert \vert \partial_{t}\xi\vert \vert^{4}_{L^{2}}-\delta_3\vert \vert \partial_{t}^{2}\xi\vert \vert^{2}_{L^{2}}\|\xi\|_{L^{2}}^{2}.
        \end{equation}
\end{theorem}

\begin{proof}
    For the first relation, it follows from the fact that $\xi_s$ is the kernel of second variation and that $\p_{t}^{i}\xi-a_{i}(t)\rho_{0}\in B$. Using Theorem 5.2, we have this formula proved. Then we prove the second relation. Actually, we want to show that for any function $w(\theta)$ such that

    \begin{equation}{\label{equ:v_0}}
        \int_{0}^{\pi}w(\theta)\rho_{0}d\theta=0,
    \end{equation}

    \noindent and that

    \begin{equation}{\label{equ:v_1}}
        \int_{0}^{\pi}w(\theta)\xi_sd\theta=0,
    \end{equation}

    \noindent we have:

    \begin{equation*}
        (w,w)_{1,\Sigma}\geq \delta \vert \vert w\vert \vert_{H^{1}}
    \end{equation*}

    We proceed by a contradiction argument. Suppose that there is a sequence $w_{n}$  such that

    \begin{equation*}
        (w_{n},w_{n})_{1,\Sigma}\leq \frac{1}{n},
    \end{equation*}

    \noindent and

    \begin{equation*}
        \vert \vert w_{n}\vert \vert_{L^{2}}=1.
    \end{equation*}

    \noindent From the definition of the $(1,\Sigma)$ inner product, we obtain that the $H^{1}$ norm of $w_{n}$ is uniformly bounded. Then $w_{n}$ converges weakly to $w_0$ in $H^{1}$ and strongly in $L^{2}$, which implies that:

    \begin{equation*}
        (w_{0},w_{0})_{1,\Sigma}\leq \liminf_{n\rightarrow +\infty}(w_{n},w_{n})_{1,\Sigma}=0
    \end{equation*}

    \noindent Therefore, $w_{n}=c\xi_s$ for some positive constant c since $\xi_{s}$ serves as the unique basis of kernel of $(1,\Sigma)$ inner product. However, we also have the following relation from equation \eqref{equ:v_1}:

    \begin{equation}
        c\int_{0}^{\pi}\xi_s^{2}d\theta=\int_{0}^{\pi}\xi_sw_0 d\theta=0
    \end{equation}

    \noindent which implies that $c=0$. This contradicts to the fact that $\vert \vert w_0\vert \vert_{L^{2}}=1$. Therefore, we have

    \begin{equation}
        (w,w)_{1,\Sigma}\geq m\vert \vert w\vert \vert_{L^{2}}
    \end{equation}

    \noindent for some constant $m>0$. The definition of inner product yileds the following result

    \begin{equation}
        (w,w)_{1,\Sigma}\geq \delta_1\vert \vert w\vert \vert_{H^{1}}-\delta_2\vert \vert w\vert \vert_{L^{2}}
    \end{equation}

    \noindent for some positive constants $\delta_{1}$ and $\delta_{2}$. Then combining the two equations above, we obtain the following positivity result

    \begin{equation*}
         (w,w)_{1,\Sigma}\geq \delta\vert \vert w\vert \vert_{H^{1}}
    \end{equation*}

    \noindent for some $\delta$. Since $\p_{t}^{i}\xi-a_{i}(t)\rho_{0}$ satisfies equation \eqref{equ:v_0} and \eqref{equ:v_1}, we have the equation \eqref{equ:p_1} proved.

    Finally, we prove inequalities \eqref{equ:p_1} to \eqref{equ:p_3}. In general, we show the following inequality

    \begin{align}
        (\partial_{t}^{i}\xi-a_{i}(t)\rho_{0},\partial_{t}^{i}\xi-a_{i}(t)\rho_{0})_{1,\Sigma}\geq \delta \vert \vert \partial_{t}^{i}\xi-a(t)\rho_{0}\vert \vert^{2}_{H^{1}}\geq \delta \vert \vert \partial_{t}^{i}\xi\vert\vert^{2}_{H^{1}}-\delta a^{2}_{i}(t)\vert \vert \rho_{0}\vert \vert^{2}_{H^{1}} \label{equ:3.1.30}
    \end{align}

    \noindent for $i=0,1,2$. For simplicity, we only show the case where $i=2$. The proof of cases $i=0,1$ follow the same manner. We have the following estimate for $a_{2}(t)$

    \begin{align}
        a_{2}(t)=\frac{-\int_{0}^{\pi} (\partial_{t}\xi)^{2} d\theta-\int_{0}^{\pi}\partial_{t}^{2}\xi \xi}{\int_{0}^{t}\rho_{0}^{2}d\theta}\lesssim  \vert \vert \partial_{t}\xi\vert \vert^{2}_{L^{2}}+\vert \vert \partial^{2}_{t}\xi\vert \vert_{L^{2}}\vert \vert \xi\vert \vert_{L^{2}} \label{equ:3.1.31}
    \end{align}

    \noindent Then plugging \eqref{equ:3.1.31} into the equation \eqref{equ:3.1.30}, we obtain

    \begin{align}
         (\partial_{t}^{2}\xi-a_{2}(t)\rho_{0},\partial_{t}^{2}\xi-a_{2}(t)\rho_{0})_{1,\sigma}\geq \delta \vert \vert \partial_{t}^{2}\xi-a_{2}(t)\rho_{0}\vert \vert_{H^{1}}\geq \delta \vert \vert \partial_{t}^{2}\xi\vert\vert_{H^{1}}-\delta a^{2}_{2}(t)\vert \vert \rho_{0}\vert \vert_{H^{1}}\notag\\
         \geq \delta_1\vert \vert \partial_{t}^{2}\xi\vert \vert^{2}_{H^{1}}-\delta_2\vert \vert \partial_{t}\xi\vert \vert^{4}_{L^{2}}-\delta_3\vert \vert \partial_{t}^{2}\xi\vert \vert^{2}_{L^{2}}\|\xi\|_{L^{2}}^{2} \label{equ:3.1.32}
    \end{align}

    \noindent for some positive constant $\delta_1$, $\delta_2$ and $\delta_{3}$. Using the fact that $\mathcal{E}(u,\xi,p)\leq \delta\ll1$, inequality \eqref{equ:p_3} is proved.
\end{proof}

Now we have the positivity and a weak form of the equation system. We then want to show an estimate for the solution function of the equation system \eqref{equ:4.1.23}. Before this, we show estimates for the position of the pole and its temporal derivatives $\mathfrak{n}^{(i)}(t)$ first. We have the following theorem.

\begin{theorem}{\label{thm:gam}}

\noindent Suppose that $\mathcal{E}(u,p,\xi)(t)\leq \delta\ll1$ when $t\in(0,T)$ for some $T>0$. Then the following properties for $\mathfrak{n}^{i}(t)$ hold for $i=1,2,3$ and $t\in(0,T)$.
    \begin{equation}{\label{equ:3.1.33}}
        \vert \mathfrak{n}'(t)\vert\lesssim \vert \vert u\vert \vert_{H^{1}}
    \end{equation}

    \begin{equation}{\label{equ:3.1.34}}
         \vert \mathfrak{n}''(t)\vert\lesssim \vert \vert \p_{t}u\vert \vert_{H^{1}}+\mathcal{E}_{-}
    \end{equation}

    \begin{equation}{\label{equ:3.1.35}}
         \vert \mathfrak{n}'''(t)\vert\lesssim  \vert \vert \p_{t}^{2}u\vert \vert_{H^{1}}+\mathcal{E}_{-}
    \end{equation}
    \noindent where:
    \begin{align}
        \mathcal{E}_{-}=&\vert \vert u\vert \vert_{W^{2,q_{+}}}^{2}+\vert \vert \partial_{t}u\vert \vert_{H^{1+\frac{\epsilon_{-}}{2}}}+\vert \vert \partial_{t}^{2}u\vert \vert^{2}_{H^{0}}+\vert \vert p\vert \vert^{2}_{W^{1,q_{+}}}+\vert \vert \partial_{t}p\vert \vert_{L^{2}}^{2}\notag\\
       &+\sum_{k=0}^{2}(\vert \vert \partial_{t}^{k}u\vert \vert_{L^{2}}+\vert \vert \partial_{t}^{k}\xi\vert \vert_{H^{1}}^{2})+\vert \vert \xi\vert \vert_{W^{3-\frac{1}{q_{+}}}}+\vert \vert \partial_{t}\xi\vert \vert_{H^{\frac{3}{2}+\frac{\epsilon_{-}-\alpha}{2}}}^{2}+\vert \vert \partial_{t}^{2}\xi\vert \vert^{2}_{H^{1}}
    \end{align}
\end{theorem}

\begin{proof}

    \textbf{Step 1} We first show equation \eqref{equ:3.1.33}, From the definition of $\mathfrak{n}'(t)$, it suffices to show that

    \begin{align}
        \int_{0}^{\pi}\xi_{s}\xi_{3}d\theta\geq c
    \end{align}

    \noindent for some positive constant $c$. By definition of $\xi_{3}$, we have:

    \begin{align}
        \xi_{3}=\xi_{s}+(\frac{\rho'}{\rho}-\frac{\rho_{0}'}{\rho_{0}})\sin\theta
    \end{align}

    \noindent Therefore,

    \begin{align}
         \int_{0}^{\pi}\xi_{s}\xi_{3}d\theta\geq \int_{0}^{\pi}\xi_{s}^{2}d\theta-\int_{0}^{\pi}\xi_{s}(\frac{\rho'}{\rho}-\frac{\rho_{0}'}{\rho_{0}})\sin\theta\gtrsim \int_{0}^{\pi}\xi_{s}^{2}-\vert \vert \xi\vert \vert_{H^{1}}\vert \vert \xi_{s}\vert \vert_{L^{2}}\geq c,
    \end{align}

    \noindent where we used the smallness of $\vert \vert \xi(t)\vert \vert_{H^{1}}\leq \mathcal{E}(u,p,\xi)(t)$ and Cauchy inequality.

    \textbf{Step 2} We now show the estimate for the second derivative of $\mathfrak{n}(t)$. By definition,

    \begin{align}
        \vert \mathfrak{n}''(t)\vert=\vert \partial_{t}(\mathfrak{n}'(t))\vert\lesssim& \frac{\vert \partial_{t}\lambda\vert}{\lambda^{2}}\int_{0}^{\pi}\frac{1}{\rho}\vert u\vert\vert \mathcal{N}\vert\vert \xi_{s}\vert d\theta+\frac{1}{\vert \lambda\vert}\int_{0}^{\pi}\frac{\vert \partial_{t}\xi\vert}{\rho^{2}}\vert u\vert\vert \mathcal{N}\vert \vert \xi_{s}\vert d\theta\notag\\
        &\quad+\frac{1}{\vert \lambda\vert}\int_{0}^{\pi}\frac{1}{\rho}\vert \p_{t}u\vert\vert \mathcal{N}\vert \vert \xi_{s}\vert d\theta+\frac{1}{\vert \lambda\vert}\int_{0}^{\pi}\frac{1}{\rho}\vert \p_{t}u\vert\vert \p_{t}\mathcal{N}\vert \vert \xi_{s}\vert d\theta\notag\\
        =&I_{1}+I_{2}+I_{3}+I_{4}.
    \end{align}
    We then estimate $I_{1}$ to $I_{4}$ individually.
    
    \textbf{Term $I_{1}$} Using the definition of $\lambda$ and trace theorem, we have

    \begin{align}
        I_{1}\lesssim \frac{\int_{0}^{\pi}\vert \xi_{s}\vert(\vert \partial_{t}\xi'\vert+\vert \partial_{t}\xi\vert)}{\lambda^{2}}\vert \vert u\vert \vert_{L^{2}(\Sigma)}\lesssim \vert \vert \partial_{t}\xi\vert \vert_{H^{1}}\vert \vert u\vert \vert_{H^{1}}\lesssim \mathcal{E}_{-}
    \end{align}

    \textbf{Term $I_{2}$} Using the definition of $I_{2}$, H\"older's inequality and trace theorem, we have

    \begin{align}
        I_{2}=\frac{1}{\vert \lambda\vert}\int_{0}^{\pi}\frac{\vert \partial_{t}\xi\vert}{\rho^{2}}\vert u\vert\vert \mathcal{N}\vert \vert \xi_{s}\vert d\theta\lesssim \|u\|_{L^{2}(\Sigma)}\|\p_{t}\xi\|_{L^{2}}\lesssim \vert \vert u\vert \vert_{H^{1}}\vert \vert \partial_{t}\xi\vert \vert_{L^{2}}\lesssim \mathcal{E}_{-}
    \end{align}

    \textbf{Term $I_{3}$} Using trace theorem, we have

    \begin{align}
        I_{3}=\frac{1}{\vert \lambda\vert}\int_{0}^{\pi}\frac{1}{\rho}\vert \p_{t}u\vert\vert \mathcal{N}\vert \vert \xi_{s}\vert d\theta\lesssim\|\p_{t}u\|_{L^{2}(\Sigma)}\lesssim \vert \vert \p_{t}u\vert \vert_{H^{1}}
    \end{align}

    \textbf{Term $I_{4}$} Using trace theorem, H\"older's inequality and the definition of $I_{4}$, we have:

    \begin{align}
        I_{4}=\frac{1}{\vert \lambda\vert}\int_{0}^{\pi}\frac{1}{\rho}\vert \p_{t}u\vert\vert \p_{t}\mathcal{N}\vert \vert \xi_{s}\vert\lesssim \|\p_{t}u\|_{H^{1}}\|\p_{t}\xi\|_{H^{1}}
    \end{align}
    
    \noindent Therefore, combining all of the estimates above, we obtain

    \begin{align}
       \vert \mathfrak{n}''(t)\vert\lesssim \vert \vert \partial_{t}u\vert \vert_{H^{1}}+\mathcal{E}_{-}
    \end{align}

    \textbf{Step 3} In this step, we Show the estimate for the third derivative of $\mathfrak{N}$. We have the following decomposition for $\mathfrak{n}^{(3)}(t)$ by definition:

    \begin{align}
    \vert \mathfrak{n}'''(t)\vert\lesssim &2\frac{(\partial_{t}\lambda)^{2}}{\vert \lambda\vert^{3}}\int_{0}^{\pi}\frac{1}{\rho}\vert u\vert\vert \mathcal{N}\vert+\frac{\vert \partial_{t}^{2}\lambda\vert}{\vert \lambda\vert^{2}}\int_{0}^{\pi}\frac{1}{\rho}\vert u\vert\vert \mathcal{N}\vert d \theta+2\frac{\vert \partial_{t}\lambda\vert}{\lambda^{2}}\int_{0}^{\pi}\frac{\vert\partial_{t}\xi\vert}{\rho^{2}}\vert u\vert\vert \mathcal{N}\vert d\theta\notag\\
    &+2\frac{\vert \partial_{t}\lambda\vert}{\lambda^{2}}\int_{0}^{\pi}\frac{1}{\rho}\vert \p_{t}u\vert\vert \mathcal{N}\vert d\theta+\frac{\vert \partial_{t}\lambda\vert}{\lambda^{2}}\int_{0}^{\pi}\frac{1}{\rho}\vert u\vert\vert \p_t\mathcal{N}\vert d\theta+\frac{1}{\vert \lambda\vert}\int_{0}^{\pi}\frac{\vert\partial_{t}\xi\vert^{2}}{\rho^{3}}\vert u\vert\vert \mathcal{N}\vert d\theta\notag\\
    &+\frac{1}{\vert\lambda\vert}\int_{0}^{\pi}\frac{\vert\partial_{t}^{2}\xi\vert}{\rho^{2}}\vert u\vert\vert \mathcal{N}\vert d\theta
    +\frac{1}{\vert\lambda\vert}\int_{0}^{\pi}\frac{\vert\partial_{t}\xi\vert}{\rho^{2}}\vert \p_{t}u\vert\vert \mathcal{N}\vert d\theta+\frac{1}{\vert\lambda\vert}\int_{0}^{\pi}\frac{1}{\rho}\vert \p_{t}^{2}u\vert\vert \mathcal{N}\vert d\theta\\
    &+\frac{1}{\vert\lambda\vert}\int_{0}^{\pi}\frac{|\p_{t}\xi|}{|\rho|^{2}}\vert u\vert\vert \p_{t}\mathcal{N}\vert d\theta+\frac{1}{\vert\lambda\vert}\int_{0}^{\pi}\frac{1}{\rho}\vert \p_{t}u\vert\vert \p_{t}\mathcal{N}\vert d\theta+\frac{1}{\vert\lambda\vert}\int_{0}^{\pi}\frac{1}{\rho}\vert u\vert\vert \p_{t}^{2}\mathcal{N}\vert d\theta=\Sigma_{i=1}^{12} I_{i}
    \end{align}
We then estimate $I_{1}$ to $I_{12}$ individually.

    \textbf{Term $I_{1}$}  By the definition of $I_{1}$ and $\lambda$ (Equation \eqref{equ:lambda}), and trace theorem, we have

    \begin{align}
        I_{1}=2\frac{(\partial_{t}\lambda)^{2}}{\vert \lambda\vert^{3}}\int_{0}^{\pi}\frac{1}{\rho}\vert u\vert\vert \mathcal{N}\vert\lesssim \vert \vert \partial_{t}\xi\vert \vert^{2}_{H^{1}}\vert \vert u\vert \vert_{H^{1}}\lesssim \mathcal{E_{-}}^{\frac{3}{2}}
    \end{align}

    \textbf{Term $I_{2}$} By the definition of $I_{2}$ and $\lambda$ (Equation \eqref{equ:lambda}), and trace theorem, we have

    \begin{align}
        I_{2}=\frac{\vert \partial_{t}^{2}\lambda\vert}{\vert \lambda\vert^{2}}\int_{0}^{\pi}\frac{1}{\rho}\vert u\vert\vert \mathcal{N}\vert d \theta\lesssim \vert \vert \partial_{t}^{2}\xi\vert \vert_{H^{1}}\vert \vert u\vert \vert_{L^{2}}+\vert \vert \partial_{t}\xi\vert \vert^{2}_{H^{1}}\vert \vert u\vert \vert_{L^{2}}\lesssim (\mathcal{E}_{-})
    \end{align}

    \textbf{Term $I_{3}$} By the definition of $I_{3}$ and $\lambda$ (Equation \eqref{equ:lambda}), and trace theorem, we have

    \begin{align}
        I_{3}=2\frac{\vert \partial_{t}\lambda\vert}{\lambda^{2}}\int_{0}^{\pi}\frac{\vert\partial_{t}\xi\vert}{\rho^{2}}\vert u\vert\vert \mathcal{N}\vert d\theta\lesssim \vert \vert \partial_{t}\xi\vert \vert_{H^{1}}\vert \vert \partial_{t}\xi\vert \vert_{L^{2}}\vert \vert u\vert \vert_{H^{1}}\lesssim (\mathcal{E}_{-})^{\frac{3}{2}}
    \end{align}

    \textbf{Term $I_{4}$} By the definition of $I_{4}$ and $\lambda$ (Equation \eqref{equ:lambda}), and trace theorem, we have

    \begin{align}
        I_{4}=2\frac{\vert \partial_{t}\lambda\vert}{\lambda^{2}}\int_{0}^{\pi}\frac{1}{\rho}\vert \p_{t}u\vert\vert \mathcal{N}\vert d\theta\lesssim \vert \vert \partial_{t}\xi\vert\vert_{H^{1}}\vert \vert \p_{t}u\vert \vert_{H^{1}}\lesssim \mathcal{E}_{-}.
    \end{align}
    
    \textbf{Term $I_{5}$} By the definition of $I_{5}$ and $\lambda$ (Equation \eqref{equ:lambda}), and trace theorem, we have

    \begin{align}
        I_{5}=\frac{\vert \partial_{t}\lambda\vert}{\lambda^{2}}\int_{0}^{\pi}\frac{1}{\rho}\vert u\vert\vert \p_t\mathcal{N}\vert d\theta\lesssim \int_{0}^{\pi}|\p_{t}\xi'|\int_{0}^{\pi}|u||\p_{t}\xi'|d\theta\lesssim \|\p_{t}\xi\|^{2}_{H^{1}}\|u\|_{W^{2,q_{+}}}\lesssim\mathcal{E}_{-}^{\frac{3}{2}}
    \end{align}
    
    \textbf{Term $I_{6}$}By the definition of $I_{6}$ and $\lambda$ (Equation \eqref{equ:lambda}), and trace theorem, we have

    \begin{align}
        I_{6}=\frac{1}{\vert \lambda\vert}\int_{0}^{\pi}\frac{\vert\partial_{t}\xi\vert^{2}}{\rho^{3}}\vert u\vert\vert \mathcal{N}\vert d\theta\lesssim \vert\vert \xi\vert \vert^{2}_{L^{\infty}}\vert \vert u\vert \vert_{H^{1}}\lesssim(\mathcal{E}_{-})^{\frac{3}{2}}
    \end{align}

    \textbf{Term $I_{7}$} By the definition of $I_{7}$ and $\lambda$ (Equation \eqref{equ:lambda}), and trace theorem, we have

    \begin{align}
        I_{7}=\frac{1}{\vert\lambda\vert}\int_{0}^{\pi}\frac{\vert\partial_{t}^{2}\xi\vert}{\rho^{2}}\vert u\vert\vert \mathcal{N}\vert d\theta\lesssim \vert \vert \partial_{t}^{2}\xi\vert\vert_{L^{2}}\vert \vert u\vert \vert_{L^{2}}\lesssim \mathcal{E}_{-}^{\frac{3}{2}}
    \end{align}

    \textbf{Term $I_{8}$} We have the following estimate using H\"older's inequality and trace theorem

    \begin{align}
        I_{7}=\frac{1}{\vert\lambda\vert}\int_{0}^{\pi}\frac{\vert\partial_{t}\xi\vert}{\rho^{2}}\vert \p_{t}u\vert\vert \mathcal{N}\vert d\theta\lesssim\vert \vert \partial_{t}\xi\vert \vert_{L^{2}}\vert \vert \p_{t}u\vert \vert_{L^{2}(\Sigma)}\lesssim \vert \vert \partial_{t}\xi\vert \vert_{L^{2}}\vert \vert \p_{t}u\vert \vert_{H^{1}}\lesssim \mathcal{E}_{-}
    \end{align}

    \textbf{Term $I_{9}$}: We have the following estimate using trace theorem

    \begin{align}
        I_{9}=\frac{1}{\vert\lambda\vert}\int_{0}^{\pi}\frac{1}{\rho}\vert \p_{t}^{2}u\vert\vert \mathcal{N}\vert d\theta\lesssim \vert \vert \p_{t}^{2}u\vert \vert_{H^{1}}
    \end{align}

     \textbf{Term $I_{10}$}:  We have the following estimate using H\"older's inequality and trace theorem

    \begin{align}
        I_{10}=\frac{1}{\vert\lambda\vert}\int_{0}^{\pi}\frac{|\p_{t}\xi|}{|\rho|^{2}}\vert u\vert\vert \p_{t}\mathcal{N}\vert d\theta\lesssim \int_{0}^{\pi}|\p_{t}\xi||u||\p_{t}\p_{\theta}\xi|\lesssim \|\p_{t}\xi\|_{L^{\infty}}\|u\|_{L^{\infty}(\Sigma)}\|\p_{t}\xi\|_{H^{1}}\lesssim \|u\|_{W^{2,q_{+}}}\|\p_{t}\xi\|^{2}_{H^{1}}\lesssim \mathcal{E}_{-}^{\frac{3}{2}}
    \end{align}

    \textbf{Term $I_{11}$}: We have the following estimate using H\"older's inequality and trace theorem

    \begin{align}
        I_{11}=\frac{1}{\vert\lambda\vert}\int_{0}^{\pi}\frac{1}{\rho}\vert \p_{t}u\vert\vert \p_{t}\mathcal{N}\vert d\theta\lesssim \int_{0}^{\pi}|\p_{t}u||\p_{t}\xi'|\lesssim \|\p_{t}u\|_{L^{\infty}(\Sigma)}\|\p_{t}\xi\|_{H^{1}}\lesssim \|\p_{t}u\|_{W^{2,q_{+}}(\Sigma)}\|\p_{t}\xi\|_{H^{1}}\lesssim\mathcal{E}_{-}
    \end{align}

    \textbf{Term $I_{12}$}: We have the following estimate using H\"older's inequality and trace theorem

    \begin{align}
        I_{12}=\frac{1}{\vert\lambda\vert}\int_{0}^{\pi}\frac{1}{\rho}\vert u\vert\vert \p_{t}^{2}\mathcal{N}\vert d\theta\lesssim \int_{0}^{\pi}|u||\p_{t}^{2}\xi'|\lesssim \|u\|_{L^{\infty}(\Sigma)}\|\p_{t}^{2}\xi\|_{H^{1}}\lesssim \|u\|_{W^{2,q_{+}}(\Sigma)}\|\p_{t}\xi\|_{H^{1}}\lesssim\mathcal{E}_{-}
    \end{align}

    Combining all of the estimates for $I_{i}$ above, we obtain the following result

    \begin{align}
        \vert \mathfrak{n}'''(t)\vert\lesssim \vert \vert \partial_{t}^{2}u\vert \vert_{H^{1}}+\mathcal{E}_{-}
    \end{align}

    \noindent This completes the proof.
\end{proof}

\section{The functional calculus for the capillary operator $\mathcal{K}$}

In this Section, we define a capillary operator with respect to the $1,\Sigma$ inner product defined by \eqref{equ:4.1.25}. For the detailed construction, we  refer to the construction in the paper of Guo-Tice \cite{Guo}.

       Let $\mathcal{K}$ denote the gravity capillary operator in polar coordinates. The definition of this operator is given as follows

      \textbf{Definition of $\mathcal{K}$}: Consider weak solution $\rho\in\{ \psi\in H^{1}(0,\pi)|\int_{0}^{\pi}\psi\rho_{0}=0, \int_{0}^{\pi}\psi\xi_sd\theta=0\}$ to the following equation system.

      \begin{equation}{\label{equ:5.1.62}}
          \begin{cases}
              \mathcal{K}\rho=f \\
              B_{\pm}\rho=h_{\pm}
          \end{cases}
      \end{equation}
      
      \noindent for any $f\in [H^{1}(0,\pi)]^{*}$ and $h:\{0,\pi\}\rightarrow \mathbb{R}$. The weak form of the equation system \eqref{equ:5.1.62} is defined as follows

      \begin{equation}{\label{equ:5.1.63}}
          (\rho,\psi)_{1,\Sigma}=(f,\psi)_{H^{0}}+[h,\psi]_{\theta}
      \end{equation}

      \noindent for any $\psi\in\{ \psi\in H^{1}(0,\pi)|\int_{0}^{\pi}\psi\rho_{0}=0, \int_{0}^{\pi}\psi\xi_sd\theta=0\}$. 
      
      The result of existence and uniqueness can be derived from the positivity of the inner product $1,\Sigma$ and Lax-Milgram theorem. Therefore, for any $f$ and $h_{\pm}$ in the function spaces defined above, there exists a unique $\rho\in\{ \psi\in H^{1}(0,\pi)|\int_{0}^{\pi}\psi\rho_{0}=0, \int_{0}^{\pi}\psi\xi_sd\theta=0\}$ such that equation \eqref{equ:5.1.63} is satisfied. Then we define $\mathcal{K}^{-1}$ to be
      \begin{align}
          \mathcal{K}^{-1}:  [H^1(0,\pi)]^{*} \rightarrow \{ \psi\in H^{1}(0,\pi)|\int_{0}^{\pi}\psi\rho_{0}=0, \int_{0}^{\pi}\psi\xi_sd\theta=0\}
      \end{align}
      \noindent such that
      \begin{align}
         \mathcal{K}^{-1}(f)=\rho
      \end{align}
      where $\rho$ is the unique solution to \eqref{equ:5.1.63} when $h=0$. Thus, $\mathcal{K}$ can be defined as the inverse map of $\mathcal{K}^{-1}$. Moreover, using a standard elliptic estimate, we obtain the following estimate for any solution $\rho$.

     \begin{equation}{\label{equ:5.1.64}}
         \vert \vert \rho\vert \vert_{H^{m+2}}\lesssim\vert \vert f\vert \vert_{H^{m}}+\|\rho_{0}\|_{H^{m+2}}+\|\xi_{s}\|_{H^{m+2}}.
     \end{equation}

     We now define a map $f\rightarrow \rho_{f}$ satisfying the following relation

     \begin{equation}{\label{equ:5.1.65}}
         (\rho_{f},\psi)_{1,\Sigma}=(f,\psi)_{0,\Sigma}
     \end{equation}

     \noindent When restricted on the function space $L^{2}$, this is a compact map. Hence, we construct an orthonormal basis of $L^{2}$ $\{\omega_{k}\}$ where $\omega_{k}$ is the eigenfunction of this map with eigenvalue $\lambda_{k}$. Moreover, we establish the Fourier expansion of $f$ with respect to this basis $\omega_{k}$ and denote $\hat{f}(k)$ to be the corresponding Fourier coefficients of $f$. Then the Fourier decomposition of $f$ with respect to this Fourier is expressed as follows

     \begin{equation}{\label{equ:5.1.66}}
         (f,g)_{0,\Sigma}=\sum_{k=0}^{+\infty}\hat{f}(k)\hat{g}(k)~\operatorname{and}~\vert \vert f\vert \vert^{2}_{0,\Sigma}=\sum_{k=0}^{\infty}\vert \hat{f}(k)\vert^{2},
     \end{equation}

     \noindent and

     \begin{equation}{\label{equ:5.1.67}}
         (f,g)_{1,\Sigma}=\sum_{k} \lambda_{k}\hat{f}(k)\hat{g}(k).
     \end{equation}

     \noindent Now we have the definition for $(1,\Sigma)$ norm and $(0,\Sigma)$ norm in Fourier sense which can be denoted as the inner product of spaces $\mathcal{H}_{\mathcal{K}}^{i}$ for $i=0,1$. We then define a series of spaces $\mathcal{H}_{\mathcal{K}}^{s}$ based on this decomposition. We have the math expression of this construction as follows

     \begin{equation}{\label{equ:5.1.68}}
         (u,v)_{\mathcal{H}_{\mathcal{K}}^{s}}=\Sigma_{k=0}^{\infty}\lambda_{k}^{s}\hat{u}(k)\hat{v}(k)
     \end{equation}

     \noindent Using the definition above, we construct a a series of Hilbert spaces $\mathcal{H}_{\mathcal{K}}^{s}$ with respect to the eigenvalues of $\mathcal{K}$. We then establish the functional calculus corresponding to the operator $\mathcal{K}$ based on this Fourier decomposition. Let $\Sigma_{k=\{\lambda_{k}|k\geq 0\}}\subseteq (0,+\infty)$. For any $r\in \mathbb{R}$, we define the following function space

     \begin{align}{\label{equ:5.1.69}}
         \mathfrak{B}(\Sigma_{\mathcal{K}})=\{f:\Sigma_{\mathcal{K}}\rightarrow \mathcal{R}|~\vert \vert f\vert \vert_{\mathfrak{B}^{r}}<+\infty\}
     \end{align}

     \noindent where:

     \begin{equation}{\label{equ:5.1.70}}
         \vert \vert f\vert \vert_{\mathfrak{B}^{r}}:=\sup_{x\geq \lambda_{0}}\frac{\vert f(x)\vert}{x^{r}}
     \end{equation}

     \noindent Similarly, we define:

     \begin{equation}{\label{equ:5.1.71}}
         \mathfrak{B}_{0}^{r}(\Sigma_{\mathcal{K}})=\{f\in \mathfrak{B}^{r}(\Sigma_{\mathcal{K}})|\lim_{x\rightarrow +\infty}(\frac{\vert f(x)\vert}{x^{r}})=0\}
     \end{equation}

     \noindent Let $s\in\mathbb{R}$ and $r\in\mathbb{R}$. For any functions $f\in \mathfrak{B}^{r}(\Sigma_{\mathcal{K}})$ and $u\in \mathcal{H}_{\mathcal{K}}^{s+2r}((0,\pi))$, we have the following definition for $f(\mathcal{K})$

     \begin{equation}{\label{equ:5.1.72}}
         f(\mathcal{K})u:=\sum_{k=0}^{\infty}f(\lambda_{k})\hat{u}(k)\omega_{k}.
     \end{equation}

     \noindent The operator $f(\mathcal{K})$ defined above is a compact, bounded operator from $\mathcal{H}_{\mathcal{K}}^{s+2r}$ to $\mathcal{H}_{\mathcal{K}}^{s}$. The proof of compactness follows from Guo-Tice's paper \cite{Guo}. Hence, we obtain the compactness embedding from $\mathcal{H}_{\mathcal{K}}^{t}((0,\pi))\subset\subset\mathcal{H}_{\mathcal{K}}^{s}((0,\pi))$ if $s<t$. Using these results, we obtain that the map $\mathcal{K}$ from $\mathcal{H}_{\mathcal{K}}^{s+2}$ to $\mathcal{H}_{\mathcal{K}}^{s}$ is an isometric isomorphism.  Furthermore, we define operator $D_{j}^{s}$ as follows
     \begin{align}{\label{equ:D_j}}
         D_{j}^{s}u=\sum_{k=0}^{j}\lambda_{k}\hat{u}(k)\omega_{k}~\operatorname{for~any~j\in\mathbb{N}^{+}}
     \end{align}
     \noindent This operator is important for our discussion in the next section.
     
    Now we have a well defined pseudo-differential operator in the functional space $\tilde{B}$ defined as follows:

    \begin{align}
        \tilde{B}:=\{\rho\in H^{1}|\int_{0}^{\pi}\rho\rho_{0}d\theta=0,\int_{0}^{\pi}\rho\xi_2d\theta=0\}
    \end{align}

    \noindent From \cite{Guo}, we have the following embedding

    \begin{align}
        \mathcal{H}_{\mathcal{K}}^{s}\cap \tilde{B}=H^{s}\cap \tilde{B}
    \end{align}

    \noindent for any $0\leq s<\frac{3}{2}$. We then extend this map $\mathcal{K}$ to the functional space $H^{1}$ (The definition is given in Section 4.2). The extension map is defined to satisfy the following properties

    \begin{align}
        \tilde{\mathcal{K}}(\rho)=\mathcal{K}(\rho)
    \end{align}

    \noindent for any $\rho\in \tilde{B}$, and

    \begin{align}
         \tilde{\mathcal{K}}(\rho)=0
    \end{align}

    \noindent if

    \begin{align}
    \rho=c_{1}\rho_{0}+c_{2}\xi_{2}
    \end{align}
    
    \noindent for any constant $c_{1}$ and $c_{2}$. For simplicity, we still use $\mathcal{K}$ to denote this capillary map $\mathcal{K}$ and use $D_{j}^{s}$ to denote the extension of the operator defined in \eqref{equ:D_j}. Moreover, using Proposition 3.80 in \cite{Guo} paper, we have

    \begin{align}{\label{est:D_j}}
        D_{j}^{s}:H^{s-r}\cap \tilde{B}\rightarrow H^{-r}\cap \tilde{B}=(H^{r})^{*}\cap\tilde{B}
    \end{align}

    \noindent For any $0\leq r<\frac{1}{2}$ and $0\leq s<\frac{3}{2}$
    
    \section{Apriori estimate}

   Now we have proved all of the preliminary results for the estimate of equation system \eqref{equ:4.1.23} in previous sections. In this section, we proceed to establish the Apriori estimate.
   
    \subsection{Energy estimate}

    We define the energy terms as follows

    \begin{align}
        \mathcal{E}_{||}=\sum_{k=0}^{2}(\vert \vert \partial_{t}^{k}u\vert \vert_{L^{2}}+\vert \vert \partial_{t}^{k}\xi\vert \vert_{H^{1}}^{2}),
    \end{align}

    \noindent and:

    \begin{align}{\label{equ:energy}}
        \mathcal{E}=\mathcal{E}_{||}+\vert \vert u\vert \vert_{W^{2,q_{+}}}^{2}+\vert \vert \partial_{t}u\vert \vert_{H^{1+\frac{\epsilon_{-}}{2}}}+\vert \vert \partial_{t}^{2}u\vert \vert^{2}_{H^{0}}+\vert \vert p\vert \vert^{2}_{W^{1,q_{+}}}+\vert \vert \partial_{t}p\vert \vert_{L^{2}}^{2}\notag\\
        |\mathfrak{n}^{\prime}(t)|^{2}+|\mathfrak{n}^{\prime\prime}(t)|^{2}+\vert \vert \xi\vert \vert_{W^{3-\frac{1}{q_{+}}}}+\vert \vert \partial_{t}\xi\vert \vert_{H^{\frac{3}{2}+\frac{\epsilon_{-}-\alpha}{2}}}^{2}+\vert \vert \partial_{t}^{2}\xi\vert \vert^{2}_{H^{1}}.
    \end{align}

    \noindent We then give the definition of the dissipation terms as follows

    \begin{align}
        \mathcal{D}_{||}=\Sigma_{k=0}^{2}\vert \vert \partial_{t}^{k}u\vert \vert^{2}_{H^{1}}+\vert \vert \partial_{t}^{k}u\vert\vert_{L^{2}}+[\p_{t}^{k}u\cdot \mathcal{N}]_{\theta}^{2},
    \end{align}
    \noindent and
    \begin{align}{\label{equ:dis}}
        \mathcal{D}=\mathcal{D}_{||}+\vert \vert u\vert \vert^{2}_{W^{2,q_{+}}}\vert \vert \partial_{t}u\vert \vert_{W^{2,q_{-}}}+\Sigma_{k=0}^{2}[\partial_{t}^{k+1}\xi]^{2}_{\theta}+\Sigma_{k=0}^{2}[\partial_{t}^{k}\partial_{\theta}\xi]_{\theta}^{2}+\vert \vert p\vert \vert_{W^{1,q_{+}}}+\vert \vert \partial_{t}p\vert \vert_{W^{1,q_{-}}}\notag\\
        |\mathfrak{n}^{\prime}(t)|^{2}+|\mathfrak{n}^{\prime\prime}(t)|^{2}+|\mathfrak{n}^{(3)}(t)|^{2}+\Sigma_{k=0}^{2}\vert \vert \partial_{t}^{k}\xi\vert \vert^{2}_{H^{\frac{3}{2}-\alpha}}+\vert \vert \xi\vert \vert_{W^{3-\frac{1}{q_{+}},q_{+}}}^{2}+\vert \vert \partial_{t}\xi\vert \vert_{W^{3-\frac{1}{q_{-}},q_{-}}}+\vert \vert \partial_{t}^{3}\xi\vert \vert_{H^{\frac{1}{2}-\alpha}}
    \end{align}
    
      To begin with, we establish the energy-dissipation relation estimate for this problem. We first rewrite the equation system \eqref{equ:4.1.23} by separating it according to each order.

      \textbf{Zero Order} In the following system, $v=u, q=p$ and $\zeta=\xi$:
      
      \begin{equation}{\label{equ:0}}
      \begin{cases}
          \partial_{t}v+\operatorname{div}_{\mathcal{A}}S_{\mathcal{A}}(q,v)={b}^{1}~~&\operatorname{in}~~\Omega\\
        \operatorname{div}_{\mathcal{A}}v=0~~&\operatorname{in}~~\Omega\\
        S_{\mathcal{A}}(q,v)\mathcal{N}=(\mathcal{K}(\zeta)-\sigma\partial_{\theta}{b}^{3})\mathcal{N}+{b}^{4}~~&\operatorname{on}~~\Sigma\\
        \partial_{t}\zeta=\frac{1}{\rho_{0}}v\cdot \mathcal{N}-\mathfrak{n}^{\prime}(t)\xi_s+b^{6}~~&\operatorname{on}~~\Sigma\\
        (S_{\mathcal{A}}(q,v)\nu-\beta v)\cdot \tau=b^{5}~~&\operatorname{on}~~\Sigma_{s}\\
        v\cdot \nu=0~~&\operatorname{on}~~\Sigma_{s}\\
        (\mp\sigma \frac{\rho_{0}^{2}\partial_{\theta}(\zeta)}{(\rho_{0}^{2}+\rho_{0}'^{2})^{\frac{3}{2}}}\pm\sigma \frac{\rho_{0}'\rho_{0}(\zeta)}{(\rho_{0}^{2}+\rho_{0}'^2)^{\frac{3}{2}}}\mp \sigma b^{3})(\frac{\pi}{2}\pm\frac{\pi}{2})=(\kappa (\frac{1}{\rho_{0}}v\cdot \mathcal{N})-b^{7}+b^{6})(\frac{\pi}{2}\pm \frac{\pi}{2})
      \end{cases}
      \end{equation}

      \textbf{First Order} In the following system,  $v=\partial_{t}u, q=\p_{t}p$ and $\zeta=\p_{t}\xi$:

      \begin{equation}{\label{equ:1}}
      \begin{cases}
          \partial_{t}v+\operatorname{div}_{\mathcal{A}}S_{\mathcal{A}}(q,v)={b}^{1,1}~~&\operatorname{in}~~\Omega\\
        \operatorname{div}_{\mathcal{A}}v=b^{2,1}~~&\operatorname{in}~~\Omega\\
        S_{\mathcal{A}}(q,v)\mathcal{N}=(\mathcal{K}(\zeta)-\sigma\partial_{\theta}{b}^{3,1})\mathcal{N}+{b}^{4,1}~~&\operatorname{on}~~\Sigma\\
        \partial_{t}\zeta=\frac{1}{\rho_{0}}v\cdot \mathcal{N}-\mathcal{n}^{\prime \prime}(t)\xi_s+b^{6,1}~~&\operatorname{on}~~\Sigma\\
        (S_{\mathcal{A}}(q,v)\nu-\beta v)\cdot \tau=b^{5,1}~~&\operatorname{on}~~\Sigma_{s}\\
        v\cdot \nu=0~~&\operatorname{on}~~\Sigma_{s}\\
        (\mp\sigma \frac{\rho_{0}^{2}\partial_{\theta}(\zeta)}{(\rho_{0}^{2}+\rho_{0}'^{2})^{\frac{3}{2}}}\pm\sigma \frac{\rho_{0}'\rho_{0}(\zeta)}{(\rho_{0}^{2}+\rho_{0}'^2)^{\frac{3}{2}}}\mp \sigma b^{3,1})(\frac{\pi}{2}\pm\frac{\pi}{2})=(\kappa (\frac{1}{\rho_{0}}v\cdot \mathcal{N})-b^{7,1}-b^{6,1})(\frac{\pi}{2}\pm\frac{\pi}{2})
      \end{cases}
      \end{equation}

       \textbf{Second Order} In the following system,  $v=\partial_{t}^{2}u, \zeta=\p_{t}^{2}\xi$ and $q=\p_{t}^{2}p$:

       \begin{equation}{\label{equ:2}}
      \begin{cases}
          \partial_{t}v+\operatorname{div}_{\mathcal{A}}S_{\mathcal{A}}(q,v)={b}^{1,2}~~&\operatorname{in}~~\Omega\\
        \operatorname{div}_{\mathcal{A}}v=b^{2,2}~~&\operatorname{in}~~\Omega\\
        S_{\mathcal{A}}(q,v)\mathcal{N}=(\mathcal{K}(\zeta)-\sigma\partial_{\theta}{b}^{3,2})\mathcal{N}+{b}^{4,2}~~&\operatorname{on}~~\Sigma\\
        \partial_{t}\zeta=\frac{1}{\rho_{0}}v\cdot \mathcal{N}-\mathfrak{n}^{\prime\prime\prime}(t)\xi_s+b^{6,2}~~&\operatorname{on}~~\Sigma\\
        (S_{\mathcal{A}}(q,v)\nu-\beta v)\cdot \tau=b^{5,2}~~&\operatorname{on}~~\Sigma_{s}\\
        v\cdot \nu=0~~&\operatorname{on}~~\Sigma_{s}\\
        (\mp\sigma \frac{\rho_{0}^{2}\partial_{\theta}(\zeta)}{(\rho_{0}^{2}+\rho_{0}'^{2})^{\frac{3}{2}}}\pm\sigma \frac{\rho_{0}'\rho_{0}(\zeta)}{(\rho_{0}^{2}+\rho_{0}'^2)^{\frac{3}{2}}}\mp \sigma b^{3,2})(\frac{\pi}{2}\pm\frac{\pi}{2})=(\kappa (\frac{1}{\rho_{0}}v\cdot \mathcal{N})-b^{7,2}-b^{6,2})(\frac{\pi}{2}\pm\frac{\pi}{2})
      \end{cases}
      \end{equation}
      The expressions of all the non-linear terms $b^{i,j}$ are given in Appendix.
      
      We now establish the energy estimate for the 0-order equation by the following theorem.

      \begin{theorem}{\label{thm:0-order}}
          Suppose that $(u,p,\xi)$ is the solution to system \eqref{equ:0} such that $\mathcal{E}(u,p,\xi)\leq \delta\ll 1$ when $t\in (0,T)$ for some positive constant $T$. It satisfies the following energy-dissipation relation

          \begin{align}
              \mathcal{E}_{||,0}(u,p,\xi)(t)+\int_{0}^{t}\mathcal{D}_{||,0}(u,p,\xi)(s)ds\lesssim \mathcal{E}_{||,0}(0)+ \int_{0}^{t}\sqrt{\mathcal{E}}\mathcal{D}(u,p,\xi)(s)ds
          \end{align}
      \end{theorem}

      \begin{proof}

      We first write down the  Navier-Stokes equation (The first equation of \eqref{equ:4.1.23}) in weak form as follows
      
      \begin{align}
         (\partial_{t}u,Ju)+(\operatorname{div}_{\mathcal{A}}S_{\mathcal{A}}(p,u),Ju)=(b^{1},Ju)\label{equ:5.1.93}
      \end{align}

      For the first term in the equation \eqref{equ:5.1.93}, we have:

      \begin{equation}{\label{equ:5.1.94}}
          (\partial_{t}u,Ju)=\frac{d}{dt}\frac{1}{2}\int_{\Omega}J\vert u\vert^{2}-\frac{1}{2}\int_{\Omega}\partial_{t}J\vert u\vert^{2}
      \end{equation}

       For the second term in \eqref{equ:5.1.93}, using integration by part, we obtain

      \begin{align}
          &\int_{\Omega} J\mathcal{A}_{jk}\partial_{k}(S_{\mathcal{A}}(p,u)_{ij})u_{i}=\int_{\Omega}\partial_{k}(J\mathcal{A}_{jk}S_{\mathcal{A}}(p,u)_{ij})u_{i}\notag\\
          &=\int_{\Omega} -J\mathcal{A}_{jk}\partial_{k}u_{i}S_{\mathcal{A}}(p,u)_{ij}+\int_{\partial \Omega} (J\mathcal{A}\nu)\cdot (S_{\mathcal{A}}(p,u)u):=I_{1}+I_{2}\label{equ:5.1.95}
      \end{align}

      \noindent For the term $I_{1}$, we have the following computation from the fact that $\dive_{\mathcal{A}}{u}=0$

      \begin{align}
          I_{1}=\int_{\Omega}\frac{\mu}{2}\mathbb{D}_{\mathcal{A}}u:\mathbb{D}_{\mathcal{A}}uJ-p\operatorname{div}_{\mathcal{A}}u J=\int_{\Omega}\frac{\mu}{2}\mathbb{D}_{\mathcal{A}}u:\mathbb{D}_{\mathcal{A}}uJ\label{equ:5.1.96}
      \end{align}

      \noindent For the term $I_{2}$, we separate it into two parts as follows

      \begin{align}
          \int_{\partial \Omega}(J\mathcal{A}\nu)\cdot (S_{\mathcal{A}}(p,u)u)=\int_{\Sigma_{s}}(J\mathcal{A}\nu)\cdot (S_{\mathcal{A}}(p,u)u)+\int_{\Sigma}(J\mathcal{A}\nu)\cdot (S_{\mathcal{A}}(p,u)u)=I_{3}+I_{4}\label{equ:5.1.97}
      \end{align}

      \noindent For the term $I_{3}$, we have:

      \begin{align}
          \int_{\Sigma_{s}}(J\mathcal{A}\nu)\cdot (S_{\mathcal{A}}(p,u)u)=&\int_{\Sigma_{s}}J\nu\cdot (S_{\mathcal{A}}(p,u)u)=\int_{\Sigma_{s}}Ju\cdot (S_{\mathcal{A}}(p,u)\nu)\notag\\
          =&\int_{\Sigma_{s}}J\beta(u\cdot \tau)^{2} \label{equ:5.1.98},
      \end{align}

      \noindent where we used the fifth and sixth equations in \eqref{equ:4.1.23} to substitute $S_{\mathcal{A}}(p,u)\nu\cdot \tau$. For the term $I_{4}$, we have

      \begin{align}
          \int_{\Sigma}(J\mathcal{A}\nu)\cdot (S_{\mathcal{A}}(p,u)u)=\int_{\Sigma} \frac{\mathcal{N}}{\vert \mathcal{N}_{0}\vert}\cdot  (S_{\mathcal{A}}(p,u)u)=\int_{0}^{\pi} (S_{\mathcal{A}}(p,u)\mathcal{N})\cdot u d\theta=\int_{0}^{\pi}(\mathcal{K}(\xi)-\partial_{\theta}b^{3})u\cdot \mathcal{N}+b^{4}\cdot u d\theta\label{equ:5.1.99}.
      \end{align}

      \noindent For the right hand side of equation \eqref{equ:5.1.99}, using integration by part, we obtain

      \begin{align}
         \int_{0}^{\pi}(\mathcal{K}(\xi)-\frac{1}{\rho_{0}}\partial_{\theta}b^{3})u \cdot \mathcal{N} d\theta
         = (\xi,\frac{1}{\rho_{0}}u\cdot \mathcal{N})_{1,\Sigma}+\int_{0}^{\pi}b^{3}\partial_{\theta}(\frac{1}{\rho_{0}}u\cdot \mathcal{N})d\theta+\kappa[\frac{1}{\rho_{0}}u\cdot \mathcal{N},\frac{1}{\rho_0{}}u\cdot \mathcal{N}]_{\theta}\notag\\-\kappa[b^{7},\frac{1}{\rho_{0}}u\cdot\mathcal{N}]_{\theta}-\kappa[b^{6},\frac{1}{\rho_{0}}u\cdot \mathcal{N}]_{\theta} \label{equ:5.1.100}.
      \end{align}

      \noindent  Applying equations \eqref{equ:5.1.94}, \eqref{equ:5.1.96}-\eqref{equ:5.1.100} to equation \eqref{equ:5.1.93}, and using the kinematic boundary condition, we obtain the following result

      \begin{align}
          \frac{d}{dt}\frac{1}{2}(\int_{0}^{\pi}J\vert u\vert^{2}+(\xi,\xi)_{1,\Sigma})+\int_{\Sigma_{s}}J\beta(u\cdot \tau)^{2}+\frac{\mu}{2}\int_{\Omega} \mathbb{D}_{\mathcal{A}}u:\mathbb{D}_{\mathcal{A}}uJ+\kappa[\frac{1}{\rho_{0}}u\cdot \mathcal{N}]^{2}_{\theta}\notag\\
          =\int_{\Omega} b^{1}\cdot uJ+\frac{1}{2}\int_{\Omega}\partial_{t}J\vert u\vert^{2}+\kappa[b^{7}+b^{6},\frac{1}{\rho_{0}}u\cdot\mathcal{N}]_{\theta}+(\xi,b^{6})_{1,\Sigma}+(\xi,\mathfrak{n}'(t)\xi_s)_{1,\Sigma}\notag\\
          -\int_{0}^{\pi}(b^{3}\partial_{\theta}(\frac{1}{\rho_{0}}u\cdot \mathcal{N}))-\int_{0}^{\pi} b^{4}\cdot u  d\theta \label{equ:5.1.104}
      \end{align}

     \noindent For the term $(\xi,\mathfrak{n}'(t)\xi_{s})$, we recall from the conservation of total mass

     \begin{equation}
         \int_{0}^{\pi}\rho^{2}=\int_{0}^{\pi}\rho_{0}^{2},
     \end{equation}

     \noindent which implies that:

     \begin{equation}{\label{equ:5.1.101}}
         2\int_{0}^{\pi}\xi\rho_{0} d\theta=-\int_{0}^{\pi}\xi^{2}d\theta
     \end{equation}

     \noindent We the define
     \begin{align}
     a_{0}=-\frac{\frac{1}{2}\int_{0}^{\pi}\xi^{2}d\theta}{\int_{0}^{\pi}\rho_{0}^{2}d\theta}
     \end{align}
     With this definition, we have
     \begin{equation}{\label{equ:5.1.102}}
         (\xi,\mathfrak{n}'(t)\xi_s)_{1,\Sigma}=(\xi-a_0\rho_{0},\mathfrak{n}'(t)\xi_s)_{1,\Sigma}+a_0\mathfrak{n}'(t)(\rho_{0},\xi_s)_{1,\Sigma}
     \end{equation}

     \noindent Using Theorem \ref{thm:pos}, we have $(\xi-a_0\rho_{0},\mathfrak{n}'(t)\xi_s)_{1,\Sigma}=0$. Therefore, equation \eqref{equ:5.1.102} can be rewritten as

     \begin{equation}{\label{equ:5.1.105}}
         (\xi,\mathfrak{n}'(t)\xi_s)_{1,\Sigma}=a_{0}\mathfrak{n}'(t)(\rho_{0},\xi_s)_{1,\Sigma}
     \end{equation}

     \noindent Substituting \eqref{equ:5.1.105} into equation \eqref{equ:5.1.104}, we finally obtain that

     \begin{align}
          &\frac{d}{dt}\frac{1}{2}(\int_{0}^{\pi}J\vert u\vert^{2}+(\xi,\xi)_{1,\Sigma})+\int_{\Sigma_{s}}J\beta(u\cdot \tau)^{2}+\frac{\mu}{2}\int_{\Omega} \mathbb{D}_{\mathcal{A}}u:\mathbb{D}_{\mathcal{A}}uJ+\kappa[u\cdot \mathcal{N}]^{2}_{\theta}\notag\\
          &=\int_{\Omega} b^{1}\cdot uJ+\frac{1}{2}\int_{\Omega}\partial_{t}J\vert u\vert^{2}+\kappa[b^{7}+b^{6},\frac{1}{\rho_{0}}u\cdot\mathcal{N}]_{\theta}+(\xi,b^{8})_{1,\Sigma}+a_{0}\mathfrak{n}'(t)(\rho_{0},\xi_s)_{1,\Sigma}\notag\\
           &\quad-\int_{0}^{\pi}(b^{3}\partial_{\theta}(\frac{1}{\rho_{0}}u\cdot \mathcal{N}))-\int_{0}^{\pi} b^{4}\cdot u  d\theta\label{equ:5.1.106}
     \end{align}

      For the right-hand side of the equation \eqref{equ:5.1.106}, we now estimate each term individually.

     \textbf{Terms included in $b^{1}$}

     Using $b^{1}$ in Appendix and the estimate for $\mathfrak{n}'(t)$ in Theorem \ref{thm:gam}, we establish the estimate for terms with $b^{1}$. We have the following estimates

     \begin{equation}{\label{equ:5.1.107}}
         \int_{\Omega}\mathfrak{n}'(t)\partial_{x}u\lesssim \vert \vert u\vert \vert_{W^{1,1}}\vert \vert u\cdot \mathcal{N}\vert \vert_{L^{1}(\Sigma)} \lesssim\|u\|_{W^{2,q_{+}}}^{2}\lesssim \mathcal{E}^{2}
     \end{equation}

     \begin{equation}{\label{equ:5.1.108}}
         \int_{\Omega} (\cos\theta W\partial_{t}\bar{\xi},\sin\theta W\partial_{t}\bar{\xi})\mathcal{A}(\partial_{x}u,\partial_{y}u)^{T}Ju\lesssim \vert \vert \partial_{t}\bar{\xi}\vert \vert_{L^{\infty}}  \vert \vert u\vert \vert_{H^{1}}\vert \vert u\vert \vert_{L^{2}}\lesssim \vert \vert \partial_{t}\xi\vert \vert_{H^{\frac{3}{2}+\frac{\epsilon_{-}-\alpha}{2}}}\vert \vert u\vert \vert_{H^{1}}\vert \vert u\vert \vert_{L^{2}}\lesssim \mathcal{E}^{\frac{3}{2}}
     \end{equation}

     \begin{equation}{\label{equ:5.1.109}}
         \int_{\Omega} u\cdot \nabla_{\mathcal{A}}uu\lesssim \vert \vert u\vert \vert_{L^{\infty}}\vert \vert u\vert \vert_{H^{1}}\vert \vert u\vert \vert_{L^{2}}\lesssim \|u\|_{W^{2,q_{+}}}^{3}\lesssim \mathcal{E}^{\frac{3}{2}}
     \end{equation}

     \textbf{Term including $\p_{t}J$}:

     By definition $J=\operatorname{det}(\mathcal{A})$, we have the following estimate for this term.

    \begin{equation}{\label{equ:5.1.110}}
    \begin{aligned}
        \frac{1}{2}\int_{\Omega} \partial_{t}J\vert u\vert^{2}&\lesssim \vert \vert u\vert \vert_{L^{\infty}}\vert \vert u\vert \vert_{L^{2}}\vert \vert \partial_{t}J\vert \vert_{L^{2}}\lesssim \vert \vert u\vert \vert_{W^{2,q_{+}}}\vert \vert u\vert \vert_{L^{2}}(\vert \vert \partial_{t}\p_{\theta}\bar{\xi}\vert \vert_{L^{2}}+\vert \vert \partial_{t}\p_{\theta}\bar{\xi}\vert \vert_{L^{2}})\\
         &\lesssim \|u\|_{W^{2,q_{+}}}^{2}(\|\p_{t}\xi\|_{H^{1}})\lesssim \mathcal{E}^{\frac{3}{2}}
        \end{aligned}
    \end{equation}

   \textbf{Term $(\xi,b^{6})_{1,\Sigma}$}

   We have the following estimate using the definition of $b^{6}$ in Appendix, and Theorem \ref{thm:gam}:
   
    \begin{equation}{\label{equ:5.1.111}}
    \begin{aligned}
        (\xi,b^{6})_{1,\Sigma}&=((\frac{1}{\rho}-\frac{1}{\rho_{0}})u\cdot \mathcal{N},\xi)_{1,\Sigma}+(\mathfrak{n}^{\prime}(t)\sin\theta(\frac{\rho'}{\rho}-\frac{\rho_{0}'}{\rho_{0}}),\xi)_{1,\Sigma}\\
        &\lesssim \vert \vert \xi\vert \vert^{2}_{W^{1,+\infty}}  \vert \vert u\cdot \mathcal{N}\vert \vert_{W^{1,\frac{1}{1-\epsilon_{+}}}(\Sigma)}+\vert \vert u\cdot \mathcal{N}\vert \vert_{W^{1,\frac{1}{1-\epsilon_{+}}}(\Sigma)}(\vert \vert \xi\vert \vert_{W^{1,\frac{1}{\epsilon_{+}}}}+\|\xi\|_{L^{\frac{1}{\epsilon_{+}}}})\|\xi\|_{H^{1}} \\
        &\quad+\vert \vert u\cdot \mathcal{N}\vert \vert_{L^{\frac{1}{\epsilon_{+}}}(\Sigma)}(\vert \vert \xi\vert \vert_{W^{2,\frac{1}{1-\epsilon_{+}}}}+\|\xi\|_{W^{1,\frac{1}{1-\epsilon_{+}}}})\|\xi\|_{H^{1}}\\&\lesssim \|\xi\|_{W^{3-\frac{1}{q_{+}},q_{+}}}^{2}\|u\|_{W^{2,q_{+}}}
        \lesssim \mathcal{E}^{\frac{3}{2}}
        \end{aligned}
    \end{equation}

 \textbf{Term $a_{0}\mathfrak{n}'(t)(\rho_{0},\xi_{s})_{1,\Sigma}$}

 Using Theorem \ref{thm:pos} and Theorem \ref{thm:gam}, we obtain the following result
    \begin{equation}{\label{equ:5.1.112}}
        a_{0}\mathfrak{n}'(t)(\rho_{0},\xi_s)_{1,\Sigma}\lesssim |\mathfrak{n}^{\prime}(t)|\vert \vert\xi\vert \vert^{2}_{L^{2}}\lesssim\mathcal{E}^{\frac{3}{2}}
    \end{equation}

    \textbf{Terms on the boundary}

    We have the following estimate using the definition of $b^{7}$ and $b^{6}$
    \begin{align}
        [b^{7}+b^{6},\frac{1}{\rho_{0}}u\cdot \mathcal{N}]_{\theta}&\lesssim [u\cdot \mathcal{N}]_{\theta}([\mathfrak{n}'(t)+\partial_{t}\xi]_{\theta}^{2}+[ u\cdot \mathcal{N}]_{\theta}\vert[ \xi]_{\theta}+\vert \mathfrak{n}'(t)\vert[ \partial_{\theta}\xi]_{\theta}+|\mathfrak{n}'(t)|[\xi]_{\theta})\notag\\
        &\lesssim (\vert \vert u\cdot \mathcal{N}\vert \vert^{2}_{L^{2}(\Sigma)}+\vert \vert \partial_{t}\xi\vert \vert^{2}_{H^{1}})\vert [u\cdot \mathcal{N}]_{\theta}\vert + \vert [u\cdot \mathcal{N}]_{\theta}\vert^{2}\vert \vert \xi\vert \vert_{H^{1}}\notag\\
        &\quad+|[u\cdot \mathcal{N}]_{\theta}|\|\xi\|_{W^{3-\frac{1}{q_{+}},q_{+}}}\|u\cdot \mathcal{N}\|_{L^{2}(\Sigma)}\notag\\
        &\lesssim |[u\cdot \mathcal{N}]_{\theta}|(\|\xi\|_{W^{3-\frac{1}{q_{+}},q_{+}}}\|u\|_{W^{2,q_{+}}}+|[u\cdot \mathcal{N}]_{\theta}|\|\xi\|_{H^{1}}+\|u\|^{2}_{W^{2,q_{+}}}+\|\p_{t}\xi\|_{H^{1}}^{2})\notag\\
        &\lesssim \mathcal{E}^{\frac{1}{2}}\mathcal{D}^{\frac{1}{2}}\mathcal{D}^{\frac{1}{2}}\leq \mathcal{E}^{\frac{1}{2}}\mathcal{D}{\label{equ:5.1.113}}
    \end{align}

   \textbf{Terms with $b^{3}$ and $b^{4}$}

   Using the definition of $b^{3}$, we have
    \begin{equation}{\label{equ:5.1.114}}
        \int_{0}^{\pi}b^{3}\partial_{\theta}(u\cdot \mathcal{N})\lesssim \vert \vert \xi\vert \vert_{W^{1,\infty}}^{2}\vert \vert u\cdot \mathcal{N}\vert \vert_{W^{1,1}(\Sigma)}\lesssim \vert \vert \xi\vert \vert^{2}_{W^{3-\frac{1}{q_{+}},q_{+}}}\vert \vert u\vert \vert_{W^{2,q_{+}}}\lesssim \mathcal{E}^{\frac{1}{2}}\mathcal{D}
    \end{equation}

    \noindent The estimate for the term involving $b^{4}$ follows the same argument as that for the term involving $b^{3}$.

    \textbf{Term $(\xi,\xi)_{1,\Sigma}$}

    Using Theorem \ref{thm:pos}, we have:

    \begin{align}
        (\xi,\xi)_{1,\Sigma}&=(\xi-a_{0}(t)\rho_{0},\xi-a_{0}(t)\rho_{0})_{1,\Sigma}-2(\xi,a_{0}(t)\rho_{0})_{1,\Sigma}+(a_{0}(t)\rho_{0},a_{0}(t)\rho_{0})\notag\\
        &\gtrsim \|\xi\|^{2}_{H^{1}}-\|\xi\|^{4}_{L^{2}}-\|\xi\|^{2}_{L^{2}}\|\xi\|_{H^{1}}
    \end{align}

    \noindent Then combining all of the computation above and applying them to equation \eqref{equ:5.1.104}, and subsequently integrating it from $0$ to $T$, we have

    \begin{align}{\label{equ:e-d_0}}
        \mathcal{E}_{||,0}(T)-\mathcal{E}_{||,0}(0)-\mathcal{E}_{||,0}^{\frac{3}{2}}(T)-\mathcal{E}_{||,0}^{\frac{3}{2}}(0)+\int_{0}^{t}\mathcal{D}_{||,0}\lesssim \int_{0}^{t}\mathcal{E}^{\frac{1}{2}}\mathcal{D} 
    \end{align}

    \noindent By invoking the smallness assumption on $\mathcal{E}$ in equation \eqref{equ:e-d_0}, we then deduce the desired result of this theorem from equation \eqref{equ:e-d_0}.
    \end{proof}
    
 Now we have the zero order estimate, we then show the first order energy estimate:
 
          \begin{theorem}
           Suppose that $(\p_{t}u,\p_{t}p,\p_{t}\xi)$ is the solution to system \eqref{equ:1} such that $\mathcal{E}(u,p,\xi)\leq \delta\ll 1$ when $t\in (0,T)$ for some positive constant $T$. It satisfies the following energy-dissipation relation.

          \begin{align}
              \mathcal{E}_{||,1}(t)+\int_{0}^{t}\mathcal{D}_{||,1}\lesssim \mathcal{E}_{||,1}(0)+ \int_{0}^{t}\sqrt{\mathcal{E}}\mathcal{D}
          \end{align}

          \noindent for any $t\in (0,T]$.
      \end{theorem}

    \begin{proof}

    We first write down the weak form equation and then apply the test function $J\p_{t}u$. Using the similar computation as in Theorem \ref{thm:0-order}, we derive the following relation for the first-order equation system as follows:
    
    \begin{align}
        \frac{1}{2}\frac{d}{dt}(\vert \vert \partial_{t}u\vert \vert^{2}_{L^{2}}+\vert \vert \partial_{t}\xi\vert \vert_{1,\Sigma}^{2})+\int_{\Sigma_{s}} J\beta(\partial_{t}u\cdot \tau)^{2}+\frac{\mu}{2}\int_{\Omega}\mathbb{D}_{\mathcal{A}}\partial_{t}u:\mathbb{D}_{\mathcal{A}}\partial_{t}uJ+\kappa[\frac{1}{\rho_{0}}\partial_{t}u\cdot \mathcal{N}]^{2}_{\theta}\notag\\
        =\int_{\Omega} b^{1,1}\cdot \partial_{t}uJ+\frac{1}{2}\int_{\Omega}\partial_{t}J\vert \partial_{t}u\vert^{2}+\kappa[b^{7,1}+b^{6,1},\frac{1}{\rho_{0}}\partial_{t}u\cdot\mathcal{N}]_{\theta}+(\p_{t}\eta,b^{6,1})_{1,\Sigma}+a_{1}(t)\mathfrak{n}''(t)(\rho_{0},\xi_s)_{1,\Sigma}\notag\\
           -\int_{0}^{\pi}(b^{3,1}\partial_{\theta}(\frac{1}{\rho_{0}}\partial_{t}u\cdot \mathcal{N}))-\int_{0}^{\pi} b^{4,1}\cdot \partial_{t}u  d\theta -\int_{0}^{\pi}b^{5,1}u-\int_{0}^{\pi}b^{2,1}\partial_{t}p.
           \label{equ:5.1.115}
    \end{align}

    \noindent The function $a_{1}(t)$ appearing above is a function defined by equation \eqref{equ:a}. For completeness, we recall its expression here

    \begin{equation}{\label{equ:5.1.119}}
        a_{1}(t)=-\frac{\int_{0}^{\pi}\xi \partial_{t}\xi}{\int_{0}^{\pi}\rho_{0}^{2}d\theta}.
    \end{equation}
    
     From the definition of $\mathfrak{n}(t)$, we write its second order temporal derivative $\mathfrak{n}^{\prime\prime}(t)$ as follows
    
    \begin{align}
    \mathfrak{n}''(t)=-\frac{\partial_{t}\lambda}{\lambda^{2}}\int_{0}^{\pi}\frac{1}{\rho}u\cdot \mathcal{N}d\theta-\frac{1}{\lambda}\int_{0}^{\pi}\frac{\partial_{t}\xi}{\rho^{2}}u\cdot \mathcal{N}+\frac{1}{\lambda}\int_{0}^{\pi}\p_{t}u\cdot \mathcal{N}+\frac{1}{\lambda}\int_{0}^{\pi}u\cdot \p_{t}\mathcal{N}
    \end{align}

    We now estimate all of the nonlinear terms on the right hand side of the equation \eqref{equ:5.1.115}. 
    
    \textbf{Step 1} We first show the estimate for $\int_{\Omega} b^{1,1}\p_{t}u\cdot J$. Specifically, we aim to prove the following bound

    \begin{align}
        \int_{\Omega} b^{1,1}\cdot vJ\lesssim \mathcal{E}\vert \vert v\vert \vert_{H^{1}}
    \end{align}
    \noindent for any $v\in H^{1}$. Using the definition of $b^{1,1}$, we estimate each component individually.

    \textbf{Term $\int_{\Omega} \partial_{t}b^{1}\cdot vJ$}

    We have the following equation from the definition of $b^{1}$ given in Appendix.
    \begin{align}
        \int_{\Omega}(\partial_{t}b^{1})\cdot uJ&=\int_{\Omega} (\mathfrak{n}''(t)\partial_{x}u+\mathfrak{n}'(t)\partial_{x}\p_{t}u-\p_{t}u\cdot \nabla_{\mathcal{A}}u-u\cdot \nabla_{\p_{t}\mathcal{A}}u-u\cdot \nabla_\mathcal{A}\p_{t}u)\cdot vJ \notag\\
        &\quad+\int_{\Omega} (\cos\theta W \partial_{tt}\bar{\xi},\sin\theta W\partial_{tt}\bar{\xi})\tilde{K}(\partial_{x}u,\partial_{y}u)^{T}\cdot vJ\notag\\
        &\quad+\int_{\Omega}(\cos\theta \p_{t}W\p_{t}\bar{\xi},\sin\theta\p_{t}W\p_{t}\bar{\xi})\tilde{K}(\p_{x}u,\p_{y}u)^{T}\cdot vJ\notag\\
        &\quad+\int_{\Omega} (\cos\theta W\partial_{t}\bar{\xi},\sin\theta W\partial_{t}\bar{\xi})\partial_{t}\tilde{K}(\partial_{x}u,\partial_{y}u)^{T}\cdot vJ \notag\\
        &\quad+\int_{\Omega}(\cos\theta W\partial_{t}\bar{\xi},\sin\theta W\partial_{t}\bar{\xi})\tilde{K}(\partial_{x}\partial_{t}u,\partial_{y}\partial_{t}u)\cdot vJ=\Sigma_{i=1}^{9} I_{i} \label{equ:5.1.120}
    \end{align}
    
    \noindent We then estimate the estimate all of the terms on the right hand side of equation \eqref{equ:5.1.120}. Using Theorem \ref{thm:gam} to bound $\mathfrak{n}''(t)$, we have:

    \begin{equation}{\label{equ:5.1.121}}
        I_{1}=\int_{\Omega} \mathfrak{n}''(t)\partial_{x}u\cdot vJ \lesssim \mathfrak{n}''(t) \vert \vert u\vert \vert_{H^{1}}\vert \vert v\vert \vert_{H^{1}}\lesssim \sqrt{\mathcal{E}}\vert \vert u\vert \vert_{H^{1}}(\vert \vert v\vert \vert_{H^{1}})\lesssim {\mathcal{E}}\|v\|_{H^{1}}
    \end{equation}

    \noindent Then using \ref{thm:gam} again to bound $\mathfrak{n}'(t)$ in $I_{2}$, we obtain the following estimate for $I_{2}$.
    \begin{equation}{\label{equ:5.1.124}}
        I_{2}=\int_{\Omega}\mathfrak{n}'(t)\partial_{x}\p_{t}u \cdot vJ=\int_{\Omega} \mathfrak{n}'(t)\partial_{x}\p_{t}u\cdot vJ \lesssim \vert \vert u\vert \vert_{H^{1}}\vert\vert \p_{t}u\vert \vert_{H^{1}}\vert \vert v\vert \vert_{L^{2}}\lesssim \vert \vert u\vert \vert_{H^{1}}\|\p_{t}u\|_{H^{1}}\vert \vert v\vert \vert_{H^{1}}\lesssim \mathcal{E}\|v\|_{H^{1}}
    \end{equation}
    Then for the term $I_{3}$, we have the following estimate by H\"older's inquality
    \begin{align}
        I_{3}=-\int_{\Omega}\p_{t}u\cdot \nabla_{\mathcal{A}}u\cdot vJ\lesssim \vert\vert u\vert \vert_{W^{1,\frac{2}{2-\epsilon_{+}}}}\vert \vert v\vert \vert_{L^{\frac{2}{\epsilon_{+}}}}\|\p_{t}u\|_{L^{2}}\lesssim\|\p_{t}u\|_{L^{2}} \vert \vert u\vert \vert_{W^{2,q_{+}}}\|v\|_{H^{1}}\lesssim \mathcal{E}\|v\|_{H^{1}} \label{equ:5.1.125}
    \end{align}

    \noindent For the term $I_{4}$, we have the following estimate from H\"older's inequality

    \begin{align}
         I_{4}=-\int_{\Omega} u\cdot \nabla_{\p_{t}\mathcal{A}}u\cdot v\lesssim \|\p_{t}\xi\|_{W^{1,+\infty}}\vert \vert u\vert \vert_{L^{\infty}}\vert \vert \p_{t}u\vert \vert_{H^{1}} \vert \vert v\vert \vert_{L^{2}}\lesssim \mathcal{E}^{\frac{3}{2}}\|v\|_{H^{1}}
    \end{align}

    \noindent For the term $I_{5}$, similarly, we have
    
    \begin{align}
        I_{5}=-\int_{\Omega} u\cdot \nabla_{\mathcal{A}}\p_{t}u\cdot v\lesssim \vert \vert u\vert \vert_{L^{\infty}}\vert \vert \p_{t}u\vert \vert_{H^{1}} \vert \vert v\vert \vert_{L^{2}}\lesssim \mathcal{E}\|v\|_{L^{2}} \label{equ:5.1.126}
    \end{align}

    \noindent For the term $I_{6}$, we have
    
    \begin{align}
        I_{6}=&\int_{\Omega} (\cos\theta W\partial_{tt}\bar{\xi},\sin\theta W\partial_{tt}\bar{\xi})\tilde{K}(\partial_{x}u,\partial_{y}u)^{T} \cdot vJ\lesssim \vert \vert \partial^{2}_{t}\bar{\xi}\vert \vert_{L^{\infty}}\vert \vert u\vert \vert_{H^{1}}\vert \vert v\vert \vert_{L^{2}}\notag\\
        &\lesssim \vert \vert \partial_{t}^{2}\xi\vert \vert_{H^{1}}\vert \vert v\vert \vert_{L^{2}}\vert \vert u\vert \vert_{H^{1}}\lesssim \mathcal{E}\|v\|_{H^{1}}\label{equ:5.1.127}
    \end{align}

    \noindent For the term $I_{7}$, we have

    \begin{align}
        I_{7}=&\int_{\Omega} (\cos\theta \p_{t}W\partial_{t}\bar{\xi},\sin\theta \p_{t}W\partial_{t}\bar{\xi})\tilde{K}(\partial_{x}u,\partial_{y}u)^{T} \cdot vJ\lesssim \vert \vert \partial_{t}\bar{\xi}\vert \vert^{2}_{W^{1,\infty}}\vert \vert u\vert \vert_{H^{1}}\vert \vert v\vert \vert_{L^{2}}\notag\\
        &\lesssim \vert \vert \partial_{t}\xi\vert \vert^{2}_{H^{\frac{3}{2}+\frac{\epsilon_{-}-\alpha}{2}}}\vert \vert v\vert \vert_{L^{2}}\vert \vert u\vert \vert_{H^{1}}\lesssim \mathcal{E}\|v\|_{H^{1}}
    \end{align}

    \noindent For the term $I_{8}$, we have
    
    \begin{align}
        I_{8}=&\int_{\Omega} (\cos\theta W\partial_{t}\bar{\xi},\sin\theta W\partial_{t}\bar{\xi})\partial_{t}\tilde{K}(\partial_{x}u,\partial_{y}u)^{T}\cdot vJ \lesssim \vert \vert \partial_{t}\bar{\xi}\vert \vert_{L^{\infty}} (\vert \vert \partial_{t}\bar{\xi}\vert \vert_{W^{1,+\infty}})\vert \vert u\vert \vert_{W^{1,\frac{2}{2-\epsilon_{+}}}}\vert \vert v\vert \vert_{L^{\frac{2}{\epsilon_{+}}}} \notag\\
        &\lesssim \vert \vert \p_{t}\eta\vert \vert_{H^{1}} \vert \vert \p_{t}\eta\vert \vert_{H^{\frac{3}{2}+\frac{\epsilon_{-}-\alpha}{2}}}\vert \vert u\vert \vert_{W^{2,q_{+}}}\vert \vert v\vert \vert_{H^{1}}\lesssim \mathcal{E}^{\frac{3}{2}}\|v\|_{H^{1}} \label{equ:5.1.128}
    \end{align}
    \noindent For the term $I_{9}$, we have
    \begin{align}
        I_{9}=\int_{\Omega}(\cos\theta W\partial_{t}\bar{\xi},\sin\theta W\partial_{t}\bar{\xi})\tilde{K}(\partial_{x}\partial_{t}u,\partial_{y}\partial_{t}u)\cdot vJ \lesssim \vert \vert \partial_{t}\bar{\xi}\vert \vert_{L^{\infty}}\vert \vert \partial_{t}u\vert \vert_{H^{1}}\vert \vert v\vert \vert_{L^{2}}\lesssim\mathcal{E}\|v\|_{L^{2}}  \label{equ:5.1.129}
    \end{align}

    \noindent Combining all the estimates from equation \eqref{equ:5.1.121} to \eqref{equ:5.1.129}, we obtain the following result

    \begin{align}
        \int_{\Omega} (\partial_{t}b^{1})\cdot vJ\lesssim \mathcal{E}\|v\|_{H^{1}}
    \end{align}

    \textbf{Term $\operatorname{div}_{\partial_{t}\mathcal{A}}S_{\mathcal{A}}(p,u)$} We have the following estimate by H\"older's inequlity

    \begin{align}
        \int_{\Omega} \operatorname{div}_{\partial_{t}\mathcal{A}}S_{\mathcal{A}}(p,u)\cdot vJ&\lesssim \vert \vert \partial_{t}\bar{\xi}\vert \vert_{W^{1,+\infty}} (\vert \vert p\vert \vert_{W^{1,q_{+}}}+\vert \vert u\vert \vert_{W^{2,q_{+}}})\vert \vert v\vert \vert_{L^{\frac{4}{\epsilon_{+}}}}\notag\\
        &\lesssim \mathcal{E} \vert \vert v\vert \vert_{H^{1}} \label{equ:5.1.133}
    \end{align}

    \textbf{Term $\mu\operatorname{div}_{\mathcal{A}}\mathbb{D}_{\partial_{t}\mathcal{A}}u$}. We have the following estimate by H\"older's inequality
    
    \begin{align}
        \int_{\Omega} \mu\operatorname{div}_{\mathcal{A}}\mathbb{D}_{\partial_{t}\mathcal{A}}u \cdot vJ \lesssim \vert \vert \partial_{t}\bar{\xi}\vert \vert_{W^{1,\infty}}\vert \vert u\vert \vert_{W^{2,q_{+}}}\vert \vert v\vert \vert_{L^{\frac{2}{\epsilon_{+}}}}+\vert \vert u\vert \vert_{W^{1,\frac{2}{1-\epsilon_{+}}}}\vert \vert \partial_{t}\bar{\xi}\vert \vert_{H^{2}}\vert \vert v\vert \vert_{L^{\frac{2}{\epsilon_{+}}}}\lesssim \mathcal{E}\|v\|_{H^{1}}
    \end{align}

   \noindent where we used Sobolev Embedding $L^{\frac{2}{\epsilon_{+}}}\hookrightarrow H^{1}$. Hence, we finished the estimate of the first term on the right hand side of the equation \eqref{equ:5.1.115}. 

    \textbf{Step 2} We now estimate the second term $\int_{\Omega}\partial_{t}J\vert \p_{t}u\vert^{2}$. We have the following estimate by the definition of $J$

    \begin{align}
        \int_{\Omega} \partial_{t}J\vert \p_{t}u\vert^{2}\lesssim \vert \vert \partial_{t}\bar{\xi}\vert \vert_{W^{1,\infty}}\vert \vert \p_{t}u\vert \vert^{2}_{L^{2}}\lesssim \mathcal{E}^{\frac{3}{2}} \label{equ:5.1.134}
    \end{align}

   \textbf{Step 3} We estimate the third term on the right hand side of the equation \eqref{equ:5.1.115}. We have:

   \begin{align}
       \kappa[b^{7,1}+b^{6,1},\frac{1}{\rho_{0}}v\cdot \mathcal{N}]_{\theta}\lesssim& \kappa[b^{7,1}]_{\theta}[\frac{1}{\rho_{0}}v\cdot \mathcal{N}]_{\theta}+[b^{6,1}]_{\theta}[\frac{1}{\rho_{0}}v\cdot \mathcal{N}]_{\theta}
   \end{align}

   \noindent We first estimate $[b^{6,1}]_{\theta}$ by its definition. It holds that

   \begin{align}
       [\partial_{t}b^{6}]_{\theta}\lesssim& [\partial_{t}\xi]_{\theta}[u\cdot \mathcal{N}]_{\theta}+[\xi]_{\theta}[\p_{t}u\cdot \mathcal{N}]_{\theta}+[\xi]_{\theta}[u\cdot \p_{t}\mathcal{N}]_{\theta}+\vert \mathfrak{n}''(t)\vert([\partial_{\theta}\xi]_{\theta}+[\xi]_{\theta})+\vert \mathfrak{n}'(t)\vert([\partial_{t}\partial_{\theta }\xi]_{\theta}+[\partial_{\theta}\xi]_{\theta}[\partial_{t}\xi]_{\theta})\notag\\
       \lesssim & \vert \vert \partial_{t}\xi\vert \vert_{H^{1}}\vert \vert u\vert\vert_{W^{2,q_{+}}}+\vert \vert \xi\vert \vert_{H^{1}}\vert \vert \p_{t}u\vert \vert_{H^{1+\frac{\epsilon_{-}}{2}}}+\vert \vert \xi\vert \vert_{H^{1}}\vert \vert u\vert \vert_{H^{1+\frac{\epsilon_{-}}{2}}}\|\p_{t}\xi\|_{H^{\frac{3}{2}+\frac{\epsilon_{-}-\alpha}{2}}}+\vert \vert \xi\vert \vert_{W^{2,1}}\sqrt{\mathcal{E}}\notag+\vert \vert \p_{t}\xi\vert \vert_{H^{\frac{3}{2}+\frac{\epsilon_{-}-\alpha}{2}}}\sqrt{\mathcal{E}}\notag\\
       \lesssim& \mathcal{E}+(\vert\vert \xi\vert \vert_{W^{3-\frac{1}{q_{+}},q_{+}}}+\|\p_{t}\xi\|_{H^{\frac{3}{2}+\frac{\epsilon_{-}-\alpha}{2}}})\mathcal{E}\label{equ:5.2.134}
  \end{align}
   
   \noindent Then for the other terms in $b^{6,1}$, using the similar computation as above, we have

   \begin{align}
       [\frac{1}{\rho_{0}}u\cdot \p_{t}\mathcal{N},\frac{1}{\rho_{0}}v\cdot \mathcal{N}]_{\theta}\lesssim \|\p_{t}\xi\|_{H^{\frac{3}{2}+\frac{\epsilon_{-}-\alpha}{2}}}\|u\|_{W^{2,q_{+}}}[v\cdot \mathcal{N}]_{\theta }\lesssim \mathcal{E}[v\cdot \mathcal{N}]_{\theta}
   \end{align}

   \noindent Therefore, we obtain the following estimate

   \begin{align}
       [b^{6,1}]_{\theta}\lesssim \mathcal{E}\label{equ:b_6}
   \end{align}

    We now estimate the terms including $b^{7,1}$, we have the following computation by its definition

   \begin{align}
       [\hat{W}'(\partial_{t}\xi+\mathfrak{n}'(t))(\partial_{tt}\xi+\mathfrak{n}''(t))]_{\theta}&\lesssim [(\vert \partial_{t}\xi\vert+\vert \mathfrak{n}'(t)\vert)(\vert \partial_{tt}\xi\vert+\vert \mathfrak{n}''(t)\vert)]_{\theta}\notag\\
       &\lesssim (\vert \vert \partial_{t}\xi\vert \vert_{W^{1,1}}+\vert \vert u\vert \vert_{H^{1}})(\vert \vert \partial_{tt}\xi\vert \vert_{W^{1,1}}+\|\p_{t}u\|_{H^{1}}+\mathcal{E}) \notag\\
       &\lesssim (\vert \vert \partial_{t}\xi\vert \vert_{H^{1}}+\vert \vert u\vert \vert_{H^{1}})(\vert \vert \partial_{tt}\xi\vert \vert_{H^{1}}+\vert \vert \p_{t}u\vert\vert_{H^{1}}+\mathcal{E})\lesssim \mathcal{E} \label{equ:5.1.136}
   \end{align}

   \noindent Therefore,by the definition of $b^{7,1}$, we obtain the following result
   
   \begin{align}
       \kappa [b^{7,1},\frac{1}{\rho_{0}}v\cdot \mathcal{N}]_{\theta}\lesssim \mathcal{E}[v\cdot \mathcal{N}]_{\theta} \label{equ:5.1.137}
   \end{align}

   Combining equation \eqref{equ:b_6} and \eqref{equ:5.1.137}, we can finally conclude that:

   \begin{align}
       \kappa[b^{7,1}+b^{6,1},\frac{1}{\rho_{0}}v\cdot \mathcal{N}]_{\theta}\lesssim \mathcal{E}[v\cdot \mathcal{N}]_{\theta}
   \end{align}
   
   \textbf{Step 4} We then estimate the fourth term on the right-hand side of the equation \eqref{equ:5.1.115}.  Decomposing $\p_{t}b^{6}$, we have the following computation:

   \begin{align}
       (\partial_{t}b^{6},\partial_{t}\xi)_{1,\Sigma}=&((-\frac{\partial_{t}\xi}{\rho\rho_{0}}+\frac{(\partial_{t}\xi)^{2}}{\rho^{2}\rho_{0}})u\cdot \mathcal{N}+\partial_{t}(\mathfrak{n}'(t)(\frac{\rho'}{\rho}-\frac{\rho_{0}'}{\rho_{0}}))\cos\theta,\partial_{t}\xi)_{1,\Sigma}+(\frac{1}{\rho}-\frac{1}{\rho_{0}})(\p_{t}u\cdot \mathcal{N}+u\cdot \p_{t}\mathcal{N}),\partial_{t}\xi)_{1,\Sigma}\notag\\
       &\lesssim  \vert \vert \partial_{t}\xi\vert \vert^{2}_{W^{1,4}}\vert \vert u\cdot \mathcal{N}\vert \vert_{L^{2}}+\vert \vert \partial_{t}\xi\vert \vert^{3}_{W^{1,6}}\vert \vert u\cdot \mathcal{N}\vert \vert_{L^{2}}+\vert \vert \partial_{t}\xi\vert \vert_{L^{\infty}}\vert \vert u\cdot \mathcal{N}\vert \vert_{W^{1,1}(\Sigma_{s})}\vert \vert \partial_{t}\xi\vert \vert_{H^{1}}\notag\\
       &\quad+\vert \vert \partial_{t}\xi\vert \vert_{L^{\infty}}^{2}\vert \vert u\cdot \mathcal{N}\vert \vert_{W^{1,1}(\Sigma_{s})}\vert \vert \partial_{t}\xi\vert \vert_{W^{1,+\infty}}+\vert \vert v\vert \vert_{H^{1}}\vert \vert \xi\vert \vert_{W^{2,\frac{1}{1-\epsilon_{+}}}}\vert \vert \partial_{t}\xi\vert \vert_{W^{1,\frac{1}{\epsilon_{+}}}}+\vert \vert u\vert \vert\vert_{H^{1}} \vert \vert\partial_{t}\xi\vert \vert^{2}_{H^{\frac{3}{2}}}\notag\\
       &\quad+\|\xi\|_{W^{1,\infty}}(\|\p_{t}u\|_{W^{1,\frac{1}{1-\epsilon_{-}}}(\Sigma)}+\|\p_{t}u\|_{L^{\infty}(\Sigma)}\|\xi\|_{W^{2,\frac{1}{1-\epsilon_{-}}}})\|\p_{t}\xi\|_{W^{1,\frac{1}{1-\epsilon_{-}}}}\notag\\&\quad+\|\xi\|_{W^{1,\infty}}(\|u\|_{W^{1,\frac{1}{1-\epsilon_{-}}}(\Sigma)}\|\p_{t}\xi\|_{W^{1,+\infty}}+\|u\|_{L^{\infty}(\Sigma)}\|\p_{t}\xi\|_{W^{2,\frac{1}{1-\epsilon_{-}}}})\|\p_{t}\xi\|_{W^{1,\frac{1}{\epsilon_{-}}}} \notag\\
       &\lesssim \vert \vert \partial_{t}\xi\vert \vert^{2}_{H^{\frac{3}{2}}}\vert \vert u\vert \vert_{H^{1}}+\vert \vert \partial_{t}\xi\vert \vert_{H^{\frac{3}{2}}}^{3}\vert \vert u\vert \vert_{H^{1}}+\vert \vert \partial_{t}\xi\vert \vert^{2}_{H^{\frac{3}{2}+\frac{\epsilon_{-}-\alpha}{2}}}\vert \vert u\vert \vert_{W^{2,q_{+}}}\notag\\
       &\quad+\vert \vert \partial_{t}\xi\vert \vert^{3}_{H^{\frac{3}{2}
       +\frac{\epsilon_{-}-\alpha}{2}}}\vert \vert u\vert \vert_{W^{2,q_{+}}}+\vert \vert v\vert \vert_{H^{1}}\vert \vert \xi\vert \vert_{W^{3-\frac{1}{q_{+}},q_{+}}}\vert \vert \partial_{t}\eta\vert \vert_{H^{\frac{3}{2}}+\frac{\epsilon_{-}-\alpha}{2}}\notag\\
       &\quad+\|\p_{t}\xi\|_{W^{3-\frac{1}{q_{-}},q_{-}}}^{2}\|u\|_{W^{2,q_{+}}}+\|\p_{t}\xi\|_{W^{3-\frac{1}{q_{-}},q_{-}}}\|\xi\|_{W^{3-\frac{1}{q_{+}},q_{+}}}\|\p_{t}u\|_{W^{2,q_{-}}}\notag\\
       &\lesssim \mathcal{E}^{\frac{1}{2}}\mathcal{D}\label{equ:5.1.138}
   \end{align}

    \noindent Then for the remaining terms included in $b^{6,1}$, we have the following estimate

   \begin{align}
       &(\frac{1}{\rho_{0}}u\cdot \p_{t}\mathcal{N},\p_{t}\xi)_{1,\Sigma}\notag\\
       &\lesssim \vert \vert \partial_{t}\xi\vert \vert_{W^{2,\frac{1}{1-\epsilon_{-}}}}\vert \vert u\vert \vert_{L^{\infty}(\Sigma)}\vert \vert\partial_{t}\xi \vert \vert_{W^{1,\frac{1}{\epsilon_{-}}}}+\vert \vert u\vert \vert_{W^{1,\frac{1}{1-\epsilon_{-}}}(\Sigma)}\vert \vert \partial_{t}\eta\vert \vert^{2}_{W^{1,+\infty}}\notag\\
       &\lesssim \vert \vert \partial_{t}\xi\vert \vert_{W^{3-\frac{1}{q_{+}},q_{+}}}\vert \vert u\vert \vert_{W^{2,q_{+}}}\vert \vert\partial_{t}\xi \vert \vert_{W^{1,\frac{1}{\epsilon_{-}}}}+\vert \vert u\vert \vert_{W^{2,q_{+}}}\vert \vert \partial_{t}\xi\vert \vert^{2}_{W^{1,+\infty}}\lesssim \sqrt{\mathcal{E}}\mathcal{D}
   \end{align}
   
   \textbf{Step 5} We estimate the fifth term as follows by Theorem \ref{thm:gam} and Theorem \ref{thm:pos}. It holds that

   \begin{align}
       a_{1}\mathfrak{n}''(t)(\rho_{0},\xi_s)\lesssim a_1\mathfrak{n}''(t)\lesssim |\mathfrak{n}^{\prime\prime}(t)|\vert \vert \xi\vert \vert_{L^{2}}\vert \vert \partial_{t}\xi\vert \vert_{L^{2}}\lesssim \mathcal{E}^{\frac{3}{2}} \label{equ:5.1.139}
   \end{align}

   \textbf{Step 6} We estimate the sixth term by the definition of $b^{3,1}$ as follows

   \begin{align}
       \int_{0}^{\pi}b^{3,1}\partial_{\theta}(\p_{t}u\cdot \mathcal{N})&\lesssim \int_{\Omega}\vert \partial_{\theta}\xi\vert\vert \partial_{t}\partial_{\theta}\xi\vert \vert \partial_{\theta}(\p_{t}u\cdot \mathcal{N})\vert +\int_{\Omega}\vert \xi\vert\vert \partial_{t}\xi\vert\vert \partial_{\theta}(\p_{t}u\cdot \mathcal{N})\vert\notag\\
       &\lesssim \vert \vert \xi\vert \vert_{W^{1,+\infty}}\vert \vert \partial_{t}\xi\vert \vert_{W^{1,+\infty}}\vert \vert \p_{t}u\cdot \mathcal{N}\vert \vert_{W^{1,1}(\Sigma)} \label{equ:5.1.140}
   \end{align}

   \noindent By trace theorem, we have

   \begin{align}
       \vert \vert \p_{t}u\cdot \mathcal{N}\vert \vert_{W^{1,1}(\Sigma)}&\lesssim \vert \vert \p_{t}u\vert \vert_{W^{1,\frac{1}{1-\epsilon_{-}}}}\lesssim \|\p_{t}u\|_{W^{2,q_{-}}} \label{equ:5.1.141}
   \end{align}

 Then applying equation \eqref{equ:5.1.141} to equation \eqref{equ:5.1.140}, we have

   \begin{equation}{\label{equ:5.1.142}}
        \int_{0}^{\pi}b^{3,1}\partial_{\theta}(\p_{t}u\cdot \mathcal{N})\lesssim \|\xi\|_{W^{3-\frac{1}{q_{-}},q_{-}}}\|\p_{t}\xi\|_{W^{3-\frac{1}{q_{-}},q_{-}}}\|\p_{t}u\|_{W^{2,q_{-}}}\lesssim\mathcal{E}^{\frac{1}{2}}\mathcal{D}
   \end{equation}

   \textbf{Step 7} In this step, we estimate the seventh term in equation \eqref{equ:5.1.115}. We aim to show the following two results

   \begin{align}
       \int_{0}^{\pi} b^{4,1}\cdot v d\theta\lesssim \mathcal{E}\|v\|_{H^{1}},
   \end{align}
   and
   \begin{align}
       \int_{0}^{\pi} b^{5,1}\cdot v d\theta\lesssim \mathcal{E}\|v\|_{H^{1}}.
   \end{align}
   \noindent for any $v\in H^{1}$. This result will be achieved by estimating each component of $b^{4,1}$ individually.

   \textbf{Term $\sigma \partial_{t}\mathcal{R}_{2}$}:

   The definition of $\mathcal{R}_{2}$ implies the following computation
   \begin{align}
       \int_{0}^{\pi}\partial_{t}\mathcal{R}_{2}\cdot v d\theta&\lesssim \int_{0}^{\pi} \vert \partial_{t}{\xi}'\vert(\vert {\xi}'\vert+|\xi|)\vert v\vert d\theta+\int_{0}^{\pi} \vert \partial_{t}\xi\vert(\vert \xi\vert+|\p_{\theta}\xi|)\vert v\vert d\theta\notag\\
       &\lesssim \vert \vert \partial_{t}\xi\vert \vert_{W^{1,4}}\vert \vert \xi\vert \vert_{W^{1,4}}\vert \vert v\vert \vert_{L^{2}(\Sigma_{s})}+ \vert \vert \partial_{t}\xi\vert \vert_{L^{4}}\vert \vert \xi\vert \vert_{W^{1,4}}\vert \vert v\vert \vert_{L^{2}(\Sigma_{s})}\notag\\
       &\lesssim \vert \vert \partial_{t}\xi\vert \vert_{H^{\frac{3}{2}-\frac{\epsilon_{-}-\alpha}{2}}}\vert \vert \xi\vert \vert_{W^{3-\frac{1}{q_{+}},q_{+}}}\vert \vert v\vert \vert_{H^{1}}\lesssim \mathcal{E}\|v\|_{H^{1}} \label{equ:5.1.143}
   \end{align}
   
   \textbf{Term $\mu \mathbb{D}_{\partial_{t}\mathcal{A}}u\mathcal{N}$}:

   We have the following estimate by trace theorem and H\"older's inequality 
   \begin{align}
       \int_{0}^{\pi} \mu \mathbb{D}_{\partial_{t}\mathcal{A}}u\mathcal{N}\cdot v d\theta\lesssim \vert \vert \partial_{t}\partial_{\theta}{\xi}\vert\vert_{L^{\infty}}\vert \vert u\vert \vert_{W^{1,\frac{1}{1-\epsilon_{+}}}(\Sigma)}\vert \vert v\vert \vert_{L^{\frac{1}{\epsilon_{+}}}(\Sigma )}\lesssim \vert \vert v\vert \vert_{H^{1}}\vert \vert u\vert \vert_{W^{2,q_{+}}}\vert \vert \partial_{t}\xi\vert \vert_{\frac{3}{2}+\frac{\epsilon_{-}-\alpha}{2}}\lesssim \mathcal{E}\|v\|_{H^{1}}\label{equ:5.1.144}
   \end{align}

   \textbf{Term $\mathcal{K}(\xi-\partial_{\theta}b^{3})\partial_{t}\mathcal{N}$} We have the following computation by trace theorem and H\"older's inequality 

   \begin{align}
       \int_{0}^{\pi}(\mathcal{K}(\xi)-\partial_{\theta}b^{3})\partial_{t}\mathcal{N} \cdot v\lesssim \vert \vert \xi\vert \vert_{W^{2,\frac{1}{1-\epsilon_{+}}}}\vert \vert \partial_{t}\xi\vert \vert_{W^{1,\infty}}\vert \vert v\vert \vert_{L^{\frac{1}{\epsilon_{+}}}(\Sigma)}+\int_{0}^{\pi} \vert\partial_{t} \partial_{\theta}\xi\vert^{2}\vert \partial_{\theta}^{2}\xi\vert \vert v\vert d\theta+\int_{0}^{\pi} \vert \partial_{t}\partial_{\theta}\xi\vert \vert \partial_{\theta}\xi\vert\vert \partial_{\theta}^{2}\xi\vert \vert v\vert \notag\\
       \lesssim \vert \vert \xi\vert \vert_{W^{3-\frac{1}{q_{+}},q_{+}}}\vert \vert \partial_{t}\xi\vert \vert_{H^{\frac{3}{2}+\frac{\epsilon_{-}-\alpha}{2}}}\vert \vert v \vert \vert_{H^{1}}+\vert\vert \partial_{t}\xi\vert \vert_{W^{1,+\infty}}^{2}\vert \vert \xi\vert \vert_{W^{2,\frac{1}{1-\epsilon_{+}}}}\vert \vert v\vert \vert_{L^{\frac{1}{\epsilon_{+}}}}+\vert \vert \partial_{t}\xi\vert \vert_{W^{1,\infty}}\vert \vert \xi\vert \vert_{W^{1,\infty}}\vert \vert \xi\vert \vert_{W^{2,\frac{1}{1-\epsilon_{-}}}}\vert \vert v\vert \vert_{L^{\frac{1}{\epsilon_{-}}}} \notag\\
       \lesssim \mathcal{E}^{\frac{3}{2}}+\vert \vert \partial_{t}\xi\vert \vert^{2}_{H^{\frac{3}{2}+\frac{\epsilon_{-}-\alpha}{2}}}\vert \vert \xi\vert \vert_{W^{3-\frac{1}{q_{+}},q_{+}}}\vert \vert v\vert \vert_{H^{1}}+\vert \vert \partial_{t}\xi\vert \vert_{H^{\frac{3}{2}+\frac{\epsilon_{-}-\alpha}{2}}}\vert \vert\xi \vert \vert_{W^{3-\frac{1}{q_{-}},q_{-}}}\vert \vert v\vert \vert_{H^{1}}\vert \vert \xi\vert \vert_{W^{3-\frac{1}{q_{+}},q_{+}}}\lesssim \mathcal{E}^{\frac{3}{2}}+\mathcal{E}^{2}\label{equ:5.1.145}
   \end{align}

   The estimate for $b^{5,1}$ follows from the similar estimate.
   
   \textbf{Term $S_{\mathcal{A}}(p,u)\partial_{t}\mathcal{N}$} We have the following estimate by trace theorem and H\"older's inequality 

   \begin{align}
       \int_{0}^{\pi} S_\mathcal{A}(p,u)\partial_{t}\mathcal{N}\cdot v d\theta\lesssim \vert \vert p\vert \vert_{L^{\frac{1}{1-\epsilon_{+}}}(\Sigma)}\vert \vert \partial_{t}\xi\vert \vert_{W^{1,+\infty}}\vert \vert v\vert \vert_{L^{\frac{1}{\epsilon_{+}}}(\Sigma)}+ \vert \vert u\vert \vert_{W^{1,\frac{1}{1-\epsilon_{+}}}(\Sigma)}\vert \vert \partial_{t}\xi\vert \vert_{W^{1,\infty}}\vert \vert v\vert\vert_{L^{\frac{1}{\epsilon_{+}}}} \notag\\
       \lesssim \vert \vert p\vert \vert_{W^{1,q_{+}}}\vert \vert \partial_{t}\xi\vert\vert_{W^{3-\frac{1}{q_{-}},q_{-}}}\vert \vert v\vert \vert_{H^{1}}+\vert \vert u\vert \vert_{W^{2,q_{+}}}\vert \vert \partial_{t}\xi\vert \vert_{H^{\frac{3}{2}+\frac{\epsilon_{-}-\alpha}{2}}}\vert \vert v\vert \vert_{H^{1}}\lesssim \mathcal{E}\|v\|_{H^{1}}
   \end{align}

   \textbf{Step 8} In this step, we establish the estimate for the eighth term on the right hand side of equation \eqref{equ:5.1.115} as follows

   \begin{align}
       \int_{\Sigma_{s}} \mu \mathbb{D}_{\partial_{t}\mathcal{A}}u\nu\cdot \tau v\lesssim \vert \vert \xi\vert \vert_{W^{1,+\infty}}\vert \vert u\vert \vert_{W^{1,\frac{1}{1-\epsilon_{+}}}(\Sigma)} \vert\vert v\vert \vert_{L^{\frac{1}{\epsilon_{+}}}(\Sigma)}\lesssim \|\xi\|_{W^{3-
       \frac{1}{q_{+}},q_{+}}}\|u\|_{W^{3-\frac{1}{q_{+}},q_{+}}} \|v\|_{H^{1}}\lesssim \mathcal{E}\|v\|_{H^{1}}
   \end{align}

   \textbf{Step 9} In this step, we estimate the estimate for the pressure term on the right hand side of \eqref{equ:5.1.115}

   \noindent Using the definition of $b^{2,1}$, we have:

   \begin{align}
      \int_{\Omega} \operatorname{div}_{\partial_{t}\mathcal{A}}u \partial_{t}p\lesssim \vert \vert \partial_{t}p\vert \vert_{L^{2}}\vert \vert \partial_{t}\xi\vert \vert_{W^{1,+\infty}}\vert \vert u\vert \vert_{H^{1}}\lesssim \mathcal{E}^{\frac{3}{2}}
   \end{align}

   \textbf{Step 10} 

   In this step, we show the estimate for $(\p_{t}\xi,\p_{t}\xi)_{1,\Sigma}$. Using Theorem \ref{thm:pos} and the definition of $a_{1}(t)$ in \eqref{equ:a}, we have:

   \begin{align}
       (\p_{t}\xi,\p_{t}\xi)_{1,\Sigma}&=(\p_{t}\xi-a_{1}(t)\rho_{0},\p_{t}\xi-a_{1}(t)\rho_{0})_{1,\Sigma}-2(a_{1}(t)\rho_{0},\p_{t}\xi)_{1,\Sigma}+a_{1}^{2}(t)(\rho_{0},\rho_{0})_{1,\Sigma}\notag\\
       &\gtrsim \|\p_{t}\xi\|^{2}_{H^{1}}-\|\p_{t}\xi\|_{L^{2}}^{2}\|\xi\|^{2}_{L^{2}}-\|\p_{t}\xi\|^{2}_{H^{1}}\|\xi\|_{L^{2}}\gtrsim \mathcal{E}_{||,1}-\mathcal{E}_{||,1}^{\frac{3}{2}}
   \end{align}
   
  Combining all of the computation in Step 1 to Step 10, and integrating both sides of the equation \eqref{equ:5.1.115} from $0$ to $t$, we obtain the following result

  \begin{align}
      \mathcal{E}_{||,1}(t)+\int_{0}^{t}\mathcal{D}_{||,1}-\mathcal{E}_{||,1}^{\frac{3}{2}}(t)-\mathcal{E}^{\frac{3}{2}}(0)\lesssim \mathcal{E}_{||,1}(0)+\int_{0}^{t}\mathcal{E}^{\frac{1}{2}}\mathcal{D}
  \end{align}
  \noindent which implies the final result of this theorem. 
   \end{proof}

   \textbf{Remark}: Step 1 and Step 7 in Theorem 5.2 yields that:

   \begin{align}
       \int_{\Omega}b^{1,1}\cdot \omega J-\int_{\Sigma_{s}}J(\omega\cdot \tau)b^{5,1}+\int_{0}^{\pi}b^{4,1}\cdot \omega\lesssim \mathcal{E}\vert \vert \omega\vert \vert_{H^{1}}
   \end{align}

   \begin{lemma}{\label{thm:B}}
       Suppose that $0<T\leq+\infty$ and that $\sup_{0\leq t< T}\mathcal{E}(t)\leq \delta\ll1$. Let $b^{2,2}$ be as in the equation \eqref{equ:5.1.92}, i.e.

       \begin{align}
           \langle g\rangle_{\Omega}=\frac{1}{\vert \Omega\vert}\int_{\Omega}g
       \end{align}

       \noindent Then there exists $\omega:\Omega\times [0,T)\rightarrow \mathbb{R}^{2}$ satisfying the following properties

       (1) $\omega(\cdot,t)\in H_{0}^{1}(\Omega;\mathbb{R}^{2})$ for $0\leq t<T$ and

       \begin{align}
           J\operatorname{div}_{\mathcal{A}}\omega=Jb^{2,2}-\langle Jb^{2,2}\rangle 
       \end{align}

       (2)$\omega$ obeys the estimates:

       \begin{align}
           \vert \vert \omega\vert \vert_{W_{0}^{1,\frac{4}{3-2\epsilon_{+}}}}\lesssim\mathcal{E},
       \end{align}

       \noindent and

       \begin{align}
           \vert \vert \omega\vert \vert_{W_{0}^{1,\frac{2}{1-\epsilon_{-}}}}+\vert \vert \partial_{t}\omega\vert\vert_{L^{\frac{2}{1-\epsilon_{-}}}}\lesssim (\sqrt{\mathcal{E}}+\mathcal{E})\sqrt{\mathcal{D}}.
       \end{align}

       (3) The following interaction estimates hold

       \begin{align}
           \vert \int_{\Omega}\partial_{t}^{2}uJ\omega\vert\lesssim \mathcal{E}^{\frac{3}{2}}
       \end{align}

       \noindent Moreover, we have

       \begin{align}
           \vert \int_{\Omega}\partial_{t}^{2}uJ\omega\vert \lesssim \mathcal{E}^{\frac{3}{2}},
       \end{align}

       \noindent and

       \begin{align}{\label{equ:B_1}}
           \vert \int_{\Omega}\partial_{t}^{2}u\partial_{t}(J\omega)\vert+\vert (b^{1,2},\omega)_{0}\vert+\vert ((\partial_{t}^{2}u,\omega))\vert\lesssim (\sqrt{\mathcal{E}}+\mathcal{E})\mathcal{D}
       \end{align}
   \end{lemma}

   \begin{proof}
   Here we use the same construction and proof as in Proposition 10.1 in \cite{Guo}. The only new contribution we need to consider is the term $\vert\int_{\Omega} Jb^{1,2}\omega\vert$ in equation \eqref{equ:B_1}. Compared to $F^{1,2}$ in \cite{Guo}, the expression for $b^{1,2}$ includes an additional term $\partial_{t}^{2}(\mathfrak{n}'(t)\partial_{x}u)$. We compute

   \begin{align}
       \partial_{t}^{2}(\mathfrak{n}'(t)\partial_{x}u)=\mathfrak{n}'''(t)\partial_{x}u+2\mathfrak{n}''(t)\partial_{x}\partial_{t}u+\mathfrak{n}'(t)\partial_{x}\partial_{t}^{2}u
   \end{align}

   \noindent which implies that

   \begin{align}
       \vert \int_{\Omega}Jb^{1,2}\cdot \omega\vert\lesssim \vert \mathfrak{n}'''(t)\vert\vert \vert u\vert \vert_{H^{1}}\vert \vert \omega\vert \vert_{L^{2}}+\vert \partial_{t}^{2}\mathfrak{n}(t)\vert\vert \vert \partial_{t}u\vert \vert_{H^{1}}\vert \vert \omega\vert\vert_{L^{2}}+\vert \mathfrak{n}'(t)\vert\vert \partial_{t}^{2}u\vert \vert_{H^{1}} \vert\vert \omega\vert \vert_{H^{1}}\lesssim \mathcal{E}\mathcal{D}
   \end{align}

   \noindent This finishes the proof.
   \end{proof}

   Now we have the new constructed function $\omega$. We then use this function to establish the following Theorem establishing the second-order energy estimate
   
   \begin{theorem}{\label{thm:2-order}}
        Suppose that $(\p_{t}^{2}u,\p_{t}^{2}p,\p_{t}^{2}\xi)$ is the solution to system \eqref{equ:1} such that $\mathcal{E}(u,p,\xi)\leq \delta\ll 1$ when $t\in (0,T)$ for some positive constant $T$. It satisfies the following energy-dissipation relation. 

          \begin{align}
              \mathcal{E}_{||,2}(t)+\int_{0}^{t}\mathcal{D}_{||,2}\lesssim \mathcal{E}_{||,2}(0)+ \int_{0}^{t}\sqrt{\mathcal{E}}\mathcal{D}
          \end{align}

          \noindent for any $0<t\leq T$
   \end{theorem}

   \begin{proof}

        Testing the second order equation system by function $(\p_{t}^{2}u-\omega)$ and using integration by part as in Theorem \ref{thm:0-order}, we obtain the following equation:

       \begin{align}
           &\frac{d}{dt}(\vert \vert \p_{t}^{2}u\vert \vert^{2}_{L^{2}}+\vert \vert \partial_{tt}\xi\vert \vert^{2}_{1,\Sigma}-(\partial_{t}^{2}u,\omega)_{0})+\int_{\Sigma}J\beta (\p_{t}^{2}u\cdot \tau)^{2}+\frac{\mu}{2}\int_{\Omega} \mathbb{D}_{\mathcal{A}}\partial_{t}^{2}u:\mathbb{D}_{\mathcal{A}}\partial_{t}^{2}uJ+\kappa[\frac{1}{\rho_{0}}\p^{2}_{t}u\cdot \mathcal{N}]^{2}_{\theta}\notag\\
           &=\int_{\Omega} b^{1,2}\cdot \p_{t}^{2}uJ+\frac{1}{2}\int_{\Omega}\partial_{t}J\vert \p^{2}_{t} u\vert^{2}+\kappa[b^{7,2}+b^{6,2},\frac{1}{\rho_{0}}v\cdot \mathcal{N}]_{\theta}+(\p_{t}^{2}\xi,b^{6,2})_{1,\Sigma}+a_{2}\mathfrak{n}'''(t)(\rho_{0},\xi_s)_{1,\Sigma}\notag\\
           &\quad-\int_{0}^{\pi}(b^{3,2}\partial_{\theta}(\p_{t}^{2}u\cdot \mathcal{N}))-\int_{0}^{\pi}b^{4,2}\cdot \p^{2}_{t}ud\theta-\int_{\Sigma_{s}}b^{5,2}\p_{t}^{2}u
           + \int_{\Omega}\partial_{t}^{2}u\partial_{t}(J\omega)+ (b^{1,2},\omega)_{0}+ ((\partial_{t}^{2}u,\omega))\notag\\
           &\quad-\int_{\Omega}\p_{t}^{2}p\langle Jb^{2,2}\rangle_{\Omega}
 \end{align}

We estimate the terms on the right-hand side individually. We begin by discussing the first term involving $b^{1,2}$.
As shown in the Appendix, the expression for $b^{1,2}$ is quite complicated and contains many components.
We divide the estimation into Steps 1–3. Our goal is to establish the following bound:

 \begin{align}
        \int_{\Omega}b^{1,2}v\lesssim \mathcal{E}^{\frac{1}{2}}\mathcal{D}^{\frac{1}{2}}\|v\|_{H^{1}}
    \end{align}
    \noindent for any $v\in H^{1}$.
 
    \textbf{Step 1} $\int_{\Omega} \partial_{tt}b^{1}v$
    
    \textbf{Term $\partial_{tt}(\mathfrak{n}'(t)\partial_{x}u)$} We have the following estimate by Sobolev embedding and H\"older's inequality
    
    \begin{align}
        \int_{\Omega}\partial_{tt}(\mathfrak{n}'(t)\partial_{x}u)\cdot vJ&=\int_{\Omega} (2\mathfrak{n}''(t)\partial_{x}\partial_{t}u+\mathfrak{n}'(t)\partial_{x}\p_{t}^{2}u+\mathfrak{n}'''(t) \partial_{x}u)\cdot vJ \notag\\
       & \lesssim \mathcal{E}^{\frac{1}{2}}(\vert \vert \partial_{t}u\vert \vert_{H^{1}}\vert \vert v\vert \vert_{L^{2}}+\vert \vert v\vert \vert_{L^{2}}\vert \vert \p_{t}^{2}u\vert \vert_{H^{1}})+ \vert \vert u\vert \vert_{W^{1,\frac{2}{1-\epsilon_{+}}}}\vert \vert v\vert \vert _{L^{\frac{4}{\epsilon_{+}}}}\mathcal{D}^{\frac{1}{2}}\notag\\
        &\lesssim \mathcal{E}^{\frac{1}{2}}(\vert \vert \partial_{t}u\vert \vert_{H^{1}} \vert \vert v\vert \vert_{L^{2}}+\vert \vert v\vert \vert_{L^{2}}\vert \vert \p_{t}^{2}u\vert \vert_{H^{1}})+\mathcal{D}^{\frac{1}{2}}\vert \vert v\vert \vert_{H^{1}}\vert \vert u\vert \vert_{W^{2,q_{+}}}\lesssim \mathcal{E}^{\frac{1}{2}}\mathcal{D}^{\frac{1}{2}}\|v\|_{H^{1}} \label{equ:5.1.171}
    \end{align}

    \textbf{Term $\partial_{t}^{2}((\cos\theta W\partial_{t}\bar{\eta},\sin\theta W\partial_{t}\bar{\eta})\mathcal{A}(\partial_{x}u,\partial_{y}u)^{T})$}. 
    
    We have the following estimate using H\"older's inequality and Sobolev embedding
    
    \begin{align}
        &\int_{\Omega} \partial_{t}^{2}((\cos\theta W\partial_{t}\bar{\xi},\sin\theta W\partial_{t}\bar{\xi})\tilde{K}(\partial_{x}u,\partial_{y}u)^{T})\cdot vJ \notag\\
        &\lesssim \int_{\Omega} \vert \partial_{t}^{3}\bar{\xi}\vert \vert \nabla u\vert\vert v\vert+\int_{\Omega} \vert \partial_{t}\bar{\xi}\vert \vert \partial_{t}^{2}\tilde{K}\vert \vert \nabla u\vert\vert v\vert+\int_{\Omega} \vert \partial_{t}\bar{\xi}\vert \vert \nabla \p_{t}^{2}u\vert\vert v\vert\notag\\
        &\quad+\int_{\Omega} \vert \partial_{t}^{2}\bar{\xi}\vert \vert \partial_{t}\tilde{K}\vert \vert \nabla u\vert\vert v\vert+\int_{\Omega} \vert \partial_{t}^{2}\bar{\xi}\vert \vert \partial_{t}\nabla u\vert\vert v\vert+\int_{\Omega} \vert \partial_{t}\bar{\xi}\vert \vert \partial_{t}\tilde{K}\vert \vert\partial_{t} \nabla u\vert\vert v\vert \notag\\
        &\quad+\int_{\Omega} \vert \partial_{t}^{2}\bar{\xi}\vert \vert \partial_{t}W\vert \vert \nabla u\vert\vert v\vert+\int_{\Omega} \vert \partial_{t}\bar{\xi}\vert |\p_{t}^{2}W| \vert \nabla u\vert\vert v\vert+\int_{\Omega} \vert \partial_{t}\bar{\xi}\vert \vert \partial_{t}\tilde{K}\vert \vert\partial_{t} \nabla u\vert\vert v\vert+\int_{\Omega} \vert \partial_{t}\bar{\xi}\vert \vert \partial_{t}\tilde{K}\vert \vert \nabla u\vert|\p_{t}W|\vert v\vert \notag\\
        &\lesssim \vert \vert \partial_{t}^{3}\bar{\xi}\vert \vert_{L^{\frac{2}{\alpha}}}\vert \vert  u\vert \vert_{H^{1}} \vert \vert v\vert \vert_{L^{\frac{2}{1-\alpha}}}+\vert \vert \partial_{t}\bar{\xi}\vert \vert_{L^{\infty}}\vert \vert \partial_{t}^{2}\bar{\xi}\vert \vert_{W^{1,4}}\vert \vert u\vert \vert_{H^{1}}\vert \vert v \vert \vert_{L^{4}}+\vert \vert \partial_{t}\bar{\xi}\vert \vert_{L^{+\infty}} \vert \vert \p_{t}^{2}u\vert\vert_{H^{1}}\vert \vert v\vert \vert_{L^{2}}\notag\\
        &\quad+\vert \vert \partial_{t}^{2}\bar{\xi}\vert\vert_{L^{4}}\vert \vert \partial_{t}\bar{\xi}\vert \vert_{W^{1,+\infty}}\vert \vert u\vert\vert_{H^{1}}\vert \vert v\vert \vert_{L^{4}}+\vert \vert \partial_{t}^{2}\bar{\xi}\vert \vert_{L^{4}}\vert \vert \partial_{t}u\vert \vert_{H^{1}}\vert \vert v\vert \vert_{L^{2}}\notag\\
        &\quad+\vert \vert \partial_{t}\bar{\xi}\vert \vert_{L^{\infty}}\vert \vert \partial_{t}\bar{\xi}\vert \vert_{W^{1,+\infty}}\vert \vert \partial_{t}u\vert \vert_{H^{1}}\vert \vert v\vert \vert_{L^{2}}\notag\\
        &\lesssim \vert \vert \partial_{t}^{3}\bar{\xi}\vert \vert_{H^{1-\alpha}}\mathcal{E}^{\frac{1}{2}}\|v\|_{H^{1}}+\vert \vert \partial_{t}\bar{\xi}\vert \vert_{H^{\frac{3}{2}}}\vert \vert \partial_{t}^{2}\bar{\xi}\vert \vert_{H^{2-\alpha}}\mathcal{E}^{\frac{1}{2}}\|v\|_{H^{1}}+\mathcal{E}\|v\|_{H^{1}}+\vert \vert \partial_{t}^{2}\bar{\xi}\vert \vert_{H^{1}}\vert \vert \partial_{t}\bar{\xi}\vert \vert_{H^{\frac{3}{2}}} \mathcal{E}^{\frac{1}{2}}\|v\|_{H^{1}}\notag\\
        &\quad+\vert \vert \partial_{t}^{2}\bar{\xi}\vert \vert_{H^{1}}\vert \vert \partial_{t}u\vert \vert_{H^{1}}\vert \vert v\vert \vert_{L^{2}}+\vert \vert \partial_{t}\bar{\xi}\vert \vert^{2}_{H^{\frac{3}{2}+\frac{\epsilon_{-}-\alpha}{2}}}\mathcal{E}^{\frac{1}{2}}\|v\|_{L^{2}}\lesssim \mathcal{E}^{\frac{1}{2}}\mathcal{D}^{\frac{1}{2}}\|v\|_{H^{1}}\label{equ:5.1.172}
    \end{align}

    \noindent where we have used the fact that $\alpha<\frac{1}{2}$. 

    \textbf{Term $\partial^{2}_{t}(u\cdot \nabla_{\mathcal{A}}u)$} 
    
    We have the following computation by H\"o

    \begin{align}
        \int_{\Omega} \partial_{t}^{2}(u\cdot \nabla_{\mathcal{A}}u)\cdot vJ\lesssim& \int_{\Omega} (\p_{t}^{2}u\cdot \nabla_{\mathcal{A}}u)\cdot vJ+\int_{\Omega}(u\cdot \nabla_{\mathcal{A}}\p_{t}^{2}u)\cdot v+(\partial_{t}u\cdot \nabla_{\mathcal{A}}\partial_{t}u)\cdot v \notag\\
        &+\int_{\Omega}(u\cdot \nabla_{\partial_{t}^{2}\mathcal{A}}u)\cdot v+\int_{\Omega}(u\cdot \nabla_{\partial_{t}\mathcal{A}}\partial_{t}u)\cdot v+\int_{\Omega} (\partial_{t}u\cdot \nabla_{\partial_{t}\mathcal{A}}u)\cdot v\notag\\
        \lesssim& \|\p_{t}^{2}u\|_{L^{4}}\vert \vert v\vert \vert_{L^{4}}\vert \vert u\vert \vert_{W^{1,2}}
        +\vert \vert u\vert \vert_{L^{\infty}}\vert \vert \p_{t}^{2}u\vert \vert_{H^{1}}\vert \vert v\vert \vert_{L^{2}}+\vert \vert \partial_{t}u\vert \vert_{L^{\infty}}\vert \vert \partial_{t}u\vert \vert_{H^{1}}\vert \vert v\vert \vert_{L^{2}}\notag\\
        &+\vert \vert u\vert \vert_{L^{\infty}} \vert \vert \partial_{t}u\vert \vert_{H^{1}}\vert \vert \partial_{t}\bar{\eta}\vert \vert_{W^{1,4}}\vert \vert v\vert \vert_{L^{4}}+\vert \vert \partial_{t}\bar{\eta}\vert \vert_{W^{1,+\infty}}\vert \vert \partial_{t}u\vert \vert_{L^{\infty}}\vert \vert u\vert \vert_{H^{1}} \vert \vert v\vert \vert_{L^{2}}\notag\\
        &+\vert \vert u\vert \vert_{L^{\infty}}\vert \vert \partial_{t}^{2}\bar{\xi}\vert \vert_{W^{1,4}}\vert \vert u\vert \vert_{W^{1,2}}\vert \vert v\vert \vert_{L^{4}}\\
        \lesssim& \mathcal{E}^{\frac{1}{2}}\mathcal{D}^{\frac{1}{2}}\|v\|_{H^{1}}\label{equ:5.1.173},
    \end{align}

    \noindent where we used the Sobolev embedding $W^{1,4}\hookrightarrow H^{2-\alpha}$ since $\alpha<\frac{1}{2}$. Combining all of the estimates in this step, we obtain the following estimate

    \begin{align}
        \int_{\Omega} \partial_{tt}b^{1}\cdot vJ\lesssim \mathcal{E}^{\frac{1}{2}}\mathcal{D}^{\frac{1}{2}}\|v\|_{H^{1}} \label{equ:5.1.174}
    \end{align}

    \textbf{Step 2} Other terms of $\partial_{t}b^{1,1}$:

    \textbf{Term $\partial_{t}(\operatorname{div}_{\partial_{t}\mathcal{\mathcal{A}}}S_{\mathcal{A}}(p,u))$} We have the following estimate using H\"older's inequality

    \begin{align}
        \int_{\Omega} \partial_{t}(\operatorname{div}_{\partial_{t}\mathcal{\mathcal{A}}}S_{\mathcal{A}}(p,u))\cdot vJ=&\int_{\Omega}\operatorname{div}_{\partial_{t}^{2}\mathcal{A}}S_{\mathcal{A}}(p,u)\cdot vJ+\int_{\Omega} \operatorname{div}_{\partial_{t}\mathcal{A}}\mathbb{D}_{\partial_{t}\mathcal{A}}u\cdot vJ\notag\\
        &\quad+\int_{\Omega} \operatorname{div}_{\partial_{t}\mathcal{A}}S_{\mathcal{A}}(\partial_{t}p,\p_{t}u)\cdot vJ=I_{1}+I_{2}+I_{3} \label{equ:5.1.183}
    \end{align}

    \noindent We then estimate $I_{1}-I_{3}$ individually. First, $I_{1}$ can be estimated as follows

    \begin{align}
        I_{1}&\lesssim \int_{\Omega}\vert \partial_{t}^{2}\nabla \bar{\xi}\vert\vert \nabla p\vert\vert v\vert+\int_{\Omega} \vert \partial_{t}^{2}\nabla \bar{\xi}\vert\vert \nabla^{2} u\vert\vert v\vert+\int_{\Omega} \vert \partial_{t}^{2}\nabla \bar{\xi}\vert\vert \nabla^{2}\bar{\xi}\vert\vert  p\vert\vert v\vert+\int_{\Omega} \vert \partial_{t}^{2}\nabla \bar{\xi}\vert\vert \nabla^{2}\bar{\xi}\vert\vert  \nabla u\vert\vert v\vert \notag\\
       & \lesssim \vert \vert \partial_{t}^{2}\bar{\xi}\vert \vert_{W^{1,\frac{2}{\alpha}}}\vert \vert p\vert \vert_{W^{1,q_{+}}}\vert \vert v\vert \vert_{L^{\frac{2}{\epsilon_{+}-\alpha}}}+\vert \vert \partial_{t}^{2}\bar{\xi}\vert \vert_{W^{1,\frac{2}{\alpha}}}\vert \vert u\vert \vert_{W^{2,q_{+}}}\vert \vert v\vert \vert_{L^{\frac{2}{\epsilon_{+}-\alpha}}}\notag\\
        &\quad+\vert \vert \partial_{t}^{2}\bar{\xi}\vert \vert_{W^{1,\frac{2}{\alpha}}}\vert \vert p\vert \vert_{L^{\frac{2}{1-\epsilon_{+}}}}\vert \vert v\vert \vert_{L^{\frac{2}{2\epsilon_{+}-\alpha}}}\vert \vert \bar{\xi}\vert \vert_{W^{2,\frac{2}{1-\epsilon_{+}}}}+\vert \vert \partial_{t}^{2}\bar{\xi}\vert \vert_{W^{1,\frac{2}{\alpha}}}\vert \vert u\vert \vert_{W^{1,\frac{2}{1-\epsilon_{+}}}}\vert \vert v\vert \vert_{L^{\frac{2}{2\epsilon_{+}-\alpha}}}\vert \vert \bar{\xi}\vert \vert_{W^{2,\frac{2}{1-\epsilon_{+}}}}\notag\\
        &\lesssim   \vert \vert \partial_{t}^{2}\bar{\xi}\vert \vert_{H^{2-\alpha}}\vert \vert p\vert \vert_{W^{1,q_{+}}}\vert \vert v\vert \vert_{H^{1}}+\vert \vert \partial_{t}^{2}\bar{\xi}\vert \vert_{H^{2-\alpha}}\vert \vert u\vert \vert_{W^{2,q_{+}}}\vert \vert v\vert \vert_{H^{1}}\notag\\
        &\quad+\vert \vert \partial_{t}^{2}\bar{\xi}\vert \vert_{H^{2-\alpha}}\vert \vert p\vert \vert_{W^{1,q_{+}}}\vert \vert v\vert \vert_{H^{1}}\vert \vert \bar{\xi}\vert \vert_{W^{3,q_{+}}}+\vert \vert \partial_{t}^{2}\bar{\eta}\vert \vert_{H^{2-\alpha}}\vert \vert u\vert \vert_{W^{2,q_{+}}}\vert \vert v\vert \vert_{H^{1}}\vert \vert \bar{\xi}\vert \vert_{W^{3,q_{+}}}\lesssim \mathcal{E}^{\frac{1}{2}}\mathcal{D}^{\frac{1}{2}}\|v\|_{H^{1}}\label{equ:5.1.184}
    \end{align}

    \noindent Similar to the estimate of $I_1$, term $I_{2}$ is estimated as follows
    
    \begin{align}
        I_{2}&\lesssim \int_{\Omega} \vert \partial_{t}\nabla \bar{\xi}\vert^{2}(\vert \nabla p\vert+\vert \nabla^{2} u\vert) \vert v\vert+\int_{\Omega} \vert \partial_{t}\nabla \bar{\xi}\vert\vert \partial_{t}\nabla^{2}\bar{\xi}\vert(\vert  p\vert+\vert \nabla u\vert) \vert v\vert \notag\\
        &\lesssim \vert \vert \partial_{t}\bar{\xi}\vert \vert^{2}_{W^{1,+\infty}}(\vert \vert p\vert \vert_{W^{1,q_{+}}}+\vert \vert u\vert \vert_{W^{2,q_{+}}})\vert \vert v\vert \vert_{L^{\frac{2}{\epsilon_{+}}}}\notag\\
        &\quad+\vert \vert \partial_{t}\bar{\xi}\vert \vert_{W^{1,+\infty}}\vert \vert \partial_{t}\bar{\xi}\vert \vert_{W^{2,\frac{2}{1-\epsilon_{-}}}}(\vert \vert p\vert \vert_{L^{\frac{2}{1-\epsilon_{+}}}}+\vert \vert u\vert \vert_{W^{1,\frac{2}{1-\epsilon_{+}}}})\vert \vert v\vert \vert_{L^{\frac{2}{\epsilon_{-}+\epsilon_{+}}}}\notag\\
        &\lesssim \vert \vert \partial_{t}\bar{\xi}\vert \vert^{2}_{H^{2+\frac{\epsilon_{-}-\alpha}{2}}}(\vert \vert p\vert \vert_{W^{1,q_{+}}}+\vert \vert u\vert \vert_{W^{2,q_{+}}})\vert \vert v\vert \vert_{H^{1}}\notag\\
        &\quad+\vert \vert \partial_{t}\bar{\xi}\vert \vert_{H^{2+\frac{\epsilon_{-}-\alpha}{2}}}\vert \vert \partial_{t}\bar{\xi}\vert \vert_{W^{3,q_{-}}}(\vert \vert p\vert \vert_{W^{1,q_{+}}}+\vert \vert u\vert \vert_{W^{2,q_{+}}})\vert \vert v\vert \vert_{H^{1}} \lesssim \mathcal{E}\mathcal{D}^{\frac{1}{2}}\|v\|_{H^{1}}\label{equ:5.1.185}
    \end{align}

    \noindent Similarly, for $I_{3}$, we have the following estimate
    \begin{align}
        I_{3}\lesssim& \|\p_{t}\bar{\xi}\|_{W^{1,\infty}}(\|\p_{t}u\|_{W^{2,q_{-}}}+\|\p_{t}p\|_{W^{1,q_{-}}})\|v\|_{L^{\frac{2}{\epsilon_{-}}}}+\|\p_{t}\bar{\xi}\|_{W^{1,+\infty}}\|\bar{\xi}\|_{W^{2,\frac{2}{1-\epsilon_{+}}}}(\|\p_{t}u\|_{W^{1,\frac{1}{1-\epsilon_{-}}}}+\|\p_{t}p\|_{L^{\frac{2}{1-\epsilon_{-}}}})   \|v\|_{L^{\frac{2}{\epsilon_{-}+\epsilon_{+}}}}\notag\\
        \lesssim& \|\p_{t}\xi\|_{H^{\frac{3}{2}+\frac{\epsilon_{-}-\alpha}{2}}}\|v\|_{H^{1}}(\|\p_{t}u\|_{W^{2,q_{-}}}+\|\p_{t}p\|_{W^{1,q_{-}}})+\|\p_{t}\xi\|_{H^{\frac{3}{2}+\frac{\epsilon_{-}-\alpha}{2}}}\|v\|_{H^{1}}(\|\p_{t}u\|_{W^{2,q_{-}}}+\|\p_{t}p\|_{W^{1,q_{-}}})\|\xi\|_{W^{3-\frac{1}{q_{+}},q_{+}}}\notag\\
        \lesssim& \mathcal{E}^{\frac{1}{2}}\mathcal{D}^{\frac{1}{2}}\|v\|_{H^{1}}\label{equ:5.1.186}
    \end{align}

    \noindent Hence, combining  \eqref{equ:5.1.184},\eqref{equ:5.1.185} and \eqref{equ:5.1.186}, we have

    \begin{align}
        \int_{\Omega} \partial_{t}(\operatorname{div}_{\partial_{t}\mathcal{\mathcal{A}}}S_{\mathcal{A}}(p,u))\cdot vJ\lesssim\mathcal{E}^{\frac{1}{2}}\mathcal{D}^{\frac{1}{2}}\|v\|_{H^{1}}\label{equ:5.1.187}
    \end{align}

     \textbf{Term $\partial_{t}(\operatorname{div}_{\mathcal{A}}\mathbb{D}_{\partial_{t}\mathcal{A}}u)$}

      Using the similar computation as for the last term, we obtain
     
    \begin{align}
        \int_{\Omega}\partial_{t}(\operatorname{div}_{\mathcal{A}}\mathbb{D}_{\partial_{t}\mathcal{A}}u)\cdot vJ\lesssim& \vert \vert \partial_{t}\nabla \bar{\xi}\vert \vert^{2}_{L^{\infty}}\vert \vert u\vert \vert_{W^{2,q_{+}}}\vert \vert v\vert \vert_{L^{\frac{2}{\epsilon_{+}}}}\notag\\
        &\quad+\vert \vert \partial_{t}\nabla \bar{\xi}\vert \vert_{L^{+\infty}}\vert \vert u\vert \vert_{W^{1,\frac{2}{1-\epsilon_{+}}}}\vert \vert \partial_{t}\bar{\xi}\vert \vert_{W^{2,\frac{2}{1-\epsilon_{-}}}}\vert \vert v\vert \vert_{L^{\frac{2}{\epsilon_{+}+\epsilon_{-}}}}\notag\\
        &\quad+\vert \vert v\vert \vert_{L^{\frac{2}{\epsilon_{+}-\alpha}}}(\vert \vert \partial_{t}\bar{\xi}\vert \vert_{W^{1,\frac{2}{\alpha}}}+\vert \vert \partial_{t}^{2}\bar{\xi}\vert\vert_{W^{1,\frac{2}{\alpha}}} )\vert \vert \nabla^{2}u\vert \vert_{L^{\frac{2}{2-\epsilon_{+}}}}\notag\\
        &\quad+\vert \vert v\vert \vert_{L^{\frac{2}{\epsilon_{+}-\alpha}}}(\vert \vert \partial_{t}\bar{\xi}\vert \vert_{W^{1,\frac{2}{\alpha+1}}}+\vert \vert \partial_{t}^{2}\bar{\xi}\vert\vert_{W^{2,\frac{2}{\alpha+1}}} )\vert \vert \nabla u\vert \vert_{L^{\frac{2}{1-\epsilon_{+}}}}\notag\\
        &\quad+\vert \vert \partial_{t}\bar{\xi}\vert \vert_{W^{1,+\infty}}\vert \vert \p_{t}u\vert \vert_{W^{2,q_{-}}}\vert \vert v\vert \vert_{\frac{2}{\epsilon_{-}}}+\vert \vert \partial_{t}\bar{\xi}\vert \vert_{W^{2,,\frac{2}{1-\epsilon_{-}}}}\vert \vert \p_{t}u\vert \vert_{W^{1,\frac{2}{1-\epsilon_{-}}}}\vert \vert v\vert \vert_{L^{\frac{2}{2\epsilon_{-}}}}\notag\\
        &\lesssim \mathcal{E}^{\frac{1}{2}}\mathcal{D}^{\frac{1}{2}}\|v\|_{H^{1}}\label{equ:5.1.188}
    \end{align}

   \textbf{Step 3} Remaining  terms included in $b^{1,2}$. 

   \textbf{Term $\operatorname{div}_{\partial_{t}\mathcal{A}}S_{\mathcal{A}}(\partial_{t}p,\partial_{t}u)$} 

   We have the following estimate using H\"older's inequality and trace theorem 

   \begin{align}
       \int_{\Omega} \operatorname{div}_{\partial_{t}\mathcal{A}}S_{\mathcal{A}}(\partial_{t}p,\partial_{t}u)\cdot vJ\lesssim \vert \vert \partial_{t}\bar{\xi}\vert \vert_{W^{1,+\infty}}\vert\vert \nabla^{2}\bar{\xi}\vert \vert_{L^{\frac{2}{1-\epsilon_{+}}}}(\vert \vert \partial_{t}p\vert \vert_{L^{\frac{2}{1-\epsilon_{-}}}}+\vert\vert \partial_{t}u\vert \vert_{W^{1,\frac{2}{1+\epsilon_{+}}}})\vert \vert v\vert \vert_{L^{\frac{1}{1-\epsilon_{+}}}}\notag\\
       +\vert \vert \partial_{t}\bar{\xi}\vert \vert_{W^{1,+\infty}}(\vert \vert \partial_{t}p\vert \vert_{W^{1,q_{-}}}+\vert\vert \partial_{t}u\vert \vert_{W^{2,q_{-}}})\vert \vert v\vert \vert_{L^{\frac{2}{\epsilon_{-}}}}\notag\\
       \lesssim \vert \vert \partial_{t}\bar{\xi}\vert \vert_{W^{1,+\infty}}\vert\vert \bar{\xi}\vert \vert_{W^{3,q_{+}}}(\vert \vert \partial_{t}p\vert \vert_{W^{1,q_{-}}}+\vert\vert \partial_{t}u\vert \vert_{W^{2,q_{+}}})\vert \vert v\vert \vert_{H^{1}}\notag\\
       +\vert \vert \partial_{t}\bar{\xi}\vert \vert_{W^{1,+\infty}}(\vert \vert \partial_{t}p\vert \vert_{W^{1,q_{-}}}+\vert\vert \partial_{t}u\vert \vert_{W^{2,q_{+}}})\vert \vert v\vert \vert_{H^{1}}\lesssim \mathcal{E}^{\frac{1}{2}}\mathcal{D} \label{equ:5.1.193}
   \end{align}

   \textbf{Term $\operatorname{div}_{\mathcal{A}}S_{\partial_{t}\mathcal{A}}(\partial_{t}p,\partial_{t}u)$} 

   Similarly, we have the following estimate using H\"older's inequality

   \begin{align}
       \int_{\Omega} \operatorname{div}_{\mathcal{A}}S_{\partial_{t}\mathcal{A}}(\partial_{t}p,\partial_{t}u)\lesssim\mathcal{E}^{\frac{1}{2}}\mathcal{D} \label{equ:5.1.194}
   \end{align}

   \textbf{Step 4} The estimate for $\int_{\Omega}\partial_{t}J\vert \p_{t}^{2}u\vert^{2}$

   We have the following estimate from the definition of $J=\det(\mathcal{A})$
   \begin{align}
       \int_{\Omega} \partial_{t}J\vert \p_{t}^{2}u\vert^{2}\lesssim \vert \vert \partial_{t}\bar{\eta}\vert \vert_{W^{1,+\infty}}\vert \vert \p_{t}u\vert \vert^{2}_{L^{2}}\lesssim \mathcal{E}^{\frac{3}{2}} \label{equ:5.1.195}
   \end{align}

   \textbf{Step 5} The estimate for $\kappa [b^{7,2},\frac{1}{\rho_{0}}v\cdot \mathcal{N}]_{\theta}$. 
   
   Using the definition of $b^{7,2}$ in Appendix, and Theorem \ref{thm:gam}, we have the following computation

   \begin{align}
       [b^{7,2},\frac{1}{\rho_{0}}v\cdot \mathcal{N}]_{\theta}\lesssim& [v\cdot \mathcal{N}]_{\theta}(([\partial_{t}\xi]_{\theta}+[\mathfrak{n}'(t)]_{\theta})([\partial_{t}^{3}\xi]_{\theta}+[\mathfrak{n}'''(t)]_{\theta})+([\partial_{t}^{2}\xi]^{2}_{\theta}+[\mathfrak{n}''(t)]^{2}_{\theta}))\notag\\
       \lesssim& [v\cdot \mathcal{N}]_{\theta}((\vert \vert \partial_{t}\xi\vert \vert_{W^{1,1}}+\vert \vert u\vert \vert_{H^{1}}+\mathcal{E})([\p_{t}^{3}\xi]_{\theta}+\vert \vert \p_{t}^{2}u\vert \vert_{H^{1}}+\mathcal{E})+\vert \vert \p_{t}u\vert \vert^{2}_{H^{1}}+\vert \vert \p_{t}^{2}\xi\vert \vert^{2}_{H^{1}}+\mathcal{E}^{2})\notag\\
       \lesssim& \mathcal{E}^{\frac{1}{2}}\mathcal{D}^{\frac{1}{2}}[v\cdot \mathcal{N}]_{\theta}\label{equ:5.1.196}
   \end{align}

   \textbf{Step 6} In this step, we now estimate the term $\kappa [b^{6,2}\,\frac{1}{\rho_{0}}v\cdot \mathcal{N}]_{\theta}$. 
   We have the following argument using the definition of $b^{6,2}$, trace theorem and Sobolev embedding

   \begin{align}
       \kappa [b^{6,2},\frac{1}{\rho_{0}}v\cdot \mathcal{N}]_{\theta}&\lesssim [v\cdot \mathcal{N}]_{\theta}[\p_{t}u]_{\theta}[\p_{t}\p_{\theta}\xi]_{\theta}+[v\cdot \mathcal{N}]_{\theta}[u]_{\theta}[\p_{t}^{2}\p_{\theta}\xi]_{\theta}+\mathcal{E}[v\cdot \mathcal{N}]_{\theta}(\sum_{i=0}^{2}[\p_{\theta}\p_{t}^{i}\xi]_{\theta})\notag\\
       &\quad+(\|\p_{t}\xi\|_{H^{\frac{3}{2}+\frac{\epsilon_{-}-\alpha}{2}}}+\|\p_{t}^{2}\xi\|_{H^{1}})(\|u\|_{W^{2,q_{+}}}+\|\p_{t}u\|_{H^{1+\frac{\epsilon_{-}}{2}}}+\|u\|_{W^{2,q_{+}}}\|\p_{t}\xi\|_{H^{\frac{3}{2}+\frac{\epsilon_{-}-\alpha}{2}}})[v\cdot \mathcal{N}]_{\theta}\notag\\
       &\quad+\|\xi\|_{W^{3-\frac{1}{q_{-}},q_{-}}}([\p_{t}^{2}u\cdot \mathcal{N}]_{\theta}+\|\p_{t}u\|_{H^{1+\frac{\epsilon_{-}}{2}}}\|\p_{t}\xi\|_{H^{\frac{3}{2}+\frac{\epsilon_{-}-\alpha}{2}}}+\|u\|_{W^{2,q_{+}}}[\p_{t}^{2}\p_{\theta}\xi]_{\theta})[v\cdot \mathcal{N}]_{\theta}\notag\\
       &\lesssim \mathcal{E}^{\frac{1}{2}}\mathcal{D}^{\frac{1}{2}}[v\cdot \mathcal{N}]_{\theta}
   \end{align}
   
   \textbf{Step 7} In this step, we estimate the term $(\p_{t}^{2}\xi,b^{6,2})_{1,\Sigma}$:

   We aim to show the following result

   \begin{align}
       |(\p_{t}^{2}\xi,b^{6,2})_{1,\Sigma}-H|\lesssim \mathcal{E}^{\frac{1}{2}}\mathcal{D}{\label{equ:H}},
   \end{align}
   \noindent where $H$ denotes a collection of terms involving $(\p_{t}^{i}u,\p_{t}^{i}p,\p_{t}^{i}\xi)$ with $i=0,1,2$, which satisfies the following estimate
   \begin{align}
       |\int_{0}^{t}H|\lesssim \mathcal{E}^{\frac{3}{2}}(t)+\mathcal{E}^{\frac{3}{2}}(0)
   \end{align}
   
   We now estimate $(\p_{t}^{2}\xi,b^{6,2})_{1,\Sigma}$ as follows.

   \begin{align}
        (\partial_{tt}\xi,b^{6,2})_{1,\Sigma}=& (\partial_{tt}\xi,(\frac{\xi}{\rho\rho_{0}}\p_{t}^{2}u\cdot \mathcal{N}-\frac{(\partial_{t}\xi)\xi}{\rho^{2}\rho_{0}}\p_{t}u\cdot \mathcal{N}+\frac{(\partial_{t}\xi)}{\rho\rho_{0}}\p_{t}u\cdot \mathcal{N}))_{1,\Sigma}\notag\\
        &+(\partial_{tt}\xi,(\frac{\xi}{\rho\rho_{0}}u\cdot \p_{t}^{2}\mathcal{N}-\frac{(\partial_{t}\xi)\xi}{\rho^{2}\rho_{0}}u\cdot \p_{t}\mathcal{N}+\frac{(\partial_{t}\xi)}{\rho\rho_{0}}u\cdot \p_{t}\mathcal{N}))_{1,\Sigma}\notag\\
        &+(\partial_{tt}\xi,(\frac{\partial_{t}^{2}\xi}{\rho\rho_{0}}-\frac{(\partial_{t}\xi)^{2}}{\rho^{2}\rho_{0}}-\frac{\partial_{t}^{2}\xi\xi}{\rho^{2}\rho_{0}}+2\frac{\xi(\partial_{t}\xi)^{2}}{\rho^{3}\rho_{0}})u\cdot \mathcal{N})_{1,\Sigma}\notag\\
        &+(\partial_{tt}\xi,\mathfrak{n}'(t)\partial_{tt}(\frac{\rho'}{\rho})\cos\theta+\mathfrak{n}''(t)\partial_{t}(\frac{\rho'}{\rho})\cos\theta+\mathfrak{n}'''(t)(\frac{\rho'}{\rho}-\frac{\rho_{0}'}{\rho_{0}})\cos\theta)\notag\\
        &+(\partial_{t}^{2}\xi,\frac{2}{\rho_{0}}\partial_{t}u\cdot\partial_{t}\mathcal{N}+\frac{1}{\rho_{0}}u\cdot \partial_{tt}\mathcal{N})_{1,\Sigma}\label{equ:5.1.197}
   \end{align}

   \noindent Then we estimate each term in \eqref{equ:5.1.197}, individually. \\

   \textbf{Term $\frac{\xi}{\rho\rho_{0}}\p_{t}^{2}u\cdot \mathcal{N}$}

   We first estimate the following term
   \begin{align}
       (\partial_{tt}\xi,\frac{\xi}{\rho_{0}\rho}\p_{t}^{2}u\cdot \mathcal{N})_{1,\Sigma}\label{equ:5.1.198}.
   \end{align}
   \noindent The kinematic boundary condition yields

   \begin{align}
   \frac{1}{\rho}\p_{t}^{2}u\cdot \mathcal{N}=\partial_{ttt}\xi+\mathfrak{n}'''(t)\xi_s+\tilde{b}^{8,2} -\frac{1}{\rho
       }u\cdot \p_{t}^{2}\mathcal{N}\label{equ:5.1.199},
   \end{align}

   \noindent where

   \begin{align}
       \tilde{b}^{8,2}=-\partial_{tt}(\frac{1}{\rho})u\cdot \mathcal{N}-2\partial_{t}(\frac{1}{\rho})\p_{t}u\cdot \mathcal{N}-2\p_{t}(\frac{1}{\rho})u\cdot \p_{t}\mathcal{N} \label{equ:5.1.200}.
   \end{align}

   \noindent We show the following boundedness for $\tilde{b}^{8,2}$ from its definition

   \begin{align}
       \vert \vert \tilde{b}^{8,2}\vert \vert_{W^{1,\frac{1}{1-\alpha}}}&\lesssim \vert \vert \partial_{t}^{2}\xi\vert \vert_{W^{1,\frac{1}{1-\alpha}}}\vert \vert u\vert \vert_{W^{2,q_{+}}}+\vert \vert \partial_{t}^{2}\xi\vert \vert_{L^{\infty}}\vert \vert u\vert \vert_{W^{1,\frac{1}{1-\epsilon_{+}}}(\Sigma)}+\vert \vert \partial_{t}\xi\vert \vert_{W^{1,\frac{1}{1-\alpha}}}\vert \vert \p_{t}u\vert \vert_{W^{2,q_{-}}}\notag\\
      &\quad+\vert \vert \partial_{t} \xi \vert \vert_{L^{\infty}}\vert \vert \p_{t}u\vert \vert_{W^{1,\frac{1}{1-\epsilon_{-}}}(\Sigma)}+\|\p_{t}\xi\|_{W^{1,\frac{1}{1-\alpha}}}\|u\|_{L^{\infty}(\Sigma)}\|\p_{t}\xi\|_{W^{1,\infty}}+\|\p_{t}\xi\|_{L^{\infty}}\|u\|_{W^{1,\frac{1}{1-\epsilon_{+}}}(\Sigma)}\|\p_{t}\xi\|_{W^{1,\infty}}\notag\\
      &\quad+\|\p_{t}\xi\|_{L^{\infty}}\|u\|_{L^{\infty}(\Sigma)}\|\p_{t}\xi\|_{W^{2,\frac{1}{1-\epsilon_{-}}}}\notag\\
      &\lesssim \vert \vert \partial_{t}^{2}\xi\vert \vert_{H^{\frac{3}{2}-\alpha}}\vert \vert u\vert \vert_{W^{2,q_{+}}}+\vert \vert \partial_{t}^{2}\xi\vert \vert_{H^{\frac{3}{2}-\alpha}}\vert \vert u\vert \vert_{W^{2,q_{+}}}+\vert \vert \partial_{t}\xi\vert \vert_{H^{\frac{3}{2}-\alpha}}\vert \vert \p_{t}u\vert \vert_{W^{2,q_{-}}}\notag\\
      &\quad+\vert \vert \partial_{t} \xi \vert \vert_{H^{\frac{3}{2}-\alpha}}\vert \vert \p_{t}u\vert \vert_{W^{2,q_{-}}}+\|\p_{t}\xi\|_{H^{\frac{3}{2}-\alpha}}\|u\|_{W^{2,q_{+}}}\|\p_{t}\xi\|_{H^{\frac{3}{2}+\frac{\epsilon_{-}-\alpha}{2}}}+\|\p_{t}\xi\|_{H^{\frac{3}{2}-\alpha}}\|u\|_{W^{2,q_{+}}}\|\p_{t}\xi\|_{H^{\frac{3}{2}+\frac{\epsilon_{-}-\alpha}{2}}}\notag\\
      &\quad+\|\p_{t}\xi\|_{H^\frac{3}{2}-\alpha }\|u\|_{W^{2,q_{+}}}\|\p_{t}\xi\|_{W^{3-\frac{1}{q_{-}},q_{-}}}\notag\\
      &\lesssim \mathcal{D}^{\frac{1}{2}}\mathcal{E}^{\frac{1}{2}}\label{equ:5.1.201}
   \end{align}
   
   \noindent Substituting equation \eqref{equ:5.1.199} into equation \eqref{equ:5.1.198}, we obtain
   
   \begin{align}
       (\partial_{tt}\xi,\frac{\xi}{\rho_{0}\rho}\p_{t}^{2}u\cdot \mathcal{N})_{1,\Sigma}=&(\partial_{tt}\xi,\frac{\xi}{\rho_{0}}\partial_{ttt}\xi)_{1,\Sigma}+(\partial_{tt}\xi,\frac{\xi}{\rho_{0}}\gamma'''(t)u\xi_s)_{1,\Sigma}+(\partial_{tt}\xi,\frac{\xi}{\rho_{0}}\tilde{b}^{8,2})_{1,\Sigma}-(\p_{t}^{2}\xi,\frac{\xi}{\rho_{0}\rho}u\cdot \p_{t}^{2}\mathcal{N})_{1,\Sigma}\notag\\
       &-(\p_{t}^{2}\xi,\frac{\xi}{\rho_{0}\rho}u\cdot \p_{t}^{2}\mathcal{N})_{1,\Sigma}\notag\\
       =&\frac{1}{2}\partial_{t}(\partial_{t}^{2}\xi,\frac{\xi}{\rho_{0}}\partial_{tt}\xi)_{1,\Sigma}-(\partial_{t}^{2}\xi,\frac{\partial_{t}\xi}{\rho_{0}}\partial_{t}^{2}\xi)_{1,\Sigma}+(\partial_{tt}\xi,\frac{\xi}{\rho_{0}}\mathfrak{n}'''(t)u\xi_2)_{1,\Sigma}+(\partial_{tt}\xi,\frac{\xi}{\rho_{0}}\tilde{b}^{8,2})_{1,\Sigma} \notag\\
        &-(\p_{t}^{2}\xi,\frac{\xi}{\rho_{0}\rho}u\cdot \p_{t}^{2}\mathcal{N})_{1,\Sigma}\notag\\
       \lesssim& \frac{1}{2}\partial_{t}(\partial_{t}^{2}\xi,\frac{\xi}{\rho_{0}}\partial_{tt}\xi)_{1,\Sigma}+\vert \vert \partial_{t}^{2}\xi\vert \vert^{2}_{H^{1}}\vert \vert \partial_{t}\xi\vert \vert_{W^{1,+\infty}}+\vert \vert \partial_{t}^{2}\xi\vert \vert_{H^{1}}\vert \vert \xi
       \vert \vert_{W^{1,+\infty}}(\mathcal{E}+\|\p_{t}^{2}u\|_{H^{1}})\notag\\&+\vert \vert \partial_{t}^{2}\xi\vert \vert_{W^{1,\frac{1}{\alpha}}}\vert \vert\xi \vert \vert_{W^{1,\infty}}\vert \vert \tilde{b}^{8,2}\vert \vert_{W^{1,\frac{1}{1-\alpha}}}-(\p_{t}^{2}\xi,\frac{\xi}{\rho_{0}\rho}u\cdot \p_{t}^{2}\mathcal{N})_{1,\Sigma}\label{equ:5.1.202}
   \end{align}

   \noindent Using \eqref{equ:5.1.201} in \eqref{equ:5.1.202}, we have

   \begin{align}
        (\partial_{tt}\xi,\frac{\xi}{\rho_{0}\rho}\p_{t}^{2}u\cdot \mathcal{N})_{1,\Sigma}-\frac{1}{2}\partial_{t}(\partial_{t}^{2}\xi,\frac{\xi}{\rho_{0}}\partial_{t}^{2}\xi)_{1,\Sigma}-(\p_{t}^{2}\xi,\frac{\xi}{\rho_{0}\rho}u\cdot \p_{t}^{2}\mathcal{N})_{1,\Sigma}\lesssim \mathcal{E}^{\frac{1}{2}}\mathcal{D} \label{equ:5.1.203}
   \end{align}
   \noindent It then remains to deal with the second and the third terms on the left hand side of \eqref{equ:5.1.203}. For the second term, we have the following estimate:
   \begin{align}
       (\partial_{t}^{2}\xi,\frac{\xi}{\rho_{0}}\partial_{t}^{2}\xi)_{1,\Sigma}\lesssim\|\p_{t}^{2}\xi\|_{H^{1}}\|\xi\|_{W^{3-\frac{1}{q_{+}},q_{+}}}\lesssim\mathcal{E}^{\frac{3}{2}}
   \end{align}
   \noindent Using integration by part, we have the following estimate for the third term on the left hand side of \eqref{equ:5.1.203}
   \begin{align}
       (\p_{t}^{2}\xi,\frac{\xi}{\rho_{0}\rho}u\cdot \p_{t}^{2}\mathcal{N})_{1,\Sigma}=&-\frac{1}{2}\int_{0}^{\pi}(\p_{t}^{2}\xi')^{2}\p_{\theta}(u_{\theta}\frac{\xi}{\rho_{0}\rho})+\int_{0}^{\pi}(\p_{t}^{2}\xi')^{2}\frac{\xi}{\rho_{0}\rho}+\int_{0}^{\pi}(\p_{t}^{2}\xi')\p_{\theta}(\frac{\xi}{\rho_{0}\rho}u)\cdot \p_{t}^{2}\mathcal{N}\notag\\
       &+[\p_{t}^{2}\xi]_{\theta}^{2}[\xi]_{\theta}\|u\|_{W^{2,q_{+}}}\notag\\
       \lesssim& \|\p_{t}^{2}\xi\|_{H^{\frac{3}{2}-\alpha}}\|u\|_{W^{1,\frac{1}{1-\epsilon_{+}}}(\Sigma)}\|\xi\|_{W^{1,\infty}}+\mathcal{E}^{2}\lesssim \|\p_{t}^{2}\xi\|_{H^{\frac{3}{2}-\alpha}}\|u\|_{W^{2,q_{+}}}\|\xi\|_{W^{3-\frac{1}{q_{+}},q_{+}}}+\mathcal{E}^{2}\lesssim \mathcal{E}\mathcal{D}\label{equ:int_p}
   \end{align}
   \noindent Therefore, we obtain the following estimate
   \begin{align}
       |(\p_{t}^{2}\xi,\frac{\xi}{\rho_{0}\rho}\p_{t}^{2}u\cdot \mathcal{N})-(\partial_{t}^{2}\xi,\frac{\xi}{\rho_{0}}\partial_{t}^{2}\xi)_{1,\Sigma}|\lesssim \mathcal{E}^{\frac{1}{2}}\mathcal{D},
   \end{align}
   \noindent where
   \begin{align}
       (\partial_{t}^{2}\xi,\frac{\xi}{\rho_{0}}\partial_{t}^{2}\xi)_{1,\Sigma}\lesssim \mathcal{E}^{\frac{3}{2}}.
   \end{align}
   
   \textbf{Term $\frac{\partial_{t}\xi\xi}{\rho^{2}\rho_{0}}\p_{t}u\cdot \mathcal{N}$}. 
   We have the following estimate by Sobolev embedding and H\"older's inequality

   \begin{align}
       (\partial_t^{2}\xi,\frac{\partial_{t}{\xi}\xi}{\rho^{2}\rho_{0}}\p_{t}u\cdot \mathcal{N})_{1,\Sigma}\lesssim \vert \vert \partial_{t}^{2}\xi\vert \vert_{W^{1,\frac{1}{\alpha}}}\vert \vert \partial_{t}\xi\vert \vert_{W^{1,+\infty}}\vert \vert \xi\vert \vert_{W^{1,+\infty}}\vert \vert \p_{t}u\vert \vert_{W^{1,\frac{1}{1-\epsilon_{-}}}(\Sigma)}\notag\\
       \lesssim \vert \vert \partial_{t}^{2}\xi\vert \vert_{H^{\frac{3}{2}-\alpha}}\vert \vert \partial_{t}\xi\vert \vert_{H^{\frac{3}{2}+\frac{\epsilon_{-}-\alpha}{2}}}\vert \vert \xi\vert \vert_{W^{3-\frac{1}{q_{+}},q_{+}}}\vert \vert \p_{t}u\vert \vert_{W^{2,q_{-}}}\lesssim \mathcal{E}\mathcal{D} \label{equ:5.1.204}
   \end{align}

   \textbf{Term $\frac{\partial_{t}\xi}{\rho_{0}\rho^{2}}\p_{t}u\cdot \mathcal{N}$} 

   We have the following estimate for this term by Sobolev embedding and H\"older's inequality

   \begin{align}
       \int_{\Omega} (\frac{\partial_{t}\xi}{\rho_{0}\rho^{2}}\p_{t}u\cdot \mathcal{N},\partial_{t}^{2}\xi)_{1,\Sigma}&\lesssim \|\p_{t}\xi\|_{H^{1}}\|\p_{t}^{2}\xi\|_{H^{1}}\|\p_{t}u\|_{L^{\infty}(\Sigma)}+\|\p_{t}\xi\|_{L^{\infty}}\|\p_{t}^{2}\xi\|_{W^{1,\frac{1}{\alpha}}}\|\p_{t}u\|_{W^{1,\frac{1}{\epsilon_{-}}}(\Sigma)}\notag\\
       &\lesssim \|\p_{t}\xi\|_{H^{\frac{3}{2}+\frac{\epsilon_{-}-\alpha}{2}}}\|\p_{t}^{2}\xi\|_{H^{\frac{3}{2}-\alpha}}\|\p_{t}u\|_{W^{2,q_{-}}} \lesssim \mathcal{E}^{\frac{1}{2}}\mathcal{D} \label{equ:5.1.205}
   \end{align}

   \textbf{Term $\frac{\partial_{t}^{2}\xi}{\rho\rho_{0}}u\cdot \mathcal{N}$} 

   We have the following computation:

   \begin{align}
       (\partial_{t}^{2}\xi,\frac{\partial_{t}^{2}\xi}{\rho\rho_{0}}u\cdot \mathcal{N})_{1,\Sigma}&\lesssim \vert \vert \partial_{t}^{2}\xi\vert \vert_{H^{1}}\vert \vert u\vert \vert_{L^{+\infty}(\Sigma)}\vert \vert \partial_{t}^{2}\xi\vert \vert_{H^{1}}+\vert \vert \partial_{t}^{2}\xi\vert \vert_{W^{1,\frac{1}{\alpha}}}\vert \vert \partial_{t}^{2}\xi\vert \vert_{L^{\infty}}\vert \vert u\vert \vert_{W^{1,\frac{1}{1-\epsilon_{+}}}} \notag\\
      &\lesssim \|\p_{t}^{2}\xi\|_{H^{\frac{3}{2}-\alpha}}^{2}\|u\|_{W^{2,q_{+}}} \lesssim \mathcal{E}^{\frac{1}{2}}\mathcal{D} \label{equ:5.1.206}
   \end{align}

   \textbf{Term $\frac{(\partial_{t}\xi)^{2}}{\rho^{2}\rho_{0}}u\cdot \mathcal{N}$} 
   
   We have the following estimate for this term

   \begin{align}
       (\partial_{t}^{2}\xi,\frac{(\partial_{t}\xi)^{2}}{\rho^{2}\rho_{0}}u\cdot \mathcal{N})_{1,\Sigma}\lesssim& \vert \vert \partial_{t}^{2}\xi\vert \vert_{H^{1}}\vert \vert \partial_{t}\xi\vert \vert_{W^{1,+\infty}}^{2}\vert \vert u\vert \vert_{L^{\infty}(\Sigma)}+\vert \vert \partial_{t}^{2}\xi\vert \vert_{W^{1,\frac{1}{\alpha}}}\vert \vert \partial_{t}\xi\vert \vert^{2}_{L^{\infty}}\vert\vert u\vert\vert_{W^{1,\frac{1}{1-\epsilon_{+}}}(\Sigma)}\notag\\ 
      &\lesssim \|\p_{t}^{2}\xi\|_{H^{\frac{3}{2}-\alpha}}\|u\|_{W^{2,q_{+}}}\|\p_{t}\xi\|^{2}_{H^{\frac{3}{2}+\frac{\epsilon_{-}-\alpha}{2}}} \lesssim \mathcal{D}\mathcal{E}^{\frac{1}{2}}\label{equ:5.1.207}
   \end{align}

   \textbf{Term $(\frac{\partial_{t}^{2}\xi\xi }{\rho^{2}\rho_{0}})u\cdot \mathcal{N}$ and $\frac{\eta(\partial_{t}\xi)^{2}}{\rho^{3}\rho}u\cdot \mathcal{N}$} 
   
   The estimate for these two terms follow from \eqref{equ:5.1.206} and \eqref{equ:5.1.207}. We omit the details here.

   \textbf{Term $\frac{\xi}{\rho\rho_{0}}u\cdot \p_{t}\mathcal{N}^{2}$}

  This estimate is given in \eqref{equ:int_p}.

  \textbf{Term $\frac{(\p_{t}\xi)\xi}{\rho^{2}\rho_{0}}u\cdot \p_{t}\mathcal{N}$}

  We have the estimate:

  \begin{align}
      (\p_{t}^{2}\xi,\frac{(\p_{t}\xi)\xi}{\rho^{2}\rho_{0}}u\cdot \p_{t}\mathcal{N})_{1,\Sigma}&\lesssim \|\p_{t}^{2}\xi\|_{H^{1}}\|\p_{t}\xi\|^{2}_{W^{1,+\infty}}\|\xi\|_{W^{1,+\infty}}\|u\|_{L^{\infty}(\Sigma)}+\|\p_{t}^{2}\xi\|_{H^{\frac{3}{2}-\alpha}}\|\p_{t}\xi\|^{2}_{W^{1,\infty}}\|\xi\|_{W^{1,\infty}}\|u\|_{W^{1,\frac{1}{1-\epsilon_{+}}}}\notag\\
      &\quad+\|\p_{t}^{2}\xi\|_{H^{\frac{3}{2}-\alpha}}\|\p_{t}\xi\|_{W^{1,\infty}}\|\xi\|_{W^{1,\infty}}\|u\|_{L^{\infty}}\|\p_{t}\xi\|_{W^{2,\frac{1}{1-\epsilon_{+}}}}\notag\\
      &\lesssim\|\p_{t}^{2}\xi\|_{H^{1}}\|\p_{t}\xi\|_{H^{\frac{3}{2}+\frac{\epsilon_{-}-\alpha}{2}}}^{2}\|\xi\|_{W^{3-\frac{1}{q_{+}},q_{+}}}\|u\|_{W^{2,q_{+}}}+\|\p_{t}^{2}\xi\|_{H^{\frac{3}{2}-\alpha}}\|\p_{t}\xi\|_{H^{\frac{3}{2}+\frac{\epsilon_{-}-\alpha}{2}}}\|\xi\|_{W^{3-\frac{1}{q_{+}},q_{+}}}\|u\|_{W^{2,q_{+}}}\notag\\
      &\lesssim\mathcal{E}^{\frac{1}{2}}\mathcal{D}
  \end{align}

  \textbf{Term $\frac{\p_{t}\xi}{\rho^{2}\rho_{0}}\p_{t}u\cdot \mathcal{N}$}

  We have the computation similar to the estimate for term above

  \begin{align}
      (\p_{t}^{2}\xi,\frac{\p_{t}\xi}{\rho^{2}\rho_{0}}\p_{t}u\cdot \mathcal{N})_{1,\Sigma}\lesssim \|\p_{t}^{2}\xi\|_{H^{1}}\|\p_{t}\xi\|_{H^{1}}\|\p_{t}u\|_{W^{2,q_{+}}}+\|\p_{t}^{2}\xi\|_{H^{\frac{3}{2}-\alpha}}\|\p_{t}\xi\|_{H^{\frac{3}{2}+\frac{\epsilon_{-}-\alpha}{2}}}\|\p_{t}u\|_{W^{2,q_{-}}}\lesssim \mathcal{E}^{\frac{1}{2}}\mathcal{D}
  \end{align}
  
   \textbf{Term $(\partial_{t}^{2}\xi,\mathfrak{n}'(t)\partial_{t}^{2}(\frac{\rho'}{\rho}))_{1,\Sigma}$}

   We have the following computation :

   \begin{align}
       (\partial_{t}^{2}\xi,\mathfrak{n}'(t)\partial_{t}^{2}(\frac{\rho'}{\rho}))_{1,\Sigma}=&\mathfrak{n}'(t)(\partial_{t}^{2}\xi,\partial_{t}(\frac{\rho\partial_{t}\xi'-\rho'\partial_{t}\xi}{\rho^{2}}))_{1,\Sigma}=\mathfrak{n}'(t)(\partial_{t}^{2}\xi,\frac{\rho\partial_{t}^{2}\xi'-\rho'\partial_{t}^{2}\xi}{\rho^{2}}-2\frac{(\rho\partial_{t}\xi'-\rho'\partial_{t}\xi)\partial_{t}\xi}{\rho^{3}})_{1,\Sigma}\notag\\
       &=\mathfrak{n}'(t)(\partial_{t}^{2}\xi,\frac{1}{\rho}\partial_{t}^{2}\xi')_{1,\Sigma}+\mathfrak{n}'(t)(\partial_{t}^{2}\xi,-\frac{\rho'\partial_{t}^{2}\xi}{\rho^{2}}-2\frac{(\rho\partial_{t}\xi'-\rho'\partial_{t}\xi)\partial_{t}\xi}{\rho^{3}})_{1,\Sigma}=I_{1}+I_{2}
   \end{align}

   \noindent For the term $I_{2}$, we have the following estimate using Theorem \ref{thm:gam}

   \begin{align}
       I_{2}=&\mathfrak{n}'(t)(\partial_{t}^{2}\xi,-\frac{\rho'\partial_{t}^{2}\xi}{\rho^{2}}-2\frac{(\rho\partial_{t}\xi'-\rho'\partial_{t}\xi)\partial_{t}\xi}{\rho^{3}})_{1,\Sigma}\notag\\
       &\lesssim \vert \vert u\vert \vert_{H^{1}}\vert \vert \partial_{t}^{2}\xi\vert \vert_{H^{1}}^{2}
       +\vert \vert u\vert \vert_{H^{1}}\vert \vert \partial_{t}\xi\vert \vert_{W^{1,+\infty}}\vert \vert \partial_{t}\xi\vert \vert_{W^{2,\frac{1}{1-\epsilon_{-}}}}\vert \vert \partial_{t}^{2}\xi\vert \vert_{W^{1,\frac{1}{\epsilon_{-}}}}\lesssim \mathcal{E}^{\frac{1}{2}}\mathcal{D}
   \end{align}

   \noindent Then for the term $I_{1}$, using spatial integration by part, we have

   \begin{align}
       I_{1}=\mathfrak{n}'(t)(\partial_{t}^{2}\xi,\frac{1}{\rho}\partial_{t}^{2}\xi')_{1,\Sigma}\lesssim \vert \vert u\vert \vert_{H^{1}}\vert \vert \partial_{t}^{2}\xi\vert \vert_{H^{1}}^{2}+\vert \vert u\vert \vert_{H^{1}}\vert [\partial_{t}^{2}\xi']_{\theta}\vert^{2}\lesssim \mathcal{E}^{\frac{1}{2}}\mathcal{D}
   \end{align}
   
   \textbf{Term $\mathfrak{n}''(t)\partial_{t}(\frac{\rho'}{\rho})\cos\theta$}

   We have the computation:

   \begin{align}
       (\partial_{t}^{2}\xi,\mathfrak{n}''(t)\partial_{t}(\frac{\rho'}{\rho})\cos\theta)_{1,\Sigma}\lesssim \mathcal{E}^{\frac{1}{2}}\vert \vert \partial_{t}^{2}\xi\vert \vert_{W^{1,\frac{1}{\epsilon_{-}}}}\vert \vert \partial_{t}\xi\vert\vert_{W^{2,\frac{1}{1-\epsilon_{-}}}}\lesssim \mathcal{E}^{\frac{1}{2}}\mathcal{D}
   \end{align}

   \textbf{Term $\mathfrak{n}'''(t)(\frac{\rho'}{\rho}-\frac{\rho_{0}'}{\rho_{0}})$}. We have the following computation by Theorem \ref{thm:gam}

   \begin{align}
      (\partial_{t}^{2}\xi,\mathfrak{n}'''(t)(\frac{\rho'}{\rho}-\frac{\rho_{0}'}{\rho_{0}}))_{1,\Sigma}\lesssim \mathcal{D}^{\frac{1}{2}}\vert \vert \partial_{t}^{2}\xi\vert \vert_{W^{1,\frac{1}{\epsilon_{+}}}}\vert \vert \xi\vert \vert_{W^{2,\frac{1}{1-\epsilon_{+}}}}\lesssim \mathcal{E}^{\frac{1}{2}}\mathcal{D}
   \end{align}

   Combining all of the estimates above we obtain equation\eqref{equ:H} by setting $H=\frac{1}{2}\p_{t}(\p_{t}^{2}\xi,\frac{\xi}{\rho_{0}}\p_{t}^{2}\xi)_{1,\Sigma}$
   
   \textbf{Step 8} We estimate the term $a_{2}\mathfrak{n}'''(t)(\rho_{0},\xi_s)_{1,\Sigma}$ in this step. We have the following computation using Theorem \ref{thm:gam}.

   \begin{align}
       a_{2}\mathfrak{n}'''(t)(\rho_{0},\xi_2)\lesssim a_{2}\mathcal{D}^{\frac{1}{2}} \label{equ:5.1.208}
   \end{align}

   \noindent By definition of $a_{2}$in \eqref{equ:a} and Theorem \ref{thm:pos}, we obtain the following estimate

   \begin{align}
       a_{2}=\frac{\int_{0}^{\pi}\partial_{tt}\xi\rho_{0}d\theta}{\int_{0}^{\pi}\rho_{0}^{2}d\theta}=\frac{-\int_{0}^{\pi}\partial_{tt}\xi\xi d\theta-\int_{0}^{\pi}(\partial_{t}\xi)^{2}d\theta}{\int_{0}^{\pi}\rho_{0}^{2}d\theta}\lesssim \vert \vert \partial_{t}^{2}\xi\vert \vert_{L^{2}}\vert \vert \xi\vert \vert_{L^{2}}+\vert \vert \partial_{t}\xi\vert \vert^{2}_{L^{2}}\lesssim \mathcal{E} \label{equ:5.1.210}
   \end{align}

   \noindent Then applying equation \eqref{equ:5.1.210} to equation \eqref{equ:5.1.208}, we obtain

   \begin{equation}{\label{equ:5.1.211}}
       a_{2}\mathfrak{n}'''(t)(\rho_{0},\xi_2)\lesssim \mathcal{E}^{\frac{1}{2}}\mathcal{D}
   \end{equation}

   \textbf{Step 9} In this step, we estimate the following term:

   \begin{align}
       \int_{0}^{\pi}b^{4,2}(v\cdot \mathcal{N})
   \end{align}

   \noindent Our goal is to show the following result
   \begin{align}
        \int_{0}^{\pi}b^{4,2}(v\cdot \mathcal{N})\lesssim \mathcal{E}^{\frac{1}{2}}\mathcal{D}^{\frac{1}{2}}\|v\|_{H^{1}}
   \end{align}
   
   \noindent We write down the definition of $b^{4,2}$ and each term included in $b^{4,2}$ individually.

   \textbf{Term $\partial_{tt}\mathcal{R}_{2}$}:

   We have 
   \begin{align}\
       \int_{0}^{\pi}(\partial_{tt}\mathcal{R}_{2})(v\cdot \mathcal{N})\lesssim \int_{0}^{\pi} ((\vert \partial_{\theta}\xi\vert+|\xi|)^{2}(\vert \partial_{\theta}\partial_{t}^{2}\xi\vert+|\p_{t}^{2}\xi|^{2})+(\vert \partial_{\theta}\partial_{t}\xi\vert+|\p_{t}\xi|)^{2} )\vert v\vert\notag\\
       \lesssim \vert \vert v\vert \vert_{H^{1}}\vert\vert  \xi\vert\vert_{W^{1,+\infty}}\vert \vert \partial_{t}^{2}\xi\vert \vert_{H^{1}}+\vert \vert v\vert \vert_{H^{1}}\vert \vert \partial_{t}\xi
       \vert\vert^{2}_{W^{1,4}}\lesssim \mathcal{E}^{\frac{1}{2}}\mathcal{D}^{\frac{1}{2}}\|v\|_{H^{1}} \label{equ:5.1.212}
   \end{align}

   \textbf{Term $\partial_{t}(\mathbb{D}_{\partial_{t}\mathcal{A}}u\mathcal{N})$}

   We have the following computation:

   \begin{align}
       \int_{0}^{\pi} \partial_{t}(\mathbb{D}_{\partial_{t}\mathcal{A}}u\mathcal{N}) \cdot v=&\int_{0}^{\pi}\mathbb{D}_{\partial_{t}^2{\mathcal{A}}}u\mathcal{N}\cdot v+\int_{0}^{\pi} \mathbb{D}_{\partial_{t}\mathcal{A}}\partial_{t}u\mathcal{N}\cdot v+\int_{0}^{\pi} \mathbb{D}_{\partial_{t}\mathcal{A}}u\partial_{t}\mathcal{N}\cdot v\notag\\
       &\lesssim \vert \vert \partial_{t}^{2}\xi\vert \vert_{W^{1,\frac{1}{\alpha}}}\vert \vert u\vert \vert_{W^{1,\frac{1}{1-\epsilon_{+}}}}\vert \vert v\vert \vert_{L^{\frac{1}{\epsilon_{+}-\alpha}}}+\vert \vert \partial_{t}\xi\vert\vert_{W^{1,+\infty}}\vert \vert \partial_{t}u\vert \vert_{W^{1,\frac{1}{1-\epsilon_{-}}}}\vert \vert v\vert \vert_{L^{\frac{1}{{\epsilon_{+}}}}}\notag\\
       &+\vert \vert \partial_{t}\xi\vert \vert^{2}_{W^{1,+\infty}}\vert \vert u\vert \vert_{W^{1,\frac{1}{1-\epsilon_{+}}}}\vert \vert v\vert \vert_{\frac{1}{\epsilon_{+}}}\notag\\
       &\lesssim \|\p_{t}^{2}\xi\|_{H^{\frac{3}{2}-\alpha}}(\|u\|_{W^{2,q_{+}}}+\|\p_{t}u\|_{H^{1+\frac{\epsilon_{-}}{2}}})\|v\|_{H^{1}}+\|\p_{t}\xi\|_{H^{\frac{3}{2}+\frac{\epsilon_{-}-\alpha}{2}}}^{2}\|u\|_{W^{2,q_{+}}}\|v\|_{H^{1}}\lesssim \mathcal{E}^{\frac{1}{2}}\mathcal{D}^{\frac{1}{2}}\|v\|_{H^{1}} \label{equ:5.1.213}
   \end{align}

   \textbf{Term $\partial_{t}((\mathcal{K}(\xi)-\partial_{\theta}b^{3})\partial_{t}\mathcal{N})$} 

   We have the computation using the definition of $b^{3}$
   \begin{align}
       \int_{0}^{\pi} \partial_{t}((\mathcal{K}(\xi)-\partial_{\theta}b^{3})\partial_{t}\mathcal{N})\cdot v\lesssim& \vert \vert v\vert \vert_{L^{\frac{2}{1+\epsilon_{-}}}}\vert \vert \partial_{t}\xi\vert \vert_{W^{2,\frac{2}{1-\epsilon_{-}}}}\vert \vert \partial_{t}\xi\vert \vert_{W^{1,+\infty}}+\vert \vert \partial_{t}^{2}\xi\vert \vert_{W^{1,\frac{1}{\alpha}}}\vert \vert \xi\vert \vert_{W^{2,\frac{2}{1-\epsilon_{+}}}}\vert \vert v\vert \vert_{L^{\frac{2}{1+\epsilon_{+}-\alpha}}}\vert \vert \xi\vert \vert_{W^{1,\infty}}\notag\\
       \lesssim&\|v\|_{H^{1}}\|\p_{t}\xi\|_{H^{\frac{3}{2}+\frac{\epsilon_{-}-\alpha}{2}}}\|\p_{t}\xi\|_{W^{3-\frac{1}{q_{-}},q_{-}}}+\|\p_{t}^{2}\xi\|_{H^{\frac{3}{2}-\alpha}}\|\xi\|^{2}_{W^{3-\frac{1}{q_{+}},q_{+}}}\|v\|_{H^{1}}\lesssim \mathcal{E}^{\frac{1}{2}}\mathcal{D}^{\frac{1}{2}}\|v\|_{H^{1}} \label{equ:5.1.214}
   \end{align}

   \textbf{Term $\partial_{t}(S_{\mathcal{A}}(p,u)\partial_{t}\mathcal{N})$} 
   
   We have the following estimate

   \begin{align}
       \int_{0}^{\pi}\partial_{t}(S_{\mathcal{A}}(p,u)\partial_{t}\mathcal{N})\cdot v&\lesssim \vert \vert \partial_{t}\xi\vert \vert^{2}_{W^{1,+\infty}}(\vert \vert p\vert \vert_{L^{\frac{1}{1-\epsilon_{+}}}(\Sigma)}+\vert \vert u\vert \vert_{W^{1,\frac{1}{1-\epsilon_{+}}}(\Sigma)})\vert \vert v\vert\vert_{L^{\frac{1}{1+\epsilon_{+}}}(\Sigma)}\notag\\
       &\quad+(\vert \vert\partial_{t}p \vert \vert_{L^{\frac{2}{1-\epsilon_{-}}}(\Sigma)}+\vert \vert \partial_{t}u\vert \vert_{W^{1,\frac{2}{1-\epsilon_{-}}}(\Sigma)})\vert \vert \partial_{t}{\xi}\vert\vert_{W^{1,+\infty}}\vert \vert v\vert \vert_{L^{\frac{2}{1+\epsilon_{-}}}(\Sigma)}\notag\\
       &\quad+(\vert \vert p\vert \vert_{L^{\frac{2}{1-\epsilon_{+}}}(\Sigma)}+\vert \vert u\vert \vert_{W^{1,\frac{1}{1-\epsilon_{+}}}(\Sigma)})\vert \vert \partial_{t}^{2}\xi\vert \vert_{W^{1,\frac{1}{\alpha}}}\vert \vert v\vert \vert_{L^{\frac{2}{1-\epsilon_{+}-2\alpha}}(\Sigma)}\notag\\
       &\lesssim\vert \vert \partial_{t}\xi\vert \vert^{2}_{H^{\frac{3}{2}+\frac{\epsilon_{-}-\alpha}{2}}}(\vert \vert p\vert \vert_{W^{1,q_{+}}(\Sigma)}+\vert \vert u\vert \vert_{W^{2,q_{+}}(\Sigma)})\vert \vert v\vert\vert_{H^{1}}\notag\\
       &+(\vert \vert\partial_{t}p \vert \vert_{W^{1,q_{-}}(\Sigma)}+\vert \vert \partial_{t}u\vert \vert_{W^{2,q_{-}}})\vert \vert \partial_{t}{\xi}\vert\vert_{H^{\frac{3}{2}+\frac{\epsilon_{-}-\alpha}{2}}}\vert \vert v\vert \vert_{H^{1}}\notag\\
       &+(\vert \vert p\vert \vert_{W^{1,q_{+}}}+\vert \vert u\vert \vert_{W^{2,q_{-}}})\vert \vert \partial_{t}^{2}\eta\vert \vert_{H^{\frac{3}{2}-\alpha}}\vert \vert v\vert \vert_{H^{1}}\notag\\
       &\lesssim \mathcal{E}^{\frac{1}{2}}\mathcal{D}^{\frac{1}{2}}\|v\|_{H^{1}}\label{equ:5.1.215}
   \end{align}

   \textbf{Term $(\mathcal{K}(\partial_{t}\xi)-\partial_{\theta}b^{3,1})\partial_{t}\mathcal{N}$}.
   
   We have the following computation for this term
   
   \begin{align}
       \int_{0}^{\pi}(\mathcal{K}(\partial_{t}\xi)-\partial_{\theta}b^{3,1})\partial_{t}\mathcal{N}\cdot v=\int_{0}^{\pi} \mathcal{K}(\partial_{t}\xi)\partial_{t}\mathcal{N}\cdot v-\partial_{\theta}b^{3,1}\partial_{t}\mathcal{N}\cdot v=I_{1}+I_{2} \label{equ:5.1.216}
   \end{align}

   \noindent We then estimate $I_{1}$ and $I_{2}$ individually. We have:

   \begin{align}
       I_{1}\lesssim \vert \vert \partial_{t}\xi\vert \vert_{W^{1,+\infty}}\vert \vert \partial_{t}\xi\vert \vert_{W^{2,\frac{1}{1-\epsilon_{-}}}}\vert \vert v\vert \vert_{L^{\frac{1}{\epsilon_{-}}}(\Sigma)}\lesssim \vert \vert \partial_{t}\xi\vert \vert_{W^{1,+\infty}}\vert \vert \partial_{t}\xi\vert \vert_{W^{3-\frac{1}{q_{-}},q_{-}}}\vert \vert v\vert \vert_{H^{1}}\lesssim \mathcal{E}^{\frac{1}{2}}\mathcal{D} ^{\frac{1}{2}}\|v\|_{H^{1}}\label{equ:5.1.217},
   \end{align}

   \noindent and

   \begin{align}
       I_{2}=&\int_{0}^{\pi} \partial_{\theta}(\partial_{b}\mathcal{R}\partial_{t}\partial_{\theta}\xi+\partial_{c}\mathcal{R}\partial_{t}\xi)\partial_{t}\mathcal{N}\cdot vd\theta\notag\\
       =& \int_{0}^{\pi} (\partial_{ab}\mathcal{R}\partial_{t}\partial_{\theta}\xi+\partial_{bb}\mathcal{R}\partial_{t}\partial_{\theta}\xi\partial_{\theta}^{2}\xi+\partial_{bc}\mathcal{R}\partial_{\theta}\xi\partial_{t\theta}\xi+\partial_{b}\mathcal{R}\partial_{t}\partial^{2}_{\theta}\xi)\partial_{t}\mathcal{N}\cdot vd\theta\notag\\
       &+\int_{0}^{\pi}(\partial_{ac}\mathcal{R}\partial_{t}\xi+\partial_{bc}\mathcal{R}\partial_{t}\xi\partial_{\theta}^{2}\xi+\partial_{cc}\mathcal{R}\partial_{t}\partial_{\theta}\xi+\partial_{c}\mathcal{R}\partial_{t}\partial_{\theta}\xi) \partial_{t}\mathcal{N}\cdot vd\theta \notag\\
       \lesssim& \vert \vert \xi\vert \vert_{W^{1,+\infty}}\vert \vert \partial_{t}\xi\vert \vert^{2}_{W^{1,+\infty}}\vert \vert v\vert \vert_{L^{2}}+\vert \vert \xi\vert \vert_{W^{2,\frac{1}{1-\epsilon_{+}}}}\vert \vert \partial_{t}\xi\vert \vert^{2}_{W^{1,+\infty}}\vert \vert v\vert  \vert_{L^{\frac{1}{\epsilon_{+}}}}\notag\\
       &+\vert \vert \xi\vert \vert_{W^{1,\infty}}\vert \vert \partial_{t}\xi\vert \vert_{W^{1,\infty}}\vert \vert\partial_{t}\xi \vert \vert_{W^{2,\frac{1}{1-\epsilon_{-}}}}\vert \vert v\vert \vert _{L^{\frac{1}{\epsilon_{-}}}}\lesssim \mathcal{E}\mathcal{D}^{\frac{1}{2}}\|v\|_{H^{1}} \label{equ:5.1.218}.
   \end{align}

   \noindent Combining equation \eqref{equ:5.1.127} and \eqref{equ:5.1.128}, we obtain the following estimate:

   \begin{align}
        \int_{0}^{\pi}(\mathcal{K}(\partial_{t}\xi)-\partial_{\theta}b^{3,1})\partial_{t}\mathcal{N}\cdot v\lesssim \mathcal{E}^{\frac{1}{2}}\mathcal{D}^{\frac{1}{2}}\|v\|_{H^{1}}
   \end{align}
   
   \textbf{Term $S_{\mathcal{A}}(\partial_{t}p,\partial_{t}u)\partial_{t}\mathcal{N}$}

   We have the following estimate
   \begin{align}
       \int_{0}^{\pi}S_{\mathcal{A}}(\partial_{t}p,\partial_{t}u)\partial_{t}\mathcal{N}\cdot v&\lesssim \vert \vert \partial_{t}\eta\vert \vert_{W^{1,+\infty}} (\vert \vert \partial_{t}p\vert \vert_{L^{\frac{1}{1-\epsilon_{-}}}}+\vert \vert \partial_{t}u\vert \vert_{W^{1,\frac{1}{1-\epsilon_{-}}}(\Sigma)}) \vert \vert v\vert \vert_{L^{\frac{1}{\epsilon_{-}}}(\Sigma)}\notag\\
       &\lesssim\vert \vert \partial_{t}\xi\vert \vert_{H^{\frac{3}{2}+\frac{\epsilon_{-}-\alpha}{2}}} (\vert \vert \partial_{t}p\vert \vert_{W^{1,q_{-}}}+\vert \vert \partial_{t}u\vert \vert_{W^{2,q_{-}}}) \vert \vert v\vert \vert_{H^{1}} \lesssim \mathcal{E}^{\frac{1}{2}}\mathcal{D}^{\frac{1}{2}}\|v\|_{H^{1}} \label{equ:5.1.219}
   \end{align}

   \textbf{Term $\mu \mathbb{D}_{\partial_{t}\mathcal{A}}\partial_{t}u\mathcal{N}$}

   Using the similar computation as above, we have:
   
   \begin{align}
       \int_{0}^{\pi} \mu \mathbb{D}_{\partial_{t}\mathcal{A}}\partial_{t}u\mathcal{N} \cdot v\lesssim \vert \vert \partial_{t}\xi\vert \vert_{W^{1,+\infty}}\vert \vert \partial_{t}u\vert \vert_{W^{1,\frac{1}{1-\epsilon_{-}}}(\Sigma)}\vert \vert v\vert \vert_{L^{\frac{1}{\epsilon_{-}}}(\Sigma)}\lesssim \mathcal{E}^{\frac{1}{2}}\mathcal{D}^{\frac{1}{2}}\|v\|_{H^{1}} \label{equ:5.1.220}
   \end{align}

   \textbf{Step 10} We estimate terms included in $b^{5,2}$ on the boundary.

   Using the definition of $b^{5,2}$, we obtain
   \begin{align}
       \int_{\Sigma_{s}}b^{5,2}v=&2\int_{\Omega} \mu\mathbb{D}_{\partial_t\mathcal{A}}\partial_{t}u\nu\cdot \tau v+\int_{\Omega} \mu \mathbb{D}_{\partial_{t}^{2}\mathcal{A}}u\nu\cdot \tau\notag\\
       \lesssim& \vert \vert \partial_{t}\xi\vert \vert_{W^{1,\infty}}\vert \vert \partial_{t}u\vert \vert_{W^{1,\frac{1}{1-\epsilon_{-}}}(\Sigma_{s})}\vert \vert v\vert \vert_{L^{\frac{1}{\epsilon_{-}}}}+ \vert \vert \partial_{t}^{2}\xi\vert \vert_{W^{1,\frac{1}{\alpha}}}\vert \vert u\vert \vert_{W^{1,\frac{1}{1-\epsilon_{+}}}(\Sigma_{s})}\vert \vert v\vert \vert_{L^{\frac{1}{\epsilon_{+}-\alpha}}}\notag\\
       \lesssim&\|\p_{t}\xi\|_{H^{\frac{3}{2}+\frac{\epsilon_{-}-\alpha}{2}}}\|\p_{t}u\|_{W^{2,q_{-}}}\|v\|_{H^{1}}+\|\p_{t}^{2}\xi\|_{H^{\frac{3}{2}-\alpha}}\|u\|_{W^{2,q_{+}}}\|v\|_{H^{1}}\lesssim\mathcal{E}^{\frac{1}{2}}\mathcal{D}^{\frac{1}{2}}\|v\|_{H^{1}} \label{equ:5.1.221}
   \end{align}

   \textbf{Step 11} We estimate the term including $b^{3,2}$. 

   We aim to estimate the following integral
   
   \begin{align}
       \int_{0}^{\pi}b^{3,2}\partial_{\theta}(\p_{t}^{2}u\cdot \mathcal{N})d\theta \label{equ:5.1.226}
   \end{align}

   \noindent Using the definition of $b^{3,2}$, we have:

   \begin{align}
       b^{3,2}=\partial_{tt}\mathcal{R}_{2}=\partial_{bb}\mathcal{R} (\partial_{t}\partial_{\theta}\xi)^{2}+\partial_{b}\mathcal{R}(\partial_{t}^{2}\partial_{\theta}\xi)+\partial_{cc}\mathcal{R}(\partial_{t}\eta)^{2}+\partial_{c}\mathcal{R}(\partial_{t}^{2}\xi) \label{equ:5.1.227}
   \end{align}

   \noindent We notice that for the last two terms in equation \eqref{equ:5.1.227} has better regularity than the first two terms. Hence, it suffices to analyze the first two terms in this step. Applying the kinematic boundary condition

   \begin{align}
       \p_{t}^{2}u\cdot \mathcal{N}=\rho\partial_{ttt}\xi+\rho\mathfrak{n}'''(t)\xi_{s}+\rho\tilde{b}_{8,2}- u\cdot \p_{t}^{2}\mathcal{N}
   \end{align}

   \noindent we have

   \begin{align}
        \int_{0}^{\pi}b^{3,2}\partial_{\theta}(\frac{1}{\rho_{0}}\p_{t}^{2}u\cdot \mathcal{N})d\theta=&\int_{0}^{\pi} \partial_{bb}\mathcal{R} (\partial_{t}\partial_{\theta}\xi)^{2}\partial_{\theta}(\rho\partial_{ttt}\xi+\rho\mathfrak{n}'''(t)\xi_{s}+\rho\tilde{b}_{8,2}-\frac{1}{\rho_{0}}u\cdot \p_{t}^{2}\mathcal{N})\notag\\
        &+\int_{0}^{\pi}\partial_{b}\mathcal{R}(\partial_{t}^{2}\partial_{\theta}\xi)\partial_{\theta}(\rho\partial_{ttt}\xi+\rho\mathfrak{n}'''(t)\xi_{s}+\rho\tilde{b}_{8,2}-\frac{1}{\rho_{0}}u\cdot \p_{t}^{2}\mathcal{N}) \notag
        =\Sigma_{i=1}^{8}I_{i} \label{equ:5.1.228}.
   \end{align}

   \noindent We estimate each term in \eqref{equ:5.1.228} individually.

   \textbf{Term $I_{1}$}:

   We have the following estimate for $I_{1}$
   \begin{align}
      I_{1}=&\partial_{t}(\int_{0}^{\pi}\partial_{bb}\mathcal{R}(\partial_{t}\partial_{\theta}\xi)^{2}\rho\partial_{tt\theta}\xi)-\int_{0}^{\pi}\partial_{t}\partial_{bb}\mathcal{R}(\partial_{t}\partial_{\theta}\xi)^{2}\rho\partial_{tt\theta}\xi)\notag\\
      &-2\int_{0}^{\pi}\partial_{bb}\mathcal{R}\partial_{t}^{2}\partial_{\theta}\xi\partial_{t}\partial_{\theta}\xi \rho\partial_{tt\theta}\xi-\int_{0}^{\pi}\partial_{bb}\mathcal{R}(\partial_{t}\partial_{\theta}\xi)^{2}\partial_{t}\xi\partial_{tt\theta}\xi+\int_{0}^{\pi} \partial_{bb}\mathcal{R}(\partial_{t}\partial_{\theta}\xi)^{2}\partial_{\theta}\rho\partial_{t}^{3}\xi,
   \end{align}

   \noindent which implies that:

   \begin{align}
       &\vert I_{1}-\partial_{t}(\int_{0}^{\pi}\partial_{bb}\mathcal{R}(\partial_{t}\partial_{\theta}\xi)^{2}\rho\partial_{tt\theta}\xi)\vert\notag\\
       &\lesssim \int_{0}^{\pi}\vert\partial_{t}\partial_{bb}\mathcal{R}(\partial_{t}\partial_{\theta}\xi)^{2}\rho\partial_{tt\theta}\xi\vert+2\int_{0}^{\pi}\vert\partial_{bb}\mathcal{R}\partial_{t}^{2}\partial_{\theta}\xi\partial_{t}\partial_{\theta}\xi \rho\partial_{tt\theta}\xi\vert\notag\\
      &\quad+\int_{0}^{\pi}\vert\partial_{bb}\mathcal{R}(\partial_{t}\partial_{\theta}\eta)^{2}\partial_{t}\xi\partial_{tt\theta}\xi\vert+\int_{0}^{\pi} \vert\partial_{bb}\mathcal{R}(\partial_{t}\partial_{\theta}\xi)^{2}\partial_{\theta}\rho\partial_{t}^{3}\xi\vert\\
      &\lesssim \vert \vert \partial_{t}\xi\vert \vert_{W^{1,+\infty}}\vert \vert \partial_{t}^{2}\xi\vert \vert_{W^{1,\frac{1}{\alpha}}}\vert \vert \partial_{t}^{2}\xi\vert \vert_{H^{1}}+\vert \vert \partial_{t}\xi\vert \vert_{W^{1,+\infty}}\vert \vert \partial_{t}^{2}\xi\vert\vert^{2}_{H^{1}}\notag\\
      &\quad+\vert \vert \partial_{t}\xi\vert \vert_{W^{1,+\infty}}^{3}\vert \vert \partial_{t}^{2}\xi\vert \vert_{H^{1}}+\vert \vert \partial_{t}\xi\vert \vert_{W^{1,+\infty}}^{2}\vert\vert \partial_{t}^{3}\xi\vert \vert_{L^{2}}\lesssim \mathcal{E}^{\frac{3}{2}}+\mathcal{E}\mathcal{D}^{\frac{1}{2}}\label{equ:5.1.229}
   \end{align}

   \textbf{Term $I_{2}$}
   
    For the term $I_{2}$, we have:
   
   \begin{align}
       I_{2}=&\int_{0}^{\pi}\partial_{bb}\mathcal{R}(\partial_{t}\partial_{\theta}\xi)^{2}\partial_{\theta}\rho \mathfrak{n}'''(t)\xi_{s}d\theta+\int_{0}^{\pi}\partial_{bb}\mathcal{R}(\partial_{t}\partial_{\theta}\xi)^{2} \rho\partial_{\theta}\xi_{s}\mathfrak{n}'''(t)\notag\\
       \lesssim& \vert \vert \partial_{t}\xi\vert \vert^{2}_{W^{1,\infty}}(\mathcal{E}+\|\p_{t}^{2}u\|_{H^{1}})\lesssim \mathcal{E}\vert \vert \p_{t}^{2}u\vert \vert_{H^{1}}+\mathcal{E}^{2}\lesssim \mathcal{E}\mathcal{D}^{\frac{1}{2}}\label{equ:5.1.230}
   \end{align}

   \textbf{Term $I_{3}$}
   
    For the term $I_{3}$, we have

   \begin{align}
       I_{3}=\int_{0}^{\pi} \partial_{bb}\mathcal{R}(\partial_{t}\partial_{\theta}\xi)^{2}\partial_{\theta} \rho \tilde{b}_{8,2}+\int_{0}^{\pi} \partial_{bb}\mathcal{R}(\partial_{t}\partial_{\theta}\xi)^{2} \rho \partial_{\theta}\tilde{b}_{8,2}\label{equ:5.1.231}
   \end{align}

   \noindent To complete the estimate, we need to bound the term $\vert \vert \tilde{b}^{8,2}\vert \vert_{H^{1}}$. Using equation \eqref{equ:5.1.200} together with its corresponding estimate, we obtain

   \begin{align}
       \vert \vert \tilde{b}^{8,2}\vert \vert_{W^{1,\frac{1}{1-\alpha}}}\lesssim \mathcal{D}^{\frac{1}{2}}\mathcal{E}^{\frac{1}{2}} \label{equ:5.1.232}
   \end{align}

   \noindent Substituting equation \eqref{equ:5.1.232} into \eqref{equ:5.1.231}, we deduce that:
   
   \begin{align}
       I_{3}\lesssim \vert \vert \partial_{t}\xi \vert \vert^{2}_{W^{1,+\infty}}\mathcal{E}\mathcal{D}^{\frac{1}{2}}\lesssim \mathcal{E}^{\frac{3}{2}}\mathcal{D}^{\frac{1}{2}}
   \end{align}

   \textbf{Term $I_{4}$ and Term $I_{8}$}
    For the terms $I_{4}$ and $I_{8}$,  we apply integration by parts in the same manner as in \eqref{equ:int_p} to obtain

   \begin{align}
       I_{4}+I_{8}\lesssim \mathcal{E}^{\frac{1}{2}}\mathcal{D}
   \end{align}

   \textbf{Term $I_{5}$}
   
   For the term $I_{5}$, we have the estimate as follows
   
   \begin{align}
       I_{5}&=\int_{0}^{\pi} \partial_{b}\mathcal{R}(\partial_{t}^{2}\partial_{\theta}\xi)(\rho\partial_{\theta}\partial_{t}^{3}\xi)\rho d\theta+\int_{0}^{\pi}\partial_{b}\mathcal{R}(\partial_{t}^{2}\partial_{\theta}\xi)(\p_{t}^{3}\xi)\partial_{\theta}\rho d\theta\notag\\
       &=\frac{1}{2}\partial_{t}(\int_{0}^{\pi}\partial_{b}\mathcal{R}(\partial_{t}^{2}\partial_{\theta}\xi)(\partial_{t}^{2}\partial_{\theta}\xi)\rho d\theta)-\frac{1}{2}\int_{0}^{\pi}\partial_{t}\partial_{b}\mathcal{R}(\partial_{t}^{2}\partial_{\theta}\xi)(\partial_{t}^{2}\partial_{\theta}\xi)\rho d\theta\notag\\
       &\quad-\frac{1}{2}\int_{0}^{\pi}\partial_{b}\mathcal{R}(\partial_{t}^{2}\partial_{\theta}\xi)(\partial_{t}^{2}\partial_{\theta}\xi)\p_{t}\rho d\theta+\int_{0}^{\pi}\partial_{b}\mathcal{R}(\partial_{t}^{2}\partial_{\theta}\xi)\p_{t}^{3}\xi\partial_{\theta}\rho d\theta\label{equ:5.1.233}
   \end{align}

   \noindent which implies that:

   \begin{align}
    &\vert I_{5}-\frac{1}{2}\partial_{t}(\int_{0}^{\pi}\partial_{b}\mathcal{R}(\partial_{t}^{2}\partial_{\theta}\xi)(\partial_{t}^{2}\partial_{\theta}\xi)\rho d\theta)\vert\notag\\
    &\lesssim  \int_{0}^{\pi}|\partial_{t}\partial_{b}\mathcal{R}(\partial_{t}^{2}\partial_{\theta}\xi)||(\partial_{t}^{2}\partial_{\theta}\xi)||\p_{t}\xi| d\theta+\int_{0}^{\pi}|\partial_{b}\mathcal{R}||(\partial_{t}^{2}\partial_{\theta}\xi)||\p_{t}^{3}\xi||\partial_{\theta}\rho| d\theta+\int_{0}^{\pi}\partial_{b}\mathcal{R}(\partial_{t}^{2}\partial_{\theta}\xi)(\partial_{t}^{2}\partial_{\theta}\xi)\p_{t}\xi d\theta\notag\\
   & \lesssim \vert \vert \partial_{t}\xi\vert \vert_{W^{1,+\infty}}\vert \vert \partial_{t}^{2}\xi\vert \vert^{2}_{H^{1}}+\vert \vert \xi\vert \vert_{W^{1,+\infty}}\vert \vert \partial_{t}^{3}\xi\vert \vert_{L^{\frac{1}{\alpha}}}\vert \vert \partial_{t}^{2}\xi\vert \vert^{2}_{H^{1}}\vert \vert \xi\vert \vert_{W^{1,+\infty}}+\vert \vert \xi\vert \vert_{W^{1,+\infty}}\vert \vert \partial_{t}\xi\vert\vert_{L^{\infty}}\vert \vert\partial_{t}^{2}\xi \vert \vert^{2}_{H^{1}}\notag\\
   &\lesssim\vert \vert \partial_{t}\xi\vert \vert_{W^{3-\frac{1}{q_{-}},q_{-}}}\vert \vert \partial_{t}^{2}\xi\vert \vert^{2}_{H^{1}}+\vert \vert \xi\vert \vert_{W^{3-\frac{1}{q_{-}},q_{-}}}\vert \vert \partial_{t}^{3}\xi\vert \vert_{H^{\frac{1}{2}-\alpha}}\vert \vert \partial_{t}^{2}\xi\vert \vert^{2}_{H^{1}}\vert \vert \xi\vert \vert_{W^{3-\frac{1}{q_{+}},q_{+}}}+\vert \vert \xi\vert \vert_{W^{3-\frac{1}{q_{+}},q_{+}}}\vert \vert \partial_{t}\xi\vert\vert_{H^{\frac{3}{2}-\alpha}}\vert \vert\partial_{t}^{2}\xi \vert \vert^{2}_{H^{1}}\notag\\
   &\lesssim \mathcal{E}^{\frac{1}{2}}\mathcal{D}\label{equ:5.1.234}
   \end{align}

   \textbf{Term $I_{6}$}
   
     For the term $I_{6}$, we have the following computation using Theorem \ref{thm:gam}:
   
   \begin{align}
       I_{6}=&\int_{0}^{\pi} \partial_{b}\mathcal{R} \partial_{t}^{2}\partial_{\theta}\xi\partial_{\theta}\rho\mathfrak{n}'''(t)\xi_{s}d\theta+\int_{0}^{\pi}\partial_{b}\mathcal{R} \partial_{t}^{2}\partial_{\theta}\xi\rho\mathfrak{n}'''(t)\partial_{\theta}\xi_{s}d\theta \notag\\
    &\lesssim \vert \vert \xi\vert \vert_{W^{1,+\infty}}\vert \vert \partial_{t}^{2}\xi\vert \vert_{H^{1}}\mathcal{D}^{\frac{1}{2}}+\vert \vert \xi\vert \vert_{W^{1,+\infty}}\vert \vert \partial_{t}^{2}\xi\vert \vert_{H^{1}}\mathcal{D}^{\frac{1}{2}}\notag\\
    &\lesssim \|\xi\|_{W^{3-\frac{1}{q_{-}},q_{-}}}\|\p_{t}^{2}\xi\|_{H^{1}}\mathcal{D}^{\frac{1}{2}}+\|\xi\|_{W^{3-\frac{1}{q_{+}},q_{+}}}\|\p_{t}^{2}\xi\|_{H^{1}}\mathcal{D}^{\frac{1}{2}} \notag\\
    &\lesssim \mathcal{E}^{\frac{1}{2}}\mathcal{D} \label{5.1.235}
   \end{align}

   \textbf{Term $I_{7}$}
   
   Finally, for the term $I_{7}$, we have

   \begin{align}
       I_{7}=\int_{0}^{\pi} \partial_{b}\mathcal{R}(\partial_{t}^{2}\partial_{\theta}\xi)^{2}\partial_{\theta}\rho\tilde{b}_{8,2}+\int_{0}^{\pi} \partial_{b}\mathcal{R}(\partial_{t}^{2}\partial_{\theta}\xi)\textbf{}^{2}\rho\partial_{\theta}\tilde{b}_{8,2}\lesssim \mathcal{E}^{\frac{1}{2}}\vert \vert \partial_{t}^{2}\xi\vert \vert_{H^{\frac{3}{2}-\alpha}}\vert \vert \tilde{b}_{8,2}\vert \vert_{W^{1,\frac{1}{1-\alpha}}} \label{equ:5.1.236}
   \end{align}

   \noindent We use the similar computation as in \eqref{equ:5.1.232} to obtain the boundedness of $\vert \vert \tilde{b}_{8,2}\vert \vert_{H^{1}}\lesssim \mathcal{E}^{\frac{1}{2}}\mathcal{D}^{\frac{1}{2}}$. Applying this result to equation \eqref{equ:5.1.236}, we then obtain

   \begin{align}
       I_{7}\lesssim \mathcal{E}\mathcal{D}
   \end{align}

  Now combining the estimates for $I_{1}$ to $I_{8}$, we have

   \begin{align}
       \vert (\int_{0}^{\pi}b^{3,2}\partial_{\theta}(v\cdot \mathcal{N}))d\theta-\partial_{t}(\int_{0}^{\pi}\partial_{bb}\mathcal{R}(\partial_{t}\partial_{\theta}\xi)^{2}\rho\partial_{tt\theta}\xi)-\frac{1}{2}\partial_{t}(\int_{0}^{\pi}\partial_{b}\mathcal{R}(\partial_{t}^{2}\partial_{\theta}\xi)(\partial_{t}^{2}\partial_{\theta}\xi)\rho d\theta)\vert\notag\\
       \lesssim \mathcal{E}^{\frac{1}{2}}\mathcal{D} \label{equ:5.1.237}.
   \end{align}

   \noindent For the temporal derivative terms appearing in \eqref{equ:5.1.237}, we further note that

   \begin{align}
       (\int_{0}^{\pi}\partial_{bb}\mathcal{R}(\partial_{t}\partial_{\theta}\xi)^{2}\rho\partial_{tt\theta}\xi)\lesssim \vert \vert \partial_{t}\xi\vert \vert^{2}_{W^{1,\infty}}\vert \vert \partial_{t}^{2}\xi\vert \vert_{H^{1}}\lesssim \mathcal{E}^{\frac{3}{2}} \label{equ:5.1.238},
   \end{align}

   \noindent and

   \begin{align}
       \frac{1}{2}(\int_{0}^{\pi}\partial_{b}\mathcal{R}(\partial_{t}^{2}\partial_{\theta}\xi)(\partial_{t}^{2}\partial_{\theta}\xi)\rho d\theta)\vert\lesssim \vert \vert \partial_{t}\xi\vert \vert_{W^{1,+\infty}}\vert \vert \partial_{t}^{2}\xi\vert \vert^{2}_{H^{1}}\lesssim \mathcal{E}^{\frac{3}{2}}.
   \end{align}

   \textbf{Step 12} In this step, we estimate terms involving $\omega$. We have the following computation using Lemma \ref{thm:B} and H\"older's inequality

   \begin{align}
       |\int_{\Omega}\p_{t}^{2}u\omega J|\lesssim\int_{\Omega} |\p_{t}^{2}u||\omega|\lesssim \|\p_{t}^{2}u\|_{L^{2}}\|\omega\|_{W^{1,\frac{4}{3-2\epsilon_{+}}}}\lesssim \mathcal{E}^{\frac{3}{2}}.
    \end{align}
    \noindent Similarly,
    \begin{align}
        |\int_{\Omega}\p_{t}^{2}u\p_{t}(J\omega)|\lesssim \int_{\Omega}|\p_{t}^{2}u|(|\nabla \p_{t}\bar{\xi}||\omega|+|\p_{t}\omega|)\lesssim \|\p_{t}^{2}u\|_{L^{2}}(\|\p_{t}\bar{\xi}\|_{W^{1,+\infty}}\|\omega\|_{L^{2}}+\|\p_{t}\omega\|_{L^{2}})\lesssim (\sqrt{\mathcal{E}}+\mathcal{E})\mathcal{D}.
    \end{align}
    \noindent Using the result in Step 1, we further obtain
    \begin{align}
        (b^{1,2},\omega)_{0}\lesssim \mathcal{E}^{\frac{1}{2}}\mathcal{D}^{\frac{1}{2}}\|\omega\|_{H^{1}}\lesssim \mathcal{E}^{\frac{1}{2}}\mathcal{D}. 
    \end{align}
    \noindent Finally, for the term $((\p_{t}^{2}u,\omega))$, we have
    \begin{align}
        |((\p_{t}^{2}u,\omega))|\lesssim \|\p_{t}^{2}u\|_{H^{1}}\|\omega\|_{H^{1}}\lesssim \mathcal{E}^{\frac{1}{2}}\mathcal{D}
    \end{align}

    \textbf{Step 13} In this step, we estimate terms involving pressure. Using Lemma \ref{lem:F_2} which will be proved after the main proof of this theorem, we have:

    \begin{align}
        \int_{\Omega}\p_{t}^{2}p\langle Jb^{2,2}\rangle_{\Omega}=\frac{d}{dt}(\langle Jb^{2,2}\rangle_{\Omega}\int_{\Omega}\p_{t}p)-\p_{t}\langle Jb^{2,2}\rangle_{\Omega}\int_{\Omega}\p_{t}p=\frac{d}{dt}I_{1}-I_{2}
    \end{align}

    \noindent We then estimate $I_{1}$ and $I_{2}$ individually. We have

    \begin{align}
        |I_{1}|\lesssim \mathcal{E}\|\p_{t}p\|_{L^{2}}\lesssim \mathcal{E}^{\frac{3}{2}},
    \end{align}

    \noindent and similarly

    \begin{align}
        |I_{2}|\lesssim \mathcal{E}^{\frac{1}{2}}\mathcal{D}.
    \end{align}

    \textbf{Step 14} In this step, we estimate the term $(\p_{t}^{2}\xi,\p_{t}^{2}\xi)_{1,\Sigma}$.

    Using Theorem \ref{thm:pos}, we have:

    \begin{align}
        (\p_{t}^{2}\xi,\p_{t}^{2}\xi)_{1,\Sigma}&=(\p_{t}^{2}\xi-a_{2}(t)\rho_{0},\p_{t}^{2}\xi-a_{2}(t)\rho_{0})_{1,\Sigma}-2(\p_{t}^{2}\xi,a_{2}(t)\rho_{0})_{1,\Sigma}+(a_{2}(t)\rho_{0},a_{2}(t)\rho_{0})_{1,\Sigma}\notag\\
        &\gtrsim \|\p_{t}^{2}\xi\|^{2}_{H^{1}}-\mathcal{E}^{\frac{3}{2}}
    \end{align}

    \noindent This finishes the estimate
    
    \textbf{Conclusion}

    Combining all of steps in this proof, we obtain
      \begin{align}
       \mathcal{E}_{||,2}(t)-\mathcal{E}^{\frac{3}{2}}(t)+\int_{0}^{t}\mathcal{D}_{||,2}\lesssim \mathcal{E}_{||,2}(0)+\mathcal{E}^{\frac{3}{2}}(0)+\int_{0}^{t}\mathcal{E}^{\frac{1}{2}}\mathcal{D} 
   \end{align}
   \end{proof}

   Now it remains to establish the boundedness for $Jb^{2,2}$ via the following lemma

   \begin{lemma}{\label{lem:F_2}}
   
       Let $b^{2,2}$ be defined as in \eqref{equ:5.1.92}, we have the following estimates:

       \begin{align}{\label{equ:F_2_1}}
           \int_{\Omega}Jb^{2,2}\lesssim\mathcal{E}
       \end{align}
       \noindent and:
       \begin{align}{\label{equ:F_2_2}}
           \int_{\Omega}\p_{t}(Jb^{2,2})\lesssim \mathcal{E}^{\frac{1}{2}}\mathcal{D}^{\frac{1}{2}}
       \end{align}
       
   \end{lemma}

   \begin{proof}
   
       We begin by estimating \eqref{equ:F_2_1}. We have the following estimate using the regularity results for $\xi$ and $u$

       \begin{align}
           \int_{\Omega}|Jb^{2,2}|&\lesssim \|\operatorname{div}_{\p_{t}^{2}\mathcal{A}}u\|_{L^{1}}+\|\operatorname{div}_{\p_{t}\mathcal{A}}\p_{t}u\|_{L^{1}}\notag\\
           &\lesssim \|\p_{t}^{2}\bar{\xi}\|_{W^{1,4}}\|u\|_{W^{1,\frac{2}{1-\epsilon_{+}}}}+\|\p_{t}\bar{\xi}\|_{W^{1,\infty}}\|\p_{t}u\|_{W^{1,\frac{2}{1-\epsilon_{-}}}}\notag\\
           &\lesssim \|\p_{t}^{2}\xi\|_{H^{1}}\|u\|_{W^{2,q_{+}}}+\|\p_{t}\xi\|_{H^{\frac{3}{2}+\frac{\epsilon_{-}-\alpha}{2}}}\|\p_{t}u\|_{W^{2,q_{-}}}\lesssim \mathcal{E}
       \end{align}

       \noindent We then establish the estimate for \eqref{equ:F_2_2}. We have:

       \begin{align}
           |\int_{\Omega}\p_{t}(Jb^{2,2})|&\lesssim \int_{\Omega}|\p_{t}Jb^{2,2}|+\int_{\Omega}J\p_{t}b^{2,2}\notag\\
           &\lesssim\|\p_{t}\xi\|_{W^{3-\frac{1}{q_{-}},q_{-}}}\int_{\Omega}|b^{2,2}|+\int_{\Omega}|\p_{t}b^{2,2}|
       \end{align}
       \noindent Using the estimate for \eqref{equ:F_2_1}, we deduce $\int_{\Omega}|b^{2,2}|\lesssim\mathcal{E}$. Hence, applying this result to the equation above, we obtain

       \begin{align}
           |\int_{\Omega}\p_{t}(Jb^{2,2})|&\lesssim \mathcal{E}\mathcal{D}^{\frac{1}{2}}+\|\dive_{\p_{t}^{3}\mathcal{A}}u\|_{L^{1}}+\|\dive_{\p_{t}^{2}\mathcal{A}}\p_{t}u\|_{L^{1}}+\|\dive_{\p_{t}\mathcal{A}}\p_{t}^{2}u\|_{L^{1}}\notag\\
           &\lesssim \mathcal{E}\mathcal{D}^{\frac{1}{2}}+\|\p_{t}^{3}\bar{\xi}\|_{W^{1,\frac{1}{1+2\alpha}}}\|u\|_{W^{1,\frac{1}{1-\epsilon_{+}}}}+\|\p_{t}^{2}\bar{\xi}\|_{L^{\frac{2}{\alpha}}}\|\p_{t}u\|_{L^{\frac{4}{2-\epsilon_{-}}}}+\|\p_{t}\bar{\xi}\|_{W^{1,+\infty}}\|\p_{t}^{2}u\|_{H^{1}}\notag\\
           &\lesssim \mathcal{E}\mathcal{D}^{\frac{1}{2}}+\|\p_{t}^{3}\xi\|_{H^{\frac{1}{2}-\alpha}}\|u\|_{W^{2,q_{+}}}+\|\p_{t}^{2}\xi\|_{H^{\frac{3}{2}-\alpha}}\|\p_{t}u\|_{H^{1+\frac{\epsilon_{-}}{2}}}+\|\p_{t}\xi\|_{H^{\frac{3}{2}+\frac{\epsilon_{-}-\alpha}{2}}}\|\p_{t}^{2}u\|_{H^{1}}\notag\\
           &\lesssim \mathcal{E}^{\frac{1}{2}}\mathcal{D}^{\frac{1}{2}}
       \end{align}
   \end{proof}
   
   Combining Theorem 7.2, Theorem 7.3 and Theorem 7.4, we obtain the following result

   \begin{align}
       \mathcal{E}_{||}(t)-\mathcal{E}^{\frac{3}{2}}(t)+\int_{0}^{t}\mathcal{D}_{||}\lesssim \mathcal{E}_{||}(0)+\mathcal{E}^{\frac{3}{2}}(0)+\int_{0}^{t}\mathcal{E}^{\frac{1}{2}}\mathcal{D} \label{equ:5.1.239}
   \end{align}

   To complete the apriori estimate, we need to replace $\mathcal{E}_{||}$ and $\mathcal{D}_{||}$ by $\mathcal{E}(t)$ and $\mathcal{D}(t)$ in the estimate above. In the next section, we establish the necessary elliptic estimates, which will allow us to control the remaining terms in the energy and dissipation.

   \subsection{Elliptic Estimate}

  In this subsection, we rely on Theorem 4.7 from the work of Guo–Tice \cite{Guo} to apply the elliptic theory. For our purposes here, we focus only on the necessary estimates and computations.For the full details of the elliptic theory, we refer the reader to \cite{Guo}. We write down the esquation as follows
   
   \begin{equation}{\label{equ:5.1.240}}
       \begin{cases}
           \operatorname{div}_{\mathcal{A}}S_{\mathcal{A}}(p,u)=G^{1}~~~~&\operatorname{in} \Omega\\
           J\operatorname{div}_{\mathcal{A}}u=G^{2}~~&\operatorname{in} \Omega\\
           \frac{u\cdot \mathcal{N}}{\vert \mathcal{N}_{0}\vert}=G_{+}^{3}~~&\operatorname{on}\Sigma\\
           S_{\mathcal{A}}(p,u)\mathcal{N}=[\mathcal{K}(\xi)]\mathcal{N}+G_{+}^{4}\mathcal{T}+G^{5}\mathcal{N}~~&\operatorname{on} \Sigma\\
           u\cdot J\nu=G_{-}^{3}~~~&\operatorname{on} \Sigma_{s}\\
           \mu\mathbb{D}_{\mathcal{A}}u\nu\cdot \tau+\beta u\cdot \tau=G_{-}^{4}~~&\operatorname{on} \Sigma_{s}\\
           \mp \sigma(\frac{\rho_{0}^{2}\partial_{\theta}\xi}{(\rho_{0}^{2}+\rho_{0}'^{2})^{\frac{3}{2}}}-\frac{\rho_{0}\rho_{0}'\xi}{(\rho_{0}^{2}+\rho_{0}'^{2})^{\frac{3}{2}}})(\frac{\pi}{2}\pm\frac{\pi}{2})=G_{\pm}^{7}
       \end{cases}
   \end{equation}

  \noindent where expressions for  $G^{1}$ to $G^{5}$ are as given in the Appendix.

    \begin{theorem}
       Suppose that $(u, p, \xi) \in W^{2,q_{+}} \times W^{1,q_{+}} \times W^{3-\frac{1}{q_{+}},,q_{+}}$ is a strong solution to the system \eqref{equ:5.1.240}. Then the following elliptic estimate holds:

       \begin{align}
           &\vert \vert u\vert \vert_{W^{2,q_{+}}}+\vert \vert u\vert \vert_{H^{1+\epsilon_{+}}}+\vert \vert p\vert \vert_{W^{1,q_{\epsilon_{+}}}}+\vert \vert p\vert \vert_{H^{\epsilon_{+}}}+\vert \vert \xi\vert\vert_{W^{3-\frac{1}{q_{+}},q_{+}}} \notag\\
           &\lesssim \sqrt{\mathcal{E}_{||}}+\sqrt{\mathcal{E}}\vert \vert \partial_{t}\xi\vert \vert_{H^{1}}+\vert \vert \partial_{t}\xi\vert \vert_{H^{\frac{3}{2}-\alpha}}+\sqrt{\mathcal{E}}\vert \vert \xi\vert \vert_{W^{2-\frac{1}{q_{+}},q_{+}}} \notag\\
           &\quad+\sqrt{\mathcal{E}}\vert \vert \xi\vert \vert_{W^{3-\frac{1}{q_{+}},q_{+}}}+\mathcal{E}_{||}+\vert \vert u\vert \vert_{H^{1}}\label{equ:5.1.254}
       \end{align}
       
    \end{theorem}

    \begin{proof}
    
     \textbf{Step 1} In this step, we estimate each term included $G^{1}$ individually

        We have the following estimates

       \begin{align}
           \vert \vert \partial_{t}u\vert \vert_{L^{q_{+}}}\lesssim \vert \vert \partial_{t}u\vert \vert_{L^{2}}\lesssim \sqrt{\mathcal{E}_{||}} \label{equ:5.1.241}
       \end{align}

       \begin{align}
           \vert \vert u\cdot \nabla_{\mathcal{A}}u\vert \vert_{L^{q_{+}}}\lesssim \vert \vert \nabla u\vert \vert_{L^{\frac{2}{1-\epsilon_{+}}}}\vert \vert u\vert\vert_{L^{2}}\lesssim \|u\|_{W^{2,q_{+}}}\|u\|_{L^{2}}\lesssim \sqrt{\mathcal{E}}\sqrt{\mathcal{E}_{||}} \label{equ:5.1.242}
       \end{align}

       \begin{align}
           \vert \vert \mathfrak{n}'(t)\partial_{x}u\vert \vert_{L^{q_{+}}}\lesssim \vert \vert \partial_{x}u\vert \vert_{L^{q_{+}}}\mathcal{E}^{\frac{1}{2}}\lesssim \sqrt{\mathcal{E}}\vert \vert u\vert \vert_{H^{1}}\lesssim \mathcal{E}\label{equ:5.1.243}
       \end{align}
      
       \begin{align}
           \vert \vert (\cos\theta W\partial_{t}\bar{\xi},\sin\theta W\partial_{t}\bar{\xi})\mathcal{A}(\partial_{x}u,\partial_{y}u)^{T}\vert \vert_{L^{q_{+}}}\lesssim \vert \vert \partial_{t}\bar{\xi}\vert \vert_{L^{2}}\vert \vert \nabla u\vert\vert_{L^{\frac{2}{1-\epsilon_{+}}}}\lesssim \sqrt{\mathcal{E}}\vert \vert \partial_{t}\xi\vert \vert_{H^{1}} \lesssim \mathcal{E}\label{equ:5.1.244}
       \end{align}

       \textbf{Step 2} 
       
       In this step, we show the estimate for $G_{+}^{3}$:

       \textbf{Term $ \vert \vert \frac{\rho\partial_{t}\xi}{\vert \mathcal{N}_{0}\vert}\vert \vert_{W^{2-\frac{1}{q_{+}},q_{+}}}$}

       \begin{align}
           \vert \vert \frac{\rho\partial_{t}\xi}{\vert \mathcal{N}_{0}\vert}\vert \vert_{W^{2-\frac{1}{q_{+}},q_{+}}}\lesssim \vert \vert \partial_{t}\xi\vert \vert_{W^{2-\frac{1}{q_{+}},q_{+}}}
           \label{equ:5.1.245}
       \end{align}

      \noindent Using the assumption that $2\alpha+\epsilon_{+}<1$, we have:

      \begin{align}
          2-\frac{1}{q_{+}}=2-\frac{2-\epsilon_{+}}{2}=1+\frac{\epsilon_{+}}{2}\leq  \frac{3}{2}-\alpha,
      \end{align}

       \noindent which implies that

       \begin{align}
           \frac{1}{q_{+}}=\frac{2-\epsilon_{+}}{2}\geq \frac{2\alpha+\epsilon_{+}}{2}=\frac{1}{2}-\frac{1}{1}(\frac{3}{2}-\alpha-1-\frac{\epsilon_{+}}{2}).
       \end{align}

       \noindent Therefore,  using Sobolev embedding and the index relation derived above, we obtain the following result

       \begin{align}
           H^{\frac{3}{2}-\alpha}((0,\pi))\hookrightarrow W^{2-\frac{1}{q_{+}},\frac{2}{2\alpha+\epsilon_{+}}}((0,\pi))\hookrightarrow W^{2-\frac{1}{q_{+}},q_{+}}((0,\pi))
       \end{align}

       \noindent Applying the embedding result above to inequality \eqref{equ:5.1.245}, we obtain

       \begin{align}
           \vert \vert \frac{\rho\partial_{t}\xi}{\vert \mathcal{N}_{0}\vert}\vert \vert_{W^{2-\frac{1}{q_{+}},q_{+}}}\lesssim \vert \vert \partial_{t}\xi\vert \vert_{H^{\frac{3}{2}-\alpha}} \label{equ:5.1.246}
       \end{align}

       \textbf{Term $ \vert \vert \frac{\mathfrak{n}'(t)\rho\xi_{s}}{\vert \mathcal{N}_{0}\vert}\vert \vert_{W^{2-\frac{1}{q_{+}},q_{+}}}$}
       
        We have the following estimate for this term

      \begin{align}
           \vert \vert \frac{\mathfrak{n}'(t)\rho\xi_{s}}{\vert \mathcal{N}_{0}\vert}\vert \vert_{W^{2-\frac{1}{q_{+}},q_{+}}}\lesssim |\mathfrak{n}'(t)|\lesssim \vert \vert u\vert \vert_{H^{1}}
      \end{align}

      \textbf{Term $ | \mathfrak{n}'(t)|\vert \vert \frac{(\rho'-\frac{\rho\rho_{0}'}{\rho_{0}})\cos\theta}{\vert \mathcal{N}_{0}\vert}\vert \vert_{W^{2-\frac{1}{q_{+}},q_{+}}}$}
      
       We have the following estimate

      \begin{align}
          | \mathfrak{n}'(t)|\vert \vert \frac{(\rho'-\frac{\rho\rho_{0}'}{\rho_{0}})\cos\theta}{\vert \mathcal{N}_{0}\vert}\vert \vert_{W^{2-\frac{1}{q_{+}},q_{+}}}\lesssim\vert \vert u\vert \vert_{H^{1}}\vert \vert \xi\vert \vert_{W^{3-\frac{1}{q_{+}},q_{+}}}\lesssim \mathcal{E}
      \end{align}

      Therefore, combining the estimates for all of terms above, we obtain

      \begin{align}
          \vert \vert G_{+}^{3}\vert \vert_{W^{2-\frac{1}{q_{+}},q_{+}}}\lesssim \vert \vert \partial_{t}\xi\vert \vert_{H^{\frac{3}{2}-\alpha}}+\vert \vert u\vert \vert_{H^{1}}+\mathcal{E}
      \end{align}
      
       \textbf{Step 3} The estimate for $G^{5}$

       We aim to estimate the following norms
       
       \begin{align}
           \vert \vert G^{5}\cdot\frac{\mathcal{N}}{\vert \mathcal{N}\vert^{2}}\vert \vert_{W^{1-\frac{1}{q_{+}},q_{+}}}+\vert \vert G^{5}\cdot\frac{\mathcal{T}}{\vert \mathcal{T}\vert^{2}}\vert \vert_{W^{1-\frac{1}{q_{+}},q_{+}}} \label{equ:5.1.247}
       \end{align}

       \noindent By the definition of $G^{5}$ we have:

       \begin{align}
             \vert \vert \mathcal{R}_{2}\cdot\frac{\mathcal{N}}{\vert \mathcal{N}\vert^{2}}\vert \vert_{W^{1-\frac{1}{q_{+}},q_{+}}}+\vert \vert \mathcal{R}_{2}\cdot\frac{\mathcal{T}}{\vert \mathcal{T}\vert^{2}}\vert \vert_{W^{1-\frac{1}{q_{+}},q_{+}}}&\lesssim\vert \vert \mathcal{R}_{2}\vert \vert_{W^{1-\frac{1}{q_{+}},q_{+}}}(\|\frac{\mathcal{N}}{|\mathcal{N}|^{2}}\|_{W^{1,q_{+}}}+\|\frac{T}{|\mathcal{T}|^{2}}\|_{W^{1,q_{+}}})\notag\\&\lesssim \vert \vert (\partial_{\theta}\xi)^{2}\vert \vert_{W^{1-\frac{1}{q_{+}},q_{+}}}+\vert \vert (\xi)^{2}\vert \vert_{W^{1-\frac{1}{q_{+}},q_{+}}}+\vert \vert (\partial_{\theta}\xi)\xi\vert \vert_{W^{1-\frac{1}{q_{+}},q_{+}}}
             \notag\\
             &\lesssim \vert \vert\xi\vert \vert_{W^{2-\frac{1}{q_{+}},q_{+}}}\vert \vert \xi\vert \vert_{W^{2,q_{+}}}
             \lesssim\|\xi\|_{W^{3-\frac{1}{q_{+}},q_{+}}}^{2}\lesssim \mathcal{E}\label{equ:5.1.248}
       \end{align}
       \noindent where we used Theorem B.2 in \cite{Guo}, and the fact that $1>\max(1-\frac{1}{q_{+}},\frac{1}{q_{+}})$ and $3-\frac{1}{q_{+}}>2$.
     
     Similarly, we have the following estimate for $\p_{\theta}\mathcal{R}_{1}$ in $G^{5}$ 
     
       \begin{align}
            \vert \vert \partial_{\theta}(\mathcal{R}_{1})\vert \vert_{W^{1-\frac{1}{q_{+}},q_{+}}}\lesssim \|\mathcal{R}_{1}\|_{W^{2-\frac{1}{q_{+}},q_{+}}}\lesssim\vert \vert \xi\vert \vert_{W^{2,q_{+}}}\vert \vert\xi\vert \vert_{W^{3-\frac{1}{q_{+}},q_{_{+}}}}\lesssim \sqrt{\mathcal{E}}\vert \vert \xi\vert \vert_{W^{3-\frac{1}{q_{+}}q_{+}}}\lesssim \mathcal{E}\label{equ:5.1.251}
       \end{align}

       \noindent

       \textbf{Step 5} In this step, we then estimate $G^{7}$. We have the following estimate using the definition of $G^{7}$ in the Appendix

       \begin{align}
           [\mathcal{R}_{1}]_{\theta}\lesssim \vert \vert \mathcal{R}_{1}\vert \vert_{W^{1,q_{+}}}\lesssim \vert \vert \xi\vert \vert_{W^{1,+\infty}}\vert \vert\xi \vert \vert_{W^{2,q_{+}]}}\lesssim \sqrt{\mathcal{E}}\vert \vert \xi\vert \vert_{W^{2,q_{+}}}\lesssim\mathcal{E}\label{equ:5.1.252},
       \end{align}

       \noindent and

       \begin{align}
           [\alpha(\partial_{t}\xi+\mathfrak{n}'(t))]_{\theta}\lesssim [\partial_{t}\xi]_{\theta}+[\mathfrak{n}'(t)]_{\theta}\lesssim \vert \vert \partial_{t}\xi\vert \vert_{H^{1}}+\vert \mathfrak{n}'(t)\vert\lesssim \mathcal{E}_{||}+\vert \vert u\vert \vert_{H^{1}} \label{equ:5.1.253},
       \end{align}

       \noindent and

       \begin{align}
           [\hat{W}(\mathfrak{n}'(t)+\p_{t}\xi)]_{\theta}\lesssim |\mathfrak{n}'(t)|^{2}+|[\p_{t}\xi]_{\theta}|^{2}\lesssim \|\p_{t}\xi\|_{H^{1}}^{2}+\|u\|^{2}_{H^{1}}.
       \end{align}

       \textbf{Final result}
       
      Therefore, combining all of the steps in this proof, and applying Theorem 4.7 in \cite{Guo}, we obtain

       \begin{align}
           &\vert \vert u\vert \vert_{W^{2,q_{+}}}+\vert \vert u\vert \vert_{H^{1+\frac{\epsilon_{+}}{2}}}+\vert \vert p\vert \vert_{W^{1,q_{+}}}+\vert \vert p\vert \vert_{H^{\frac{\epsilon_{+}}{2}}}+\vert \vert \xi\vert\vert_{W^{3-\frac{1}{q_{+}},q_{+}}} \notag\\
           &\lesssim \vert \vert G^{1}\vert \vert_{L^{q_{+}}}+\vert \vert G^{2}\vert \vert_{W^{1,q_{+}}}+\vert \vert G_{+}^{3}\vert \vert_{W^{2-\frac{1}{q_{+}},q_{+}}}+\vert \vert G_{-}^{3}]\vert \vert_{W^{2-\frac{1}{q_{+}},q_{+}}}\notag\\
           &\quad+\vert \vert G_{+}^{4}\vert \vert_{W^{1-\frac{1}{q_{+}},q_{+}}}+\vert \vert G_{-}^{4}\vert \vert_{W^{1-\frac{1}{q_{+}},q_{+}}}+\vert \vert G^{5}\vert \vert_{W^{1-\frac{1}{q_{+}},q_{+}}}\notag\\
           &\quad+\vert \vert \partial_{\theta}G^{6}\vert \vert_{W^{1-\frac{1}{q_{+}},q_{+}}}+[G^{7}]_{\theta} \notag\\
           &\lesssim \sqrt{\mathcal{E}_{||}}+\sqrt{\mathcal{E}}\vert \vert \partial_{t}\xi\vert \vert_{H^{1}}+\vert \vert \partial_{t}\xi\vert \vert_{H^{\frac{3}{2}-\alpha}}+\sqrt{\mathcal{E}}\vert \vert \xi\vert \vert_{W^{2-\frac{1}{q_{+}},q_{+}}} \notag\\
           &\quad+\sqrt{\mathcal{E}}\vert \vert \xi\vert \vert_{W^{3-\frac{1}{q_{+}},q_{+}}}+\mathcal{E}_{||}+\vert \vert u\vert \vert_{H^{1}}
       \end{align}
    \end{proof}

    We now establish the elliptic estimate for the first order equation system. The first order equation system can be expressed as follows

     \begin{equation}{\label{equ:ellip_1}}
       \begin{cases}
           \operatorname{div}_{\mathcal{A}}S_{\mathcal{A}}(\p_{t}p,\p_{t}u)=G^{1,1}~~~~&\operatorname{in} \Omega\\
           J\operatorname{div}_{\mathcal{A}}\p_{t}u=G^{2,1}~~&\operatorname{in} \Omega\\
           \frac{\p_{t}u\cdot \mathcal{N}}{\vert \mathcal{N}_{0}\vert}=G_{+}^{3,1}~~&\operatorname{on}\Sigma\\
           S_{\mathcal{A}}(\p_{t}p,\p_{t}u)\mathcal{N}=[\mathcal{K}(\p_{t}\xi)]\mathcal{N}+G_{+}^{4,1}\mathcal{T}+G^{5,1}\mathcal{N}~~&\operatorname{on} \Sigma\\
           \p_{t}u\cdot J\nu=G_{-}^{3,1}~~~&\operatorname{on} \Sigma_{s}\\
           \mu\mathbb{D}_{\mathcal{A}}\p_{t}u\nu\cdot \tau+\beta \p_{t}u\cdot \tau=G_{-}^{4,1}~~&\operatorname{on} \Sigma_{s}\\
           \mp \sigma(\frac{\rho_{0}^{2}\partial_{\theta}\p_{t}\xi}{(\rho_{0}^{2}+\rho_{0}'^{2})^{\frac{3}{2}}}-\frac{\rho_{0}\rho_{0}'\p_{t}\xi}{(\rho_{0}^{2}+\rho_{0}'^{2})^{\frac{3}{2}}})(\frac{\pi}{2}\pm\frac{\pi}{2})=G_{\pm}^{7,1}
       \end{cases}
   \end{equation}
   \noindent where the expressions for $G^{1,1}$ to $G^{7,1}$ are given in the Appendix.

    \begin{theorem}
       Suppose that $(\p_{t}u, \p_{t}p, \p_{t}\xi) \in W^{2,q_{-}} \times W^{1,q_{-}} \times W^{3-\frac{1}{q_{-}},,q_{-}}$ is a strong solution to the system \eqref{equ:5.1.240}. Then the following elliptic estimate holds:

       \begin{align}
            & \vert \vert \partial_{t}u\vert \vert_{W^{2,q_{-}}}+\vert \vert \partial_{t}u\vert \vert_{H^{1+\epsilon_{-}}}+\vert \vert \partial_{t}p\vert \vert_{W^{1,q_{-}}}+\vert \vert \partial_{t}p\vert \vert_{H^{\epsilon_{-}}}+\vert \vert \partial_{t}\xi\vert\vert_{W^{3-\frac{1}{q_{-}},q_{-}}}\notag\\
           &\lesssim \mathcal{E}_{||}+\sqrt{\mathcal{E}}\sqrt{\mathcal{D}}+\vert \vert \partial_{t}^{2}\xi\vert \vert_{H^{\frac{3}{2}-\alpha}}+\vert \vert \partial_{t}\xi\vert \vert_{H^{\frac{3}{2}-\alpha}}^{2}+\vert \vert \partial_{t}u\vert \vert_{H^{1}}+\vert \vert u\vert \vert_{H^{1}}\sqrt{\mathcal{D}}+\vert \vert \partial_{t}\xi\vert \vert_{H^{1}}+\mathcal{E}
       \end{align}
    \end{theorem}

    \begin{proof}
        
    \textbf{Step 1} We estimate each term in $G^{1,1}$ individually.

    \textbf{Term $\partial_{t}^{2}u$}. We have the following computation:

    \begin{align}
        \vert \vert \partial_{t}^{2}u\vert \vert_{L^{q_{-}}}\lesssim \vert \vert \partial_{t}^{2}u\vert \vert_{L^{2}}\lesssim \mathcal{E}_{||}^{\frac{1}{2}} \label{equ:5.1.255} 
    \end{align}

    \textbf{Term $\partial_{t}u\cdot \nabla_{\mathcal{A}}u$}. We have the following computation:

    \begin{align}
        \vert \vert \partial_{t}u\cdot \nabla_{\mathcal{A}}u\vert \vert_{L^{q_{-}}}\lesssim \vert \vert \partial_{t}u\vert\vert_{L^{+\infty}}\vert \vert u\vert \vert_{W^{1,q_{-}}}\lesssim \mathcal{E}\label{equ:5.1.256} 
    \end{align}

   \textbf{Term $u\cdot \nabla_{\partial_{t}\mathcal{A}}u$}. We have the following computation:

   \begin{align}
       \vert \vert u\cdot \nabla_{\partial_{t}\mathcal{A}}u\vert \vert_{L^{q_{-}}}\lesssim \vert \vert \partial_{t}\bar{\xi}\vert \vert_{W^{1+\infty}}\vert \vert u\vert \vert_{L^{\infty}}\vert \vert u\vert \vert_{W^{1,q_{-}}}\lesssim \vert \vert \partial_{t}{\xi}\vert \vert_{W^{3-\frac{1}{q_{-}},q_{-}}}\vert \vert u\vert \vert_{W^{2,q_{+}}}\lesssim \mathcal{E}^{\frac{3}{2}}\label{equ:5.1.257} 
   \end{align}

   \textbf{Term $u\cdot \nabla_{\mathcal{A}}\partial_{t}u$}

   We have the following estimate by Sobolev embedding and H\"older's inequality
   \begin{align}
       \vert \vert u\cdot \nabla_{\mathcal{A}}\partial_{t}u\vert \vert_{L^{q_{-}}}\lesssim \vert \vert u\vert \vert_{L^{\infty}}\vert \vert \partial_{t}u\vert \vert_{W^{1,q_{-}}}\lesssim \|u\|_{W^{2,q_{+}}}\|\p_{t}u\|_{H^{1}}\lesssim\mathcal{E}\label{equ:5.1.258} 
   \end{align}

      \textbf{Term $\mathfrak{n}''(t)\partial_{x}u$}:

      We have
      
      \begin{align}
          \vert \vert \mathfrak{n}''(t)\partial_{x}u\vert \vert_{L^{q_{-}}}\lesssim |\mathfrak{n}^{\prime\prime}(t)|\vert \vert u\vert \vert_{W^{1,q_{-}}}\lesssim {\mathcal{E}}\label{equ:5.1.259} 
      \end{align}

      \textbf{Term $\mathfrak{n}'(t)\partial_{x}\partial_{t}u$}:

      We have
      \begin{align}
            \vert \vert \mathfrak{n}'(t)\partial_{x}\partial_{t}u\vert \vert_{L^{q_{-}}}\lesssim \vert \vert u\vert \vert_{H^{1}}\vert \vert \partial_{t}u\vert \vert_{W^{1,q_{-}}}\lesssim \|u\|_{H^{1}}\|\p_{t}u\|_{H^{1}}\lesssim\sqrt{\mathcal{E}}\sqrt{\mathcal{D}}\label{equ:5.1.260} 
      \end{align}

       \textbf{Term $\partial_{t}((\cos\theta W\partial_{t}\bar{\xi},\sin\theta W\partial_{t}\bar{\xi})\tilde{K}(\partial_{x}u,\partial_{y}u)^{T})$}. We have the following computation by Sobolev embedding and H\"older's inequality

       \begin{align}
           &\vert \vert \partial_{t}((\cos\theta W\partial_{t}\bar{\xi},\sin\theta W\partial_{t}\bar{\xi})\tilde{K}(\partial_{x}u,\partial_{y}u)^{T})\vert \vert_{L^{q_{-}}}\notag\\
           \lesssim& \vert \vert \partial_{t}^{2}\bar{\xi}\vert \vert_{L^{\infty}}\vert \vert u\vert \vert_{W^{1,q_{-}}}
           +\vert \vert \partial_{t}\bar{\xi}\vert \vert_{L^{\infty}}\vert \vert \partial_{t}\bar{\xi}\vert \vert_{W^{1,\infty}}\vert \vert u\vert \vert_{W^{1,q_{-}}}+\vert \vert \partial_{t}\bar{\xi}\vert \vert_{L^{\infty}}\vert \vert u\vert \vert_{W^{1,q_{-}}}\notag\\
           \lesssim&\|\p_{t}^{2}{\xi}\|_{H^{\frac{3}{2}-\alpha}}\|u\|_{H^{1}}+\|\p_{t}\xi\|_{H^{\frac{3}{2}+\frac{\epsilon_{-}-\alpha}{2}}}^{2}\|u\|_{W^{2,q_{-}}}+\|\p_{t}\xi\|_{H^{\frac{3}{2}+\frac{\epsilon_{-}-\alpha}{2}}}\|u\|_{H^{1}}\lesssim \sqrt{\mathcal{E}}\sqrt{\mathcal{D}}\label{equ:5.1.261} 
       \end{align}

       \textbf{Term $\operatorname{div}_{\partial_{t}\mathcal{A}}S_{\mathcal{A}}(\p_{t}p,\p_{t}u)$}. 
       
       We have the following estimate

       \begin{align}
           \vert \vert \operatorname{div}_{\partial_{t}\mathcal{A}}S_{\mathcal{A}}(p,u)\vert \vert_{L^{q_{-}}}\lesssim& \vert \vert \partial_{t}\bar{\xi}\vert \vert_{W^{1,+\infty}}(\vert \vert u\vert \vert_{W^{2,q_{-}}}+\vert \vert p\vert \vert_{W^{1,q_{-}}})\notag\\
           &+\vert \vert \partial_{t}\bar{\xi}\vert \vert_{W^{1,+\infty}}(\vert \vert u\vert \vert_{W^{1,\frac{2}{1-\epsilon_{+}}}}+\vert \vert p\vert \vert_{L^{\frac{2}{1-\epsilon_{+}}}})\|\bar{\xi}\|_{W^{2,\frac{2}{1-\epsilon_{+}}}}\notag\\
           \lesssim&\|\p_{t}\xi\|_{W^{3-\frac{1}{q_{-}},q_{-}}}(\|u\|_{W^{2,q_{-}}}+\|p\|_{W^{1,q_{-}}})\|\xi\|_{W^{3-\frac{1}{q_{-}},q_{-}}}\lesssim \mathcal{E}\label{equ:5.1.262} 
       \end{align}

        \textbf{Term $\operatorname{div}_{\mathcal{A}}S_{\partial_{t}\mathcal{A}}(p,u)$}

        We have the following estimate by Sobolev embedding and H\"older's inequality
       \begin{align}
           \vert \vert \operatorname{div}_{\mathcal{A}}S_{\partial_{t}\mathcal{A}}(p,u)\vert \vert_{L^{q_{-}}}&\lesssim \vert \vert \vert \nabla \partial_{t}\mathcal{A}\vert\vert \nabla u\vert\vert \vert_{L^{q_{-}}}+\vert \vert \vert \partial_{t}\mathcal{A}\vert \vert \nabla^{2}u\vert\vert \vert_{L^{q_{-}}}\notag\\
           &\lesssim (\vert \vert \partial_{t}\bar{\xi}\vert \vert_{W^{2,\frac{2}{1-\epsilon_{-}}}})\vert \vert \nabla u\vert \vert_{L^{\frac{2}{1-\epsilon_{+}}}}+\vert \vert \partial_{t}\bar{\xi}\vert \vert_{W^{1,\infty}}\vert \vert u\vert \vert_{W^{2,q_{+}}}\notag\\
           &\lesssim \|\p_{t}\bar{\xi}\|_{W^{3,q_{-}}}\|u\|_{W^{2,q_{+}}}+\|\p_{t}\bar{\xi}\|_{W^{3,q_{-}}}\|u\|_{W^{2,q_{+}}} \lesssim \sqrt{\mathcal{E}}\sqrt{\mathcal{D}}\label{equ:5.1.263} 
       \end{align}

        \textbf{Step 2} In this step, we show the estimate for $G^{2,1}$. We have:

        \begin{align}
            \vert \vert J\operatorname{div}_{\partial_{t}\mathcal{A}}u\vert \vert_{W^{1,q_{-}}}\lesssim \vert \vert \partial_{t}\bar{\xi}\vert \vert_{W^{2,\frac{2}{1-\epsilon_{+}}}}\vert \vert u\vert\vert_{W^{1,\frac{2}{1-\epsilon_{-}}}}+\vert \vert \partial_{t}\bar{\xi}\vert \vert_{W^{1,+\infty}}\vert \vert u\vert \vert_{W^{2,q_{-}}}\lesssim \|\p_{t}{\xi}\|_{W^{3-\frac{1}{q_{-}},q_{-}}}\|u\|_{W^{2,q_{+}}}\lesssim\sqrt{\mathcal{E}}\sqrt{\mathcal{D}} \label{equ:5.1.264}
        \end{align}

        \textbf{Step 3} In this step, we estimate each term in $G_{+}^{3}$ individually. We have:

        \textbf{Term $\frac{-u\cdot \partial_{t}\mathcal{N}}{\vert \mathcal{N}_{0}\vert}$} 
        
        Using the fact that $W^{2-\frac{1}{q_{-}},q_{-}}$ is an algebra in one dimension, we have
        
        \begin{align}
            \vert \vert \frac{u\cdot \partial_{t}\mathcal{N}}{\vert \mathcal{N}_{0}\vert}\vert \vert_{W^{2-\frac{1}{q_{-}},q_{-}}}\lesssim \vert \vert u\vert \vert_{W^{2-\frac{1}{q_{-}},q_{-}}(\Sigma)}\vert \vert \partial_{\theta}\partial_{t}\xi\vert \vert_{W^{2-\frac{1}{q_{-}},q_{-}}}\lesssim \vert \vert u\vert \vert_{W^{2,q_{-}}}\vert \vert\partial_{t}\xi \vert \vert_{W^{3-\frac{1}{q_{-}},q_{-}}}\lesssim \sqrt{\mathcal{E}}\sqrt{\mathcal{D}}\label{equ:5.1.265}
        \end{align}

        \textbf{Term $\partial_{t}(\frac{\rho\partial_{t}\xi}{\vert \mathcal{N}_{0}\vert})$} 
        
        We have the following computation:

        \begin{align}
           \vert \vert \partial_{t}(\frac{\rho\partial_{t}\xi}{\vert \mathcal{N}_{0}\vert})\vert \vert_{W^{2-\frac{1}{q_{-}},q_{-}}} \lesssim&  \vert \vert (\partial_{t}\xi)^{2}\vert \vert_{W^{2-\frac{1}{q_{-}},q_{-}}}+\vert \vert \partial_{t}^{2}\xi\vert \vert_{W^{2-\frac{1}{q_{-}},q_{-}}}\notag\\
           \lesssim& \vert \vert \partial_{t}\xi\vert \vert_{W^{2-\frac{1}{q_{-}},q_{-}}}^{2}+\vert \vert \partial_{t}^{2}\xi\vert \vert_{H^{\frac{3}{2}-\alpha}}\lesssim \vert \vert \partial_{t}\xi\vert \vert_{H^{\frac{3}{2}-\alpha}}^{2}+\vert \vert \partial_{t}^{2}\xi\vert \vert_{H^{\frac{3}{2}-\alpha}}\lesssim \mathcal{E}+\|\p_{t}^{2}\xi\|_{H^{\frac{3}{2}-\alpha}}\label{equ:5.1.266}
        \end{align}

        \textbf{Term $\partial_{t}(\mathfrak{n}'(t)\xi_s\rho\frac{1}{\vert \mathcal{N}_{0}\vert})$} 
        
        We have the following computation by Sobolev embedding 

        \begin{align}
            &\vert \vert \partial_{t}(\mathfrak{n}'(t)\xi_s\rho\frac{1}{\vert \mathcal{N}_{0}\vert})\vert \vert_{W^{2-\frac{1}{q_{-}},q_{-}}}\lesssim \|\mathfrak{n}''(t)\xi_{s}\rho\frac{1}{|\mathcal{N}_{0}|}\|_{W^{2-\frac{1}{q_{-}},q_{-}}}+ \|\mathfrak{n}'(t)\xi_{s}\p_{t}\xi\frac{1}{|\mathcal{N}_{0}|}\|_{W^{2-\frac{1}{q_{-}},q_{-}}}\notag\\
            &\lesssim (\vert \vert \partial_{t}u\vert \vert_{H^{1}}+\mathcal{E})+\vert \vert u\vert \vert_{H^{1}}\vert \vert \partial_{t}\xi\vert \vert_{W^{2-\frac{1}{q_{-}},q_{-}}}\lesssim {\mathcal{E}}+\vert \vert\partial_{t}u \vert \vert_{H^{1}} \label{equ:5.1.267}
        \end{align}

        \textbf{Term $\partial_{t}(\mathfrak{n}'(t)\frac{(\rho'-\frac{\rho\rho_{0}'}{\rho_{0}})\cos\theta}{\vert \mathcal{N}_{0}\vert})$}

        We have the following computation by Sobolev embedding 

        \begin{align}
            \vert \vert \partial_{t}(\mathfrak{n}'(t)\frac{(\rho'-\frac{\rho\rho_{0}'}{\rho_{0}})\cos\theta}{\vert \mathcal{N}_{0}\vert})\vert \vert_{W^{2-\frac{1}{q_{-}},q_{-}}}\lesssim& \|\mathfrak{n}''(t)\xi\|_{W^{2-\frac{1}{q_{-}},q_{-}}}+\|\mathfrak{n}''(t)\xi'\|_{W^{2-\frac{1}{q_{-}},q_{-}}}+\|\mathfrak{n}'(t)\p_{t}\xi'\|_{W^{2-\frac{1}{q_{-}},q_{-}}}\notag\\
            &+\|\mathfrak{n}'(t)\p_{t}\xi\|_{W^{2-\frac{1}{q_{-}},q_{-}}}\notag\\
            \lesssim& (\vert \vert \p_{t}u\vert \vert_{H^{1}}+\mathcal{E})\vert \vert \xi\vert \vert_{W^{3-\frac{1}{q_{-}},q_{-}}}+\vert \vert u\vert \vert_{H^{1}}\vert \vert \partial_{t}\xi\vert \vert_{W^{3-\frac{1}{q_{-}},q_{-}}}\lesssim \sqrt{\mathcal{E}\mathcal{D}}
        \end{align}

        Combining the computation for all of the terms above, we obtain the estimate for $G_{+}^3$

       \begin{align}
           \vert \vert G_{+}^{3}\vert \vert_{W^{2-\frac{1}{q_{-}},q_{-}}}\lesssim \vert \vert \partial_{t}\xi\vert \vert_{H^{\frac{3}{2}-\alpha}}^{2}+\vert \vert \partial_{t}^{2}\xi\vert \vert_{H^{\frac{3}{2}-\alpha}}+\vert \vert \partial_{t}u\vert \vert_{H^{1}}+\sqrt{\mathcal{ED}}
       \end{align}
       
       \textbf{Step 4} In this step, we estimate each term in $G^{4,1}_{+}$ individually. Our goal is to bound

       \begin{align}
           \vert \vert \mathbb{D}_{\partial_{t}\mathcal{A}}u\mathcal{N}\cdot \frac{\mathcal{T}}{\vert \mathcal{T}\vert^{2}}\vert \vert_{W^{1-\frac{1}{q_{-}},q_{-}}}\label{equ:5.1.277}
       \end{align}

       \noindent where
       $\mathcal{N}=\rho(\theta)\hat{e}_{r}+\rho'(\theta)\hat{e}_{\theta}$ and $\mathcal{T}=\rho'(\theta)\hat{e}_{r}-\rho(\theta)\hat{e}_{\theta}$. From the definition of $\mathcal{N}$ and $\mathcal{T}$, we have

       \begin{align}
           \vert \vert \frac{\mathcal{T}}{\vert \mathcal{T}\vert^{2}}\vert \vert_{W^{1,q_{-}}}+\vert \vert \frac{\mathcal{N}}{\vert \mathcal{N}\vert}\vert \vert_{W^{1,q_{-}}}+\vert \vert \mathcal{N}\vert \vert_{W^{1,q_{-}}}\lesssim 1+\vert \vert \xi\vert \vert_{W^{2,q_{-}}}\lesssim 1 \label{equ:5.1.278}
       \end{align}

       \noindent  Using \eqref{equ:5.1.279} and Theorem 4.7 in \eqref{equ:5.1.277}, we obtain

       \begin{align}
           \vert \vert \mathbb{D}_{\partial_{t}\mathcal{A}}u\mathcal{N}\cdot \frac{\mathcal{T}}{\vert \mathcal{T}\vert^{2}}\vert \vert_{W^{1-\frac{1}{q_{-}},q_{-}}}&\lesssim \vert \vert \mathbb{D}_{\partial_{t}\mathcal{A}}u\vert \vert_{W^{1-\frac{1}{q_{-}},q_{-}}}\vert \vert \mathcal{N}\vert \vert_{W^{1,q_{-}}}\vert \vert \frac{\mathcal{T}}{\vert \mathcal{T}\vert^{2}}\vert \vert_{W^{1-\frac{1}{q_{-}},q_{-}}}\notag\\
           &\lesssim \vert \vert \nabla u\vert \vert_{W^{1,q_{-}}}\vert \vert \partial_{t}\xi\vert \vert_{W^{2-\frac{1}{q_{-}},q_{-}}}\lesssim \sqrt{\mathcal{E}}\sqrt{\mathcal{D}} \label{equ:5.1.279}
       \end{align}

       \textbf{Term $[\mathcal{K}(\xi)-S_{\mathcal{A}}(p,u)]\partial_{t}\mathcal{N}\cdot \frac{\mathcal{T}}{\vert \mathcal{T}\vert^{2}}$}

      By a similar computation used for the preceding term, we have 
       \begin{align}
           &\vert \vert [\mathcal{K}(\xi)-S_{\mathcal{A}}(p,u)]\partial_{t}\mathcal{N}\cdot \frac{\mathcal{T}}{\vert \mathcal{T}\vert^{2}}\vert \vert_{W^{1-\frac{1}{q_{-}},q_{-}}}\notag\\
           &\lesssim \vert \vert \partial_{\theta}\xi\partial_{\theta}\partial_{t}\xi\vert\vert_{W^{1-\frac{1}{q_{-}},q_{-}}}+\vert \vert \partial_{\theta}^{2}\xi\partial_{\theta}\partial_{t}\xi\vert \vert_{W^{1-\frac{1}{q_{-}},q_{-}}}
           +\vert \vert p\partial_{t}\partial_{\theta}\xi\vert \vert_{W^{1-\frac{1}{q_{-}},q_{-}}(\Sigma)}+\vert \vert \nabla u \partial_{t}\partial_{\theta}\xi\vert \vert_{W^{1-\frac{1}{q_{-}}q_{-}}(\Sigma)}
           \notag\\
           &\lesssim \vert\vert \xi\vert \vert_{W^{3-\frac{1}{q_{+}},q_{+}}}\vert \vert \partial_{t}\xi\vert \vert_{W^{2,q_{-}}}+\vert \vert p\vert \vert_{W^{1,q_{+}}}\vert \vert \partial_{t}\xi\vert \vert_{W^{2-\frac{1}{q_{-}},q_{-}}}+\vert \vert u\vert \vert_{W^{2,q_{+}}}\vert \vert \partial_{t}\xi\vert \vert_{W^{2,q_{-}}}\lesssim \mathcal{\sqrt{\mathcal{E}}\sqrt{\mathcal{D}}} \label{equ:5.1.280}
       \end{align}

       \textbf{Term $\mathcal{R}_{2}\partial_{t}\mathcal{N}\cdot \frac{\mathcal{T}}{\vert \mathcal{T}\vert^{2}}$}

       By an argument analogous to the one for the preceding term, we obtain
       \begin{align}
           \vert \vert \mathcal{R}_{2}\partial_{t}\mathcal{N}\cdot \frac{\mathcal{T}}{\vert \mathcal{T}\vert^{2}}\vert \vert_{W^{1-\frac{1}{q_{-}},q_{-}}}\lesssim& \vert \vert (\partial_{\theta}\xi)^{2}\partial_{t}\partial_{\theta}\xi\vert \vert_{W^{1-\frac{1}{q_{-}},q_{-}}}+\vert\vert \xi^{2}\partial_{t}\partial_{\theta}\xi\vert \vert_{W^{1-\frac{1}{q_{-}},q_{-}}}\notag\\
           \lesssim& \vert \vert \partial_{\theta}\xi\vert \vert^{2}_{W^{1-\frac{1}{q_{-}},q_{-}}}\vert \vert \partial_{t}\xi\vert \vert_{W^{2,q_{-}}}+\vert \vert \xi\vert \vert^{2}_{W^{1-\frac{1}{q_{-}},q_{-}}}\vert \vert \partial_{t}\xi\vert \vert_{W^{2,q_{-}}}\lesssim \vert \vert \xi\vert \vert_{W^{2,q_{+}}}^{2}\vert \vert \partial_{t}\xi\vert \vert_{W^{2,q_{-}}}\lesssim \mathcal{E}\sqrt{\mathcal{D}}\label{equ:5.1.281}
       \end{align}

        \textbf{Term $\p_{\theta}\mathcal{R}_{1}\p_{t}\mathcal{N}\cdot \frac{\mathcal{T}}{|\mathcal{T}|^{2}}$}

       We have the following computation:
       \begin{align}
           \|\p_{\theta}\mathcal{R}_{1}\p_{t}\mathcal{N}\cdot \frac{\mathcal{T}}{|\mathcal{T}|^{2}}\|_{W^{1-\frac{1}{q_{-}},q_{-}}}\lesssim&  \|\p_{\theta}^{2}\xi\p_{\theta}\xi\p_{t}\p_{\theta}\xi\|_{W^{1-\frac{1}{q_{-}},q_{-}}}+ \|\p_{\theta}^{2}\xi\xi\p_{t}\p_{\theta}\xi\|_{W^{1-\frac{1}{q_{-}},q_{-}}}\notag\\
           &+ \|(\p_{\theta}\xi)^{2}\p_{t}\p_{\theta}\xi\|_{W^{1-\frac{1}{q_{-}},q_{-}}}+ \|\p_{\theta}\xi\xi\p_{t}\p_{\theta}\xi\|_{W^{1-\frac{1}{q_{-}},q_{-}}}\notag\\
           \lesssim& \|\xi\|_{W^{3-\frac{1}{q_{-}},q_{-}}}\|\p_{t}\xi\|_{W^{2,q_{-}}}\|\xi\|_{W^{2,q_{-}}}\lesssim \mathcal{E}^{\frac{1}{2}}\mathcal{D}
       \end{align}
       
       \textbf{Step 5} In this step, we establish the estimate for the term $G_{-}^{4,1}$, which can be expressed as $D_{\partial_{t}\mathcal{A}}u\nu\cdot \tau$. We have

       \begin{align}
           \vert \vert D_{\partial_{t}\mathcal{A}}u\nu\cdot \tau\vert \vert_{W^{1-\frac{1}{q_{-}},q_{-}}(\Sigma_{s})}\lesssim  \vert \vert u\vert \vert_{W^{2,q_{-}}} \vert \vert \partial_{t}\eta\vert \vert_{W^{2-\frac{1}{q_{-}},q_{-}}}\lesssim \sqrt{\mathcal{E}}\sqrt{\mathcal{D}}\label{equ:5.1.282}
       \end{align}
       
       \textbf{Step 6} In this step, we establish estimate for $G^{5,1}$. The boundedness can be shown by the same computations used in Step 4, with the only modification of replacing $\frac{\mathcal{T}}{|\mathcal{T}|^{2}}$ by $\frac{\mathcal{N}}{|\mathcal{N}|^{2}}$. The only exceptions are the terms $\partial_{t}\mathcal{R}_{2}$ and $\p_{t}\p_{\theta}\mathcal{R}_{1}$, for which additional care is required. We bound these two terms as follows

       \begin{align}
           \vert\vert \partial_{t}\mathcal{R}_{2}\vert \vert_{W^{1-\frac{1}{q_{-}},q_{-}}}\lesssim \vert \vert \partial_{t}\xi\vert \vert_{W^{2-\frac{1}{q_{-}},q_{-}}}\vert \vert \xi\vert \vert_{W^{2,q_{-}}}\lesssim \sqrt{\mathcal{E}}\sqrt{\mathcal{D}}\label{equ:5.1.283},
       \end{align}

       \noindent and

       \begin{align}
           \vert \vert \partial_{\theta}\partial_{t}\mathcal{R}_1\vert \vert_{W^{1-\frac{1}{q_{-}},q_{-}}}=&\vert \vert \partial_{\theta}(\partial_{b}\mathcal{R}_1\partial_{t}\partial_{\theta}\xi)\vert \vert_{W^{1-\frac{1}{q_{-}},q_{-}}}+\vert \vert \partial_{\theta}(\partial_{c}\mathcal{R}_1\partial_{t}\xi)\vert \vert_{W^{1-\frac{1}{q_{-}},q_{-}}}\notag\\
           &\lesssim \vert \vert \xi\vert \vert_{W^{2,q_{-}}}\vert \vert \partial_{t}\xi\vert\vert_{W^{3-\frac{1}{q_{-}},q_{-}}}+\vert \vert \partial_{t}\xi\vert \vert_{W^{2,q_{-}}}\vert \vert \xi\vert \vert_{W^{3-\frac{1}{q_{+}},q_{+}}}\lesssim \sqrt{\mathcal{E}}\sqrt{\mathcal{D}}\label{equ:5.1.284}.
       \end{align}

       \textbf{Step 8} In this step, we establish the estimate for $G^{7,1}$. We have the following boundedness by using the definition of $G^{7}$ and Theorem \ref{thm:gam}:

       \begin{align}
           [\partial_{t}G^{7}]_{\theta}&\lesssim \vert \mathfrak{n}''(t)\vert+\vert \partial_{t}^{2}\xi\vert+\vert \partial_{t}\mathcal{R}_1\vert+(\vert \mathfrak{n}''(t)\vert+\vert \partial_{t}^{2}\xi\vert)[\hat{W}']_{\theta}\notag\\
           &\lesssim \vert \vert \partial_{t}u\vert \vert_{H^{1}}+\vert \vert \partial_{t}^{2}\xi\vert \vert_{H^{1}}+\vert \vert \partial_{t}\xi\vert \vert_{W^{2,q_{-}}}\vert \vert \xi\vert \vert_{W^{2,q_{-}}}+(\vert \mathfrak{n}''(t)\vert+[\partial_{t}^{2}\xi]_{\theta})[\hat{W}']_{\theta}+\mathcal{E}\notag\\
           &\lesssim \vert \vert \partial_{t}u\vert \vert_{H^{1}}+\vert \vert \partial_{t}^{2}\xi\vert \vert_{H^{1}}+\sqrt{\mathcal{E}}\sqrt{\mathcal{D}}+(\vert \vert u\vert\vert_{H^{1}}+\vert \vert \partial_{t}\xi\vert \vert_{W^{2,q_{+}}})(\vert \vert \partial_{t}u\vert \vert_{H^{1}}+\vert \vert \partial_{t}^{2}\xi\vert \vert_{H^{\frac{3}{2}-\alpha}}+\mathcal{E})\notag\\
           &\lesssim \vert \vert \partial_{t}u\vert \vert_{H^{1}}+\vert \vert \partial_{t}^{2}\xi\vert \vert_{H^{1}} +\sqrt{\mathcal{E}}\sqrt{\mathcal{D}}+\vert \vert u\vert\vert_{H^{1}}\sqrt{\mathcal{D}}\label{equ:5.1.285}
       \end{align}

       Combining all of the estimate above, and using elliptic estimate derived by Theorem 4.7 in \cite{Guo}, we obtain

       \begin{align}
          & \vert \vert \partial_{t}u\vert \vert_{W^{2,q_{-}}}+\vert \vert \partial_{t}u\vert \vert_{H^{1+\epsilon_{-}}}+\vert \vert \partial_{t}p\vert \vert_{W^{1,q_{-}}}+\vert \vert \partial_{t}p\vert \vert_{H^{\epsilon_{-}}}+\vert \vert \partial_{t}\xi\vert\vert_{W^{3-\frac{1}{q_{-}},q_{-}}}\notag\\
           &\lesssim \mathcal{E}_{||}+\sqrt{\mathcal{E}}\sqrt{\mathcal{D}}+\vert \vert \partial_{t}^{2}\xi\vert \vert_{H^{\frac{3}{2}-\alpha}}+\vert \vert \partial_{t}\xi\vert \vert_{H^{\frac{3}{2}-\alpha}}^{2}+\vert \vert \partial_{t}u\vert \vert_{H^{1}}+\vert \vert u\vert \vert_{H^{1}}\sqrt{\mathcal{D}}+\vert \vert \partial_{t}\xi\vert \vert_{H^{1}}+\mathcal{E}\label{equ:5.1.286}
       \end{align}
       \end{proof}

       After establishing the elliptic estimate, we now derive the enhanced estimate for $\xi$, $\partial_{t}\xi$ and $\partial_{t}^{2}\xi$. 

       \begin{theorem}
           Suppose that $\xi$ is the perturbation function. The it satisfies enhanced $\frac{3}{2}-$ estimate.

           \begin{align}
              \int_{0}^{T}\vert \vert \xi\vert \vert^{2}_{H^{\frac{3}{2}-\alpha}}\lesssim \mathcal{E}_{||}(0)+\mathcal{E}_{||}(t)+\int_{0}^{T}(\mathcal{D}_{||}+(\mathcal{E}_{||})^{\frac{1}{2}}\mathcal{D}+(\mathcal{E}_{||})\mathcal{D})
          \end{align}
       \end{theorem}

       \begin{proof}
           Let test function $\psi$ satisfying the following conditions:

           \begin{align}
                -\Delta \psi=0~in~\Omega~~\partial_{\nu}\psi=\rho_{0}\frac{D_{j}^{s}(\xi-a_{0}\rho_{0})}{\vert \mathcal{N}_{0}\vert}~on~\Sigma, ~~\partial_{\nu}\psi=0~~on~\Sigma_{s}\label{equ:5.1.287}
           \end{align}

           \noindent where $a_{0}=a_{0}(t)$ is a function defined by \eqref{equ:a}. We recall its expression:

           \begin{align}
               a_0(t)=\frac{-\frac{1}{2}\int_{0}^{\pi}\xi^{2}}{\int_{0}^{\pi}\rho_{0}^{2}}\lesssim \mathcal{E}_{||},
           \end{align}

           \noindent which implies that:

           \begin{align}
               \int_{0}^{\pi}\rho_{0}(\xi-a_{0}(t)\rho_{0})=0.
           \end{align}

           \noindent For the test function $\psi$, it satisfies the following three estimates

           \begin{align}
               \vert \vert \psi\vert \vert_{H^{1}}\lesssim \vert \vert \xi\vert \vert_{H^{s-\frac{1}{2}}}+a_{0}(t)\vert \vert \rho_{0}\vert \vert_{H^{s-\frac{1}{2}}}\lesssim \vert \vert \xi\vert \vert_{H^{s-\frac{1}{2}}}+\mathcal{E}_{||},
           \end{align}

           \noindent and

           \begin{align}
               \|\psi\|_{H^{2}}\lesssim \|D_{j}^{s}\xi\|_{\mathcal{H}_{\mathcal{K}}^{\frac{1}{2}}},
           \end{align}

           \noindent and:

           \begin{align}
               \|\p_{t}\psi\|_{H^{1}}\lesssim \|\p_{t}^{k+1}\xi\|_{H^{s-\frac{1}{2}}}.
           \end{align}
           
           \noindent Then applying test function $\phi=M\nabla\psi$ to the equation system \eqref{equ:0}, we have

           \begin{align}
               &(\partial_{t}u,J\phi)+((u\cdot \nabla_{\mathcal{A}}u+\mathfrak{n}'(t)\partial_{x}u+(\cos\theta W\partial_{t}\bar{\xi},\sin\theta W\partial_{t}\bar{\xi})\tilde{K}(\partial_{x}u,\partial_{y}u)^{T}),J\phi)_{0}\notag\\
               &+((u,\phi))+(\xi,\frac{1}{\rho_{0}}\phi\cdot \mathcal{N})_{1,\Sigma}+\kappa [\mathfrak{n}'(t)+\partial_{t}\xi,\frac{1}{\rho_{0}}\phi\cdot \mathcal{N}]\notag\\
               &=\kappa[b^{7},\frac{1}{\rho_{0}}\phi\cdot \mathcal{N}]-\int_{0}^{\pi}b^{3}\partial_{\theta}(\phi\cdot \mathcal{N})-\int_{0}^{\pi} b^{4}\cdot \phi d\theta\label{equ:5.1.288}
           \end{align}

           \noindent We then estimate each term in \eqref{equ:5.1.288} individually

           \textbf{Term $\int_{\Omega}\partial_{t}uJ\phi$}. 
           
           We have the following estimate for this term
           
           \begin{align}
               \int_{\Omega}\partial_{t}uJ\phi=\int_{\Omega}\partial_{t}u\cdot \nabla \Phi\nabla \psi=\frac{d}{dt}\int_{\Omega} u\cdot \nabla \Phi\nabla \psi-\int_{\Omega}u \cdot \nabla \Phi\nabla \partial_{t}\psi-\int_{\Omega}u\cdot \partial_{t}(\nabla \Phi)\nabla \psi
           \end{align}

           For the first term on the right hand side of equation above, we apply the same estimate as in Theorem 9.2 of Guo–Tice paper \cite{Guo} to obtain

           \begin{align}
               \vert \int_{\Omega}\partial_{t}uJ\phi-\frac{d}{dt}\int_{\Omega} u\cdot \nabla \Phi\nabla \psi\vert\lesssim  \vert \vert u\vert \vert_{L^{2}}(\vert \vert \xi\vert \vert_{H^{1-\frac{s}{2}}}+\vert \vert \partial_{t}\xi\vert \vert_{H^{1-\frac{s}{2}}}),
           \end{align}

           \noindent and

           \begin{align}
               \int_{\Omega} u\cdot \nabla \Phi\nabla \psi\lesssim \mathcal{E}_{||}.
           \end{align}

           \textbf{Term $((u\cdot \nabla_{\mathcal{A}}u+\mathfrak{n}'(t)\partial_{x}u+(\cos\theta W\partial_{t}\bar{\xi},\sin\theta W\partial_{t}\bar{\xi})\tilde{K}(\partial_{x}u,\partial_{y}u)^{T}),J\phi)_{0}$}
           
           We have:

           \begin{align}
               &((u\cdot \nabla_{\mathcal{A}}u+\mathfrak{n}'(t)\partial_{x}u+(\cos\theta W\partial_{t}\bar{\xi},\sin\theta W\partial_{t}\bar{\xi})\mathcal{A}(\partial_{x}u,\partial_{y}u)^{T}),J\phi)_{0}\notag\\
              & \lesssim \vert \vert u\vert \vert_{L^{\infty}}\vert \vert u\vert \vert_{H^{1}}\vert \vert \phi\vert \vert_{L^{2}}+\vert \vert u\vert \vert_{H^{1}}\vert \vert u\vert \vert_{H^{1}}\vert \vert \phi\vert \vert_{L^{2}}+\vert \vert \partial_{t}\bar{\xi}\vert \vert_{L^{+\infty}}\vert \vert u\vert \vert_{H^{1}}\vert \vert \phi\vert \vert_{L^{2}}\notag\\
               &\lesssim \mathcal{E}\|\xi\|_{H^{1}}\lesssim\mathcal{E}^{\frac{3}{2}}
           \end{align}

           \textbf{Term $((u,\phi))$} 
           
           We have the following estimate:

           \begin{align}
               ((u,\phi))\lesssim \vert\vert u\vert \vert_{H^{1}}\vert \vert \psi\vert \vert_{H^{2}}\lesssim \vert \vert u\vert \vert_{H^{1}}\vert \vert D_{j}^{s}\xi\vert \vert_{H_{\mathcal{K}}^{\frac{1}{2}}}\lesssim\|u\|_{H^{1}}\|\xi\|_{H^{\frac{3}{2}-\alpha}}
           \end{align}

           \textbf{Term $(\xi,\frac{1}{\rho_{0}}\phi\cdot \mathcal{N})_{1,\Sigma}$} 
           
           By definition of $\phi$, we have

           \begin{align}
               \frac{1}{\rho_{0}}\phi\cdot \mathcal{N}=\frac{1}{\rho_{0}}M\nabla \psi\cdot\mathcal{N}=\frac{1}{\rho_{0}}\partial_{\nu} \psi \vert  \mathcal{N}_{0}\vert =D_{j}^{s}(\xi-a_{0}(t)\rho_{0})
           \end{align}

           \noindent Therefore, we obtain
           
           \begin{align}
               (\xi,\frac{1}{\rho_{0}}\phi\cdot \mathcal{N})_{1,\Sigma}=(\xi,D_{j}^{s}(\xi-a_{0}(t)\rho_{0}))_{1,\Sigma}=(D_{j}^{\frac{s}{2}}(\xi-a_{0}(t)\rho_{0}),D_{j}^{\frac{s}{2}}(\xi-a_{0}(t)\rho_{0}))_{1,\Sigma}\notag\\
             \gtrsim \vert \vert D_{j}^{\frac{s}{2}}\xi\vert \vert^{2}_{H^{1}}\gtrsim \vert \vert \xi\vert \vert^{2}_{H^{1+\frac{s}{2}}}-a_{0}(t)^{2}\vert \vert \rho_{0}\vert \vert^{2}_{H^{1+\frac{s}{2}}}\gtrsim \vert \vert \xi\vert \vert^{2}_{H^{1+\frac{s}{2}}}-\mathcal{E}^{2}_{||}
           \end{align}

           \noindent  For the estimate above, to show that $\vert \vert D_{j}^{\frac{s}{2}}\xi\vert \vert^{2}_{H^{1}}\gtrsim \vert \vert \xi\vert \vert^{2}_{H^{1+\frac{s}{2}}}-a_{0}^{2}(t)\vert \vert \rho_{0}\vert \vert^{2}_{H^{1+\frac{s}{2}}}$, we need to use the fact that $\xi \in H^{1}$ and we use $\tilde{\xi}$ to denote its projection on the function space ${B}$, which is $\xi-a_{0}\rho_{0}$. ($B$ is the function spaces defined in Section 4). Hence, by this idea, we have:

           \begin{align}
               \vert \vert D_{j}^{\frac{s}{2}}\xi\vert \vert^{2}_{H^{1}}=\vert \vert D^{\frac{s}{2}}\tilde{\xi}\vert \vert_{H^{1}}^{2}\gtrsim \vert \vert \tilde{\xi}\vert \vert_{H^{1+\frac{s}{2}}}=\vert \vert \xi-a_{0}(t)\rho_{0}\vert \vert_{H^{1+\frac{s}{2}}}\gtrsim \vert \vert \xi\vert \vert^{2}_{H^{1+\frac{s}{2}}}-a_{0}^{2}(t)\vert \vert \rho_{0}\vert \vert^{2}_{H^{1+\frac{s}{2}}}
           \end{align}
           
           \textbf{Term $\kappa [\mathfrak{n}'(t)+\partial_{t}\xi,\phi\cdot \mathcal{N}]_{\theta}$} 
           
           We have the following estimate:

           \begin{align}
               \kappa [\mathfrak{n}'(t)+\partial_{t}\xi,\phi\cdot \mathcal{N}]_{\theta}&\lesssim  \frac{1}{\epsilon}[\partial_{t}\xi+\mathfrak{n}'(t)\xi_{s}]^{2}_{\theta}+\epsilon\vert \vert \xi\vert \vert_{H^{1+\frac{s}{2}}}^{2}\notag\\
               &\lesssim \vert \vert\partial_{t}\xi \vert \vert^{2}_{H^{1}}+\vert \vert u\vert \vert^{2}_{H^{1}}+\epsilon\vert \vert \eta\vert \vert_{H^{1+\frac{s}{2}}}^2{}\lesssim \mathcal{E}_{||}+\|u\|_{H^{1}}^{2}+\epsilon\vert \vert \xi\vert \vert_{H^{1+\frac{s}{2}}}^{2}\label{equ:5.1.289}
           \end{align}

           \textbf{Term $[b^{7},\frac{1}{\rho_{0}}\phi\cdot \mathcal{N}]_{\theta}$} Applying the similar estimate as in the equation \eqref{equ:5.1.289}, we have the following computation:

           \begin{align}
               [b^{7},\frac{1}{\rho_{0}}\phi\cdot \mathcal{N}]_{\theta}\lesssim \frac{1}{\epsilon}([\mathfrak{n}'(t)]_{\theta }^{2}+[\partial_{t}\xi]^{2}_{\theta})^{2}+\epsilon \vert \vert \xi\vert \vert_{H^{1+\frac{s}{2}}}^{2}\lesssim \|u\|_{H^{1}}^{2}+\mathcal{E}_{||}+\epsilon \vert \vert \xi\vert \vert^{2}_{H^{1+\frac{s}{2}}}
           \end{align}

           \textbf{Term $\int_{0}^{\pi}b^{3}\partial_{\theta}(\phi\cdot \mathcal{N})$} 
           
           We have the following by the definition of $b^{3}$

           \begin{align}
               \int_{0}^{\pi}b^{3}\partial_{\theta}(\phi\cdot \mathcal{N})&\lesssim \vert \vert b^{3}\vert \vert_{H^{\frac{s}{2}}}\vert \vert \xi\vert \vert_{\mathcal{H}_{\mathcal{K}}^{1+\frac{s}{2}}}\lesssim \vert \vert \xi\vert \vert_{H^{1+\frac{s}{2}}}\vert \vert \mathcal{R}_1\vert \vert_{H^{\frac{s}{2}}}\lesssim \vert \vert \xi\vert \vert_{H^{1+\frac{s}{2}}}\vert \vert \xi\vert \vert_{H^{1+\frac{s}{2}}}\vert \vert \partial_{\theta}\xi\vert \vert_{H^{\frac{s}{2}}}\notag\\
               &\lesssim \vert \vert \xi\vert \vert^{2}_{H^{1+\frac{s}{2}}}\vert \vert \xi\vert \vert_{W^{2,q_{+}}}\label{equ:5.1.290}
           \end{align}

           \textbf{Term $\int_{0}^{\pi}b^{4}\cdot \phi$} 
           
           We have the following computation by the definition of $b^{4}$

           \begin{align}
               \int_{0}^{\pi}b^{4}\cdot \phi d\theta\lesssim \mathcal{E}\|\phi\|_{H^{1}}\lesssim \mathcal{E}\vert \vert \xi\vert \vert_{H^{s+\frac{1}{2}}}\lesssim\mathcal{E}\|\xi\|_{H^{\frac{3}{2}-\alpha}}
           \end{align}

          \noindent Combining all of the computation above and choose a small $\epsilon$, we obtain the following result

          \begin{align}
              \int_{0}^{T}\vert \vert \xi\vert \vert^{2}_{H^{\frac{3}{2}-\alpha}}\lesssim \mathcal{E}_{||}(0)+\mathcal{E}_{||}(t)+\int_{0}^{T}(\vert \vert u\vert \vert^{2}_{H^{1}}+[u\cdot \mathcal{N}]^{2}_{\theta}+\mathcal{D}_{||}+(\mathcal{E}_{||})^{\frac{3}{2}}+(\mathcal{E}_{||})^{2}+\vert \vert\partial_{t}\xi \vert \vert_{H^{s-\frac{1}{2}}}^{2})\notag\\
              \lesssim \mathcal{E}_{||}(0)+\mathcal{E}_{||}(t)+\int_{0}^{T}(\mathcal{D}_{||}+(\mathcal{E}_{||})^{\frac{1}{2}}\mathcal{D}+(\mathcal{E}_{||})\mathcal{D}+\vert \vert\partial_{t}\xi \vert \vert_{H^{s-\frac{1}{2}}}^{2})\label{equ:5.1.291}
          \end{align}

           Using the fact that $\partial_{t}\xi=\frac{1}{\rho}u\cdot \mathcal{N}-\mathfrak{n}'(t)\xi_{2}$, we have:

          \begin{align}
              \vert \vert \partial_{t}\xi\vert \vert_{H^{s-\frac{1}{2}}}\lesssim \vert \vert u\vert \vert_{H^{s-\frac{1}{2}}(\Sigma)}+\vert \vert u\vert \vert_{H^{1}}+\mathcal{E}\lesssim \vert \vert u\vert \vert_{H^{1}}+\mathcal{E}\label{equ:5.1.292}
          \end{align}

          \noindent Then substituting \eqref{equ:5.1.292} into the equation \eqref{equ:5.1.291}, we obtain 

          \begin{align}
              \int_{0}^{T}\vert \vert \xi\vert \vert^{2}_{H^{\frac{3}{2}-\alpha}}\lesssim \mathcal{E}_{||}(0)+\mathcal{E}_{||}(t)+\int_{0}^{T}(\mathcal{D}_{||}+(\mathcal{E}_{||})^{\frac{1}{2}}\mathcal{D}+(\mathcal{E}_{||})\mathcal{D})\label{equ:5.1.293}
          \end{align}

       \end{proof}

       \begin{theorem}
          $\vert \vert \partial_{t}\xi\vert \vert_{H^{\frac{3}{2}-\alpha}}$ and $\|\partial_{t}^{2}\xi\|_{H^{\frac{3}{2}-\alpha}}$ satisfy the following enhanced estimate:

          \begin{align}
               \int_{0}^{T}\vert \vert \partial_{t}^{2}\xi\vert \vert^{2}_{H^{\frac{3}{2}-\alpha}}+\int_{0}^{T}\|\p_{t}\xi\|^{2}_{H^{\frac{3}{2}-\alpha}}\lesssim (\mathcal{E}_{||}(T))+(\mathcal{E}_{||}(0))+\int_{0}^{T}(\mathcal{D}_{||}+\sqrt{\mathcal{E}}\mathcal{D})
           \end{align}
       \end{theorem}

       \begin{proof}
           For simplicity, we only establish the boundedness of $\|\p_{t}^{2}\xi\|_{H^{\frac{3}{2}-\alpha}}$. We define the test function as follows:

           \begin{align}
                -\Delta \psi=0~in~\Omega~~\partial_{\nu}\psi=\rho_{0}\frac{D_{j}^{s}(\partial_{t}^{2}\xi-a_{2}\rho_{0})}{\vert \mathcal{N}_{0}\vert}~on~\Sigma, ~~\partial_{\nu}\psi=0~~on~\Sigma_{s}
           \end{align}

           \noindent where:

           \begin{align}
               a_{2}(t)=-\frac{\int_{0}^{\pi}(\partial_{t}\xi)^{2}d\theta+\int_{0}^{\pi}\xi\partial_{t}^{2}\xi d\theta}{\int_{0}^{\pi}\rho_{0}^{2}d\theta }\lesssim \|\p_{t}^{2}\xi\|_{L^{2}}^{2}+\|\xi\|_{L^{2}}\|\p_{t}^{2}\xi\|_{L^{2}}\lesssim  \mathcal{E}_{||}
           \end{align}
           \noindent Using standard elliptic estimate for $\psi$, we have

           \begin{align}
               \vert \vert \psi\vert \vert_{H^{1}}\lesssim \vert \vert \partial_{t}^{2}\xi\vert \vert_{H^{s-\frac{1}{2}}}+\mathcal{E}_{||},~~\vert \vert \psi\vert \vert_{H^{2}}\lesssim\vert \vert \partial_{t}^{2}\xi\vert \vert_{H^{s+\frac{1}{2}}}+\mathcal{E}_{||},~\operatorname{and}~~\vert \vert \partial_{t}\psi\vert \vert_{H^{1}}\lesssim \vert \vert \partial_{t}^{3}\xi\vert \vert_{H^{s-\frac{1}{2}}}+\mathcal{E}_{||}
           \end{align}

           \noindent Let $s=1-2\alpha$ and let the test $\phi$ be $\phi=M\nabla \psi$. We have the following equation from the weak-form equation \eqref{equ:4.1.24}.

           \begin{align}
                (\partial_{t}^{3}u,J\phi)
               +((\partial_{t}^{2}u,\phi))+(\partial_{t}^{2}\xi,\frac{1}{\rho_{0}}\phi\cdot \mathcal{N})_{1,\Sigma}+\kappa [\frac{1}{\rho_{0}}\p_{t}^{2}u\cdot \mathcal{N},\frac{1}{\rho_{0}}\phi\cdot \mathcal{N}]\notag\\
               =\int_{\Omega}b^{1,2}\cdot \phi J-\int_{\Sigma_{s}}J(\omega\cdot \tau)b^{5,2}+\kappa[b^{7,2}+b^{8,2},\frac{1}{\rho_{0}}\phi\cdot \mathcal{N}]-\int_{0}^{\pi}b^{3,2}\partial_{\theta}(\phi\cdot \mathcal{N})-\int_{0}^{\pi} b^{4,2}\cdot \phi d\theta\label{equ:3/2}
           \end{align}

           \noindent We estimate each term on in equation \eqref{equ:3/2} individually.

           \textbf{Term $ (\partial_{t}^{3}u,J\phi)$}  We have the following computation:

           \begin{align}
               (\partial_{t}^{3}u,\nabla \Phi\nabla \psi)= \partial_{t}(\partial_{t}^{2}u,\nabla \Phi\nabla \psi)-(\partial_{t}^{2}u,\partial_{t}\nabla \Phi\nabla\psi)-(\partial_{t}^{2}u,\nabla \Phi\nabla \partial_{t}\psi)\label{equ:32}
           \end{align}

           \noindent For the second and third term on the right hand side of the equation above, we have the following estimate

           \begin{align}
               (\partial_{t}^{2}u,\partial_{t}\nabla \Phi\nabla\psi)+(\partial_{t}^{2}u,\nabla \Phi\nabla \partial_{t}\psi)&\lesssim \vert \vert \partial_{t}^{2}u\vert \vert_{L^{2}}\vert \vert \partial_{t}\xi\vert \vert_{W^{1,+\infty}}\vert \vert \psi\vert \vert_{H^{1}}+\vert \vert \partial_{t}^{2}u\vert \vert_{L^{2}}\vert \vert \partial_{t}\psi\vert \vert_{H^{1}}+\mathcal{E}_{||}\vert \vert \partial_{t}^{2}u\vert \vert_{L^{2}}\notag\\
               &\lesssim \sqrt{\mathcal{E}_{||}}(\vert \vert \partial_{t}^{2}\xi\vert \vert_{H^{s-\frac{1}{2}}}+\vert \vert \partial_{t}^{3}\xi\vert \vert_{H^{s-\frac{1}{2}}})+(\mathcal{E}_{||})^{\frac{3}{2}}
           \end{align}

           \noindent For the first term in \eqref{equ:32}, we have

           \begin{align}
               (\partial_{t}^{2}u,\nabla \Phi\nabla \psi)\lesssim \vert \vert \partial_{t}^{2}u\vert \vert_{L^{2}}\vert \vert \partial_{t}^{2}\xi\vert \vert_{H^{s-\frac{1}{2}}}\lesssim \vert \vert \partial_{t}^{2}u\vert \vert_{L^{2}}\vert \vert \partial_{t}^{2}\xi\vert \vert_{H^{s-\frac{1}{2}}}\lesssim {\mathcal{E}_{||}}
           \end{align}

           \textbf{Term $((\partial_{t}^{2}u,\phi))$} 
           
           We have the following computation:

           \begin{align}
               ((\partial_{t}^{2}u,\phi))\lesssim \vert \vert \partial_{t}^{2}u\vert \vert_{H^{1}}\vert \vert \phi\vert \vert_{H^{1}}\lesssim \sqrt{\mathcal{D}_{||}}\vert \vert \partial_{t}^{2}\xi\vert \vert_{H^{1+\frac{s}{2}}}
           \end{align}

           \textbf{Term $(\partial_{t}^{2}\xi,\frac{1}{\rho_{0}}\phi\cdot \mathcal{N})_{1,\Sigma}$} 
           
           We have the following computation:

           \begin{align}
               (\partial_{t}^{2}\xi,\frac{1}{\rho_{0}}\phi\cdot \mathcal{N})_{1,\Sigma}=(\partial_{t}^{2}\xi,D_{j}^{s}(\partial_{t}^{2}\xi-a_{2}(t)\rho_{0}))_{1,\Sigma}\gtrsim \vert \vert \partial_{t}^{2}\xi\vert \vert^{2}_{H^{1+\frac{s}{2}}}-\mathcal{E}^{2}_{||} 
           \end{align}

           \noindent where where the final inequality follows from the same argument used in the proof of Theorem 7.8 to estimate $(\partial_{t}^{2}\xi,D_{j}^{s}(\partial_{t}^{2}\xi-a_{2}(t)\rho_{0}))_{1,\Sigma}$.
           
           \textbf{Term $\kappa[\frac{1}{\rho_{0}}\p_{t}^{2}u\cdot \mathcal{N},D_{j}^{s}\partial_{t}^{2}\xi]_{\theta}$}. 
           
           We have the following computation

           \begin{align}
               \kappa[\frac{1}{\rho_{0}}\p_{t}^{2}u\cdot \mathcal{N},D_{j}^{s}\partial_{t}^{2}\xi]_{\theta}\lesssim \frac{1}{\epsilon}[\p_{t}^{2}u\cdot \mathcal{N}]^{2}_{\theta}+\epsilon\vert \vert D_{j}^{s}\partial_{t}^{2}\xi\vert \vert_{H^{1-\frac{s}{2}}}\lesssim \frac{1}{\epsilon}\mathcal{D}_{||}+\epsilon\vert \vert \partial_{t}^{2}\xi\vert \vert_{H^{1+\frac{s}{2}}}
           \end{align}

           \textbf{Term $\int_{\Omega}b^{1,2}\cdot \phi J$} Using the result derived by step 1 in Theorem 7.4, we have
           \begin{align}
               \int_{\Omega} b^{1,2}\cdot (M\nabla \psi)J\lesssim \sqrt{\mathcal{E}}\sqrt{\mathcal{D}}\vert \vert \nabla\psi\vert \vert_{H^{1}}\lesssim \sqrt{\mathcal{E}}\sqrt{\mathcal{D}}\vert \vert \psi\vert \vert_{H^{2}}\lesssim \sqrt{\mathcal{E}}\sqrt{\mathcal{D}}\vert \vert \partial_{t}^{2}\xi\vert \vert_{H^{s+\frac{1}{2}}} 
           \end{align}

           \textbf{Term $[b^{7,2}+b^{8,2},\frac{1}{\rho_{0}}\phi\cdot \mathcal{N}]_{\theta}$} 
           
           Using Step 5 and Step 6 in the proof of Theorem 7.4, we have

           \begin{align}
               [b^{7,2}+b^{8,2},\frac{1}{\rho_{0}}\phi\cdot \mathcal{N}]_{\theta}\lesssim& \frac{1}{\epsilon}\sqrt{\mathcal{E}}\sqrt{\mathcal{D}}[\frac{1}{\rho_{0}}\phi\cdot \mathcal{N}]_{\theta}\lesssim \mathcal{E}\mathcal{D}+\epsilon[\frac{1}{\rho_{0}}\phi\cdot \mathcal{N}]^{2}_{\theta}\notag\\
               \lesssim& \frac{1}{\epsilon}\mathcal{E}\mathcal{D}+\epsilon\vert \vert \partial_{t}^{2}\xi\vert \vert^{2}_{H^{1+\frac{s}{2}}} 
           \end{align}

           \textbf{Term $\int_{\Sigma_{s}} J(\phi\cdot \tau)b^{5,2} $}
           
           Using the result derived by Step 5 of Theorem 7.4, we have

           \begin{align}
              \int_{\Sigma_{s}} J(\phi\cdot \tau)b^{5,2}\lesssim\mathcal{E}^{\frac{1}{2}}\mathcal{D}^{\frac{1}{2}} \vert \vert \phi\vert \vert_{H^{1}}\lesssim \mathcal{E}^{\frac{1}{2}}\mathcal{D}^{\frac{1}{2}}\vert \vert \phi\vert \vert_{H^{1}}\lesssim \mathcal{E}^{\frac{1}{2}}\mathcal{D}^{\frac{1}{2}}\vert \vert \partial_{t}^{2}\xi\vert \vert_{H^{s-\frac{1}{2}}}
           \end{align}

           \textbf{Term $\int_{0}^{\pi} b^{4,2}(\phi\cdot \mathcal{N})$}: 
           
           Using the result derived in energy estimate, we have

           \begin{align}
               \int_{0}^{\pi} b^{4,2}(\phi\cdot \mathcal{N}) d\theta\lesssim \sqrt{\mathcal{E}}\sqrt{\mathcal{D}} \vert \vert \phi\vert \vert_{H^{1}}\lesssim \sqrt{\mathcal{E}}\sqrt{\mathcal{D}} \vert \vert \partial_{t}^{2}\xi\vert \vert_{H^{\frac{1}{2}-s}}
           \end{align}

           \textbf{Term $\int_{0}^{\pi}b^{3,2}\partial_{\theta}(\phi\cdot \mathcal{N})$}. 
           
           We have the following estimate by the definition of $b^{3,2}$:

           \begin{align}
               \int_{0}^{\pi} b^{3,2}\partial_{\theta}(\phi\cdot \mathcal{N})d\theta\lesssim \vert \vert D_{j}^{s}\partial_{t}^{2}\xi\vert \vert_{H^{1-\frac{s}{2}}}(\vert \vert \partial_{t}^{2}\mathcal{R}_1\vert \vert_{H^{s}})
           \end{align}

           \noindent For the boundedness of $\vert \vert \partial_{t}^{2}\mathcal{R}_1\vert \vert_{H^{s}}$, applying the same argument as in Theorem 9.3 of Guo-Tice's paper \cite{Guo}, we obtain the following boundedness result

           \begin{align}
               \vert \vert \partial_{t}^{2}\mathcal{R}_1\vert \vert_{H^{s}}\lesssim \vert \vert\partial_{t}^{2}\xi \vert \vert_{H^{1+\frac{s}{2}}}\sqrt{\mathcal{E}}+\mathcal{E}\sqrt{\mathcal{D}}
           \end{align}

           Combining all of the computation above, and choosing a small $\epsilon$, we have

           \begin{align}
               \int_{0}^{T}\vert \vert \partial_{t}^{2}\xi\vert \vert^{2}_{H^{1+\frac{s}{2}}}\lesssim (\mathcal{E}_{||}(T))+(\mathcal{E}_{||}(0))+\int_{0}^{T}\mathcal{D}_{||}+\sqrt{\mathcal{E}}\mathcal{D}
           \end{align}
       \end{proof}

       We then use the energy terms to bound $\|\partial_{t}p\|_{L^{2}}$, we have the following theorem:

       \begin{theorem}
           $\p_{t}p$ satisfies the following boundedness result
          \begin{align}
              \|\p_{t}p\|_{L^{2}}^{2}\lesssim \mathcal{E}_{||}+\mathcal{E}^{\frac{3}{2}}
          \end{align}
       \end{theorem}

       \begin{proof}
           Let $\psi\in H^{2}(\Omega)$ solve the following equation system:

          \begin{equation}
              \begin{cases}
                  -\Delta \psi=\partial_{t}p ~~&\operatorname{in}~ \Omega\\
                  \psi=0~~&\operatorname{on}~\Sigma\\
                  \partial_{\nu} \psi=0~~&\operatorname{on}~\Sigma_{s}
              \end{cases}
          \end{equation}

          \noindent which exists and enjoys $H^{2}$ regularity as follows:

          \begin{align}
              \vert \vert \psi\vert \vert_{H^{2}}\lesssim\vert \vert \partial_{t}p\vert \vert_{L^{2}}
          \end{align}

          \noindent We choose $\phi=M\nabla \psi$ as the test function in weak form equation \eqref{equ:4.1.24}. It satisfies the following key property

          \begin{align}
              \operatorname{div}_{\mathcal{A}}\phi=\operatorname{div}_{\mathcal{A}}M\nabla \psi=K\Delta \psi=K\partial_{t}p
          \end{align}

         Applying $\phi$ as the test function for the first order equation, we obtain the following equation

          \begin{align}
              (\partial_{t}^{2}u,J\phi)+((\partial_{t}u,\phi))-(\partial_{t}p,\operatorname{div}_{\mathcal{A}}\phi)=\int_{\Omega} b^{1,1}\cdot \phi J-\int_{\Sigma_{s}}J(\phi\cdot \tau) b^{5,2}\notag\\
              -\int_{0}^{\pi}g(\mathcal{K}(\partial_{t}\xi)-\partial_{\theta}b^{3,1})\phi\cdot \mathcal{N}+b^{4,1}\cdot \omega \label{equ:5.1.300}
          \end{align}

          For the terms on the left hand side of the equation \eqref{equ:5.1.300}, we estimate each term individually as follows

          \begin{align}
              (\partial_{t}^{2}u,J\phi)\lesssim \vert \vert \partial_{t}^{2}u\vert \vert_{L^{2}}\vert \vert J\phi\vert \vert_{L^{2}}\lesssim \mathcal{E}_{||}\vert \vert \phi\vert \vert_{L^{2}}\lesssim \mathcal{E}_{||}\vert \vert \psi\vert \vert_{H^{1}}\lesssim \mathcal{E}_{||}\vert \vert \partial_{t}p\vert \vert_{L^{2}}
          \end{align}

          \begin{align}
              ((\partial_{t}u,J\phi))\lesssim \vert \vert \partial_{t}u\vert \vert_{H^{1}}\vert \vert J\phi\vert \vert_{H^{1}}\lesssim \sqrt{\mathcal{E}}\vert \vert \partial_{t}p\vert \vert_{L^{2}}
          \end{align}

          \begin{align}
              (\partial_{t}p,\operatorname{div}_{\mathcal{A}}\phi)_{L^{2}}=\int_{\Omega}J\partial_{t}p\operatorname{div}_{\mathcal{A}}\phi=\vert \vert \partial_{t}p\vert \vert^{2}_{L^{2}}
          \end{align}

           For the terms on the right-hand side of the equation \eqref{equ:5.1.300}, using the remark following Theorem 7.3, we have

          \begin{align}
             \vert  \int_{\Omega} b^{1,1}\cdot \phi J-\int_{\Sigma_{s}}J(\phi\cdot \tau) b^{5,2}-\int_{0}^{\pi}b^{4,1}\cdot \phi \vert \lesssim \mathcal{E}\vert \vert \phi\vert \vert_{H^{1}}\lesssim\mathcal{E}\vert\vert \partial_{t}p\vert \vert_{L^{2}}
          \end{align}

          \noindent Moreover, we have following estimate by setting $s=1-(\epsilon_{-}-\alpha)$:
          \begin{align}
             & \vert \int_{0}^{\pi} (g\mathcal{K}(\partial_{t}\xi)-\partial_{\theta}b^{3,1})\phi\cdot \mathcal{N}\vert \notag\\
              &\lesssim \vert \vert \partial_{t}\xi\vert \vert_{L^{2}}\vert \vert \phi\vert\vert+(\vert \vert \partial_{\theta}\partial_{t}\xi\vert \vert_{H^{1-\frac{s}{2}}}+\vert \vert b^{3}\vert \vert_{H^{1-\frac{s}{2}}} )\vert \vert \phi\cdot \mathcal{N}\vert \vert_{H^{\frac{s}{2}}}\notag\\
            &\lesssim (\vert \vert \partial_{t}\xi \vert \vert_{H^{\frac{3}{2}+\frac{\epsilon_{-}-\alpha}{2}}}+\vert \vert \partial_{t}\xi\vert \vert_{H^{\frac{3}{2}+\frac{\epsilon_{-}-\alpha}{2}}}\vert \vert \xi\vert \vert_{H^{\frac{3}{2}+\frac{\epsilon_{-}-\alpha}{2}}})\vert \vert \phi\vert \vert_{H^{\frac{1}{2}}(\Sigma)}\notag\\
              &\lesssim (1+\sqrt{\mathcal{E}})\vert \vert \phi\vert \vert_{H^{1}(\Omega)}\vert \vert \partial_{t}\xi\vert \vert_{H^{\frac{3}{2}+\frac{\epsilon_{-}-\alpha}{2}}}
              \lesssim \mathcal{E}
          \end{align}

          Combining all of the estimate above we obtain the following result we want

          \begin{align}
              \|\p_{t}p\|_{L^{2}}^{2}\lesssim \mathcal{E}_{||}+\mathcal{E}^{\frac{3}{2}}
            \end{align}
       \end{proof}

       We  now establish apriori estimate by the following theorem.

      \subsection{The Proof of Theorem \ref{thm:base_0}}

       \begin{proof}
            From the energy estimate \eqref{equ:5.1.239} we have:

            \begin{align}
                 \mathcal{E}_{||}(t)-\mathcal{E}^{\frac{3}{2}}(t)+\int_{0}^{t}\mathcal{D}_{||}\lesssim \mathcal{E}_{||}(0)+\mathcal{E}^{\frac{3}{2}}(0)+\int_{0}^{t}\mathcal{E}^{\frac{1}{2}}\mathcal{D} \label{equ:5.1.301}
            \end{align}

            \noindent Applying Theorem 7.8 and Theorem 7.9 to equation \eqref{equ:5.1.301}, we obtain

            \begin{align}
                 \mathcal{E}_{||}(t)-\mathcal{E}^{\frac{3}{2}}(t)+\int_{0}^{t}(\mathcal{D}_{||}+\Sigma_{k=0}^{2}\vert \vert \partial_{t}^{k}\xi\vert \vert_{H^{\frac{3}{2}-\alpha}})\notag\\
                 \lesssim \mathcal{E}_{||}(0)+\mathcal{E}^{\frac{3}{2}}(0)+\int_{0}^{t}\mathcal{E}^{\frac{1}{2}}\mathcal{D} \label{equ:5.1.302}
            \end{align}

            \textbf{Step 2} Use the elliptic estimate of the zero-order system and the first-order system, we have

            \begin{align}
               & \vert \vert \partial_{t}u\vert \vert_{W^{2,q_{-}}}+\vert \vert \partial_{t}p\vert \vert_{W^{1,q_{-}}}+\vert \vert \partial_{t}\xi\vert \vert_{W^{3-\frac{1}{q_{-}},q_{-}}}\notag\\
               & \lesssim \vert \vert \partial_{t}^{2}u\vert \vert_{L^{q_{-}}}+\vert \vert \partial_{t}^{2}\xi\vert \vert_{H^{\frac{3}{2}-\alpha}}+\vert \vert \partial_{t}\xi\vert \vert_{W^{1-\frac{1}{q_{-}},q_{-}}}+\sqrt{\mathcal{E}}\sqrt{\mathcal{D}}\notag\\
                &\lesssim \vert \vert \partial_{t}^{2}u\vert \vert_{H^{1}}+\vert\vert \partial_{t}^{2}\xi\vert \vert_{H^{\frac{3}{2}-\alpha}}+\vert \vert \partial_{t}\xi\vert \vert_{H^{1}}+\sqrt{\mathcal{E}}\sqrt{\mathcal{D}}\notag\\
                &\lesssim \vert\vert \partial_{t}^{2}\xi\vert \vert_{H^{\frac{3}{2}-\alpha}}+\sqrt{\mathcal{D}_{||}}+\sqrt{\mathcal{E}}\sqrt{\mathcal{D}} \label{equ:5.1.303},
            \end{align}

            \noindent and

            \begin{align}
                &\vert \vert u\vert \vert_{W^{2,q_{+}}}+\vert \vert p\vert\vert_{W^{1,q_{+}}}+\vert \vert \xi\vert \vert_{W^{3-\frac{1}{q_{+}},q_{+}}}\lesssim \vert \vert \partial_{t}u\vert \vert_{L^{q_{+}}}+\vert \vert \partial_{t}\xi\vert \vert_{H^{\frac{3}{2}-\alpha}}+\sqrt{\mathcal{E}}\sqrt{\mathcal{D}}\notag\\
                &\lesssim \vert \vert \xi\vert \vert_{H^{1}}+\vert \vert \partial_{t}u\vert \vert_{L^{2}}+\vert \vert \partial_{t}\xi\vert \vert_{H^{\frac{3}{2}-\alpha}}+\sqrt{\mathcal{E}}\sqrt{\mathcal{D}}\lesssim \sqrt{\mathcal{D}_{||}}+\vert \vert \partial_{t}\xi\vert \vert_{H^{\frac{3}{2}-\alpha}}+\sqrt{\mathcal{E}}\sqrt{\mathcal{D}}\label{equ:5.1.304}.
            \end{align}

            \noindent Applying equations \eqref{equ:5.1.303} and \eqref{equ:5.1.304} to equation \eqref{equ:5.1.302}, we have

            \begin{align}
                 \mathcal{E}_{||}(t)-\mathcal{E}^{\frac{3}{2}}(t)+\int_{0}^{t}(\mathcal{D}_{||}+\Sigma_{k=0}^{2}\vert \vert \partial_{t}^{k}\xi\vert \vert_{H^{\frac{3}{2}-\alpha}})+\int_{0}^{t}( \vert \vert u\vert \vert^{2}_{W^{2,q_{+}}}+\vert \vert p\vert\vert^{2}_{W^{1,q_{+}}}+\vert \vert \xi\vert \vert^{2}_{W^{3-\frac{1}{q_{+}},q_{+}}})\notag\\
                 +\int_{0}^{t}(\vert \vert \partial_{t}u\vert \vert^{2}_{W^{2,q_{-}}}+\vert \vert \partial_{t}p\vert \vert^{2}_{W^{1,q_{-}}}+\vert \vert \partial_{t}\xi\vert \vert^{2}_{W^{3-\frac{1}{q_{-}},q_{-}}})\lesssim \mathcal{E}_{||}(0)+\mathcal{E}^{\frac{3}{2}}(0)+\int_{0}^{t}\mathcal{E}^{\frac{1}{2}}\mathcal{D} \label{equ:5.1.305}
            \end{align}

            \noindent Then we bound the term $[\partial_{t}^{k+1}\xi]_{\theta}$ and $[\partial_{t}^{k}\partial_{\theta}\xi]_{\theta}$. For $k=0,1$, we use standard trace theorem to show the boundedness. It suffices to prove the case when $k=2$. Using the kinematic boundary condition we have:

            \begin{align}
                [\partial_{t}^{3}\xi]^{2}_{\theta}&\lesssim [\p_{t}^{2}u\cdot \mathcal{N}]_{\theta}^{2}+\mathfrak{n}'''(t)^{2}+ \vert \partial_{tt}(\mathfrak{n}'(t)(\frac{\rho'}{\rho}-\frac{\rho_{0}'}{\rho_{0}}))\vert^{2}+[\p_{t}u\cdot \p_{t}\mathcal{N}]_{\theta}^{2}+[u\cdot \p_{t}^{2}\mathcal{N}]^{2}_{\theta}\notag\\
                &\lesssim [\p_{t}^{2}u\cdot \mathcal{N}]^{2}_{\theta}+ \Sigma_{i=0}^{2}\vert\vert \p_{t}^{k}u\vert\vert_{H^{1}}+\vert \vert u\vert \vert^{2}_{H^{1}}([\partial_{t}^{2}\partial_{\theta}\xi]_{\theta}+[\partial_{t}^{2}\eta]_{\theta})^{2}\notag\\
                &+[\p_{t}^{2}\p_{\theta}\xi]_{\theta}^{2}\|u\|^{2}_{W^{2,q_{+}}}+\|\p_{t}u\|^{2}_{H^{1+\frac{\epsilon_{-}}{2}}}\|\p_{t}\xi\|^{2}_{H^{\frac{3}{2}+\frac{\epsilon_{-}-\alpha}{2}}}+\mathcal{E}^{2}\notag\\
               &\lesssim \mathcal{D}_{||}+\vert \vert u\vert \vert^{2}_{W^{2,q_{+}}}([\partial_{t}^{2}\partial_{\theta}\xi]_{\theta})^{2}+\mathcal{E}\mathcal{D} \label{equ:5.1.306}
            \end{align}

            \noindent The contact point condition yields

           \begin{align}
               \mp\sigma \frac{\rho_{0}^{2}\partial_{\theta}(\partial_{t}^{2}\xi)}{(\rho_{0}^{2}+\rho_{0}'^{2})^{\frac{3}{2}}}(\frac{\pi}{2}\pm \frac{\pi}{2})\pm\sigma \frac{\rho_{0}'\rho_{0}(\partial_{t}^{2}\xi)}{(\rho_{0}^{2}+\rho_{0}'^2)^{\frac{3}{2}}}=\kappa(\partial_{t}^{3}\xi\pm \mathfrak{n}'''(t))-b^{7,2}(\pm l)\label{equ:5.1.400}
           \end{align}

           \noindent from which we obtain the following result

           \begin{align}
              [\partial_{t}^{2}\partial_{\theta}\xi]_{\theta}\lesssim [\partial_{t}^{3}\xi]_{\theta}+\vert \vert \p_{t}^{2}u\vert \vert_{H^{1}}+\vert \vert \partial_{t}^{2}\xi\vert \vert_{H^{\frac{3}{2}-\alpha}}+\sqrt{\mathcal{E}\mathcal{D}} \label{equ:5.1.401}
        \end{align}

           \noindent Using equation\eqref{equ:5.1.401} in equation \eqref{equ:5.1.306} and the smallness of the energy $\mathcal{E}(u,p,\xi)$, we have

           \begin{align}
                [\partial_{t}^{3}\xi]^{2}_{\theta}\lesssim \mathcal{D}_{||}+\mathcal{E}\mathcal{D} \label{equ:5.1.402}
           \end{align}

           \noindent Using equation \eqref{equ:5.1.402} in \eqref{equ:5.1.401}, we obtain the following estimate

           \begin{align}
              [\p_{t}^{2}\p_{\theta}\xi]^{2}_{\theta}\lesssim D_{||}+\vert \vert \partial_{t}^{2}\xi\vert \vert_{H^{\frac{3}{2}-\alpha}}+\sqrt{\mathcal{E}\mathcal{D}} \label{equ:5.1.403}
           \end{align}
           
             \noindent The kinematic boundary condition implies that

            \begin{align}
                \vert \vert \partial_{t}^{3}\xi\vert \vert^{2}_{H^{\frac{1}{2}-\alpha }}&\lesssim \vert \vert \p_{t}^{2}u\cdot \mathcal{N}\vert \vert_{H^{\frac{1}{2}}(\Sigma)}^{2}+\mathfrak{n}'''(t)^{2}+\vert \vert\partial_{t}^{2}\xi \vert \vert^{2}_{H^{\frac{3}{2}-\alpha}}+\mathcal{E}\mathcal{D}\notag\\
                &\lesssim\|u\|^{2}_{H^{1}}+\vert \vert\partial_{t}^{2}\xi \vert \vert^{2}_{H^{\frac{3}{2}-\alpha}}+\mathcal{E}\mathcal{D}\lesssim \mathcal{D}_{||}+\|\p_{t}^{2}\xi\|_{H^{\frac{3}{2}-\alpha}}^{2}+\mathcal{E}\mathcal{D}\label{equ:5.1.307}
            \end{align}

            \noindent Combining \eqref{equ:5.1.402},\eqref{equ:5.1.403}, \eqref{equ:5.1.307} and theorem \ref{thm:gam}, we obtain

            \begin{align}
                 \mathcal{E}_{||}(t)-\mathcal{E}^{\frac{3}{2}}(t)+\int_{0}^{t}(\mathcal{D})\lesssim \mathcal{E}_{||}(0)+\mathcal{E}^{\frac{3}{2}}(0)+\int_{0}^{t}\mathcal{E}^{\frac{1}{2}}\mathcal{D}\label{equ:5.1.308}
            \end{align}

            \textbf{Step 3} In this step, we enhance the energy term. We have:

            \begin{align}
                \vert \vert \partial_{t}\xi(t)\vert \vert_{H^{\frac{3}{2}+\frac{\epsilon_{-}-\alpha}{2}}}^{2}&\lesssim \vert \vert \partial_{t}\xi(s)\vert \vert_{H^{\frac{3}{2}+\frac{\epsilon_{-}-\alpha}{2}}}^{2}+\int_{0}^{t}\vert \vert \partial_{t}\xi\vert \vert_{H^{\frac{3}{2}+\epsilon_{-}}}^{2}+\vert \vert \partial_{t}^{2}\xi\vert \vert_{H^{\frac{3}{2}-\alpha}}\notag\\
                &\lesssim \vert \vert\partial_{t}\xi(s) \vert \vert_{H^{\frac{3}{2}+\frac{\epsilon_{-}-\alpha}{2}}}+\int_{0}^{t}\mathcal{D}\label{equ:5.1.309}
            \end{align}

            \noindent and:

            \begin{align}
                \vert \vert \partial_{t}u(t)\vert \vert^{2}_{H^{1+\frac{\epsilon_{-}}{2}}}\lesssim \vert \vert \partial_{t}u(s)\vert \vert^{2}_{H^{1+\frac{\epsilon_{-}}{2}}}+\int_{0}^{t}\vert \vert \partial_{t}u\vert\vert^{2}_{H^{1+\epsilon_{-}}}+\vert \vert \partial_{t}^{2}u\vert \vert_{H^{1}}^{2}\lesssim \vert \vert \partial_{t}u(s)\vert \vert_{H^{\frac{1}{2}+\frac{\epsilon_{-}}{2}}}+\int_{0}^{t}\mathcal{D}\label{equ:5.1.310}
            \end{align}

            \noindent Therefore,

            \begin{align}
                \tilde{\mathcal{E}}(t)-(\mathcal{E}(t))^{\frac{3}{2}}+\int_{0}^{t}\mathcal{D}\lesssim \tilde{\mathcal{E}}(0)+(\mathcal{E}(0))^{\frac{3}{2}}\label{equ:5.1.311},
            \end{align}

            \noindent where 

            \begin{align}
                \tilde{\mathcal{E}}:=\mathcal{E}_{||}+\vert \vert \partial_{t}u\vert \vert_{H^{1+\frac{\epsilon_{-}}{2}}}+\vert \vert \partial_{t}\xi\vert \vert^{2}_{H^{\frac{3}{2}+\frac{\epsilon_{-}-\alpha}{2}}} \label{equ:5.1.312}
            \end{align}

             The elliptic estimate implies

            \begin{align}
               \vert \vert u\vert \vert_{W^{2,q_{+}}}+\vert \vert p\vert \vert_{W^{1,q_{+}}}+\vert \vert \xi\vert \vert_{W^{3-\frac{1}{q_{+}},q_{+}}}\lesssim \sqrt{\tilde{\mathcal{E}}}+\mathcal{E}\label{equ:5.1.313}
            \end{align}

            \noindent and

            \begin{align}
                \vert \vert \partial_{t}p\vert\vert_{L^{2}}\lesssim \vert \vert \partial_{t}^{2}u\vert \vert_{H^{1}}+\vert \vert \partial_{t}u\vert \vert_{H^{1}}+\vert \vert \partial_{t}\xi\vert \vert_{H^{\frac{3}{2}+\frac{\epsilon_{-}-\alpha}{2}}}+\mathcal{E}+(\mathcal{E})^{\frac{3}{2}}\lesssim \sqrt{\tilde{\mathcal{E}}}+\mathcal{E}+\mathcal{E}^{\frac{3}{2}} \label{equ:5.1.314}.
            \end{align}

            \noindent Then using \eqref{equ:5.1.313} and \eqref{equ:5.1.314} in equation \eqref{equ:5.1.312}m we have

            \begin{align}
                \mathcal{E}(t)-(\mathcal{E}(t))^{\frac{3}{2}}+\int_{0}^{t}\mathcal{D}\lesssim \mathcal{E}(0)+(\mathcal{E}(0))^{\frac{3}{2}}
            \end{align}

            \noindent which implies the following relation given that energy $\mathcal{E}$ is small:

            \begin{align}
                 \mathcal{E}(t)+\int_{0}^{t}\mathcal{D}\lesssim \mathcal{E}(0) \label{equ:5.1.315}
            \end{align}

            Note that $\mathcal{E}\lesssim\mathcal{D}$. We obtain the following result from equation \eqref{equ:5.1.315}

            \begin{align}
                \mathcal{E}(t)\lesssim e^{-\lambda t}\mathcal{E}(0)
            \end{align}

            \noindent Hence, we have the theorem proved.
       \end{proof}    

    \begin{center}
        {\large A}PPENDIX
    \end{center}

    \textbf{The representation of non-linear terms}
    
    The zero order non-linear terms are expressed as follows

      \begin{align}{\label{equ:5.1.90}}
      \begin{aligned}
      b^{1}=&\mathfrak{n}'(t)\partial_{x}u+(\cos\theta W\partial_{t}\bar{\xi},\sin\theta W\partial_{t}\bar{\xi})\mathcal{A}(\partial_{x}u,\partial_{y}u)^{T}-u\cdot \nabla_{\mathcal{A}}u\\
      b^{2}=&0\\
      b^{3}=&\sigma \mathcal{R}_1(\rho_{0},\partial_{\theta}\xi,\xi)\\
      b^{4}=&\sigma \mathcal{R}_{2}(\rho_{0},\partial_{\theta}\xi,\xi)\\
      b^{5}=&0\\
      b^{7}(\frac{\pi}{2}&\pm\frac{\pi}{2})=\kappa \hat{\mathcal{W}}(\mp\mathfrak{n}'(t)+\partial_{t}\xi)\\
      b^{6}=&(\frac{1}{\rho}-\frac{1}{\rho_{0}})u\cdot \mathcal{N}+\mathfrak{n}'(t)\sin\theta (\frac{\rho'}{\rho}-\frac{\rho_{0}'}{\rho_{0}})
      \end{aligned}
      \end{align}

      \noindent where $\mathcal{R}_1$ and $\mathcal{R}_2$ are two-second order reminder of Taylor expansion defined in Section 5.2. They both follow some good boundedness conditions. 

      The first order non-linear terms are expressed as follows

      \begin{align}{\label{equ:5.1.91}}
      \begin{aligned}
          b^{1,1}&=\partial_{t}b^{1}-\mu\operatorname{div}_{\partial_{t}\mathcal{A}}S_{\mathcal{A}}(p,u)-\mu\operatorname{div}_{\mathcal{A}}\mathbb{D}_{\partial_{t}\mathcal{A}}u\\
          b^{2,1}&=\operatorname{div}_{\partial_{t}\mathcal{A}}u\\
          b^{3,1}&=\sigma \partial_{t}\mathcal{R}_1\\b^{4,1}&=\sigma\partial_{t}\mathcal{R}_{2}+\mu\mathbb{D}_{\partial_{t}\mathcal{A}}u\mathcal{N}+(\mathcal{K}(\xi)-\partial_{\theta}b^{3})\partial_{t}\mathcal{N}-S_{\mathcal{A}}(p,u)\partial_{t}\mathcal{N} \\
          b^{5,1}&=\mu\mathbb{D}_{\partial_{t}\mathcal{A}}u\nu\cdot \tau\\
          b^{7,1}&(\frac{\pi}{2}\pm \frac{\pi}{2})=\kappa\hat{\mathcal{W}}'(\partial_{t}\xi\mp\mathfrak{n}'(t))(\partial_{tt}\xi\mp \mathfrak{n}''(t))\\
          b^{6,1}&=\partial_{t}b^{6}+\frac{1}{\rho_{0}}u\cdot \partial_{t}\mathcal{N}
          \end{aligned}
      \end{align}

      The second order non-linear terms are expressed as follows
      \begin{align}{\label{equ:5.1.92}}
          b^{1,2}&=\partial_{t}b^{1,1}-\operatorname{div}_{\partial_{t}\mathcal{\mathcal{A}}}S_{\mathcal{A}}(\partial_{t}p,\partial_{t}u)-\operatorname{div}_{\mathcal{A}}S_{\partial_{t}\mathcal{A}}(\partial_{t}p,\partial_{t}u)\\
          b^{2,2}&=\partial_{t}b^{2,1}-\operatorname{div}_{\partial_{t}\mathcal{A}}\partial_{t}u\\
          b^{3,2}&=\sigma \partial_{t}^{2}\mathcal{R}_1 \\
          b^{4,2}&=\partial_{t}b^{4,1}+(\mathcal{K}(\partial_{t}\xi)+\partial_{t}b^{3})\partial_{t}\mathcal{N}-S_{\mathcal{A}}(\partial_{t}p,\partial_{t}u)\partial_{t}\mathcal{N}+\mu\mathbb{D}_{\partial_{t}\mathcal{A}}\partial_{t}u\mathcal{N}\\
          b^{5,2}&=\partial_{t}b^{5,1}+\mu\mathbb{D}_{\partial_{t}\mathcal{A}}\partial_{t}u\nu\cdot \tau\\
          b^{7,2}&(\frac{\pi}{2}\pm\frac{\pi}{2})=\kappa \hat{\mathcal{W}}'(\p_{t}\xi\mp\mathfrak{n}'(t))(\partial_{t}^{3}\xi\pm \mathfrak{n}'''(t))+\kappa\hat{\mathcal{W}}''(\p_{t}\xi\mp\mathfrak{n}'(t))(\partial_{t}^{2}\xi\pm\mathfrak{n}''(t))^{2}\\
          b^{6,2}&=\partial_{t}b^{6,1}+\frac{1}{\rho_{0}}\partial_{t} u\cdot \partial_{t}\mathcal{N}
      \end{align}
      
      \textbf{The expressions for forcing terms in elliptic estimate}

   \begin{align}
       G^{1}&=-\partial_{t}u-u\cdot \nabla_{\mathcal{A}}u-\mathfrak{n}'(t)\partial_{x}u+(\cos\theta W\partial_{t}\bar{\xi},\sin\theta W\partial_{t}\bar{\xi})\tilde{K}(\partial_{x}u,\partial_{y}u)^{T} \\
       G^{2}&=0\\
       G_{+}^{3}&=\frac{\rho\partial_{t}\xi}{\vert\mathcal{N}_{0}\vert}+\mathfrak{n}'(t)\xi_s\rho\frac{1}{\vert \mathcal{N}_{0}\vert}+\mathfrak{n}'(t)\frac{(\rho'-\frac{\rho\rho_{0}'}{\rho_{0}})\cos\theta}{\vert \mathcal{N}_{0}\vert}\\
       G_{+}^{4}&=0\\
       G_{-}^{3}&=0\\
       G_{-}^{4}&=0\\
       G^{5}&=\mathcal{R}_{2}+\p_{\theta}\mathcal{R}_{1}\\
       G_{\pm}^{7}&=\alpha(\gamma'(t)\mp\partial_{t}\xi)\pm \mathcal{R}_1+\kappa\hat{W}(\gamma'(t)+\partial_{t}\xi)
    \end{align}

    \begin{align}
        G^{1,1}&=\partial_{t}G^{1}-\operatorname{div}_{\partial_{t}\mathcal{A}}S_{\mathcal{A}}(p,u)-\operatorname{div}_{\mathcal{A}}S_{\partial_{t}\mathcal{A}}(p,u)\notag\\
        G^{2,1}&=J\operatorname{div}_{\partial_{t}\mathcal{A}}u\notag\\
        G_{+}^{3,1}&=-\frac{\p_{t}u\cdot \partial_{t}\mathcal{N}}{\vert \mathcal{N}_{0}\vert}+\partial_{t}G_{+}^{3} \notag\\
        G_{-}^{3,1}&=0\notag\\
        G_{+}^{4,1}&=\mu\mathbb{D}_{\partial_{t}\mathcal{A}}u\mathcal{N}\cdot \frac{\mathcal{T}}{\vert \mathcal{T}\vert^{2}}+[\mathcal{K}(\eta)-S_{\mathcal{A}}(p,u)]\partial_{t}\mathcal{N}\cdot \frac{\mathcal{T}}{\vert \mathcal{T}\vert^{2}}+\mathcal{R}_{2}\partial_{t}\mathcal{N}\cdot \frac{\mathcal{T}}{\vert \mathcal{T}\vert^{2}}+\p_{\theta}\mathcal{R}_{1}\p_{t}\mathcal{N}\cdot \frac{\mathcal{T}}{|\mathcal{T}|^{2}}\notag\\
        G_{-}^{4,1}&=-\mathcal{D}_{\partial_{t}\mathcal{A}}u\nu\cdot\tau \notag\\
         G^{5,1}&=\mu\mathbb{D}_{\partial_{t}\mathcal{A}}u\mathcal{N}\cdot \frac{\mathcal{N}}{\vert \mathcal{N}\vert^{2}}+[\mathcal{K}(\eta)-S_{\mathcal{A}}(p,u)]\partial_{t}\mathcal{N}\cdot \frac{\mathcal{N}}{\vert \mathcal{N}\vert^{2}}+\partial_{t}\mathcal{R}_{2}+\mathcal{R}_{2}\partial_{t}\mathcal{N}\cdot \frac{\mathcal{N}}{\vert \mathcal{N}\vert^{2}}+\p_{\theta}\mathcal{R}_{1}\p_{t}\mathcal{N}\cdot \frac{\mathcal{T}}{|\mathcal{T}|^{2}}+\p_{t}\p_{\theta}\mathcal{R}_{1}\notag\\
         G^{7,1}&=\partial_{t}G^{7}
    \end{align}

    \begin{center}
        {\large A}CKNOWLEDGEMENTS
    \end{center}

    The author thanks his advisor Yan Guo for numerous comments. His mentorship and constructive feedback contribute significantly to the development of this work.

    This work is supported in part by NSF Grant DMS-2405051. 
    
        \bibliographystyle{plain}
\bibliography{sample}
\end{document}